\documentclass[10pt,a4paper]{amsart}
\usepackage{amssymb}
\usepackage{amsmath}
\usepackage{xspace}
\usepackage[latin1]{inputenc}
\usepackage[all,2cell]{xy}
\UseAllTwocells
\newcommand{\ct}[1]{\ensuremath{\mathbf{#1}}}              
\newcommand{\at}[1]%
            {\ensuremath{\protect\underline{\mathbf{#1}}}} 
\newcommand{\ec}[1]{\ensuremath{\mathcal{#1}}}             

\newcommand{\cl}[1]{\ensuremath{\overline{#1}}}            
\newcommand{\op}[1]{\ensuremath{\operatorname{#1}}}        
\newcommand{\tp}[1]{\langle#1\rangle}                      
\newcommand{\farg}{\cdot}                                  

\newcommand{\mawh}[1]{\widehat{#1}}                        
\newcommand{\mawt}[1]{\widetilde{#1}}

\newcommand{\maol}[1]{\overline{#1}}

\newcommand{\h}[1][]                                       
 {\ifthenelse{\boolean{mmode}}%
  {$\mathrm{h}$}%
  {h\nobreakdash#1\hspace{0pt}}}

\DeclareMathOperator{\Mor}{Mor}    
\newcommand{\id}{\mathrm{id}}        
\DeclareMathOperator{\Hom}{Hom}    
\newcommand{\comp}{\circ}          
\newcommand{\dcomp}{\diamond}      
\newcommand{\hcomp}{\ast}          
\newcommand{\adcomp}%
  {\overset{\operatorname{ad}}{\comp}} 
\newcommand{\funcomp}%
  {\overset{\operatorname{fn}}{\comp}}
\newcommand{\tcomp}{\mathbin{\Tilde{\circ}}}  
\newcommand{\thcomp}{\mathbin{\Tilde{\ast}}} 

\DeclareMathOperator{\Id}{Id}      
\DeclareMathOperator{\Nat}{Nat}    
\DeclareMathOperator{\opp}{op}     

\newcommand{\sccat}
{\mathbin{\kern-1pt\raisebox{6pt}{.}\kern-5pt
\downarrow\kern-5pt\raisebox{6pt}{.}\kern-1pt}}

\newcommand{\parrow}[1]
   {\underset{{\displaystyle \raisebox{5pt}%
   {$\longleftarrow$}}}{\op{#1}}{\,}}
\newcommand{\iarrow}[1]
   {\underset{{\displaystyle \raisebox{5pt}%
   {$\longrightarrow$}}}{\op{#1}}{\,}}

\newcommand{\ladj}{\!\dashv\!}         
\newcommand{\uadj}{\top}           

\DeclareMathOperator{\Sig}{Sig}    

\DeclareMathOperator{\Sub}{Sub}    
\DeclareMathOperator{\inc}{in}     
\newcommand{\lmapsto}{\longmapsto} 
\newcommand{\rest}%
{\mathnormal{\restriction}}           

\newcommand{\incl}{\subseteq}      
\newcommand{\tprod}{\textstyle{\prod}}     
\newcommand{\tcoprod}{\textstyle{\coprod}} 
\newcommand{\ttcoprod}{\amalg}    
\newcommand{\bprod}{\times}        
\newcommand{\inter}{\bigcap}       
\newcommand{\union}{\bigcup}       
\newcommand{\bunion}{\cup}         
\newcommand{\iso}{\cong}           
\newcommand{\function}[4]{
            \begin{array}{@{\:}c@{\:}c@{\:}l}
                   #1 &\mor& #2 \\
                   #3 &\longmapsto& #4
            \end{array} }
\newcommand{\nfunction}[4]
    {\left\{
     \function{#1}{#2}{#3}{#4}
     \right. }

\DeclareMathOperator{\Fix}{Fix}    
\newcommand{\fmon}[1]
 {\ensuremath{#1^{\star}}}
\newcommand{\ffmon}[1]
 {\ensuremath{#1^{\star\star}}}
\newcommand{\fffmon}[1]
 {\ensuremath{#1^{\star\star\star}}}
\newcommand{\bb}[1]{\ensuremath
 {\lvert #1 \rvert}}

\DeclareMathOperator{\Sg}{Sg}      
\DeclareMathOperator{\E}{E}        
\DeclareMathOperator{\Ker}{Ker}    
\newcommand{\pr}{\mathrm{pr}}        
\DeclareMathOperator{\Op}{Op}      
\newcommand{\vs}[1]{\mathbin{\downarrow}#1}

\DeclareMathOperator{\Alg}{Alg}    

\DeclareMathOperator{\p}{p}        
\DeclareMathOperator{\bconcat}
            {\curlywedge}
\newcommand{\cncat}{\curlywedge}
\newcommand{\concat}
  {\ensuremath{\text
  {\Large $\curlywedge$}}}
\newcommand{\ext}[1]
  {\ensuremath{#1^{\sharp}}}

\newcommand{\nat}{\natural}        

\DeclareMathOperator{\Ter}{Ter}    
\DeclareMathOperator{\Eq}{Eq}      

\DeclareMathOperator{\Mod}{Mod}    
\DeclareMathOperator{\Cn}{Cn}      

\newcommand{\Hall}{\ensuremath{\mathrm{H}}} 
\newcommand{\Ben}{\ensuremath{\mathrm{B}}}  

\DeclareMathOperator{\Sen}{Sen}    

\newcommand{\vacio}{\ensuremath{\varnothing}}

\newcommand{\ol}{\overline}

\newcommand{\brel}{\ensuremath{\xymatrix{{}\ar@{{*}{-}{*}}[r] & {}}}}

\newcommand{\nseq}[3]{\xymatrix@1@C=16pt{#1 \ar@{>}[r]_-{\scriptscriptstyle{#2}} & #3 }}

\NoCompileMatrices              
\def\labelstyle{\textstyle}     
\def\twocellstyle{\textstyle}   
\xymatrixrowsep = {8ex}         
\xymatrixcolsep = {10ex}        

\newcommand{\Smmallmatrix}{
  \xymatrixcolsep={7ex}
  \xymatrixrowsep={5.6ex}}

\newdir{<= }{:a(+25)\dir{-}*:a(-25)\dir{-}*!/:a(85)+1.8pt/:a(-25)\dir{-}}
\newdir{ = }{*@{=}}
\newdir{+>}{@{|}*@{-}!/-6.5pt/@{>}*!/-13pt/{}}
\newdir{ +}{{}*!/-5pt/@{+}}

\newdir{ >}{{}*!/-10pt/@{>}}

\newdir{.-}
{{}*:a(-90)h!/2.1pt/{.}}
\newdir{.->}%
{:a(-90)h!/2.5pt/{.}*!/4pt/\dir{-}*!/0pt/\dir{-}*%
!/-2pt/\dir{-}*!/-8pt/\dir{>}}
\newdir{-.->}%
{:a(-90)h!/2pt/{.}*!/8pt/\dir{-}*!/4pt/\dir{-}*!/0pt/\dir{-}*%
!/-5pt/\dir{-}*!/-11pt/\dir{>}}

\newdir{~>}{!/4.6pt/\dir{~}*!/1pt/\dir{-}*!/-5pt/\dir{>}}
\newdir{~~>}{!/7pt/\dir{~}*!/0pt/\dir{~}*!/-5pt/\dir{>}}

\newdir{=~>}{*!/1pt/@2{~}*!/6pt/{}*!/-6pt/@2{>}}
\newdir{<~~>}{!/21pt/\dir{<}*!/21pt/\dir{-}!/12pt/\dir{~}!/5pt/\dir{~}*!/1pt/\dir{-}*!/-3pt/\dir{>}}

  \newdir{=>}{!/5pt/\dir{=}!/2.5pt/\dir{=}*!/-5pt/\dir2{>}*!/7pt/\dir{ }}
 \newdir{==>}{!/-2pt/\dir{=}!/-6pt/\dir{=}%
 !/-10pt/\dir{=}*!/-18pt/\dir2{>}}
 \newdir{<==>}{!\dir2{<}!/-2pt/\dir{=}!/-6pt/\dir{=}%
   !/-10pt/\dir{=}*!/-17pt/\dir2{>}}

\newdir{< }{{}*!/12pt/\dir{<}}

\newdir{-o}{!/1.3pt/{}*!/0pt/{}@{o}*!/-5pt/{}}

\newsavebox{\xymor}  
\newsavebox{\xymon}  
\newsavebox{\xyepi}  
\newsavebox{\xytn}   
\newsavebox{\xyrel}  
\newsavebox{\xycel}  
\newsavebox{\xymdf}  
\newsavebox{\xyumor} 
\newsavebox{\xydmor} 
\newsavebox{\xyomor} 
\newsavebox{\xyemor} 

\newcommand{\xynode}{\makebox[0ex]{}}

\savebox{\xymor}{\ensuremath{%
\xymatrix@1@C=19pt{\xynode \ar@{>}[r] & \xynode }}}
\savebox{\xymon}{\ensuremath{%
\xymatrix@1@C=19pt{\xynode \ar@{{ +}{-}{>}}[r] & \xynode }}}
\savebox{\xyepi}{\ensuremath{%
\xymatrix@1@C=19pt{\xynode \ar@{{}{-}{+>}}[r] & \xynode }}}
\savebox{\xytn}{\ensuremath{%
\xymatrix@1@C=19pt{\xynode \ar[r]|(.44){\object@{.-}} & \xynode }}}
\savebox{\xyrel}{\ensuremath{%
\xymatrix@1@C=19pt{\xynode \ar@{{}{-}{-o}}[r] & \xynode }}}
\savebox{\xycel}{\ensuremath{%
\xymatrix@1@C=19pt{\xynode \ar@{=>}[r] & \xynode }}}
\savebox{\xymdf}{\ensuremath{%
\xymatrix@1@C=16pt{\xynode \ar@{}[r]|{\dir{~>}} & \xynode}}}
\savebox{\xyumor}{\ensuremath{%
\xymatrix@1@C=19pt{\xynode \ar@{{}{-}^{>}}[r] & \xynode }}}
\savebox{\xydmor}{\ensuremath{%
\xymatrix@1@C=19pt{\xynode \ar@{{}{-}_{>}}[r] & \xynode }}}
\savebox{\xyomor}{\ensuremath{%
\xymatrix@1@C=19pt{\xynode \ar@{{}{-}^{< }}[r] & \xynode }}}
\savebox{\xyemor}{\ensuremath{%
\xymatrix@1@C=19pt{\xynode \ar@{{ >}{-}{>}}[r] & \xynode }}}

\newcommand{\mor}{\usebox{\xymor}}    
\newcommand{\cel}{\usebox{\xycel}}    
\newcommand{\mdf}{\usebox{\xymdf}}    

\newcommand{\xyn}[1]{\save="#1"\restore}
\newcommand{\xymn}[3][0ex,5ex]{%
\save[]+<#1>*+{#2}="#3" \restore }

\newcommand{\functor}[9]{
 \xymatrix{
    #4 \save[]+<0ex,5ex>*+{#1}="1"  \restore
      \ar[d]_{#6}  \ar@{}[rd]|{\longmapsto}
  & #5 \save[]+<0ex,5ex>*+{#3}="3"  \restore
      \ar[d]^{#7}
  \\
   #8 & #9 \ar "1";"3"^-{#2} } }
\newcommand{\functornd}[9]{
 \xymatrix{
    #4 \save[]+<0ex,5ex>*+{#1}="1"  \restore
      \ar[d]_{#6}  \ar@{}[rd]|{\longmapsto}
  & #5 \save[]+<0ex,5ex>*+{#3}="3"  \restore
  \\
   #8 & #9 \ar[u]_{#7} \ar "1";"3"^-{#2} } }
\newcommand{\functordn}[9]{
 \xymatrix{
    #4 \save[]+<0ex,5ex>*+{#1}="1"  \restore
       \ar@{}[rd]|{\longmapsto}
  & #5 \save[]+<0ex,5ex>*+{#3}="3"  \restore
      \ar[d]^{#7}
  \\
   #8  \ar[u]^{#6}  & #9 \ar "1";"3"^-{#2} } }

\newcommand{\larr}{->}
\newcommand{\rarr}{->}

\newcommand{\xfunctor}[9]{
 \xymatrix{
    #4 \save[]+<0ex,5ex>*+{#1}="1"  \restore
      \ifthenelse{\equal{\larr}{->}}{\ar[d]_{#6}}{}
      \ifthenelse{\equal{\larr}{<-}}{\ar[d];[]^{#6}}{}
      \ifthenelse{\equal{\larr}{-<}}{\ar@{< }[d]_{#6}}{}
      \ar@{}[rd]|{\longmapsto}
  & #5 \save[]+<0ex,5ex>*+{#3}="3"  \restore
      \ifthenelse{\equal{\rarr}{->}}{\ar[d]^{#7}}{}
      \ifthenelse{\equal{\rarr}{<-}}{\ar[d];[]_{#7}}{}
      \ifthenelse{\equal{\rarr}{-<}}{\ar@{< }[d]^{#7}}{}
  \\
   #8 & #9 \ar "1";"3"^-{#2} } }

\theoremstyle{plain}

\newtheorem{corollary}{Corollary}
\newtheorem{lemma}{Lemma}
\newtheorem{proposition}{Proposition}

\theoremstyle{definition}
\newtheorem{definition}{Definition}
\newtheorem*{example}{Example}
\newtheorem*{remark}{Remark}

\theoremstyle{remark}

\numberwithin{equation}{section}

\newenvironment{narrow}[2]{%
  \begin{list}{}{%
  \setlength{\leftmargin}{#1}%
  \setlength{\rightmargin}{#2}%
  \setlength{\listparindent}{\parindent}%
  \setlength{\itemindent}{\parindent}%
  \setlength{\parsep}{\parskip}}%
 \item[]}{\end{list}}

\begin{document}

\title[Morphisms and transformations of Fujiwara]{On the morphisms and
transformations of Tsuyoshi Fujiwara (\Small{as a concretion of a
bidimensional many-sorted general algebra and its application to
the equivalence between many-sorted clones and algebraic
theories})
%
%
}


\author[J. Climent]{J. Climent Vidal}
\address{Universidad de Valencia\\
         Departamento de Lógica y Filosofía de la Ciencia\\
         E-46010 Valencia, Spain}
\email{Juan.B.Climent@uv.es}

\author[J. Soliveres]{J. Soliveres Tur}
\address{Universidad de Valencia\\
         Departamento de Lógica y Filosofía de la Ciencia\\
         E-46010 Valencia, Spain}
\email{Juan.Soliveres@uv.es}


\keywords{Clone, many-sorted algebraic theory, many-sorted set,
          many-sorted algebra, construction of
          Ehresmann-Grothendieck, Kleisli category for a monad,
          many-sorted term, many-sorted
          closure space, adjoint square, algebraic morphism,
          algebraic transformation, Hall
          algebra, Bénabou algebra, polyderivor,
          transformation of polyderivors, $2$-institution.}
\date{\today}

\begin{abstract}
For, not necessarily similar, single-sorted algebras Fujiwara defined,
through the concept of family of basic mapping-formulas between
single-sorted signatures, a notion of morphism which generalizes the
ordinary notion of homomorphism between algebras; and an equivalence
relation, the conjugation, on the families of basic mapping-formulas,
which corresponds to the relation of inner isomorphism for algebras.
In this paper we extend the theory of Fujiwara about morphisms to the,
not necessarily similar, many-sorted algebras, by defining the concept
of polyderivor between many-sorted signatures, which assigns to basic sorts,
words and to formal operations, families of derived terms, and under
which the standard signature morphisms, the basic mapping-formulas of
Fujiwara, and the derivors of Goguen-Thatcher-Wagner are subsumed.
Then, by means of the homomorphisms between Bénabou algebras, which
are the algebraic counterpart of the finitary many-sorted algebraic
theories of Bénabou, we define the composition of polyderivors from which
we get a corresponding category, and prove that it is isomorphic to
the category of Kleisli for a monad on the standard category of
many-sorted signatures.  Next, by defining the notion of
transformation between polyderivors, which generalizes the relation of
conjugation of Fujiwara, we endow the category of many-sorted
signatures and polyderivors with a structure of $2$-category.  From this
we get a derived $2$-category of many-sorted specifications in which
we prove, syntactically, the equivalence of the many-sorted
specifications of Hall and Bénabou, and, from a suitable
pseudo-functor from it into the $2$-category of categories, we deduce
the equivalence of the categories of Hall and Bénabou algebras.
Besides, by defining for each many-sorted signature its corresponding
category of generalized many-sorted terms, we prove that the
realization of these terms in the many-sorted algebras is invariant
under polyderivors and compatible with the transformations between
polyderivors, and from this we get an example, among others, of the new
concept of $2$-institution, itself an strict generalization of that
of institution by Goguen and Burstall.

\end{abstract}

\maketitle



\section{Introduction.}
The closed sets of operations, or clones, on an arbitrary set $A$,
i.e., the sets of operations on $A$ closed under the generalized
operations of composition and containing the projection mappings, were
initially defined and investigated by P. Hall, as pointed out by Cohn
in~\cite{pC81}, pp.  127 and 132 (who attended the lectures by
Professor P. Hall from 1944 to 1951), to show that the crucial
mathematical properties of a $\Sigma$-algebra $\mathbf{A} =
(A,(F_{\sigma})_{\sigma\in \Sigma})$ do not depend on the family of
primitive operations $(F_{\sigma})_{\sigma\in \Sigma}$ on $A$ defined
by the single-sorted signature $\Sigma$, but on the system of all
operations on $A$ obtainable from $(F_{\sigma})_{\sigma\in \Sigma}$ by
means of the operations of composition.

The concept of an ordinary clone, axiomatized by P. Hall as a
single-sorted partial algebra subject to satisfy some laws
(see~\cite{pC81}, p.  132) and, independently but subsequently, by M.
Lazard as a compositor (see~\cite{mL55}, p.  327), was generalized up
to that of a many-sorted clone by Goguen and Meseguer in~\cite{gm85},
and axiomatically defined by them (in~\cite{gm85}, pp.  318--319) as
any many-sorted algebra (of the appropriate signature) that satisfies
a definite system of many-sorted equational laws, concretely, the
so-called \emph{Projection Axiom}, \emph{Identity Axiom},
\emph{Associativity Axiom}, and \emph{Invariance of Constant Functions
Axiom}.  Given its origin in P. Hall, we agree to refer to the
many-sorted algebras that are models of the just named axioms as Hall
algebras.

Hall algebras, as reflected by the defining axioms, are a species of
algebraic construct in which the essential properties of the
fundamental procedures of \emph{substitution}, for the many-sorted
terms in the free many-sorted algebras, and of \emph{composition}, for
the many-sorted-operations on sorted sets are embodied.  And this is
precisely one of the reasons why Hall algebras are a powerful and
fundamental instrument to investigate many-sorted algebras.  To this
we add that Hall algebras are not only worth of study because of its
source in the above mentioned procedures.  Besides that, Hall algebras
are interesting in themselves since they furnish important examples of
equationally defined many-sorted algebras, and also because they have
been used by Goguen and Meseguer in~\cite{gm85} to prove the
Completeness Theorem of finitary many-sorted equational logic (that
generalizes the classical Completeness Theorem of finitary equational
logic of Birkhoff), providing in this way, a full algebraization of
many-sorted equational deduction.


Another approximation to the study of many-sorted algebras has been
proposed by Bénabou in~\cite{jB68}, by making use of the finitary
many-sorted algebraic theories (categories with objects the words on a
set of sorts $S$ such that, for every word $w = (w_{i})_{i\in n}$,
there exists a family of morphisms $(p^{w}_{i})_{i\in n}$, where, for
$i\in n$, $p^{w}_{i}$ is a morphism from $w$ to $(w_{i})$, the word of
length one associated to the letter $w_{i}$, such that
$(w,(p^{w}_{i})_{i\in n})$ is a product of the family $((w_{i}))_{i\in
n}$), that are the generalization to the many-sorted case of the
finitary single-sorted algebraic theories of Lawvere, see~\cite{fL63}.
The equational presentation of the finitary many-sorted algebraic
theories of Bénabou gives rise to what we have called Bénabou
algebras.  And the Bénabou algebras, even having a many-sorted
specification different from that of the Hall algebras, are also
models of the essential properties of the clones for the many-sorted
operations.

For an arbitrary, but fixed, set of sorts $S$, the many-sorted
specifications $\mathrm{H}_{S}$, for Hall algebras, and
$\mathrm{B}_{S}$, for Bénabou algebras, are \emph{not isomorphic} in
the category $\mathbf{Spf}$, of many-sorted specifications and
many-sorted specification morphisms, because between the corresponding
categories of models: $\mathbf{Alg}(\mathrm{H}_{S})$, of Hall
algebras, and $\mathbf{Alg}(\mathrm{B}_{S})$, of Bénabou algebras,
there is not any isomorphism.  However, the many-sorted specifications
$\mathrm{H}_{S}$ and $\mathrm{B}_{S}$ can be considered, in some
definite way, as being \emph{equivalent}, as a consequence of the
proof, in the fourth section about Hall and Bénabou algebras, of the
categorical equivalence between the categories
$\mathbf{Alg}(\mathrm{H}_{S})$ and $\mathbf{Alg}(\mathrm{B}_{S})$.

But, the semantical equivalence of the many-sorted specifications
$\mathrm{H}_{S}$ and $\mathrm{B}_{S}$, or, for that matter, of any two
many-sorted specifications, understood, by convention, as meaning the
categorical equivalence of the canonically associated categories of
models, can not be properly reflected at the purely syntactical level
of the many-sorted specifications and many-sorted specification
morphisms, i.e., can not be mathematically defined in the category
$\mathbf{Spf}$.  And this is so, essentially, as a consequence of the
fact of not having actually endowed $\mathbf{Spf}$ with a (non
trivial) structure of $2$-category.  Thus, if one remains anchored in
the tradition of viewing $\mathbf{Spf}$ as being, simply, a category,
then the only reasonable way of classifying many-sorted specifications
from within the category $\mathbf{Spf}$ is through the
category-theoretical concept of isomorphism, and not, by structural
impossibility, by means of some other notion of equivalence between
many-sorted specifications, itself being strictly weaker than that of
isomorphism (as it would be the case if instead of having a category,
we had a $2$-category).

Therefore, what is really needed to settle the problem of the
equivalence between many-sorted specifications (i.e., the problem of
determining whether or not two many-sorted specifications determine
equivalent categories) is to dispose of some way of comparing
many-sorted specifications that goes, strictly, beyond the mere
isomorphisms, in the same way as equivalences go beyond the
isomorphisms when comparing categories among them.  We suggest in this
paper that an adequate way of providing a solution to the just
mentioned problem is by constructing suitable $2$-categories of
many-sorted signatures and many-sorted specifications, through the
appropriate definitions of the $2$-cells between the $1$-cells.  This
bidimensionality, by supplying one additional degree of freedom,
generates a richer world, that opens the possibility to deal not only
with isomorphic but also with adjoint and equivalent many-sorted
specifications.  Thus carrying further the previous development which
was incomplete because of its restriction to categories.
The methodology we have followed in order to find a solution of the
equivalence problem will now be considered.

It consists in generalizing the theory of Fujiwara in~\cite{tF59} and
\cite{tF60} into several directions.  Firstly, by defining the concept
of \emph{morphism of Fujiwara}, from now on abbreviated to
\emph{polyderivor}, from a many-sorted signature into another, which
assigns to basic sorts, words and to formal operations, families of
derived terms, and this in such a way that under the concept of
polyderivor falls the concept of derivor, defined in~\cite{gtw76}, and
that of morphism between many-sorted algebraic theories.  Secondly, by
endowing with a structure of $2$-category the category of many-sorted
signatures and polyderivors, by defining the appropriate transformations
between the polyderivors, that generalize the equivalences defined by
Fujiwara in~\cite{tF60}, and allow richer comparisons between
many-sorted signatures than the usually considered.  Lastly, by
introducing the corresponding $2$-categories of many-sorted
specifications, polyderivors between many-sorted specifications and
transformations from such a polyderivor into a like one.

After having developed the generalized theory we prove that the
transformations between polyderivors determine natural transformations
between the functors, on the categories of many-sorted algebras and
many-sorted terms, associated to the polyderivors.  Besides, we prove
that the realization of the many-sorted terms in the many-sorted
algebras is invariant under the polyderivors and compatible with the
transformations between the polyderivors, from which we get an example
of $2$-institution, that generalizes the usual notion of institution
as defined in~\cite{gb86}.

By using the machinery introduced we prove, as an example, the
equivalence between the many-sorted specifications of Hall and
Bénabou.  And from this we get, as an immediate consequence of the
existence of a certain pseudo-functor from the $2$\nobreakdash-category
$\mathbf{Spf}_{\mathfrak{pd}}$, of many-sorted specifications, to the
$2$-category $\mathbf{Cat}$, the equivalence between the categories of
Hall and Bénabou algebras.  This, we believe, helps to understand,
from a purely category-theoretical standpoint, how some equivalences
between categories, e.g., that between clones (represented by Hall
algebras) and finitary many-sorted algebraic theories (represented by
Bénabou algebras), arise from more primitive syntactical equivalences
between some many-sorted specifications associated to them.

Note that the suggested solution to the equivalence problem between
many-sorted specifications bears some resemblance in \emph{intention}
to e.g., the classification of the homology theories (functors from a
topological category to an algebraic category satisfying some axioms)
accomplished by Eilenberg and Steenrod in~\cite{es52}, an impossible
task without the notion of natural equivalence (invertible $2$-cell)
of functors ($1$-cells).


The paper falls naturally into two parts.  The first part is divided
into three sections: \textsc{Introduction}, common to both parts,
\textsc{The many-sorted term institution}, and \textsc{Many-sorted
specifications and morphisms}; and the second is divided into four
sections: \textsc{Hall and Bénabou algebras}, \textsc{Morphisms of
Fujiwara}, \textsc{Transformations of Fujiwara}, and
\textsc{Equivalence of the speci\-fications of Hall and Bénabou}.  Next
we proceed to describe the contents of the just enumerated sections,
leaving out, obviously, the first one.

The main goal of the \emph{second section} is to construct the
many-sorted term institution.  To attain such a goal we begin by
defining $\mathbf{MSet}$, the category of many-sorted sets and
many-sorted mappings, $\mathbf{Sig}$, the category of standard
many-sorted signatures, and $\mathbf{Alg}$, the category of standard
many-sorted algebras, through the construction of
Ehresmann-Grothendieck applied, respectively, to the contravariant
functors $\mathrm{MSet}$, $\mathrm{Sig}$, and $\mathrm{Alg}$.
Then we remark that $\mathbf{MSet}$ and $\mathbf{Sig}$ are split
bifibrations on $\mathbf{Set}$ and prove that $\mathbf{MSet}$,
$\mathbf{Sig}$, and $\mathbf{Alg}$ are bicomplete, that $\mathbf{Alg}$
is concrete, univocally transportable through a \lq\lq forgetful\rq\rq
functor $\mathrm{G}$ into the fibered product
$\mathbf{MSet}\times_{\mathbf{Set}}\mathbf{Sig}$, and that the functor
$\mathrm{G}$ has a left adjoint
$\mathbf{T}\colon\mathbf{MSet}\times_{\mathbf{Set}}
\mathbf{Sig}\mor\mathbf{Alg}$ which transforms objects of
$\mathbf{MSet}\times_{\mathbf{Set}}\mathbf{Sig}$ into labelled term
algebras in $\mathbf{Alg}$ and morphisms of
$\mathbf{MSet}\times_{\mathbf{Set}}\mathbf{Sig}$ into translators
between the associated labelled term algebras in $\mathbf{Alg}$.

On the basis of the functor $\mathbf{T}$ we define, for every
many-sorted signature $\mathbf{\Sigma}$, the category
$\mathbf{Ter}(\mathbf{\Sigma})$, of generalized terms for
$\mathbf{\Sigma}$, as the dual of the Kleisli category for
$\mathbb{T}_{\mathbf{\Sigma}}$ (the standard monad derived from the
adjunction between the category $\mathbf{Alg}(\mathbf{\Sigma})$ and
the category $\mathbf{Set}^{S}$), and we extend this procedure up to a
pseudo-functor $\Ter$ from $\mathbf{Sig}$ to $\mathbf{Cat}$ which
formalizes the procedure of translation for many-sorted terms.  Then,
to account exactly for the invariant character of the process of
realization of the many-sorted terms in the many-sorted algebras,
under change of many-sorted signature, we show that there exists a
pseudo-extranatural transformation from a pseudo-functor obtained from
$\mathrm{Alg}$ and $\mathrm{Ter}$ to the constant functor
$\mathrm{K}_{\mathbf{Set}}$, both defined on
$\mathbf{Sig}^{\mathrm{op}}\times \mathbf{Sig}$ and taking values in
the $2$-category $\mathbf{Cat}$.  Finally, after generalizing the
concept of institution by means, essentially, of the notion of
pseudo-extranatural transformation from a pseudo-functor to a constant
functor, we get $\mathfrak{Tm}$, the many-sorted term institution
on $\mathbf{Set}$.

In the \emph{third section} we begin by defining, for a many-sorted
signature $\mathbf{\Sigma}$, the concept of
$\mathbf{\Sigma}$-equation, but for the generalized terms in
$\mathbf{Ter}(\mathbf{\Sigma})$, the relation of satisfaction between
many-sorted algebras and $\mathbf{\Sigma}$-equations, the consequence
operator $\mathrm{Cn}_{\mathbf{\Sigma}}$, and by translating, for a
morphism between many-sorted signatures, equations for the source
many-sorted signature into equations for the target many-sorted
signature.  Then we continue with the proof of the satisfaction
condition and, after defining a convenient pseudo-functor from
$\mathbf{Sig}$ to $\mathbf{Cat}_{\boldsymbol{\mathcal{V}}}$, for an
adequate Grothendieck universe $\boldsymbol{\mathcal{V}}$, we get the
equational institution on $\mathbf{2}$.

Next, in order to explain category-theoretically the concept of
equational deduction, we begin by defining, by means of the
construction of Ehresmann-Grothendieck applied to a suitable
contravariant functor from $\mathbf{Set}$ to $\mathbf{Cat}$, the
category $\mathbf{MClSp}$, of many-sorted closure spaces.  Then,
through the concept of adjoint square, we define, for the Grothendieck
universe $\boldsymbol{\mathcal{V}}$, the $2$-category
$\mathbf{Mnd}_{\boldsymbol{\mathcal{V}},\mathrm{alg}}$ of monads,
algebraic morphisms between monads, and transformations between
algebraic morphisms, into which the category $\mathbf{MClSp}$ is
naturally embedded.  After this we prove the existence of a
pseudo-functor $\mathrm{Cn}$ from $\mathbf{Sig}$ to
$\mathbf{Mnd}_{\boldsymbol{\mathcal{V}},\mathrm{alg}}$ that has as
components, essentially, the consequence operators
$\mathrm{Cn}_{\mathbf{\Sigma}}$ for the different signatures
$\mathbf{\Sigma}$, and make it, after generalizing the concept of
entailment system, part of the, so-called, equational consequence
entailment system.

Following this, after defining the category $\mathbf{Spf}$, of
many-sorted specifications and many-sorted specification morphisms, we
prove the existence of a contravariant functor,
$\mathrm{Alg}^{\mathrm{sp}}$ and of a pseudo-functor,
$\mathrm{Ter}^{\mathrm{sp}}$, from $\mathbf{Spf}$ to $\mathbf{Cat}$,
that extend $\mathrm{Alg}$ and $\mathrm{Ter}$, respectively.  Then we
state that from $\mathbf{Spf}^{\opp}\bprod \mathbf{Spf}$ to the
$2$-category $\mathbf{Cat}$ there exists a pseudo-functor, obtained
from $\Alg^{\mathrm{sp}}$ and $\mathrm{Ter}^{\mathrm{sp}}$, and a
pseudo-extranatural transformation from it to the functor constantly
$\mathbf{Set}$, and from this we get $\mathfrak{Spf}$, the
many-sorted specification institution of Fujiwara on $\mathbf{Set}$,
and an institution morphism from $\mathfrak{Spf}$ to $\mathfrak{Tm}$.

In the \emph{fourth section} we show that the categories of Hall
and Bénabou algebras, that are models of the essential properties of
the clones of many-sorted operations, are equivalent and also that
there exists a biunivocal correspondence between the Bénabou algebras
and the finitary many-sorted Bénabou theories, that are a
generalization of the single-sorted algebraic theories of Lawvere.
Furthermore, these algebras are interesting because they will be used,
among other things, to define, later on, the composition of the
polyderivors from a many-sorted signature into another, that are
generalizations of the usual morphisms between many-sorted signatures,
and also to exemplify an equivalence between the specifications of
Hall and Bénabou algebras in a $2$-category of specifications,
polyderivors, and transformations.

In the \emph{fifth section} after defining the morphisms of Fujiwara
between many-sorted signatures (that generalize the standard morphisms
and the derivors between many-sorted signatu\-res, as well as the
families of basic mapping-formulas defined by Fujiwara in \cite{tF59}
for the single-sorted case), and the composition of these morphisms,
we get the category $\mathbf{Sig}_{\mathfrak{pd}}$, of many-sorted
signatures and morphisms of Fujiwara, and prove that it can be
obtained, essentially, as the Kleisli category for a monad in
$\mathbf{Sig}$.  Then we define a pseudo-functor (contravariant in the
morphisms) $\Alg_{\mathfrak{pd}}\colon \mathbf{Sig}_{\mathfrak{pd}}\mor
\mathbf{Cat}$ and, by applying the construction of
Ehresmann-Grothendieck, we get a new category
$\mathbf{Alg}_{\mathfrak{pd}}$ of many-sorted algebras and morphisms
between many-sorted algebras which have, as a component, the morphisms
of Fujiwara.  Following this we define another pseudo-functor
(covariant in the morphisms) $\mathrm{Ter}_{\mathfrak{pd}}$ from
$\mathbf{Sig}_{\mathfrak{pd}}$ to $\mathbf{Cat}$ which formalizes the
procedure of translation for many-sorted terms, but now for the
morphisms of Fujiwara.  Then, to account exactly for the invariant
character of the process of realization of the many-sorted terms in
the many-sorted algebras, under change of many-sorted signature
through the morphisms of Fujiwara, we show that there exists a
pseudo-extranatural transformation from a pseudo-functor obtained from
$\mathrm{Alg}_{\mathfrak{pd}}$ and $\mathrm{Ter}_{\mathfrak{pd}}$ to the
constant functor $\mathrm{K}_{\mathbf{Set}}$, both defined on
$\mathbf{Sig}_{\mathfrak{pd}}^{\mathrm{op}}\times
\mathbf{Sig}_{\mathfrak{pd}}$ and taking values in the $2$-category
$\mathbf{Cat}$, and from this we get $\mathfrak{Tm}_{\mathfrak{pd}}$,
the many-sorted term institution of Fujiwara.

In the \emph{sixth section} we endow the category
$\mathbf{Sig_{\mathfrak{pd}}}$ with a structure of $2$-category through
the concept of transformation between morphisms of Fujiwara, which
generalizes that one of equivalence between families of basic
mapping-formulas, defined by Fujiwara in \cite{tF60} for the
single-sorted case.  Then we prove that the transformations between
morphisms of Fujiwara determine natural transformations between the
functors associated to the morphisms of Fujiwara.  From this we extend
the pseudo-functors $\Alg_{\mathfrak{pd}}$ and
$\mathrm{Ter}_{\mathfrak{pd}}$ up to the $2$-category
$\mathbf{Sig_{\mathfrak{pd}}}$, and we get, in particular, by applying
a construction of Ehresmann-Grothendieck to $\Alg_{\mathfrak{pd}}$, a
corresponding $2$-category $\mathbf{Alg}_{\mathfrak{pd}}$.  Next, after
proving that the transformations between morphisms of Fujiwara are
compatible with the realization of the many-sorted terms in the
many-sorted algebras, we show that there exists a pseudo-extranatural
transformation from a pseudo-functor obtained from
$\mathrm{Alg}_{\mathfrak{pd}}$ and $\mathrm{Ter}_{\mathfrak{pd}}$ to the
constant functor $\mathrm{K}_{\mathbf{Set}}$, both defined on the
$2$-category $\mathbf{Sig}^{\mathrm{op}}\times \mathbf{Sig}$ and
taking values in the $2$-category $\mathbf{Cat}$, and from this we get
$\mathfrak{Tm}_{\mathfrak{pd}}$, the many-sorted term $2$-institution
of Fujiwara.

Besides, we prove that the morphisms and transformations of Fujiwara
can be taken as a concretion of a bidimensional many-sorted general
algebra, on the basis of the existence of an embedding from the
$2$-category $\mathbf{Sig}_{\mathfrak{pd}}$ into the $2$-category
$\mathbf{Mnd}_{\boldsymbol{\mathcal{V}},\mathrm{alg}}$ which sends a
signature $\mathbf{\Sigma}$ to the monad
$(\mathbf{Set}^{S},\mathbb{T}_{\mathbf{\Sigma}})$, a polyderivor
$\mathbf{d}$ from $\mathbf{\Sigma}$ to $\mathbf{\Lambda}$ to the
associate algebraic morphism $\mathbb{T}_{\mathbf{d}}$ from
$(\mathbf{Set}^{S},\mathbb{T}_{\mathbf{\Sigma}})$ to
$(\mathbf{Set}^{T},\mathbb{T}_{\mathbf{\Lambda}})$, and a
transformation $\xi$ from $\mathbf{d}$ to $\mathbf{e}$ to the
corresponding algebraic transformation $\mathbb{T}_{\xi}$ from
$\mathbb{T}_{\mathbf{d}}$ to $\mathbb{T}_{\mathbf{e}}$.

In the \emph{seventh section} we define a $2$-category of
specifications, $\mathbf{Spf}_{\mathfrak{pd}}$, with objects the
specifications, $1$-cells from a specification into a like one the
polyderivors between the underlying signatures of the specifications that
are compatible with the equations, and $2$-cells from a $1$-cell into
a like one a convenient class of transforma\-tions between the
polyderivors.  Following this we prove that the contravariant
pseudo-functor $\Alg_{\mathfrak{pd}}$ and the pseudo-functor
$\mathrm{Ter}_{\mathfrak{pd}}$, both defined on the $2$-category
$\mathbf{Sig}_{\mathfrak{pd}}$, can be lifted up to the $2$-category
$\mathbf{Spf}_{\mathfrak{pd}}$ as $\Alg^{\mathrm{sp}}_{\mathfrak{pd}}$
and $\mathrm{Ter}^{\mathrm{sp}}_{\mathfrak{pd}}$, respectively.  Then
we state that from the $2$-category
$\mathbf{Spf}^{\opp}_{\mathfrak{pd}}\bprod \mathbf{Spf}_{\mathfrak{pd}}$
to the $2$-category $\mathbf{Cat}$ there exists a pseudo-functor,
obtained from $\Alg^{\mathrm{sp}}_{\mathfrak{pd}}$ and
$\mathrm{Ter}^{\mathrm{sp}}_{\mathfrak{pd}}$, and a pseudo-extranatural
transformation from it to the functor constantly $\mathbf{Set}$, and
from this we get $\mathfrak{Spf}_{\mathfrak{pd}}$, the many-sorted
specification $2$-institution of Fujiwara.

Finally, it is in the 2-category $\mathbf{Spf}_{\mathfrak{pd}}$ that we
prove, for every set of sorts $S$, the equivalence of the
specifications of Hall and Bénabou for $S$, from which, through the
pseudo-functor $\Alg^{\mathrm{sp}}_{\mathfrak{pd}}$, the equivalence
between the corresponding categories of algebras,
$\mathbf{Alg}(\mathrm{H}_{S})$ and $\mathbf{Alg}(\mathrm{B}_{S})$,
follows immediately.


Every set we consider, unless otherwise stated, will be a
$\boldsymbol{\mathcal{U}}$-small set or a
$\boldsymbol{\mathcal{U}}$-large set, i.e., an element or a subset,
respectively, of a Grothendieck(-Sonner-Tarski) universe
$\boldsymbol{\mathcal{U}}$ (as defined in~\cite{sM98}, p.  22, but
see also~\cite{dM69}, pp.  160--163, \cite{jS62}, pp.  166--167, and
\cite{aT38}, p.  84, for more details about this concept), fixed once
and for all.  Besides, we agree that $\mathbf{Set}$ denotes the
category which has as set of objects $\boldsymbol{\mathcal{U}}$ and as
set of morphisms the subset of $\boldsymbol{\mathcal{U}}$ of all
mappings between $\boldsymbol{\mathcal{U}}$-small sets, and, depending
on the context, that $\mathbf{Cat}$ denotes either, the category of
the $\boldsymbol{\mathcal{U}}$\nobreakdash-categories (i.e.,
categories $\mathbf{C}$ such that the set of objects of $\mathbf{C}$
is a subset of the Grothendieck universe $\boldsymbol{\mathcal{U}}$,
and the hom-sets of $\mathbf{C}$ elements of
$\boldsymbol{\mathcal{U}}$), and functors between
$\boldsymbol{\mathcal{U}}$\nobreakdash-categories, or the $2$-category
of the $\boldsymbol{\mathcal{U}}$-categories, functors between
$\boldsymbol{\mathcal{U}}$-categories, and natural transformations
between functors.

In all that follows we use standard concepts and constructions from
category theory, see e.g., \cite{fB94}, \cite{fB94a}, \cite{fB94b},
\cite{Ehr65}, \cite{jwG74}, \cite{Gro71}, \cite{kl65}, \cite{sM98},
and \cite{pP71}; classical universal algebra, see e.g., \cite{pC81},
\cite{gG79}, \cite{bJ72}, \cite{agK67}, and \cite{mmt87}; categorical
universal algebra, see e.g., \cite{jB68} and \cite{fL63}; many-sorted
algebra, see e.g., \cite{jB68}, \cite{bl70}, \cite{gm85},
\cite{jh30}(particularly Chapter 3), \cite{h63}, \cite{aim59},
\cite{m76} and \cite{mg85}; and set theory, see e.g., \cite{nB70},
\cite{hE77} and \cite{dM69}.  Nevertheless, we have generically
adopted the following notational and conceptual conventions:
\begin{enumerate}
\item As far as a category $\mathbf{C}$ is concerned we will
      write, for an object $x$ of $\mathbf{C}$, $x\in \mathbf{C}$
      instead of $x\in \mathrm{Ob}(\mathbf{C})$, however, for a
      morphism $f$ of $\mathbf{C}$, we will write
      $f\in\mathrm{Mor}(\mathbf{C})$ but not $f\in\mathbf{C}$.
      Furthermore, if there is no risk of confusion, in order to
      simplify the notation, we will write $1$ for the identity
      functor at $\mathbf{C}$, $F$ for the identity natural
      transformation at the functor $F$, and, under the same
      circumstance and reason, we will use the juxtaposition to denote
      both the horizontal and the vertical composition of natural
      transformations.

\item Relative to set theory we recall that between ordinals $<$
      is identified with $\in$; thus for a \emph{(von Neumann)
      ordinal} $\alpha$ we have that $\alpha = \{\,\beta\mid \beta\in
      \alpha\,\}$, and $\mathbb{N}$, the first transfinite ordinal, is
      the set of all \emph{natural numbers}.  For two sets $A$, $B$ we
      denote by $\mathrm{Hom}(A,B)$ the set of all \emph{mappings}
      $f\colon A\mor B$ from $A$ to $B$, i.e., the set of all ordered
      triples $f = (A,F,B)$ where $F$ is a \emph{function} from $A$ to
      $B$.  For a mapping $f\colon A\mor B$, a subset $X$ of $A$, and
      a subset $Y$ of $B$, we denote by $f^{-1}[Y]$ the \emph{inverse
      image of} $Y$ \emph{under} $f$, and by $f[X]$ the \emph{direct
      image of} $X$ \emph{under} $f$.  However, if $Y = \{y\}$ is a
      \emph{final} set, we will write $f^{-1}[y]$ instead of the more
      accurate $f^{-1}[\{y\}]$.  For sets $B$, $C$, a family of sets
      $(A_{i})_{i\in I}$, a family of mappings $(f_{i})_{i\in I}$ in
      $\prod_{i\in I}\mathrm{Hom}(B,A_{i})$, and a family of mappings
      $(g_{i})_{i\in I}$ in $\prod_{i\in I}\mathrm{Hom}(A_{i},C)$, we
      denote by $\left<f_{i}\right>_{i\in I}$, resp., by
      $\left[g_{i}\right]_{i\in I}$, the unique mapping from $B$ to
      $\prod_{i\in I}A_{i}$, resp., from $\coprod_{i\in I}A_{i}$ to
      $C$, such that, for every $i\in I$, $f_{i} =
      \mathrm{pr}_{i}\comp \left<f_{i}\right>_{i\in I}$, resp., $g_{i}
      = \left[g_{i}\right]_{i\in I}\comp \mathrm{in}_{i}$.

\item For a set $S$ we agree upon denoting by
      $\mathbf{T}_{\star}(S) = (\fmon{S},\curlywedge,\lambda)$ the
      \emph{free monoid on} $S$, where
      \begin{enumerate}
      \item $\fmon{S}$, the underlying set of
            $\mathbf{T}_{\star}(S)$, is $\bigcup_{n\in
            \mathbb{N}}S^{n}$, the set of all \emph{words on} $S$,
            with $S^{n}$ the set of all \emph{functions} from $n$ to
            $S$,

     \item  $\curlywedge$, the \emph{concatenation} of words on
            $S$, is the binary operation on $\fmon{S}$ which sends a
            pair of words $(w,v)$ on $S$ to the function $w\curlywedge
            v$ from $\bb{w}+\bb{v}$ to $S$, where $\bb{w}$ and
            $\bb{v}$ are the lengths of $w$ and $v$, respectively,
            defined as follows
            $$
            w\bconcat v
            \nfunction
            {\bb{w}+\bb{v}}{S}
            {i}{
            \begin{cases}
            w_{i}, & \text{if $0\leq i < \bb{w}$;}\\
            v_{i-\bb{w}}, & \text{if $\bb{w}\leq i < \bb{w}+\bb{v}$,}
            \end{cases}
            }
            $$ and

      \item $\lambda$ is the \emph{empty word}, i.e., the unique function
            from $0$ to $S$.
      \end{enumerate}
      For a mapping $\varphi\colon S\mor T$ we denote by
      $\varphi^{\star}$ the unique homomorphism from
      $\mathbf{T}_{\star}(S)$ to $\mathbf{T}_{\star}(T)$ such that
      $\varphi^{\star}\comp \between_{S} = \between_{T}\comp \varphi$,
      where $\between_{S}$, resp., $\between_{T}$, is the canonical
      embedding of $S$ into $\fmon{S}$, resp., of $T$ into $\fmon{T}$.
\end{enumerate}
More specific notational conventions will be included and explained in
the successive sections.

We point out that in, almost, all that follows we frequently draw
diagrams to provide a geometrical description of what is going on.  By
doing so we hopefully expect to aid the reader in his understanding of
the notions and constructions involved in each concrete situation, as
well as in his grasping of the displayed ideas (after all, diagrams
are intended precisely for that purpose).  Finally, we warn the reader
that dealing with entities depending on a great number of parameters
which, in addition, are allowed to vary simultaneously, has indeed a
high notational price (to say nothing of the fact that living and
working in the topos $\mathbf{MSet}$---which, in particular, is
non-Boolean, has as internal logic the trivalued logic of Heyting
(see, e.g., \cite{aH30}), and is of De Morgan---and associated
categories is much more demanding than it is in
$\mathbf{Set}^{S}$---which, among others properties, is a Boolean
topos---and associated categories, for an arbitrary, but \emph{fixed},
set of sorts $S$).  However, we have tried our best to do the notation
as uniform, simple, and clear as possible.

\section{The many-sorted term institution.}

Our main aim in this section is to show that the concept of \lq\lq
derived operation of an algebra\rq\rq, also known as \lq\lq term
operation of an algebra\rq\rq, elemental as it is, but fundamental for
universal algebra, can be naturally subsumed under the notion of
institution (see for this notion, e.g., \cite{gb86}), provided that an
institution is meant not to be an extranatural transformation (as
in~\cite{gb86}) but a pseudo-extranatural transformation (as defined
at the end of this section).

To attain the aim just mentioned we begin by a careful examination of
the different types of things that are involved around it, namely
many-sorted sets, signatures, algebras, terms, and generalized
institutions.  More specifically, in this section we define the
category $\mathbf{MSet}$ of many-sorted sets, in which the many-sorted
sets will be labelled with the sets of sorts, by applying the
construction of (at least) Ehresmann-Grothendieck (see~\cite{Ehr65},
pp.  89--91, where it is called \lq\lq produit croisé\rq\rq,
and~\cite{Gro71}, pp. (sub.)  175--177) to a contravariant functor from
$\mathbf{Set}$ to $\mathbf{Cat}$.  Following this we define the
categories $\mathbf{Sig}$, of many-sorted signatures, and
$\mathbf{Alg}$, of many-sorted algebras, by applying also the
construction of Ehresmann-Grothendieck to suitable contravariant
functors from $\mathbf{Set}$ and $\mathbf{Sig}$, respectively, to
$\mathbf{Cat}$.

Besides we prove the existence of a left adjoint $\mathbf{T}$ to a
\lq\lq forgetful\rq\rq functor $\mathrm{G}$ from $\mathbf{Alg}$ to
$\mathbf{MSet}\times_{\mathbf{Set}}\mathbf{Sig}$, and from this left
adjoint $\mathbf{T}$ we define a pseudo-functor $\Ter$ from
$\mathbf{Sig}$ to $\mathbf{Cat}$ which formalizes the procedure of
translation for many-sorted terms.

Finally, to account exactly for the invariant character of the
realization of many-sorted terms in many-sorted algebras under change
of many-sorted signature, we prove the existence of a
pseudo-extranatural transformation $(\mathrm{Tr},\theta)$ from a
pseudo-functor $\Alg(\farg)\bprod\mathrm{Ter}(\farg)$, from
$\mathbf{Sig}^{\mathrm{op}}\times \mathbf{Sig}$ to $\mathbf{Cat}$, to the functor
$\mathrm{K}_{\mathbf{Set}}$, between the same categories, that is
constantly $\mathbf{Set}$.  Then, after providing a generalization of the
ordinary concept of institution, we prove that $(\mathrm{Tr},\theta)$
is part of an institution on $\mathbf{Set}$, the so-called many-sorted
term institution.


Before stating the first proposition of this section, we agree upon
calling, from now on, for a set (of sorts) $S\in
\boldsymbol{\mathcal{U}}$, the objects of the category
$\mathbf{Set}^{S}$ (i.e., the functions $A = (A_{s})_{s\in S}$ from
$S$ to $\boldsymbol{\mathcal{U}}$) $S$-\emph{sorted
sets}; and the morphisms of the category $\mathbf{Set}^{S}$ from an
$S$-sorted set $A$ into a like one $B$ (i.e., the ordered triples
$(A,f,B)$, abbreviated to $f\colon A\mor B$, where $f$ is an element
of $\prod_{s\in S}\mathrm{Hom}(A_{s},B_{s})$, the cartesian product of
the family $(\mathrm{Hom}(A_{s},B_{s}))_{s\in S}$) $S$-\emph{sorted
mappings from} $A$ \emph{to} $B$.  Furthermore, we also agree that a
pseudo-functor $F$ from a \emph{category} $\mathbf{C}$ to a
$2$-\emph{category} $\mathbf{D}$ consists of the following data:
\begin{enumerate}
\item An object mapping
      $F\colon \mathrm{Ob}(\mathbf{C})\mor \mathrm{Ob}(\mathbf{D})$.

\item For every $x,y\in \mathbf{C}$, an hom-mapping
      $$
      F\colon \mathrm{Hom}_{\mathbf{C}}(x,y)\mor
      \mathrm{Hom}_{\mathbf{D}}(F(x),F(y)).
      $$

\item For every morphisms $f\colon x\mor y$ and $g\colon y\mor
      z$ in $\mathbf{C}$, an isomorphic $2$-cell $\gamma^{f,g}$ from
      $F(g)\comp F(f)$ to $F(g\comp f)$.

\item For every $x\in\mathbf{C}$, an isomorphic $2$-cell  $\nu^{x}$ from
      $\id_{F(x)}$ to $F(\mathrm{id}_{x})$.
\end{enumerate}
These data must satisfy the following coherence axioms:
\begin{enumerate}
\item For morphisms $f\colon x\mor y$, $g\colon y\mor z$, and $h\colon
      z\mor t$ in $\mathbf{C}$,
      $$
      \gamma^{g\comp f,h}\comp (\id_{F(h)}\ast\gamma^{f,g}) =
      \gamma^{f,h\comp g}\comp (\gamma^{g,h}\ast\id_{F(f)}).
      $$

\item For a morphism $f\colon x\mor y$ in $\mathbf{C}$,
      $$
      \id_{F(f)} = \gamma^{\id_{x},f}\comp (\id_{F(f)}\ast \nu^{x}) \quad
      \text{and} \quad
      \id_{F(f)} = \gamma^{f,\id_{y}}\comp (\nu^{y}\ast \id_{F(f)}).
      $$
\end{enumerate}

In the following proposition, that is basic for a great deal of what
follows, for a mapping $\varphi$ from $S$ to $T$, we prove the
existence of an adjunction $\coprod_{\varphi}\ladj\Delta_{\varphi}$
from the category of $S$-sorted sets to the category of $T$-sorted
sets, as well as the existence of a contravariant functor
$\mathrm{MSet}$ and of a pseudo-functor $\mathrm{MSet}^{\ttcoprod}$
(related, respectively, to the right and left components of the
adjunction) from $\mathbf{Set}$ to $\mathbf{Cat}$.

\begin{proposition}
Let $\varphi\colon S\mor T$ be a mapping.  Then the functors
$\Delta_{\varphi}$ from $\mathbf{Set}^{T}$ to $\mathbf{Set}^{S}$ and
$\tcoprod_{\varphi}$ from $\mathbf{Set}^{S}$ to $\mathbf{Set}^{T}$ defined,
respectively, as follows
\begin{enumerate}
\item $\Delta_{\varphi}$ assigns to a $T$-sorted set $A$
      the $S$-sorted set $A_{\varphi} = (A_{\varphi(s)})_{s\in
      S}$, i.e., $A\comp \varphi$, and to a $T$-sorted mapping
      $f\colon A\mor B$ the $S$-sorted mapping
      $$
      f_{\varphi} = (f_{\varphi(s)})_{s\in
      S}\colon A_{\varphi}\mor B_{\varphi};
      $$

\item $\coprod_{\varphi}$ assigns to an $S$-sorted set $A$ the
      $T$-sorted set $\coprod_{\varphi}A = (\coprod_{
      s\in\varphi^{-1}[t]}A_{s})_{t\in T}$ and to an $S$-sorted
      mapping $f\colon A\mor B$ the $T$-sorted mapping
      $$
      \textstyle{\coprod}_{\varphi}f =
      (\textstyle{\coprod}_{ s\in\varphi^{-1}[t]}f_{s})_{t\in T}\colon
      \textstyle{\coprod}_{\varphi}A\mor \textstyle{\coprod}_{\varphi}B,
      $$
\end{enumerate}
are such that, $\coprod_{\varphi}\ladj\Delta_{\varphi}$.  We agree
that $\theta^{\varphi}$, $\eta^{\varphi}$, and $\varepsilon^{\varphi}$
denote, respectively, the natural isomorphism, the unit and the counit
of the adjunction.

Besides, there exists a contravariant functor $\mathrm{MSet}$ from
$\mathbf{Set}$ to $\mathbf{Cat}$ which sends a set $S$ to the category
$\mathrm{MSet}(S) = \mathbf{Set}^{S}$, and a mapping $\varphi$ from
$S$ to $T$ to the functor $\Delta_{\varphi}$ from $\mathbf{Set}^{T}$
to $\mathbf{Set}^{S}$; and a pseudo-functor
$\mathrm{MSet}^{\ttcoprod}$ from $\mathbf{Set}$ to the
$2$\nobreakdash-category $\mathbf{Cat}$ given by the following data
\begin{enumerate}
\item The object mapping of $\mathrm{MSet}^{\ttcoprod}$ is that which
      sends a set $S$ to the category $\mathrm{MSet}^{\ttcoprod}(S) =
      \mathbf{Set}^{S}$.

\item The morphism mapping of $\mathrm{MSet}^{\ttcoprod}$ is
      that which sends a mapping $\varphi$ from $S$ to $T$ to the
      functor $\mathrm{MSet}^{\ttcoprod}(\varphi) =
      \textstyle{\coprod}_{\varphi}$ from $\mathbf{Set}^{S}$ to
      $\mathbf{Set}^{T}$.

\item For every $\varphi\colon S\mor T$ and $\psi\colon T\mor
      U$, the natural isomorphism $\gamma^{\varphi,\psi}$ from
      $\coprod_{\psi}\comp\coprod_{\varphi}$ to
      $\coprod_{\psi\comp\varphi}$ is that which is defined, for
      every $S$-sorted set $A$, as the $U$-sorted mapping that in the
      $u$-th coordinate is $((a,s),\varphi(s))\mapsto (a,s)$, if there
      exists an $s\in S$ such that $u=\psi(\varphi(s))$, and is the
      identity at $\vacio$, otherwise.

\item For every set $S$, the natural isomorphism $\nu^{S}$ from
      $\Id_{\mathbf{Set}^{S}}$ to $\coprod_{\id_{S}}$ is that which is
      defined, for every $S$-sorted set $A$ and $s\in S$, as the
      canonical isomorphism from $A_{s}$ to $A_{s}\bprod\{s\}$.
\end{enumerate}
\end{proposition}

\begin{proof}
We begin by proving that $\coprod_{\varphi}$ is a left adjoint to
$\Delta_{\varphi}$.  Let $A$ be an $S$-sorted set, then
the pair $(\eta_{A}^{\varphi},\coprod_{\varphi}A)$, where
$\eta_{A}^{\varphi}$ is the $S$-sorted mapping from $A$ to
$\Delta_{\varphi}(\coprod_{\varphi}A) = (\tcoprod_{x\in
\varphi^{-1}[\varphi(s)]} A_{x} )_{s\in S}$ whose $s$-th coordinate,
for $s\in S$, is the canonical embedding from $A_{s}$ to
$\coprod_{x\in \varphi^{-1}[\varphi(s)]}A_{x}$, is a universal
morphism from $A$ to $\Delta_{\varphi}$.  This is so because, for a
$T$\nobreakdash-sorted set $B$ and an $S$-sorted mapping $f\colon
A\mor B_{\varphi}$, the $T$-sorted mapping $f^{\S}=(f^{\S}_{t})_{t\in
T}$ from $\coprod_{\varphi}A$ to $B$, where, for $t\in T$,
$f^{\S}_{t}$ is the unique mapping $[f_{s} ]_{s\in\varphi^{-1}[t]}$
from $\tcoprod_{s\in \varphi^{-1}[t] }A_{s}$ to $B_{t} =
B_{\varphi(s)}$ such that, for every $s\in \varphi^{-1}[t]$, the
following diagram commutes
$$
\xymatrix{ A_{s}
\ar[r]^-{\inc_{s}}
\ar[rd]_{f_{s}} &
\tcoprod_{s\in \varphi^{-1}[t] }A_{s}
\ar[d]^{[f_{s}]_{s\in\varphi^{-1}[t]} = f^{\S}_{t}} \\
     &
B_{t} = B_{\varphi(s)} }
$$
is such that $f = \Delta_{\varphi}(f^{\S})\comp \eta_{A}^{\varphi}$ and
unique with such a property.

Otherwise stated, $\coprod_{\varphi}A$ is, simply,
$\mathrm{Lan}_{\varphi}A$, i.e., the left Kan extension of $A$ along
$\varphi$, recalling that every set is the set of objects of a
discrete category and every mapping between sets the object mapping of
a functor between discrete categories.

To prove that $\mathrm{MSet}^{\ttcoprod}$ is a pseudo-functor, it is
enough to verify the coherence axioms. But given the situation
$$
\xymatrix{
S \ar[r]^{\varphi} &
T \ar[r]^{\psi} &
U \ar[r]^{\xi} &
X,
}
$$
the following diagrams commute
$$
\xymatrix@C=60pt{
\coprod_{\xi}\comp\coprod_{\psi}\comp\coprod_{\varphi}
  \ar[r]^{\id_{\coprod_{\xi}}\hcomp\gamma^{\varphi,\psi}}
  \ar[d]_{\gamma^{\psi,\xi}\hcomp\id_{\coprod_{\varphi}}} &
\coprod_{\xi}\comp\coprod_{\psi\comp\varphi}
  \ar[d]^{\gamma^{\psi\comp\varphi,\xi}} \\
\coprod_{\xi\comp\psi}\comp\coprod_{\varphi}
  \ar[r]_-{\gamma^{\varphi,\xi\comp\psi}} &
\coprod_{\xi\comp\psi\comp\varphi}
}
$$

$$
\xymatrix{
\coprod_{\varphi}\comp\Id_{\mathbf{Set}^{S}}
  \ar[r]^{\id_{\coprod_{\varphi}}\hcomp\nu^{S}}
  \ar[d]_{\id_{\coprod_{\varphi}}} &
\coprod_{\varphi}\comp\coprod_{\id_{S}}
  \ar[d]^{\gamma^{\id_{S},\varphi}} \\
\coprod_{\varphi}
  \ar[r]_{\id_{\coprod_{\varphi}}} &
\coprod_{\varphi\comp\id_{S}}
}
\qquad\qquad
\xymatrix{
\Id_{\mathbf{Set}^{T}}\comp\coprod_{\varphi}
  \ar[r]^{\nu^{T}\hcomp\id_{\coprod_{\varphi}}}
  \ar[d]_{\id_{\coprod_{\varphi}}} &
\coprod_{\id_{T}}\comp\coprod_{\varphi}
  \ar[d]^{\gamma^{\varphi,\id_{T}}} \\
\coprod_{\varphi}
  \ar[r]_{\id_{\coprod_{\varphi}}} &
\coprod_{\id_{T}\comp\varphi}
}
$$

\end{proof}

From now on, when dealing with a pseudo-functor we will restrict
ourselves to define explicitly only its object and morphism mappings,
if about the remaining data and conditions involved in it, i.e., the
natural isomorphisms and the coherence conditions, there is not any
doubt.

\begin{remark}
The particular case of the just proved proposition when the sets of
sorts are $\fmon{S}\times S$ and $\fmon{S}\times \fmon{S}$, and the
mapping from $\fmon{S}\times S$ to $\fmon{S}\times \fmon{S}$ is
$$
1\times \between_{S}\nfunction
{\fmon{S}\times S}
{\fmon{S}\times \fmon{S}}
{(w,s)}
{(w,(s))}
$$
will be specially useful in the fourth section (on Hall and Bénabou
algebras) to prove the isomorphy between the free Bénabou algebra on
the $\fmon{S}\times \fmon{S}$-sorted set $\coprod_{1\times
\between_{S}}\Sigma$ associated to a given $\fmon{S}\times S$-sorted set
(below called an $S$-sorted signature) $\Sigma$ and the Bénabou
algebra of terms for $(S,\Sigma)$, as well as in the fifth section
(on the morphisms of Fujiwara) to provide an alternative, but
equivalent, definition of the concept of morphism of Fujiwara between
signatures.
\end{remark}

\begin{definition}
The category $\mathbf{MSet}$ of \emph{many-sorted sets} and
\emph{many-sorted mappings}, obtained by applying the construction of
Ehresmann-Grothendieck to the contravariant functor $\mathrm{MSet}$
from $\mathbf{Set}$ to $\mathbf{Cat}$, is $\mathbf{MSet} =
\int^{\mathbf{Set}}\mathrm{MSet}$.
\end{definition}

Therefore $\mathbf{MSet}$ has as objects the pairs $(S,A)$, where $S$
is a set and $A$ an $S$-sorted set, and as morphisms from $(S,A)$ to
$(T,B)$ the pairs $(\varphi,f)$, where $\varphi\colon S\mor T$ and
$f\colon A\mor B_{\varphi}$.  From now on, to shorten terminology, we
will say $\mathrm{ms}$-\emph{set} and $\mathrm{ms}$-\emph{mapping}
instead of \emph{many-sorted set} and \emph{many-sorted mapping},
respectively.

%

From the definition of the category $\mathbf{MSet}$ it follows,
immediately, that the projection functor $\pi_{\mathrm{MSet}}$ for
$\mathbf{MSet}$ is a split fibration (observe that, for every set $S$,
the fiber of $\pi_{\mathrm{MSet}}$ in $S$ is, essentially, the
category $\mathbf{Set}^{S}$ of $S$-sorted sets and mappings).

On the other hand, if we apply the construction of
Ehresmann-Grothendieck to the pseudo-functor
$\mathrm{MSet}^{\ttcoprod}$, then we get a category with the same
objects as $\mathbf{MSet}$, but with morphisms from $(S,A)$ to $(T,B)$
the pairs $(\varphi,f)$, where $\varphi\colon S\mor T$ and $f\colon
\coprod_{\varphi}A\mor B$.  However, for every morphism $\varphi\colon
S\mor T$, we have that $\coprod_{\varphi}\ladj\Delta_{\varphi}$, hence
$\Hom(\coprod_{\varphi}A,B)$ and $\Hom(A,B_{\varphi})$ are naturally
isomorphic, thus the categories $\int^{\mathbf{Set}}\mathrm{MSet}$ and
$\int_{\mathbf{Set}}\mathrm{MSet}^{\ttcoprod}$ are isomorphic (observe
the use, in the symbol of integration (also called the integral of
Grothendieck), of the subscript to indicate the covariant situation,
and of the superscript to indicate the contravariant one).  From this
it follows that the functor $\pi_{\mathrm{MSet}}$ is also a split
opfibration.  Therefore $\mathbf{MSet}$ is a split bifibration on
$\mathbf{Set}$.

\begin{remark}
The construction of Ehresmann-Grothendieck applied to the
contravariant functor $\mathrm{MSet}$ produces not only the category
$\mathbf{MSet}$, but also, implicitly, a logic, the internal logic of
$\mathbf{MSet}$ (which, we recall, is the trivalued logic of Heyting),
from the combination, by means of logical morphisms between the fibers
of $\pi_{\mathrm{MSet}}$, of the Boolean internal logics of the just
named fibers.  Informally speaking, we can say that globally the
category $\mathbf{MSet}$ has an intermediate logic, but that locally
(in its fibers) it is Boolean (in the same way as a manifold is a
space which locally looks like $\mathbb{R}^{n}$ (or $\mathbb{C}^{n}$)
but not necessarily globally).  Thus, in this case, we see that the
system of laws governing the world obtained by synthesizing a family
of given interwoven worlds, each of them governed by its proper system
of laws, is not necessarily identical to anyone of the local systems
of laws.
\end{remark}


To show the bicompleteness of the category $\mathbf{MSet}$ and of
some other categories defined a bit further on, in this same section, the
following definition and propositions, stated in~\cite{tbg91} and
(partially) in~\cite{jwG66}, are particularly useful, since they give
sufficient conditions that are, mostly, easily verifiable for the
cases we will be considering.

\begin{definition}(Cf., \cite{tbg91}, p. 249)
We say that a functor $F\colon \mathbf{C}^{\opp}\mor \mathbf{Cat}$ is
\emph{locally reversible} if, for every morphism $h\colon c\mor d$ in
$\mathbf{C}$, the functor $F(h)$ from $F(d)$ to $F(c)$ has a left
adjoint.
\end{definition}

\begin{proposition}\label{LimitesGrt}
Let $F\colon\mathbf{C}^{\opp}\mor \mathbf{Cat}$ be a functor.  If
$\mathbf{C}$ is complete, for every object $c\in \mathbf{C}$, the
category $F(c)$ is complete and, for every morphism $h\colon c\mor d$
in $\mathbf{C}$, the functor $F(h)$ from $F(d)$ to $F(c)$ is
continuous (i.e., preserves projective limits), then
$\int^{\mathbf{C}}F$ is complete.
\end{proposition}

\begin{proof}
See~\cite{tbg91}, pp. 247--248.
\end{proof}

\begin{proposition}\label{CoLimitesGrt}
Let $F\colon\mathbf{C}^{\opp}\mor \mathbf{Cat}$ be a functor.  If $\mathbf{C}$ is
cocomplete, for every object $c\in \mathbf{C}$, the category $F(c)$ is
cocomplete, and $F$ is locally reversible, then $\int^{\mathbf{C}}F$ is
cocomplete.
\end{proposition}

\begin{proof}
See~\cite{tbg91}, pp. 250--251.
\end{proof}

\begin{corollary}
The category $\mathbf{MSet}$ is bicomplete.
\end{corollary}

\begin{proof}
The category $\mathbf{MSet}$ is complete because $\mathbf{Set}$ is complete,
for every set $S$, $\mathrm{MSet}(S) = \mathbf{Set}^{S}$ is complete,
and, for every mapping $\varphi\colon S\mor T$, the functor
$\mathrm{MSet}(\varphi) = \Delta_{\varphi}$ from $\mathbf{Set}^{T}$ to
$\mathbf{Set}^{S}$ is continuous, since it has $\coprod_{\varphi}$ as a
left adjoint.

The category $\mathbf{MSet}$ is cocomplete because $\mathbf{Set}$ is
cocomplete, for every set $S$, $\mathrm{MSet}(S) = \mathbf{Set}^{S}$ is
cocomplete, and the contravariant functor $\mathrm{MSet}$ is locally
reversible.
\end{proof}


Our next goal is to define the category $\mathbf{Sig}$, of standard
many-sorted signatures and many-sorted signature morphisms, by
applying the construction of Ehresmann-Grothendieck to a contravariant
functor $\op{Sig}$ from $\mathbf{Set}$ to $\mathbf{Cat}$.  The
category $\mathbf{Sig}$ will show to be fundamental to get, by
means of the same construction, but applied to a contravariant functor
from $\mathbf{Sig}$ to $\mathbf{Cat}$, the category of many-sorted
algebras, and also (as will be seen in the fifth section on morphisms
of Fujiwara) to build on it, through the construction of Kleisli,
another category with the same objects that $\mathbf{Sig}$, but with a
new type of morphisms, the polyderivors, which will show to be adequate to
prove, in the last section, the category-theoretical equivalence
between the specifications for Hall and Bénabou algebras.

Before we prove the existence of the contravariant functor
$\op{Sig}$ in the following proposition, we recall that, for a set of
sorts $S$, the category of $S$-\emph{sorted signatures}, denoted by
$\mathbf{Sig}(S)$, is $\mathbf{Set}^{\fmon{S}\bprod S}$, where, we recall,
$\fmon{S}$ is the underlying set of the free monoid on $S$.  Therefore
an $S$-\emph{sorted signature} is a function $\Sigma$ from
$\fmon{S}\bprod S$ to $\boldsymbol{\mathcal{U}}$ which sends a pair
$(w,s)\in \fmon{S}\bprod S$ to the set $\Sigma_{w,s}$ of the
\emph{formal operations} of \emph{arity} $w$, \emph{sort} (or
\emph{coarity}) $s$, and \emph{rank} (or \emph{biarity}) $(w,s)$; and
an $S$-\emph{sorted signature morphism} from $\Sigma$ to $\Sigma'$ an
ordered triple $(\Sigma,d,\Sigma')$, abbreviated to $d\colon
\Sigma\mor \Sigma'$, where $d$ is an element of $\prod_{(w,s)\in
\fmon{S}\bprod S}\mathrm{Hom}(\Sigma_{w,s},\Sigma'_{w,s})$.
Thus, for $(w,s)\in \fmon{S}\bprod S$, $d_{w,s}$ is a mapping
from $\Sigma_{w,s}$ to $\Sigma'_{w,s}$ which sends a formal operation
$\sigma$ in $\Sigma_{w,s}$ to the formal operation $d_{w,s}(\sigma)$,
abbreviated to $d(\sigma)$, in $\Sigma'_{w,s}$.  Sometimes we will
write $\sigma\colon w\mor s$ to indicate that the formal operation
$\sigma$ belongs to $\Sigma_{w,s}$.

\begin{proposition}
There exists a contravariant functor $\op{Sig}$ from $\mathbf{Set}$ to
$\mathbf{Cat}$ defined as follows
\begin{enumerate}
\item $\op{Sig}$ sends a set (of sorts) $S$ to $\Sig(S) =
      \mathbf{Sig}(S)$, the category of $S$-sorted signatures.

\item $\op{Sig}$ sends a mapping $\varphi$ from $S$ to $T$ to the
      functor $\Sig(\varphi) = \Delta_{\fmon{\varphi}\bprod \varphi}$
      from $\mathbf{Sig}(T)$ to $\mathbf{Sig}(S)$ which relabels $T$-sorted
      signatures into $S$-sorted signatures, i.e., $\Sig(\varphi)$
      assigns to a $T$-sorted signature $\Lambda$ the $S$-sorted
      signature $\Sig(\varphi)(\Lambda) =
      \Lambda_{\fmon{\varphi}\bprod \varphi}$, and assigns to a
      morphism of $T$-sorted signatures $d$ from $\Lambda$ to
      $\Lambda'$ the morphism of $S$-sorted signatures
      $\Sig(\varphi)(d) = d_{\fmon{\varphi}\bprod \varphi}$ from
      $\Lambda_{\fmon{\varphi}\bprod \varphi}$ to
      $\Lambda'_{\fmon{\varphi}\bprod \varphi}$.
\end{enumerate}

\end{proposition}

\begin{proof}
Because $\op{Sig}$ is the composition of the covariant endofunctor
$(\cdot)^{\star}\times (\cdot)$ of $\mathbf{Set}$ which sends a set $S$ to
$\fmon{S}\times S$ and a mapping $\varphi\colon S\mor T$ to the
mapping $\fmon{\varphi}\bprod \varphi\colon \fmon{S}\bprod S\mor
\fmon{T}\bprod T$ (which assigns to $(w,s)$ in $\fmon{S}\bprod S$,
$(\fmon{\varphi}(w),\varphi(s))$ in $\fmon{T}\bprod T$), and the
contravariant functor $\mathrm{MSet}$ from $\mathbf{Set}$ to $\mathbf{Cat}$.
\end{proof}

\begin{definition}
The category $\mathbf{Sig}$ of \emph{many-sorted signatures} and
\emph{many-sorted signature morphisms}, obtained by applying the
construction of Ehresmann-Grothendieck to the contravariant functor
$\Sig$ on $\mathbf{Set}$ to $\mathbf{Cat}$, is $\mathbf{Sig} =
\int^{\mathbf{Set}}\Sig$.
\end{definition}

Therefore the category $\mathbf{Sig}$ has as objects the pairs
$(S,\Sigma)$, where $S$ is a set of sorts and $\Sigma$ an $S$-sorted
signature and as many-sorted signature morphisms from $(S,\Sigma)$ to
$(T,\Lambda)$ the pairs $(\varphi,d)$, where $\varphi\colon S\mor T$
is a morphism in $\mathbf{Set}$ while $d\colon \Sigma\mor
\Lambda_{\fmon{\varphi}\bprod \varphi}$ is a morphism in
$\mathbf{Sig}(S)$.  The composition of
$$
(\varphi,d)\colon (S,\Sigma)\mor (T,\Lambda)
\,\text{\,\,and\,\,}\,
(\psi,e)\colon (T,\Lambda)\mor(U,\Omega),
$$
denoted by $(\psi,e)\comp (\varphi,d)$, is $(\psi\comp\varphi,
e_{\fmon{\varphi}\bprod \varphi}\comp d)$, where
$$
e_{\fmon{\varphi}\bprod \varphi}\colon\Lambda_{\fmon{\varphi}\bprod
\varphi}\mor (\Omega_{\fmon{\psi}\bprod \psi})_{\fmon{\varphi}\bprod
\varphi}= \Omega_{\fmon{(\psi\comp\varphi)}\bprod
(\psi\comp\varphi)}.
$$
From now on, unless otherwise stated, we will write $\mathbf{\Sigma}$,
$\mathbf{\Lambda}$, $\mathbf{\Omega}$, and $\mathbf{\Xi}$ instead of
$(S,\Sigma)$, $(T,\Lambda)$, $(U,\Omega)$, and $(X,\Xi)$,
respectively, and $\mathbf{d}$, $\mathbf{e}$, and $\mathbf{h}$,
instead of $(\varphi,d)$, $(\psi,e)$, and $(\gamma,h)$, respectively.
Furthermore, to shorten terminology, we will drop the qualifying
adjective \lq\lq many-sorted\rq\rq and thus we will say
\emph{signature} and \emph{signature morphism} instead of
\emph{many-sorted signature} and \emph{many-sorted signature
morphism}, respectively.

\begin{remark}
The category $\mathbf{Sig}$, as was the case for $\mathbf{MSet}$, is
also a split bifibration on $\mathbf{Set}$ through the projection
functor $\pi_{\mathrm{Sig}}$ for $\mathbf{Sig}$.
\end{remark}

Since the category $\mathbf{Sig}$ can be identified to a subcategory
of the category $\mathbf{Sig}_{\mathfrak{pd}}$, defined in the fifth
section, we refer to that section for examples of signature morphisms.

\begin{proposition}
The category $\mathbf{Sig}$ is bicomplete.
\end{proposition}

\begin{proof}
The proof of the bicompleteness of the category $\mathbf{Sig}$ is formally
identical to that of the category $\mathbf{MSet}$.
\end{proof}


After having defined the categories $\mathbf{MSet} =
\int^{\mathbf{Set}}\mathrm{MSet}$ and $\mathbf{Sig} =
\int^{\mathbf{Set}}\Sig$ and examined some of its most useful
properties (from the standpoint of general algebra), we proceed next to
define the category $\mathbf{Alg}$ of many-sorted algebras by applying
the construction of Ehresmann-Grothendieck to a suitable contravariant
functor $\mathrm{Alg}$ defined on $\mathbf{Sig}$ and taking values in
$\mathbf{Cat}$.  Besides, we prove that the category $\mathbf{Alg}$ is
concrete and univocally transportable relative to a \lq\lq
forgetful\rq\rq functor to an adequate category, that this forgetful
functor, in addition, has a left adjoint, and that $\mathbf{Alg}$ is
a bicomplete category.

Before we realize what has been announced we recall that, for a
signature $\mathbf{\Sigma}$ and an $S$-sorted set $A$, the
$\fmon{S}\bprod S$-sorted set of the \emph{finitary operations on}
$A$, $\mathrm{HOp}_{S}(A)$ (thus denoted because, as we will show in
the fourth section, it is an example of a Hall algebra), is
$(\mathrm{Hom}(A_{w},A_{s}))_{(w,s)\in\fmon{S}\bprod S}$, where $A_{w}
= \prod_{i\in \bb{w}}A_{w_{i}}$, with $\bb{w}$ denoting the length of
the word $w$; and that a \emph{structure of}
$\mathbf{\Sigma}$-\emph{algebra} \emph{on} $A$ is a morphism $F =
(F_{w,s})_{(w,s)\in \fmon{S}\times S}$ in $\mathbf{Sig}(S)$ from
$\Sigma$ to $\mathrm{HOp}_{S}(A)$.  For a pair $(w,s)\in
\fmon{S}\times S$ and a formal operation $\sigma\in \Sigma_{w,s}$, in
order to simplify the notation, the operation from $A_{w}$ to $A_{s}$
corresponding to $\sigma$ under $F_{w,s}$ will be written as
$F_{\sigma}$ instead of $F_{w,s}(\sigma)$.  Then the category of
$\mathbf{\Sigma}$-\emph{algebras}, denoted by
$\mathbf{Alg}(\mathbf{\Sigma})$, has as objects the pairs $(A,F)$,
abbreviated to $\mathbf{A}$, where $A$ is an $S$-sorted set and $F$ a
structure of $\mathbf{\Sigma}$-algebra on $A$; and as morphisms from
$\mathbf{A}$ to $\mathbf{B}$, where $\mathbf{B} = (B,G)$, the
$\mathbf{\Sigma}$-\emph{homomorphisms}, i.e., the triples
$(\mathbf{A},f,\mathbf{B})$, abbreviated to $f\colon \mathbf{A}\mor
\mathbf{B}$, where $f$ is an $S$-sorted mapping from $A$ to $B$ such
that, for every $(w,s)\in \fmon{S}\times S$, $\sigma\in \Sigma_{w,s}$,
and $(a_{i})_{i\in \bb{w}}\in A_{w}$ we have that
$$
  f_{s}(F_{\sigma}((a_{i})_{i\in \bb{w}})) =
  G_{\sigma}(f_{w}((a_{i})_{i\in \bb{w}})),
$$
where $f_{w}$ is the mapping $\prod_{i\in \bb{w}}f_{w_{i}}$ from
$A_{w}$ to $B_{w}$ which sends $(a_{i})_{i\in \bb{w}}$ in $A_{w}$ to
$(f_{w_{i}}(a_{i}))_{i\in \bb{w}}$ in $B_{w}$, or, what is equivalent,
such that, for every $(w,s)\in \fmon{S}\times S$ and $\sigma\in
\Sigma_{w,s}$, the following diagram commutes
$$
\xymatrix@C=80pt@R=40pt{
A_{w}
  \ar[r]^{f_{w}}
  \ar[d]_{F_{\sigma}} &
B_{w}
  \ar[d]^{G_{\sigma}}
  \\
A_{s}
  \ar[r]_{f_{s}} &
B_{s}
}
$$

Sometimes, to avoid any confusion, we will denote the structures of
$\mathbf{\Sigma}$-algebra of the $\mathbf{\Sigma}$-algebras
$\mathbf{A}$, $\mathbf{B}$, \ldots, by $F^{\mathbf{A}}$,
$F^{\mathbf{B}}$, \ldots, respectively, and the components of
$F^{\mathbf{A}}$, $F^{\mathbf{B}}$, \ldots, as
$F^{\mathbf{A}}_{\sigma}$, $F^{\mathbf{B}}_{\sigma}$, \ldots,
respectively.

\begin{proposition}
There exists a contravariant functor $\Alg$ from $\mathbf{Sig}$ to
$\mathbf{Cat}$ which sends a signature $\mathbf{\Sigma}$ to
$\Alg(\mathbf{\Sigma}) = \mathbf{Alg}(\mathbf{\Sigma})$, the category
of $\mathbf{\Sigma}$-algebras, and a signature morphism
$\mathbf{d}\colon \mathbf{\Sigma}\mor \mathbf{\Lambda}$ to the functor
$\Alg(\mathbf{d}) = \mathbf{d}^{\ast}\colon
\mathbf{Alg}(\mathbf{\Lambda})\mor \mathbf{Alg}(\mathbf{\Sigma})$
defined as follows
\begin{enumerate}
\item $\mathbf{d}^{\ast}$ assigns to a
      $\mathbf{\Lambda}$-algebra $\mathbf{B} = (B,G)$ the
      $\mathbf{\Sigma}$-algebra $\mathbf{d}^{\ast}(\mathbf{B}) =
      (B_{\varphi},G^{\mathbf{d}})$, where $G^{\mathbf{d}}$ is the
      composition of the $\fmon{S}\times S$-sorted mappings
      $$
      d\colon\Sigma\mor\Lambda_{\fmon{\varphi}\bprod \varphi} \text{\quad
      and \quad}
      G_{\fmon{\varphi}\bprod
      \varphi}\colon\Lambda_{\fmon{\varphi}\bprod
      \varphi}\mor\mathrm{HOp}_{T}(B)_{\fmon{\varphi}\bprod \varphi}.
      $$
      We agree that, for a formal operation $\sigma\in\Sigma_{w,s}$,
      $G_{d(\sigma)}\colon B_{\varphi^{\star}(w)}\mor B_{\varphi(s)}$
      denotes the value of $G^{\mathbf{d}}$ at $\sigma$.

\item $\mathbf{d}^{\ast}$ assigns to a $\mathbf{\Lambda}$-homomorphism $f$ from
      $\mathbf{B}$ to $\mathbf{B}'$ the $\mathbf{\Sigma}$-homo\-mor\-phism
      $\mathbf{d}^{\ast}(f) = f_{\varphi}$ from $\mathbf{d}^{\ast}(\mathbf{B})$
      to $\mathbf{d}^{\ast}(\mathbf{B}')$.
\end{enumerate}
\end{proposition}

\begin{proof}
For every $\mathbf{\Lambda}$-algebra $(B,G)$, we have that $G\colon
\Lambda\mor \mathrm{HOp}_{T}(B)$.  Then, by composing
$d\colon\Sigma\mor\Lambda_{\fmon{\varphi}\bprod \varphi}$ and
$G_{\fmon{\varphi}\bprod \varphi}\colon\Lambda_{\fmon{\varphi}\bprod
\varphi}\mor\mathrm{HOp}_{T}(B)_{\fmon{\varphi}\bprod \varphi}$, and
taking into account that $\mathrm{HOp}_{T}(B)_{\fmon{\varphi}\bprod
\varphi}= \mathrm{HOp}_{S}(B_{\varphi})$, we have that $G^{\mathbf{d}} =
G_{\fmon{\varphi}\bprod \varphi}\comp d$ is a structure of
$\mathbf{\Sigma}$-algebra on $B_{\varphi}$.

On the other hand, given $(w,s)\in \fmon{S}\times S$ and
$\sigma\in\Sigma_{w,s}$, $d(\sigma)\in
\Lambda_{\fmon{\varphi}(w),\varphi(s)}$, thus, being $f$ a
$\mathbf{\Lambda}$-homomorphism from $(B,G)$ to $(B',G')$, we have that
$f_{\varphi(s)}\comp G_{d(\sigma)} = {G'_{d(\sigma)}}\comp
f_{\fmon{\varphi}(w)}$, therefore, because $G^{\mathbf{d}}_{\sigma} =
G_{d(\sigma)}$ and ${G'_{\sigma}}^{\!\!\mathbf{d}} = {G'_{d(\sigma)}}$,
$(f_{\varphi})_{s}\comp G^{\mathbf{d}}_{\sigma} =
{G'_{\sigma}}^{\!\!\mathbf{d}}\comp (f_{\varphi})_{w}$, hence
$f_{\varphi}$ is a $\mathbf{\Sigma}$-homo\-morphism from
$(B_{\varphi},G^{\mathbf{d}})$ to $(B'_{\varphi},{G'}^{\mathbf{d}})$.

Since the identities and the composites are, obviously, preserved by
$\mathbf{d}^{\ast}$, it follows that $\mathbf{d}^{\ast}$ is a functor from
$\mathbf{Alg}(\mathbf{\Lambda})$ to $\mathbf{Alg}(\mathbf{\Sigma})$.
\end{proof}

\begin{definition}
The category $\mathbf{Alg}$ of \emph{many-sorted algebras} and
\emph{many-sorted algebra homomorphisms}, obtained by applying the
construction of Ehresmann-Grothen\-dieck to the contravariant functor
$\Alg$ from $\mathbf{Sig}$ to $\mathbf{Cat}$, is $\mathbf{Alg} =
\int^{\mathbf{Sig}}\Alg$.
\end{definition}

Therefore the category $\mathbf{Alg}$ has as objects the pairs
$(\mathbf{\Sigma},\mathbf{A})$, where $\mathbf{\Sigma}$ is a signature
and $\mathbf{A}$ a $\mathbf{\Sigma}$-algebra, and as morphisms from
$(\mathbf{\Sigma},\mathbf{A})$ to $(\mathbf{\Lambda},\mathbf{B})$, the
pairs $(\mathbf{d},f)$, with $\mathbf{d}$ a signature morphism from
$\mathbf{\Sigma}$ to $\mathbf{\Lambda}$ and $f$ a
$\mathbf{\Sigma}$-homomorphism from $\mathbf{A}$ to
$\mathbf{d}^{\ast}(\mathbf{B})$.  Hence, for every $(w,s)\in
\fmon{S}\times S$ and $\sigma\in \Sigma_{w,s}$ the following diagram
commutes
$$
\xymatrix@C=80pt@R=40pt{
A_{w}
  \ar[r]^{f_{w}}
  \ar[d]_{F_{\sigma}} &
B_{\varphi^{\star}(w)}
  \ar[d]^{G_{d(\sigma)}}
  \\
A_{s}
  \ar[r]_{f_{s}} &
B_{\varphi(s)}
}
$$
From now on, to shorten terminology, we will say \emph{algebra} and
\emph{algebra homomorphism}, or, simply, \emph{homomorphism}, instead
of \emph{many-sorted algebra} and \emph{many-sorted algebra
homomorphism}, respectively.  Sometimes, to avoid any confusion, we
denote an algebra $(\mathbf{\Sigma},\mathbf{A})$ and an homomorphism
$(\mathbf{d},f)$ also by $(S,\Sigma,A,F)$ and $(\varphi,d,f)$,
respectively.

Since the category $\mathbf{Alg}$ can be identified to a subcategory
of the category $\mathbf{Alg}_{\mathfrak{pd}}$, defined in the fifth
section, we refer to that section for examples of homomorphisms
between algebras.

\begin{proposition}
The category $\mathbf{Alg}$ is a concrete and univocally transportable
category.
\end{proposition}

\begin{proof}
It is enough to specify  a functor from $\mathbf{Alg}$ to a convenient
category of sorted sets labelled by signatures.

Let $\mathrm{G}_{\mathbf{MSet}}$ be the forgetful functor from
$\mathbf{Alg}$ to $\mathbf{MSet}$ (that is not a fibration),
$\pi_{\mathrm{Alg}}$ the projection functor for $\mathbf{Alg}$, and
$(\mathbf{MSet}\times_{\mathbf{Set}}\mathbf{Sig},(\mathrm{P}_{0},\mathrm{P}_{1}))$
the pullback of the projection functors $\pi_{\mathrm{MSet}}$ and
$\pi_{\mathrm{Sig}}$, for $\mathbf{MSet}$ and $\mathbf{Sig}$,
respectively, where
\begin{enumerate}
\item The category
      $\mathbf{MSet}\times_{\mathbf{Set}}\mathbf{Sig}$ has as objects,
      essentially, triples $(S,\Sigma,A)$, with $(S,\Sigma)$ a
      signature and $A$ an $S$-sorted set, and as morphisms from
      $(S,\Sigma,A)$ to $(T,\Lambda,B)$ triples $(\varphi,d,f)$, such
      that $(\varphi,d)$ is a signature morphism from $(S,\Sigma)$ to
      $(T,\Lambda)$ and $(\varphi,f)$ a mapping from $(S,A)$ to
      $(T,B)$, while

\item $\mathrm{P}_{0}$ is the functor from
      $\mathbf{MSet}\times_{\mathbf{Set}}\mathbf{Sig}$ to
      $\mathbf{MSet}$ which sends a morphism $(\varphi,d,f)$ from
      $(S,\Sigma,A)$ to $(T,\Lambda,B)$ to the $\mathrm{ms}$-mapping
      $(\varphi,f)$ from $(S,A)$ to $(T,B)$, and $\mathrm{P}_{1}$ is the
      functor from $\mathbf{MSet}\times_{\mathbf{Set}}\mathbf{Sig}$ to
      $\mathbf{Sig}$ which sends a morphism $(\varphi,d,f)$ from
      $(S,\Sigma,A)$ to $(T,\Lambda,B)$ to the signature morphism
      $(\varphi,d)$ from $(S,\Sigma)$ to $(T,\Lambda)$.
\end{enumerate}

Then we have that the structural functors $\mathrm{P}_{0}$ and
$\mathrm{P}_{1}$ are fibrations, and that the unique functor
$\mathrm{G}$ from $\mathbf{Alg}$ to
$\mathbf{MSet}\times_{\mathbf{Set}}\mathbf{Sig}$ such that
$\mathrm{P}_{0}\comp \mathrm{G} = \mathrm{G}_{\mathbf{MSet}}$ and
$\mathrm{P}_{1}\comp \mathrm{G} = \pi_{\mathrm{Alg}}$, as in the
following diagram
$$
\xymatrix@C=40pt@R=35pt{
\mathbf{Alg}
  \ar@/^2pc/[rrd]^{\pi_{\mathrm{Alg}}}
  \ar[rd]|(.6)*+{\mathrm{G}}
  \ar@/_2pc/[rdd]_{\mathrm{G}_{\mathbf{MSet}}} \\
&
\mathbf{MSet}\times_{\mathbf{Set}}\mathbf{Sig}
      \ar[r]^-{\mathrm{P}_{1}}
      \ar[d]_-{\mathrm{P}_{0}} &
\mathbf{Sig}
      \ar[d]^{\pi_{\mathrm{Sig}}}  \\
&
\mathbf{MSet}
      \ar[r]_{\pi_{\mathrm{MSet}}} &
\mathbf{Set}
}
$$
makes the category $\mathbf{Alg}$ a concrete and univocally transportable
category on the category $\mathbf{MSet}\times_{\mathbf{Set}}\mathbf{Sig}$.
\end{proof}

The functor $\mathrm{G}$ from $\mathbf{Alg}$ to
$\mathbf{MSet}\times_{\mathbf{Set}}\mathbf{Sig}$ not only has the
above mentioned properties, but it also has a left adjoint, obtained
from the family $(\mathbf{T}_{\mathbf{\Sigma}})_{\mathbf{\Sigma}\in
\mathbf{Sig}}$ where, for a signature $\mathbf{\Sigma}$ in
$\mathbf{Sig}$, the functor $\mathbf{T}_{\mathbf{\Sigma}}$ from
$\mathbf{Set}^{S}$ to $\mathbf{Alg}(\mathbf{\Sigma})$ is the left
adjoint to the forgetful functor $\mathrm{G}_{\mathbf{\Sigma}}$ from
$\mathbf{Alg}(\mathbf{\Sigma})$ to $\mathbf{Set}^{S}$.  And this left
adjoint to $\mathrm{G}$, that, as we will see below, allows us to get
translations between free algebras, will be used, once defined the
many-sorted equations in the third section, to translate, for a
signature morphism, many-sorted equations for the source signature to
many-sorted equations for the target signature.  This translation of
equations, together with the invariant character of the relation of
satisfaction under change of notation, will allow us to define, also
in the third section, the many-sorted equational institution (more
general than that defined by Goguen and Burstall in~\cite{gb86}) that
embodies the essentials of semantical many-sorted equational
deduction.

Before we prove the existence of a left adjoint to $\mathrm{G}$,
we recall that
\begin{enumerate}
\item For a signature $\mathbf{\Sigma}$ and an $S$-sorted set
      of variables $X$, $\mathbf{T}_{\mathbf{\Sigma}}(X)$, the value
      of the functor $\mathbf{T}_{\mathbf{\Sigma}}$ in $X$, is the
      \emph{free} (also called the \emph{term} or \emph{word})
      $\mathbf{\Sigma}$-\emph{algebra} \emph{on} $X$, and $\eta_{X}$
      is the \emph{insertion (of the generators)} $X$ \emph{into}
      $\mathrm{T}_{\mathbf{\Sigma}}(X)$, the underlying $S$-sorted set
      of $\mathbf{T}_{\mathbf{\Sigma}}(X)$;

\item For a $\mathbf{\Sigma}$-algebra $\mathbf{A}$ and a
      \emph{valuation} $f$ \emph{of the $S$-sorted set of variables}
      $X$ \emph{in} $A$, i.e., an $S$-sorted mapping $f$ from $X$ to
      $A$, $f^{\sharp}$ denotes the \emph{canonical extension of} $f$
      \emph{up to} $\mathbf{T}_{\mathbf{\Sigma}}(X)$, i.e., the unique
      $\mathbf{\Sigma}$-homomorphism from
      $\mathbf{T}_{\mathbf{\Sigma}}(X)$ to $\mathbf{A}$ such that
      $f^{\sharp}\comp\eta_{X} = f$; and

\item For an $S$-sorted mapping $f$ from $X$ to $Y$, $f^{@}$
      denotes the unique $\mathbf{\Sigma}$-homo\-mor\-phism from
      $\mathbf{T}_{\mathbf{\Sigma}}(X)$ to
      $\mathbf{T}_{\mathbf{\Sigma}}(Y)$ such that $f^{@}\comp\eta_{X}
      = \eta_{Y}\comp f$, i.e., the value of the functor
      $\mathbf{T}_{\mathbf{\Sigma}}$ in $f$. Therefore $f^{@}$ is also
      $(\eta_{Y}\comp f)^{\sharp}$.
\end{enumerate}
Moreover, transposing to the many-sorted case the terminology coined
for the single-sorted case, we call, for $s\in S$, the elements of
$\mathrm{T}_{\mathbf{\Sigma}}(X)_{s}$, \emph{many-sorted terms for}
$\mathbf{\Sigma}$ \emph{of type} $(X,s)$, from now on abbreviated to
\emph{terms} \emph{for} $\mathbf{\Sigma}$ \emph{of type} $(X,s)$, or,
simply, to \emph{terms} \emph{of type} $(X,s)$.  We point out that
what we have called terms for $\mathbf{\Sigma}$ of type $(X,s)$ are
also known, for those following the terminology in
Gr\"{a}tzer~\cite{gG79}, p.  39, as \emph{polynomial symbols}
(\emph{for} $\mathbf{\Sigma}$) \emph{of type} $(X,s)$.

\begin{proposition}
There exists a functor $\mathbf{T}\colon
\mathbf{MSet}\times_{\mathbf{Set}}\mathbf{Sig}\mor\mathbf{Alg}$ left
adjoint to the functor $\mathrm{G}\colon \mathbf{Alg}\mor
\mathbf{MSet}\times_{\mathbf{Set}}\mathbf{Sig}$.
\end{proposition}

\begin{proof}
The functor $\mathbf{T}$ from
$\mathbf{MSet}\times_{\mathbf{Set}}\mathbf{Sig}$ to $\mathbf{Alg}$
given on objects $(S,\Sigma,X)$ by $\mathbf{T}(S,\Sigma,X) =
(\mathbf{\Sigma},\mathbf{T}_{\mathbf{\Sigma}}(X))$ and on arrows
$(\varphi,d,f)\colon (S,\Sigma,X)\mor (T,\Lambda,Y)$ as
$$
(\mathbf{d},f^{\mathbf{d}})\colon
(\mathbf{\Sigma},\mathbf{T}_{\mathbf{\Sigma}}(X))\mor
(\mathbf{\Lambda},\mathbf{T}_{\mathbf{\Lambda}}(Y)),
$$
where $f^{\mathbf{d}} = ((\eta_{Y})_{\varphi}\comp f)^{\sharp}$ is the
canonical extension of the $S$-sorted mapping
$(\eta_{Y})_{\varphi}\comp f$ from $X$ to
$\mathrm{T}_{\mathbf{\Lambda}}(Y)_{\varphi}$ up to the free
$\mathbf{\Sigma}$-algebra on $X$, is left adjoint to the functor
$\mathrm{G}$.
\end{proof}

For a morphism $(\varphi,d,f)\colon (S,\Sigma,X)\mor(T,\Lambda,Y)$ in
$\mathbf{MSet}\times_{\mathbf{Set}}\mathbf{Sig}$, the functor
$\mathbf{T}\colon\mathbf{MSet}\times_{\mathbf{Set}}\mathbf{Sig}\mor\mathbf{Alg}$
acting on $(\varphi,d,f)$ allows us to get the
$\mathbf{\Sigma}$-homo\-mor\-phism $f^{\mathbf{d}}$ from
$\mathbf{T}_{\mathbf{\Sigma}}(X)$ to
$\mathbf{T}_{\mathbf{\Lambda}}(Y)_{\varphi}$, hence, for $s\in S$, it
translates terms for $\mathbf{\Sigma}$ of type $(X,s)$, i.e., elements
$P$ of $\mathrm{T}_{\mathbf{\Sigma}}(X)_{s}$, into terms for
$\mathbf{\Lambda}$ of type $(Y,\varphi(s))$, i.e., elements
$f^{\mathbf{d}}_{s}(P)$ of
$\mathrm{T}_{\mathbf{\Lambda}}(Y)_{\varphi(s)}$.

In particular, the unit $\eta^{\varphi}$ of the adjunction
$\coprod_{\varphi}\ladj\Delta_{\varphi}$ provides, for every
$S$-sorted set $X$, the $S$-sorted mapping $\eta^{\varphi}_{X}\colon
X\mor(\coprod_{\varphi}X)_{\varphi}$ and if $\mathbf{d}\colon
\mathbf{\Sigma}\mor \mathbf{\Lambda}$ is a morphism of signatures,
then $(\varphi,d,\eta^{\varphi}_{X})\colon(S,\Sigma,X)\mor
(T,\Lambda,\textstyle{\coprod}_{\varphi}X)$ is a morphism in
$\mathbf{MSet}\times_{\mathbf{Set}}\mathbf{Sig}$, hence the functor
$\mathbf{T}$ acting on $(\varphi,d,\eta^{\varphi}_{X})$ gives rise to
the morphism
$$
(\mathbf{d},\eta^{\mathbf{d}}_{X})\colon
(\mathbf{\Sigma},\mathbf{T}_{\mathbf{\Sigma}}(X))\mor
(\mathbf{\Lambda},\mathbf{T}_{\mathbf{\Lambda}}(\textstyle\coprod_{\varphi}X)),
$$
where $\eta^{\mathbf{d}}_{X} =
((\eta_{\coprod_{\varphi}X})_{\varphi}\comp
\eta^{\varphi}_{X})^{\sharp}$ is the $\mathbf{\Sigma}$-homomorphism
from $\mathbf{T}_{\mathbf{\Sigma}}(X)$ to
$\mathbf{T}_{\mathbf{\Lambda}}(\coprod_{\varphi}X)_{\varphi}$ that
extends the $S$-sorted mapping
$(\eta_{\coprod_{\varphi}X})_{\varphi}\comp \eta^{\varphi}_{X}$ from
$X$ to $\mathrm{T}_{\mathbf{\Lambda}}(\coprod_{\varphi}X)_{\varphi}$.
Therefore, for $s\in S$, $\eta^{\mathbf{d}}_{X,s}$, the $s$-th
component of $\eta^{\mathbf{d}}_{X}$, translates terms for
$\mathbf{\Sigma}$ of type $(X,s)$ into terms for $\mathbf{\Lambda}$ of
type $(\coprod_{\varphi}X,\varphi(s))$.  The
$\mathbf{\Sigma}$-homomorphisms $\eta^{\mathbf{d}}_{X}$, as stated in
the following proposition, are in fact the components of a natural
transformation, and this contributes to explain their relevance as
translators.

\begin{proposition}\label{fundPdextranaturalbis}
Let $\mathbf{d}$ be a morphism of signatures from $\mathbf{\Sigma}$ to
$\mathbf{\Lambda}$.  Then the family $\eta^{\mathbf{d}} =
(\eta^{\mathbf{d}}_{X})_{X\in \boldsymbol{\mathcal{U}}}$, which to an
$S$-sorted set $X$ assigns the $\mathbf{\Sigma}$-homomorphism
$\eta^{\mathbf{d}}_{X}$ from $\mathbf{T}_{\mathbf{\Sigma}}(X)$ to
$\mathbf{T}_{\mathbf{\Lambda}}(\coprod_{\varphi}X)_{\varphi}$, is a
natural transformation from $\mathbf{T}_{\mathbf{\Sigma}}$ to
$\mathbf{d}^{\ast}\comp\mathbf{T}_{\mathbf{\Lambda}}\comp
\coprod_{\varphi}$, and so, for the forgetful functor
$\mathrm{G}_{\mathbf{\Sigma}}$ from $\mathbf{Alg}(\mathbf{\Sigma})$ to
$\mathbf{Set}^{S}$, the family $\mathrm{G}_{\mathbf{\Sigma}}\ast
\eta^{\mathbf{d}}$, i.e., the horizontal composition of the natural
transformation $\eta^{\mathbf{d}}$ and $\mathrm{G}_{\mathbf{\Sigma}}$, also
denoted by $\eta^{\mathbf{d}}$, is a natural transformation from
$\mathrm{T}_{\mathbf{\Sigma}} = \mathrm{G}_{\mathbf{\Sigma}}\comp
\mathbf{T}_{\mathbf{\Sigma}}$ to $\Delta_{\varphi}\comp
\mathrm{T}_{\mathbf{\Lambda}}\comp \coprod_{\varphi}$, since
$\mathrm{G}_{\mathbf{\Sigma}}\comp \mathbf{d}^{\ast} = \Delta_{\varphi}\comp
\mathrm{G}_{\mathbf{\Lambda}}$ and $\mathrm{T}_{\mathbf{\Lambda}} =
\mathrm{G}_{\mathbf{\Lambda}}\comp \mathbf{T}_{\mathbf{\Lambda}}$. 
\end{proposition}

\begin{proof}
It follows after the commutativity of the following diagram
$$
\xy
0;<1ex,0ex>:<0ex,1ex>::
\POS 0+(0,0)
\xymatrix"a"@C=35ex@R=15ex@!0{
X
  \ar[r]|*+{f}
  \ar[d]|*+{\eta^{\varphi}_{X}}
  &
Y
  \ar[d]|*+{\eta^{\varphi}_{Y}}
   \\
(\coprod_{\varphi}X)_{\varphi}
  \ar[r]|*+{(\coprod_{\varphi}f)_{\varphi}}
  &
(\coprod_{\varphi}Y)_{\varphi}
\save "a1,2"."a2,1"!C*+<5ex>\frm{}="a"\restore 
}
\POS 0+(9,11)
\xymatrix"b"@C=35ex@R=15ex@!0{
\mathrm{T}_{\mathbf{\Sigma}}(X)
  \ar[r]|*+{f^{@}}
  \ar[d]|*+{\eta^{\mathbf{d}}_{X}}&
\mathrm{T}_{\mathbf{\Sigma}}(Y)
  \ar[d]|*+{\eta^{\mathbf{d}}_{Y}}  \\
\mathrm{T}_{\mathbf{\Lambda}}(\coprod_{\varphi}X)_{\varphi}
  \ar[r]|*+{(\coprod_{\varphi}f)^{@}_{\varphi}} &
\mathrm{T}_{\mathbf{\Lambda}}(\coprod_{\varphi}Y)_{\varphi}
\save"b1,2"."b2,1"!C*+<5ex>\frm{}="b"\restore 
}
\ar "a1,1";"b1,1"|*+{\eta_{X}}
\ar "a1,2";"b1,2"|*+{\eta_{Y}}
\ar
"a2,1";"b2,1"|(.45)*-<1ex>
{\,\,\,\,\,\,\,\,\,\,(\eta_{\coprod_{\varphi}X})_{\varphi}}
\ar "a2,2";"b2,2"|(.50)*-<1ex>
{\,\,\,\,\,\,\,\,\,(\eta_{\coprod_{\varphi}Y})_{\varphi}}
\endxy
$$
\end{proof}

\begin{remark}
The natural transformation $\eta^{\mathbf{d}} =
\mathrm{G}_{\mathbf{\Sigma}}\ast \eta^{\mathbf{d}}$ from the functor
$\mathrm{T}_{\mathbf{\Sigma}}$ to the functor $\Delta_{\varphi}\comp
\mathrm{T}_{\mathbf{\Lambda}}\comp \coprod_{\varphi}$, will be used
later on, in this same section, to provide, for a signature
morphism $\mathbf{d}\colon \mathbf{\Sigma}\mor \mathbf{\Lambda}$, an
alternative, but equivalent, definition of (the morphism mapping of) a
translation functor $\mathbf{d}_{\diamond}$ from
$\mathbf{Ter}({\mathbf{\Sigma}})$ to
$\mathbf{Ter}({\mathbf{\Lambda}})$, where, for a signature
$\mathbf{\Sigma}$, we anticipate, $\mathbf{Ter}({\mathbf{\Sigma}})$ is
the category with objects the $S$-sorted sets and morphisms from $X$
to $Y$ the $S$-sorted mappings from $Y$ to
$\mathrm{T}_{\mathbf{\Sigma}}(X)$, that we will call terms for
$\mathbf{\Sigma}$ of type $(X,Y)$.
\end{remark}

Before proceeding to prove the bicompleteness of the category
$\mathbf{Alg}$, we make next a brief excursus about the utility of
some big subcategories of $\mathbf{Alg}$.  As we know, from
$\mathbf{Alg}$ to $\mathbf{Sig}$ we have the fibration
$\pi_{\mathrm{Alg}}$ and from $\mathbf{Sig}$ to $\mathbf{Set}$ the
fibration $\pi_{\mathrm{Sig}}$, hence from $\mathbf{Alg}$ to
$\mathbf{Set}$, by composing both fibrations, we get the fibration
$\pi_{\mathrm{Sig,Alg}}$.  And this composed fibration allows us to
get, for every set of sorts $S$, the corresponding fiber
$\mathbf{Alg}(S)$, that we call the category of $S$-\emph{algebras}.
We recall that $\mathbf{Alg}(S)$ has, essentially, as objects the
pairs $(\Sigma,\mathbf{A})$, where $\Sigma$ is an $S$-sorted signature
and $\mathbf{A}$ a $\Sigma$-algebra, and as morphisms from
$(\Sigma,\mathbf{A})$ to $(\Lambda,\mathbf{B})$, the pairs $(d,f)$,
with $d$ an $S$-sorted signature morphism from $\Sigma$ to $\Lambda$
and $f$ a $\Sigma$-homomorphism from $\mathbf{A}$ to
$d^{\ast}(\mathbf{B})$.

At first glance, the categories of the type $\mathbf{Alg}(S)$ can
appear to be unnecessarily or excessively general, hence almost
useless.  However, they show themselves to be useful, for example, to
get a full understanding, i.e., to complete category-theoretically
the explanation, of some classical constructions in universal algebra
as, e.g., those due to Birkhoff-Frink in~\cite{bf48}, where they
state, among others interesting results, the following representation
theorems:
\begin{enumerate}
\item Let $(A,J)$ be an algebraic closure space. Then there exists
      a single-sorted signature $\Sigma^{(A,J)}$ and a structure
      of $\Sigma^{(A,J)}$-algebra $F^{(A,J)}$ on $A$ such that
      $$
      (A,\mathrm{Sg}_{(A,F^{(A,J)})}) = (A,J),
      $$
      where $\mathrm{Sg}_{(A,F^{(A,J)})}$ is the generated subalgebra
      operator on $A$ induced by the $\Sigma^{(A,J)}$-algebra
      $(A,F^{(A,J)})$ (see~\cite{bf48}, p. 300).

\item Let $\mathbf{L}$ be a lattice.  Then $\mathbf{L}$ is an
      algebraic lattice, i.e., it is, besides, complete and compactly
      generated (see for the meaning of these terms, e.g., \cite{bs81},
      p.  17), iff there exists a single-sorted signature $\Sigma$ and
      a $\Sigma$-algebra $\mathbf{A}$ such that $\mathbf{L}$ is
      isomorphic to the algebraic lattice determined by the fixed
      points of the operator $\mathrm{Sg}_{\mathbf{A}}$
      (see~\cite{bf48}, p.  302).
\end{enumerate}

To show the role that the fibers of the functor
$\pi_{\mathrm{Sig,Alg}}$ play in the categorization of the first
theorem of Birkhoff-Frink, and because it is stated in terms of
ordinary sets, we should consider the category $\mathbf{Alg}(1)$,
i.e., the fiber of $\pi_{\mathrm{Sig,Alg}}$ in $1$, the standard final
set.  We recall that $\mathbf{Alg}(1)$ has as objects, essentially,
the pairs $(\Sigma,\mathbf{A})$, where $\Sigma = (\Sigma_{n})_{n\in
\mathbb{N}}$ is a single-sorted signature, i.e., an object of
$\mathbf{Set}^{\mathbb{N}}$, and $\mathbf{A} = (A,F)$ a
$\Sigma$-algebra, i.e., an ordinary set $A$ together with an
$\mathbb{N}$-sorted mapping $F$ from $\Sigma$ to
$(\mathrm{Hom}(A^{n},A))_{n\in \mathbb{N}}$, and as morphisms from
$(\Sigma,\mathbf{A})$ to $(\Lambda,\mathbf{B})$, where $\mathbf{B} =
(B,G)$, the pairs $(d,f)$ with $d = (d_{n})_{n\in \mathbb{N}}$ an
$\mathbb{N}$-sorted mapping from $\Sigma$ to $\Lambda$ in
$\mathbf{Set}^{\mathbb{N}}$, and $f$ a $\Sigma$-homomorphism from
$\mathbf{A} = (A,F)$ to $\mathbf{B}^{d} = (B,G\comp d)$.

To attain the aim just stated, we should also consider the category
$\mathbf{AClSp}$ with objects the algebraic closure spaces, i.e., the
pairs $(A,J)$, where $A$ is an ordinary set and $J$ an algebraic
closure operator on $A$, and morphisms from $(A,J)$ to $(B,K)$ the
ordered triples $((A,J),f,(B,K))$, abbreviated to $f\colon (A,J)\mor
(B,K)$, where $f$ is a mapping from $A$ to $B$ such that, for every
$X\subseteq A$, $f[J(X)]\subseteq K(f[X])$.

Next, let $\mathrm{Sg}$ be the functor from $\mathbf{Alg}(1)$ to
$\mathbf{AClSp}$ which sends an algebra $(\Sigma,\mathbf{A})$ to the
algebraic closure space $(A,\mathrm{Sg}_{\mathbf{A}})$ where
$\mathrm{Sg}_{\mathbf{A}}$ is the generated subalgebra operator on $A$
induced by $\mathbf{A}$; and a morphism $(d,f)$ from
$(\Sigma,\mathbf{A})$ to $(\Lambda,\mathbf{B})$ to the morphism $f$
from $(A,\mathrm{Sg}_{\mathbf{A}})$ to $(B,\mathrm{Sg}_{\mathbf{B}})$.

The action of $\mathrm{Sg}$ on the morphisms is well defined because,
for every $X\subseteq A$, $f[\mathrm{Sg}_{\mathbf{A}}(X)] \subseteq
\mathrm{Sg}_{\mathbf{B}}(f[X])$, taking into account that $f$ is a
$\Sigma$-homomorphism from $\mathbf{A} = (A,F)$ to $\mathbf{B}^{d} =
(B,G\comp d)$, and recalling that, for an arbitrary $\Sigma$-algebra
$\mathbf{C}$ and subset $Z$ of $C$, $\mathrm{Sg}_{\mathbf{C}}(Z) =
\bigcup_{n\in \mathbb{N}}\E_{\mathbf{C}}^{n}(Z)$, where
$(\E^{n}_{\mathbf{C}}(Z))_{n\in \mathbb{N}}$ is the family in
$\Sub(C)$ defined recursively as: $\E_{\mathbf{C}}^{0}(Z) = Z$, and,
for $n\geq 0$, $\E_{\mathbf{C}}^{n+1}(Z) =
\E_{\mathbf{C}}(\E_{\mathbf{C}}^{n}(Z))$, where $\E_{\mathbf{C}}$ is
the operator on $\Sub(C)$ which sends a subset $Z$ of $C$ to
$$
\textstyle Z\bunion\left(\;\union_{n\in
\mathbb{N}}\union_{\sigma\in\Sigma_{n}} F_{\sigma}[Z^{n}]\right).
$$

Then the functor $\mathrm{Sg}$ from $\mathbf{Alg}(1)$ to
$\mathbf{AClSp}$ is surjective on the objects (and this is the first
theorem of Birkhoff-Frink) and on the morphisms.  In fact, given an
algebraic closure space $(A,J)$, for the ordinary signature
$\Sigma^{(A,J)} = (\Sigma_{n}^{(A,J)})_{n\in \mathbb{N}}$ defined, for
$n\in \mathbb{N}$, as follows
$$
\Sigma_{n}^{(A,J)} = \textstyle\bigcup_{x\in A^{n}}(\{x\}\times
                     J(\mathrm{Im}(x))),
$$
where $\mathrm{Im}(x)$ is the image of the mapping $x\colon n\mor A$,
and the $\Sigma$-algebra $(A,F^{(A,J)})$, where $F^{(A,J)}$, the
structure of $\Sigma^{(A,J)}$-algebra on $A$, is defined, for $n\in
\mathbb{N}$, $x\in A^{n}$, and $a\in
J(\mathrm{Im}(x))$, as follows
$$
  F_{x,a}^{(A,J)}\nfunction
  {A^{n}}{A}{(y_{0},\ldots,y_{n-1})}
  {
  \begin{cases}
    a, &\text{if $(y_{0},\ldots,y_{n-1}) = x$;} \\
    y_{0}, &\text{if $(y_{0},\ldots,y_{n-1}) \neq x$,}
  \end{cases}
  }
$$
we have that $J = \Sg_{(A,F^{(A,J)})}$, i.e., the object mapping of
the functor $\mathrm{Sg}$ is surjective.

On the other hand, given a morphism $f$ from an algebraic closure
space $(A,J)$ to another $(B,K)$, we have that the pair $(d_{f},f)$,
where $d_{f}$ is the morphism from $\Sigma^{(A,J)}$ to
$\Sigma^{(B,K)}$ which to a pair $(x,a)$, with $x\in A^{n}$, for some
$n\in \mathbb{N}$, and $a\in J(\mathrm{Im}(x))$, assigns the pair
$(f^{n}(x),f(a))$, is a homomorphism from the algebra
$(\Sigma^{(A,J)},(A,F^{(A,J)}))$ to the algebra
$(\Sigma^{(B,K)},(B,F^{(B,K)}))$, and, obviously, it is send to the
morphism $f$ by the functor $\mathrm{Sg}$. Therefore the morphism
mapping of $\mathrm{Sg}$ is surjective.

To succeed in the categorization of the second theorem of
Birkhoff-Frink, we should take into consideration the category
$\mathbf{ALat}_{\bigwedge}$ with objects the algebraic lattices, and
morphisms the lattice morphisms which preserve arbitrary infima.

Next, let $\mathrm{Fix}$ be the contravariant functor from
$\mathbf{AClSp}$ to $\mathbf{ALat}_{\bigwedge}$ which assigns to an
algebraic closure space $(A,J)$ the algebraic lattice
$\mathbf{Fix}(J)$ of the fixed points of $J$, and to a morphism $f$
from $(A,J)$ to $(B,K)$ associates the morphism $f^{-1}[\cdot]$ (which
sends a fixed point $Y = K(Y)$ of $K$ to $f^{-1}[Y]$, its inverse
image under $f$) from $\mathbf{Fix}(K)$ to $\mathbf{Fix}(J)$.

Then we have that $\mathrm{Fix}$ is essentially
surjective.  Therefore the contravariant functor
$\mathrm{Fix}\comp \mathrm{Sg}$ from $\mathbf{Alg}(1)$ to
$\mathbf{ALat}_{\bigwedge}$ is essentially surjective.

\begin{remark}
Results similar to the just stated about the theorems
of Birkhoff-Frink, are also valid for a set of sorts $S$ with two or
more sorts, but replacing the category $\mathbf{AClSp}$ by the
category $\mathbf{UAClSp}(S)$, of uniform algebraic $S$-closure spaces,
as in~\cite{cs04}.
\end{remark}

Additional examples of the utility of the categories of the type
$\mathbf{Alg}(S)$ are the category $\mathbf{Mod}_{\mathrm{r}}$ of all (right)
modules over all rings ($\mathbf{Mod}$ is fibered over $\mathbf{Rng}$, the
fiber over each ring $\mathbf{R}$ being the category
$\mathbf{Mod}_{\mathbf{R}}$ of right $\mathbf{R}$-modules), and, generally
speaking, all those categories that can be obtained, essentially, in
the same way as was obtained $\mathbf{Mod}_{\mathrm{r}}$, i.e., starting
from some action of a mathematical construct on another one.



The category $\mathbf{Alg}$ of algebras, as was the case for the
categories $\mathbf{MSet}$ and $\mathbf{Sig}$, is also complete and
cocomplete.  These results are already known, although we are not
aware of a suitably explicit and direct proof, as that provided by us
below, of the cocompleteness of $\mathbf{Alg}$, in what has to do,
particularly, with the existence of a certain left adjoint.

\begin{proposition}\label{Algcomplete}
The category $\mathbf{Alg}$ is complete.
\end{proposition}

\begin{proof}
Let $\mathbf{d}\colon\mathbf{\Sigma}\mor \mathbf{\Lambda}$ be a signature
morphism.  Since the forgetful functors $\mathrm{G}_{\mathbf{\Sigma}}$ and
$\mathrm{G}_{\mathbf{\Lambda}}$ create projective limits
and the following diagram commutes%
$$
\xymatrix@C=70pt{
\mathbf{Alg}(\mathbf{\Sigma})
  \ar[r]^{\mathrm{G}_{\mathbf{\Sigma}}} &
\mathbf{Set}^{S} \\
\mathbf{Alg}(\mathbf{\Lambda})
  \ar[u]^{\mathbf{d}^{\ast}}
  \ar[r]_{\mathrm{G}_{\mathbf{\Lambda}}} &
\mathbf{Set}^{T}
  \ar[u]_{\Delta_{\varphi}}
}
$$
the functor $\mathbf{d}^{\ast}$ preserves projective limits, i.e., is
continuous.  But the category $\mathbf{Sig}$ is complete, and, for
every signature $\mathbf{\Sigma}$, $\mathbf{Alg}(\mathbf{\Sigma})$ is
complete.  Therefore, by Proposition~\ref{LimitesGrt}, the category
$\mathbf{Alg}$ is complete.
\end{proof}

To prove that the category $\mathbf{Alg}$ is cocomplete we begin by proving
that, for every signature morphism $\mathbf{d}\colon\mathbf{\Sigma}\mor
\mathbf{\Lambda}$, the functor $\mathbf{d}^{\ast}$ has a left adjoint
$\mathbf{d}_{\ast}$.

\begin{proposition}\label{leftadjsigmorph}
Let $\mathbf{d}\colon \mathbf{\Sigma}\mor \mathbf{\Lambda}$ be a
signature morphism.  Then there exists a functor
$\mathbf{d}_{\ast}\colon \mathbf{Alg}(\mathbf{\Sigma})\mor
\mathbf{Alg}(\mathbf{\Lambda})$ that is left adjoint to the functor
$\mathbf{d}^{\ast}$.
\end{proposition}

\begin{proof}
We begin by defining the action $\mathbf{d}_{\ast}$ on the objects.  Let
$\mathbf{A}$ be a $\mathbf{\Sigma}$-algebra.  Then
$\mathbf{d}_{\ast}(\mathbf{A})$ is the $\mathbf{\Lambda}$-algebra defined as %
$ \mathbf{T}_{\mathbf{\Lambda}}(\tcoprod_{\varphi}A)/\cl{R}{}^{\mathbf{A}} $, where
$\cl{R}{}^{\mathbf{A}}$ is the congruence on
$\mathbf{T}_{\mathbf{\Lambda}}(\coprod_{\varphi}A)$ generated by the $T$-sorted
relation $R^{\mathbf{A}}$, defined, for every $t\in T$, as
$$
R^{\mathbf{A}}_{t} =
\biggl\{
  \left( (F_{\sigma}^{\mathbf{A}}(a_{i}\mid i\in\bb{w}),s) ,
    d(\sigma)((a_{i},w_{i})\mid i\in\bb{w})
  \right)
  \biggm|
  \begin{gathered}
    s\in \varphi^{-1}[t],\, w\in\fmon{S},\,  \\
    \sigma\in\Sigma_{w,s},\, a\in A_{w}
  \end{gathered}
\biggr\}.
$$

Following this we define the action of $\mathbf{d}_{\ast}$ on the
morphisms.  Let $f$ be a $\mathbf{\Sigma}$-homomorphism from
$\mathbf{A}$ to $\mathbf{A'}$.  Then
$R^{\mathbf{A}}\incl\Ker(\pr_{\cl{R}{}^{\mathbf{A'}}}\comp
(\coprod_{\varphi}f)^{@})$ because, for $t\in T$ and $\left(
(F_{\sigma}(a_{i}\mid i\in\bb{w}),s) , d(\sigma)((a_{i},w_{i})\mid
i\in\bb{w}) \right)\in R^{\mathbf{A}}_{t} $, we have that
\begin{align*}
    [(\tcoprod_{\varphi}f)^{@}
     (F_{\sigma}^{\mathbf{A}}(a_{i}\mid i\in\bb{w}),s)]
    &=
    [(f_{s}(F_{\sigma}^{\mathbf{A}}(a_{i}\mid i\in\bb{w})),s)] \\
    &=
    [d(\sigma)(f_{w_{i}}(a_{i},w_{i})\mid i\in\bb{w})] \\
    &=
    [(\tcoprod_{\varphi}f)^{@}
      ( d(\sigma)((a_{i},w_{i})\mid i\in\bb{w}))].
\end{align*}
From this it follows that there exists a unique $\mathbf{\Lambda}$-homomorphism
$\mathbf{d}_{\ast}(f)$ from
$\mathbf{d}_{\ast}(\mathbf{A})$ to $\mathbf{d}_{\ast}(\mathbf{A'})$ such that
$\mathbf{d}_{\ast}(f)\comp\pr_{\cl{R}{}^{\mathbf{A}}} =
\pr_{\cl{R}{}^{\mathbf{A'}}}\comp (\coprod_{\varphi}f)^{@}$.

After this we prove that the functor $\mathbf{d}_{\ast}$, which sends a
$\mathbf{\Sigma}$-algebra $\mathbf{A}$ to the
$\mathbf{\Lambda}$\nobreakdash-algebra $\mathbf{d}_{\ast}(\mathbf{A}) =
\mathbf{T}_{\mathbf{\Lambda}}(\coprod_{\varphi}A)/\cl{R}{}^{\mathbf{A}}$
and a $\mathbf{\Sigma}$-homo\-morphism $f$ from $\mathbf{A}$ to
$\mathbf{A'}$ to the $\mathbf{\Lambda}$-homo\-morphism
$\mathbf{d}_{\ast}(f)$ from $\mathbf{d}_{\ast}(\mathbf{A})$ to
$\mathbf{d}_{\ast}(\mathbf{A'})$, is left adjoint to $\mathbf{d}^{\ast}$.

Let $\mathbf{A}$ be a $\mathbf{\Sigma}$-algebra.  Then we denote by
$\eta^{\mathbf{A}}$ the $S$-sorted mapping from $A$ to the underlying
$S$-sorted set of the $\mathbf{\Sigma}$-algebra
$\mathbf{d}^{\ast}(\mathbf{d}_{\ast}(\mathbf{A}))$, obtained as the inverse
image under $\theta^{\varphi}$ (the natural isomorphism of the
adjunction $\coprod_{\varphi}\ladj \Delta_{\varphi}$) of the
$T$\nobreakdash-sorted mapping
$$
\xymatrix{
\coprod_{\varphi}A
  \ar[r]^{\eta_{\coprod_{\varphi}A}} &
\mathrm{T}_{\mathbf{\Lambda}}(\coprod_{\varphi}A)
  \ar[r]^{\pr_{\cl{R}{}^{\mathbf{A}}}} &
\mathrm{T}_{\mathbf{\Lambda}}(\coprod_{\varphi}A)/\cl{R}{}^{\mathbf{A}}.
}
$$ %

The $S$-sorted mapping $\eta^{\mathbf{A}}$ is a
$\mathbf{\Sigma}$-homomorphism from $\mathbf{A}$ to
$\mathbf{d}^{\ast}(\mathbf{d}_{\ast}(\mathbf{A}))$.  Let $\sigma\colon w\mor
s$ be a formal operation and $a\in A_{w}$, then we have that
\begin{align*}
    \eta^{\mathbf{A}}_{s}(F^{\mathbf{A}}_{\sigma}(a_{i}\mid i\in\bb{w})) &=
    [(F^{\mathbf{A}}_{\sigma}(a_{i}\mid i\in\bb{w}),s)], \\
\intertext{ and, on the other hand, we also have that }
    F^{\mathbf{d}^{\ast}(\mathbf{d}_{\ast}(\mathbf{A}))}_{\sigma}
      (\eta^{\mathbf{A}}_{w_{i}}(a_{i})\mid i\in\bb{w})
    &=
    F^{\mathbf{d}_{\ast}(\mathbf{A})}_{d(\sigma)}([(a_{i},w_{i})]\mid i\in\bb{w}) \\
    &=
    [F^{\mathbf{T}_{\mathbf{\Lambda}}(\coprod_{\varphi}A)}_{d(\sigma)}
      ((a_{i},w_{i})\mid i\in\bb{w})] \\
    &=
    [d(\sigma)((a_{i},w_{i})\mid i\in\bb{w})],
\end{align*}
but, by the definition of $R^{\mathbf{A}}_{\varphi(s)}$,
$[(F^{\mathbf{A}}_{\sigma}(a_{i}\mid i\in\bb{w}),s)]=
[d(\sigma)((a_{i},w_{i})\mid i\in\bb{w})]$. Therefore we can assert
that
$$
\eta^{\mathbf{A}}_{s}(F^{\mathbf{A}}_{\sigma}(a_{i}\mid i\in\bb{w})) =
F^{\mathbf{d}^{\ast}(\mathbf{d}_{\ast}(\mathbf{A}))}_{\sigma}
      (\eta^{\mathbf{A}}_{w_{i}}(a_{i})\mid i\in\bb{w}).
$$

Finally, we prove the universal property.  For this, let $\mathbf{B}$
be a $\mathbf{\Lambda}$-algebra, $f$ a $\mathbf{\Sigma}$-homomorphism
from $\mathbf{A}$ to $\mathbf{d}^{\ast}(\mathbf{B})$, and $\mawh{f}$
the $T$-sorted mapping associated to the $S$-sorted mapping $f\colon
A\mor B_{\varphi}$.  Then there exists a unique
$\mathbf{\Lambda}$-homomorphism $f^{\natural}$ such
that the right triangle in the following diagram commutes %
$$
\xymatrix{
\coprod_{\varphi}A
  \ar[r]^{\eta_{\coprod_{\varphi}A}}
  \ar[rd]_{\mawh{f}} &
\mathrm{T}_{\mathbf{\Lambda}}(\coprod_{\varphi}A)
  \ar[r]^{\pr_{\cl{R}{}^{\mathbf{A}}}}
  \ar[d]^{\ext{\mawh{f}}} &
\mathrm{T}_{\mathbf{\Lambda}}(\coprod_{\varphi}A)/\cl{R}{}^{\mathbf{A}}
  \ar[ld]^{f^{\natural}} \\
 &
B
}
$$
since, for every $t\in T$, from $\bigl(
(F^{\mathbf{A}}_{\sigma}(a_{i}\mid i\in\bb{w}),s) ,
d(\sigma)((a_{i},w_{i})\mid i\in\bb{w}) \bigl)\in R^{\mathbf{A}}_{t}$
it follows that
\begin{align*}
    \ext{\mawh{f}}_{t}(F^{\mathbf{A}}_{\sigma}(a_{i}\mid i\in\bb{w}),s) &=
    f_{s}(F^{\mathbf{A}}_{\sigma}(a_{i}\mid i\in\bb{w})) \\
    &=
    F^{\mathbf{B}}_{d(\sigma)}
      (f_{w_{i}}(a_{i})\mid i\in\bb{w}),s) \\
    &=
    F^{\mathbf{B}}_{d(\sigma)}
      (\ext{\mawh{f}}_{\varphi(w_{i})}(a_{i},w_{i})\mid i\in\bb{w}),s) \\
    &=
    \ext{\mawh{f}}_{t}(d(\sigma)((a_{i},w_{i})\mid i\in\bb{w})),
\end{align*}
therefore, $R^{\mathbf{A}}\incl\Ker(\ext{\mawh{f}})$.

Since $\mathbf{d}^{\ast}$ is a functor,
$\mathbf{d}^{\ast}(f^{\natural})$ is a $\mathbf{\Sigma}$-homomorphism.
Besides, for every $s\in S$ and
$a\in A_{s}$, we have that  %
\begin{align*}
    \mathbf{d}^{\ast}(f^{\natural})_{s}(\eta^{\mathbf{A}}_{s}(a)) &=
    f^{\natural}_{\varphi(s)}([(a,s)])\\
    &=
    \ext{\mawh{f}}_{\varphi(s)}(a,s) \\
    &=
    \ext{\mawh{f}}_{\varphi(s)}(\eta_{\coprod_{\varphi}A,\varphi(s)}(a)) \\
    &=
    \mawh{f}_{\varphi(s)}(a)\\
    &=
    f_{s}(a),
\end{align*}
hence $f = \mathbf{d}^{\ast}(f^{\natural})\comp \eta^{\mathbf{A}}$.

It is obvious that $f^{\natural}$ is the unique
$\mathbf{\Lambda}$-homomorphism from $\mathbf{d}_{\ast}(\mathbf{A})$
into $\mathbf{B}$ such that the above diagram commutes, hence
$\mathbf{d}_{\ast}\ladj\mathbf{d}^{\ast}$.
\end{proof}

\begin{proposition}\label{Algcocomplete}
The category $\mathbf{Alg}$ is cocomplete.
\end{proposition}

\begin{proof}
The category $\mathbf{Sig}$ is cocomplete. For every signature
$\mathbf{\Sigma}$, the category $\mathbf{Alg}(\mathbf{\Sigma})$ is
cocomplete. The functor $\Alg$ is locally reversible.  Therefore,
by Proposition~\ref{CoLimitesGrt}, the category $\mathbf{Alg}$ is
cocomplete.
\end{proof}

From Propositions~\ref{Algcomplete} $\!\!\And\!\!$ \ref{Algcocomplete}
we obtain immediately the following

\begin{corollary}
The category $\mathbf{Alg}$ is bicomplete.
\end{corollary}



The contravariant functor $\Alg$ from $\mathbf{Sig}$ to $\mathbf{Cat}$
is not only useful to construct the category $\mathbf{Alg}$, actually,
as we prove in what follows, it, together with a pseudo-functor
$\mathrm{Ter}$ from $\mathbf{Sig}$ to $\mathbf{Cat}$, and a
pseudo-extranatural transformation $(\mathrm{Tr},\theta)$ (from a
pseudo-functor on $\mathbf{Sig}^{\mathrm{op}}\times \mathbf{Sig}$ to
$\mathbf{Cat}$, induced by $\Alg$ and $\mathrm{Ter}$, to the functor,
between the same categories, constantly $\mathbf{Set}$), enables us to
construct a new institution on $\mathbf{Set}$, the so-called
many-sorted term institution, denoted by $\mathfrak{Tm} =
(\mathbf{Sig},\mathrm{Alg},\mathrm{Ter},(\mathrm{Tr},\theta))$, but,
we point out, for a concept of institution that is \emph{strictly}
more general than that of generalized $\mathbf{V}$-institution
in~\cite{gb86}.  The institution $\mathfrak{Tm}$ can be qualified of
basic, or fundamental, among others, by the following reasons:
\begin{enumerate}
\item It embodies, in a coherent way, algebras, terms, and the natural
      process of realization of terms as term operations in algebras,
      and

\item The many-sorted equational institution and the many-sorted
      specification institution (both to be defined in
      the third section), i.e., the core of universal algebra, are
      built on it.
\end{enumerate}

It happens that, in the institution $\mathfrak{Tm}$ on $\mathbf{Set}$,
the existence of the pseudo-functor $\mathrm{Ter}$ follows from the
fact that, for any signature $\mathbf{\Sigma}$, the terms for
$\mathbf{\Sigma}$, understood in a generalized sense to be explained
below, have a category-theoretical interpretation as the morphisms of
a category $\mathbf{Ter}(\mathbf{\Sigma})$.  Furthermore, the
component $\mathrm{Tr}$ of the pseudo-extranatural transformation
$(\mathrm{Tr},\theta)$, in $\mathfrak{Tm}$, depends for its existence
on the fact that the generalized terms have associated generalized
term operations on the algebras.  And, finally, as it could not be
otherwise, it is the case that the properties of the generalized
terms, resp., generalized term operations on the algebras, are
function of those of the ordinary terms, resp., term operations on the
algebras.

Therefore, to proceed properly, we should begin by defining, for a
$\mathbf{\Sigma}$-algebra $\mathbf{A}$ and an $S$-sorted set $X$, the
concept of many-sorted $X$-ary operation on $\mathbf{A}$, that of
many-sorted $X$-ary term operation on $\mathbf{A}$, and, as an
immediate consequence of the universal property of the free algebras,
the procedure of realization of terms $P$ of type $(X,s)$ as term
operations $P^{\mathbf{A}}$ on $\mathbf{A}$ (i.e., the passage from a
formal operation $P$, constructed from variables and formal
operations, to their realization as a true, or substantial, operation
$P^{\mathbf{A}}$ on the algebra $\mathbf{A}$, that transforms
valuations of the variables $X$ in the underlying $\mathrm{ms}$-set
$A$ of $\mathbf{A}$, into elements, of the adequate sort, of
$\mathbf{A}$).

\begin{definition}
Let $X$ be an $S$-sorted set, $\mathbf{A}$ a
$\mathbf{\Sigma}$-algebra, $s\in S$ and $P\in
\mathrm{T}_{\mathbf{\Sigma}}(X)_{s}$ a term for $\mathbf{\Sigma}$ of
type $(X,s)$.  Then
\begin{enumerate}
\item The $\mathbf{\Sigma}$-algebra of the \emph{many-sorted}
      $X$-\emph{ary} \emph{operations on} $\mathbf{A}$,
      $\mathbf{Op}_{X}(\mathbf{A})$, is $\mathbf{A}^{A_{X}}$, i.e.,
      the direct $A_{X}$-power of $\mathbf{A}$, where $A_{X}$ is
      $\mathrm{Hom}(X,A)$, the (ordinary) set of the $S$-sorted
      mappings from $X$ to $A$.  From now on, to shorten terminology,
      we will speak of $X$-\emph{ary} \emph{operations on}
      $\mathbf{A}$ instead of \emph{many-sorted} $X$-\emph{ary}
      \emph{operations on} $\mathbf{A}$.

\item The $\mathbf{\Sigma}$-algebra of the \emph{many-sorted}
      $X$-\emph{ary} \emph{term operations on} $\mathbf{A}$,
      $\mathbf{Ter}_{X}(\mathbf{A})$, is the subalgebra of
      $\mathbf{Op}_{X}(\mathbf{A})$ generated by the subfamily
      $$
      \mathcal{P}^{A}_{X} = (\mathcal{P}^{A}_{X,s})_{s\in S} =
      (\{\,\pr^{A}_{X,s,x} \mid x\in X_{s}\,\})_{s\in S}
      $$
      of $\Op_{X}(A) = A^{A_{X}}$, where, for every $s\in S$ and $x\in
      X_{s}$, $\pr^{A}_{X,s,x}$ is the mapping from $A_{X}$ to $A_{s}$
      which sends $a\in A_{X}$ to $a_{s}(x)$.  From now on, to shorten
      terminology, we will speak of $X$-\emph{ary} \emph{term operations on}
      $\mathbf{A}$ instead of \emph{many-sorted} $X$-\emph{ary}
      \emph{term operations on} $\mathbf{A}$.

\item We denote by $\mathrm{Tr}^{X,\mathbf{A}}$ the unique
      $\mathbf{\Sigma}$-homomorphism from $\mathbf{T}_{\mathbf{\Sigma}}(X)$ to
      $\mathbf{Op}_{X}(\mathbf{A})$ such that $\pr^{A}_{X} =
      \mathrm{Tr}^{X,\mathbf{A}}\comp \eta_{X}$, where $\pr^{A}_{X}$
      is the $S$-sorted mapping $(\pr^{A}_{X,s})_{s\in S}$ from $X$ to
      $\mathrm{Op}_{X}(A)$, with $\pr^{A}_{X,s} =
      (\pr^{A}_{X,s,x})_{x\in X_{s}}$, for every $s\in S$.  Furthermore,
      $P^{\mathbf{A}}$ denotes the image of $P$ under
      $\mathrm{Tr}^{X,\mathbf{A}}_{s}$, and we call the mapping
      $P^{\mathbf{A}}$ from $A_{X}$ to $A_{s}$, the \emph{term
      operation on} $\mathbf{A}$ \emph{determined by} $P$, or the
      \emph{term realization of} $P$ \emph{on} $\mathbf{A}$.
\end{enumerate}

\end{definition}

We recall that, for the insertion of generators $\eta_{X}\colon X\mor
\mathrm{T}_{\mathbf{\Sigma}}(X)$,
$\mathrm{Tr}^{X,\mathbf{A}}[\eta_{X}[X]]$, the (direct) image of
$\eta_{X}[X]$ under $\mathrm{Tr}^{X,\mathbf{A}}$, is also
$\mathbf{Ter}_{X}(\mathbf{A})$, i.e., the term operations on an
algebra are the same as the operations determined by the terms built
from variables and formal operations denoting the primitive operations
of the algebra.  From now on, to simplify the notation, we will also
denote by $\mathrm{Tr}^{X,\mathbf{A}}$ the co-restriction of the
$\mathbf{\Sigma}$-homomorphism $\mathrm{Tr}^{X,\mathbf{A}}$ from
$\mathbf{T}_{\mathbf{\Sigma}}(X)$ to $\mathbf{Op}_{X}(\mathbf{A})$ to
the subalgebra $\mathbf{Ter}_{X}(\mathbf{A})$ of
$\mathbf{Op}_{X}(\mathbf{A})$.

\begin{remark}
What we have called \emph{term operations on} $\mathbf{A}$ are also known,
for those following the terminology in Gr\"{a}tzer~\cite{gG79}, pp.
37--45, and J\'{o}nsson~\cite{bJ72}, pp.  83--87, as \emph{polynomial
operations of} $\mathbf{A}$, and, for those following that one in
Cohn~\cite{pC81}, pp.  145--149, as \emph{derived operators of}
$\mathbf{A}$.
\end{remark}

\begin{remark}
If for an $S$-sorted set $X$, a $\mathbf{\Sigma}$-algebra
$\mathbf{A}$, and an $S$-sorted subset $M$ of $A$, we define the
$\mathbf{\Sigma}$-algebra of the $X$-\emph{ary term operations with
constants in} $M$ \emph{on} $\mathbf{A}$, denoted by
$\mathbf{Ter}_{X}(\mathbf{A},M)$, as the subalgebra of
$\mathbf{Op}_{X}(\mathbf{A})$ generated by the subfamily $
\mathcal{P}^{A}_{X}\cup \mathcal{K}^{A,M}_{X} $ of $\Op_{X}(A) =
A^{A_{X}}$, where
$$
\mathcal{K}^{A,M}_{X} = (\mathcal{K}^{A,M}_{X,s})_{s\in S}
= (\{\,\kappa^{a}_{X,s} \mid a\in M_{s}\,\})_{s\in S},
$$
and, for every $s\in S$ and $a\in M_{s}$, $\kappa^{a}_{X,s}$ is the
mapping from $A_{X}$ to $A_{s}$ that is constantly $a$, then we have
that $\mathbf{Ter}_{X}(\mathbf{A})$ is
$\mathbf{Ter}_{X}(\mathbf{A},(\vacio)_{s\in S})$, where
$(\vacio)_{s\in S}$ is the $S$-sorted set that is constantly $\vacio$.
And, using the terminology in J\'{o}nsson~\cite{bJ72}, pp.
87--89, the $\mathbf{\Sigma}$-algebra of the $X$-\emph{ary algebraic
operations on} $\mathbf{A}$, denoted by
$\mathbf{Alg}_{X}(\mathbf{A})$, is $\mathbf{Ter}_{X}(\mathbf{A},A)$.
\end{remark}

We point out that since the above concepts are defined for arbitrary
$\mathrm{ms}$-sets, they are also applicable, in particular, for a
given set of sorts $S$ and an arbitrary, but fixed, $S$-sorted set of
variables $V^{S} = (V^{S}_{s})_{s\in S}$, where, for every $s\in S$,
$V^{S}_{s} = \{\,v^{s}_{n}\mid n\in \mathbb{N}\,\}$ is an effectively
enumerated set, to the \emph{finite} $S$-sorted subsets $\vs{w}$ of
$V^{S}$ associated to the words $w\in\fmon{S}$,  where, for every word
$w\in \fmon{S}$, we agree that $\vs{w}$ is the finite subset of
$V^{S}$ defined, for every $s\in S$, as $(\vs{w})_{s} =
\{\,v^{s}_{i}\mid i\in w^{-1}[s]\,\}$ (observe that $(\vs{w})_{s}$ is
empty for those sorts $s$ that do not occur in the word $w$).

In all that follows, every proposition relative to the above
concepts will only be stated for arbitrary $\mathrm{ms}$-sets,
therefore the corresponding propositions for the finite $S$-sorted
subsets $\vs{w}$ of $V^{S}$ will not be actually stated and will
remain tacit. However, to prove the just mentioned implicit
propositions, it will be shown to be useful to know that, for a
word $w\in \fmon{S}$, a mapping $\varphi\colon S\mor T$, and its
extension $\fmon{\varphi}\colon\mathbf{T}_{\star}(S)\mor
\mathbf{T}_{\star}(T)$ to the corresponding free monoids on $S$
and $T$, the $S$-sorted set $\vs{w}$ can be embedded in the
$S$-sorted set $(\vs{\fmon{\varphi}}(w))_{\varphi}$, associated to
the $T$-sorted set $\vs{\fmon{\varphi}}(w)\cong
\coprod_{\varphi}(\vs{w})$, through the $S$-sorted mapping
$\inc^{w,\varphi}$ defined, for every $s\in S$, as follows
$$
\inc^{w,\varphi}_{s}
\nfunction
{(\vs{w})_{s}}
{(\vs{\fmon{\varphi}}(w))_{\varphi(s)}}
{v^s_{i}}{v^{\varphi(s)}_{i}}
$$
and that this embedding has as an immediate consequence that a
signature morphism $\mathbf{d}\colon\mathbf{\Sigma}\mor
\mathbf{\Lambda}$, determines a morphism
$$
(\varphi,d,\inc^{w,\varphi})\colon(S,\Sigma,\vs{w})\mor
(T,\Lambda,\vs{\fmon{\varphi}(w)})
$$
in $\mathbf{MSet}\times_{\mathbf{Set}}\mathbf{Sig}$, hence that the $s$-th
component of the $\mathbf{\Sigma}$-homomorphism
$$
(\inc^{w,\varphi})^{\mathbf{d}}\colon\mathbf{T}_{\mathbf{\Sigma}}(\vs{w})\mor
\mathbf{T}_{\mathbf{\Lambda}}(\vs{\fmon{\varphi}(w)})_{\varphi}
$$
translates terms for $\mathbf{\Sigma}$ of type $(\vs{w},s)$ into terms for
$\mathbf{\Lambda}$ of type $(\vs{\fmon{\varphi}(w)},\varphi(s))$.

For completeness we recall that for many-sorted terms, as for
single-sorted terms, we also have that
\begin{enumerate}
\item The \emph{exchange law} is valid, i.e., that given a
      valuation $a\colon X\mor A$, where $X$ is an $S$-sorted set and
      $A$ the underlying set of a $\mathbf{\Sigma}$-algebra
      $\mathbf{A}$, and a term $P$ for $\mathbf{\Sigma}$ of type
      $(X,s)$, we always have the equality $a^{\sharp}_{s}(P) =
      P^{\mathbf{A}}(a)$; and that

\item The $\mathbf{\Sigma}$-homo\-mor\-phisms \emph{commute}
      with term operations, i.e., that given a
      $\mathbf{\Sigma}$-homo\-mor\-phism $u\colon \mathbf{A}\mor
      \mathbf{B}$ and a term $P$ for $\mathbf{\Sigma}$ of type
      $(X,s)$, we always have the equality $u_{s}\comp P^{\mathbf{A}}
      = P^{\mathbf{B}}\comp u_{X}$.
\end{enumerate}

Following this we state the fundamental facts about term operations of
different arities on the same algebra.  These facts are, actually, the
generalization to the many-sorted case and categorization of some of
those stated by Schmidt in~\cite{jSch61}, pp.  107--109.

\begin{proposition}\label{Schmidt}
Let $\mathbf{A}$ be a $\mathbf{\Sigma}$-algebra and $f\colon X\mor Y$
an $S$-sorted mapping.  Then there exists a unique
$\mathbf{\Sigma}$-homomorphism $\Ter_{f}(\mathbf{A})$ from
$\mathbf{Ter}_{X}(\mathbf{A})$ to $\mathbf{Ter}_{Y}(\mathbf{A})$ such
that the following diagram commutes
$$
\xymatrix@C=14ex@R=8ex{
\mathbf{T}_{\mathbf{\Sigma}}(X)
\ar[r]^-{\mathrm{Tr}^{X,\mathbf{A}}}
\ar[d]_{f^{@}} &
\mathbf{Ter}_{X}(\mathbf{A})
\ar[d]^{\Ter_{f}(\mathbf{A})} \\
\mathbf{T}_{\mathbf{\Sigma}}(Y)
\ar[r]_{\mathrm{Tr}^{Y,\mathbf{A}}} &
\mathbf{Ter}_{Y}(\mathbf{A})
}
$$
Besides, we have that
\begin{enumerate}
\item For every $S$-sorted set $X$, it is the case that
      $$
      \Ter_{\mathrm{id}_{X}}(\mathbf{A}) =
      \mathrm{id}_{\mathbf{Ter}_{X}(\mathbf{A})}.
      $$

\item If $g\colon Y\mor Z$ is another $S$-sorted mappings, then
      $$
      \Ter_{g\comp f}(\mathbf{A})
      = \Ter_{g}(\mathbf{A})\comp \Ter_{f}(\mathbf{A}).
      $$
\end{enumerate}
\end{proposition}

\begin{proof}
As we know $\mathbf{Ter}_{X}(\mathbf{A})$ is the subalgebra of
$\mathbf{Op}_{X}(\mathbf{A})$ generated by the subfamily
$\mathcal{P}^{A}_{X}$ of $\Op_{X}(A)$.  Therefore, to prove
that there is some $\mathbf{\Sigma}$-homomorphism from
$\mathbf{Ter}_{X}(\mathbf{A})$ to $\mathbf{Ter}_{Y}(\mathbf{A})$ it
will be enough to prove that there is a $\mathbf{\Sigma}$-homo\-morph\-ism
$\mathrm{Op}_{f}(\mathbf{A})$ from $\mathbf{Op}_{X}(\mathbf{A})$ to
$\mathbf{Op}_{Y}(\mathbf{A})$ such that
$\mathrm{Op}_{f}(\mathbf{A})[\mathcal{P}^{A}_{X}]\subseteq
\mathcal{P}^{A}_{Y}$.

Let $\mathrm{Op}_{f}(\mathbf{A})$ be the $S$-sorted mapping from
$\mathrm{Op}_{X}(A)$ to $\mathrm{Op}_{Y}(A)$ whose
$s$-th coordinate mapping $\mathrm{Op}_{f}(\mathbf{A})_{s}$, for $s\in
S$, sends a mapping $P$ in $A^{A_{X}}_{s}$ to the mapping $P\comp
A_{f}$ in $A^{A_{Y}}_{s}$, where $A_{f}$ is the mapping from $A_{Y}$
to $A_{X}$ which assigns to an $S$\nobreakdash-sorted mapping $u$ in $A_{Y}$ the
$S$-sorted mapping $u\comp f$ in $A_{X}$.  Thus defined
$\mathrm{Op}_{f}(\mathbf{A})$ is a $\mathbf{\Sigma}$-homomorphism from
$\mathbf{Op}_{X}(\mathbf{A})$ to $\mathbf{Op}_{Y}(\mathbf{A})$.

Furthermore, for every $s\in S$, the action of
$\mathrm{Op}_{f}(\mathbf{A})_{s}$ on a generator $\pr^{A}_{X,s,x}$ of
$\mathbf{Ter}_{X}(\mathbf{A})$ is the mapping
$\mathrm{Op}_{f}(\mathbf{A})_{s}(\pr^{A}_{X,s,x}) =
\pr^{A}_{X,s,x}\comp A_{f}$ from $A_{Y}$ to $A_{s}$.  But for an
$S$-sorted mapping $u$ in $A_{Y}$, we have that
$(\mathrm{Op}_{f}(\mathbf{A})_{s}(\pr^{A}_{X,s,x}))(u) =
u_{s}(f_{s}(x))$, therefore
$\mathrm{Op}_{f}(\mathbf{A})_{s}(\pr^{A}_{X,s,x}) =
\pr^{A}_{Y,s,f_{s}(x)}$, that is a member of $\mathcal{P}^{A}_{Y,s}$,
the $s$\nobreakdash-th component of the set of generators
$\mathcal{P}^{A}_{Y}$ of $\mathbf{Ter}_{Y}(\mathbf{A})$.  From this we
can assert that $\mathrm{Op}_{f}(\mathbf{A})$ has a bi-restriction,
$\Ter_{f}(\mathbf{A})$, to $\mathbf{Ter}_{X}(\mathbf{A})$ and
$\mathbf{Ter}_{Y}(\mathbf{A})$.  Thus defined $\Ter_{f}(\mathbf{A})$
is the unique $\mathbf{\Sigma}$-homomorphism from
$\mathbf{Ter}_{X}(\mathbf{A})$ to $\mathbf{Ter}_{Y}(\mathbf{A})$ such
that $\Ter_{f}(\mathbf{A})\comp \mathrm{Tr}^{X,\mathbf{A}} =
\mathrm{Tr}^{Y,\mathbf{A}}\comp f^{@}$.

The remaining properties follow easily from the definition of
$\Ter_{f}(\mathbf{A})$.
\end{proof}

The proposition just stated can be interpreted as meaning that, for a
$\mathbf{\Sigma}$-algebra $\mathbf{A}$, we have
\begin{enumerate}
\item A functor $\Ter_{(\cdot)}(\mathbf{A})$ from $\mathbf{Set}^{S}$
      to $\mathbf{Alg}(\mathbf{\Sigma})$ which sends an $S$-sorted set $X$ to
      the $\mathbf{\Sigma}$-algebra $\mathbf{Ter}_{X}(\mathbf{A})$, and an
      $S$-sorted mapping $f$ from $X$ to $Y$ to the
      $\mathbf{\Sigma}$-homomorphism $\Ter_{f}(\mathbf{A})$ from
      $\mathbf{Ter}_{X}(\mathbf{A})$ to
      $\mathbf{Ter}_{Y}(\mathbf{A})$, and

\item A natural transformation
      $\mathrm{Tr}^{(\cdot),\mathbf{A}}$ from
      $\mathbf{T}_{\mathbf{\Sigma}}$ to $\Ter_{(\cdot)}(\mathbf{A})$
      which sends an $S$\nobreakdash-sorted set $X$ to the
      $\mathbf{\Sigma}$-homomorphism $\mathrm{Tr}^{X,\mathbf{A}}$ from
      $\mathbf{T}_{\mathbf{\Sigma}}(X)$ to
      $\mathbf{Ter}_{X}(\mathbf{A})$,
\end{enumerate}
summarized in the following diagram
$$
\xymatrix@C=20ex@R=8ex{
\mathbf{Set}^{S}
\ar@/^15pt/[r]^{\mathbf{T}_{\mathbf{\Sigma}}}="f"
  \ar@/^-15pt/[r]_{\Ter_{(\cdot)}(\mathbf{A})}="g" &
\mathbf{Alg}(\mathbf{\Sigma}).
\ar @{} "f";"g" |{\dir{=>}}^{\,\mathrm{Tr}^{(\cdot),\mathbf{A}}}
}
$$

What we want to prove now is the compatibility between the
translation of terms and their realization as term operations on
the algebras.  But for this it will be shown to be useful to take
into account the following auxiliary functors and natural
transformation.

\begin{definition}
For a mapping $\varphi\colon S\mor T$, an $S$-sorted set $X$, a
$T$-sorted set $Y$, and an $S$-sorted mapping $f\colon X\mor
Y_{\varphi}$, we have the following functors and natural
transformation
\begin{enumerate}
\item $\mathrm{H}(Y,\cdot)$ is the covariant hom-functor from
      $\mathbf{Set}^{T}$ to $\mathbf{Set}$ which, we recall, sends a
      $T$-sorted set $A$ to the set $\mathrm{H}(Y,\cdot)(A) = A_{Y}$,
      and a $T$-sorted mapping $u$ from $A$ to $B$ to the mapping
      $\mathrm{H}(Y,\cdot)(u)$ from $A_{Y}$ to $B_{Y}$ which assigns
      to a $T$-sorted mapping $t$ from $Y$ to $A$ the mapping $u\comp
      t$ from $Y$ to $B$.

\item $\mathrm{H}(X,\cdot)\comp \Delta_{\varphi}$ is the functor
      from $\mathbf{Set}^{T}$ to $\mathbf{Set}$ which sends a
      $T$-sorted set $A$ to the set $(A_{\varphi})_{X}$, and a
      $T$-sorted mapping $u$ from $A$ to $B$ to the mapping
      $\mathrm{H}(X,\cdot)(u_{\varphi})$ from $(A_{\varphi})_{X}$ to
      $(B_{\varphi})_{X}$ which assigns to an $S$-sorted mapping
      $\ell$ from $X$ to $A_{\varphi}$ the mapping $u_{\varphi}\comp
      \ell$ from $X$ to $B_{\varphi}$.

\item $\vartheta^{\varphi,f}$ is the natural transformation from
      $\mathrm{H}(Y,\cdot)$ to $\mathrm{H}(X,\cdot)\comp
      \Delta_{\varphi}$, as in the following diagram
      $$
      \xymatrix@C=20ex@R=8ex{
      \mathbf{Set}^{T}
      \ar@/^15pt/[r]^{\mathrm{H}(Y,\cdot)}="f"
      \ar@/^-15pt/[r]_{\mathrm{H}(X,\cdot)\comp\Delta_{\varphi}}="g" &
      \mathbf{Set}
      \ar @{} "f";"g" |{\dir{=>}}^{\,\vartheta^{\varphi,f}}
      }
      $$
      which sends a $T$-sorted set $A$ to the
      mapping $\vartheta^{\varphi,f}_{A}$ from $A_{Y}$ to
      $(A_{\varphi})_{X}$ which assigns to a morphism $t\colon Y\mor
      A$ in $A_{Y}$ the morphism $t_{\varphi}\comp f$ in
      $(A_{\varphi})_{X}$.
\end{enumerate}
\end{definition}

From this definition, for a $T$-sorted set $A$, we get the
$S$-sorted mapping $\Upsilon^{\varphi,f}_{A}$ from
$\Op_{X}(A_{\varphi}) = A_{\varphi}^{(A_{\varphi})_{X}}$ to
$\Op_{Y}(A)_{\varphi} = (A^{A_{Y}})_{\varphi} = A^{A_{Y}}_{\varphi}$
which sends, for $s\in S$, a mapping $a\colon (A_{\varphi})_{X}\mor
A_{\varphi(s)}$ to the mapping $a\comp \vartheta^{\varphi,f}_{A}\colon
A_{Y}\mor A_{\varphi(s)}$, that we will use in the proof of the
following proposition and corollary.

\begin{proposition}\label{RealizacionNatural}
Let $(\varphi,d,f)\colon(S,\Sigma,X)\mor (T,\Lambda,Y)$ be a morphism
in the category $\mathbf{MSet}\times_{\mathbf{Set}}\mathbf{Sig}$.
Then, for every $\mathbf{\Lambda}$-algebra $\mathbf{A}$ and term $P\in
\mathrm{T}_{\mathbf{\Sigma}}(X)_{s}$ for $\mathbf{\Sigma}$ of type
$(X,s)$, the following diagram commutes
$$
\xymatrix@C=19ex@R=8ex{
(A_{\varphi})_{X}
  \ar[r]^{P^{\mathbf{d}^{\ast}(\mathbf{A})}} &
A_{\varphi(s)}
  \ar@{=}[d]\\
A_{Y}
  \ar[u]^{\vartheta^{\varphi,f}_{A}}
  \ar[r]_{f^{\mathbf{d}}_{s}(P)^{\mathbf{A}}} &
A_{\varphi(s)}
}
$$

\end{proposition}

\begin{proof}
Let $a\in A_{Y}$ be a $T$-sorted mapping from $Y$ to $A$.  Then the
following diagram commutes
$$
\xymatrix@C=70pt{
X \ar[r]^-{\eta_{X}}
  \ar[d]_{f} &
\mathrm{T}_{\mathbf{\Sigma}}(X)
  \ar[d]^{f^{\mathbf{d}}}
  \ar`r[dd]+/r5pc/ `[dd]^{\ext{(a_{\varphi}\comp f)}} [dd]   \\
Y_{\varphi}
  \ar[r]^-{(\eta_{Y})_{\varphi}}
  \ar[rd]_-{a_{\varphi}} &
\mathrm{T}_{\mathbf{\Lambda}}(Y)_{\varphi}
  \ar[d]^{(\ext{a})_{\varphi}}  \\
&
A_{\varphi}
}
$$
hence, for every $P\in \mathrm{T}_{\mathbf{\Sigma}}(X)_{s}$, we have that
\begin{align*}
    f^{\mathbf{d}}_{s}(P)^{\mathbf{A}}(a) &=
    (\ext{a})_{\varphi(s)}\comp f^{\mathbf{d}}_{s}(P) \\
    &= \ext{(a_{\varphi}\comp f)}_{s}(P) \\
    &= P^{\mathbf{d}^{\ast}(\mathbf{A})}(a_{\varphi}\comp f) \\
    &= P^{\mathbf{d}^{\ast}(\mathbf{A})}\comp\vartheta^{\varphi,f}_{A}(a).
\end{align*}
Therefore $f^{\mathbf{d}}_{s}(P)^{\mathbf{A}} =
P^{\mathbf{d}^{\ast}(\mathbf{A})}\comp\vartheta^{\varphi,f}_{A}$, as
asserted.
\end{proof}


We gather in the following corollary some useful consequences of the last
proposition.

\begin{corollary}\label{fundPdextranatural}
Let $(\varphi,d,f)\colon (S,\Sigma,X)\mor(T,\Lambda,Y)$ be a morphism
in the category $\mathbf{MSet}\times_{\mathbf{Set}}\mathbf{Sig}$, $\mathbf{A}$ a
$\mathbf{\Lambda}$-algebra, and $P\in \mathrm{T}_{\mathbf{\Sigma}}(X)_{s}$ a
term for $\mathbf{\Sigma}$ of type $(X,s)$.  Then we
have that
\begin{enumerate}
\item The following diagrams commute
      $$
      \begin{aligned}
      \xymatrix@C=14ex@R=8ex{
      \mathbf{T}_{\mathbf{\Sigma}}(X)
      \ar[r]^-{\mathrm{Tr}^{X,\mathbf{d}^{\ast}(\mathbf{A})}}
      \ar[d]_{f^{\mathbf{d}}} &
      \mathbf{Ter}_{X}(\mathbf{d}^{\ast}(\mathbf{A}))
      \ar[d]^{\Upsilon^{\varphi,f}_{A}} \\
      \mathbf{T}_{\mathbf{\Lambda}}(Y)_{\varphi}
      \ar[r]_{\mathrm{Tr}^{Y,\mathbf{A}}_{\varphi}} &
      \mathbf{Ter}_{Y}(\mathbf{A})_{\varphi}
      }
      \end{aligned}
      \quad \quad\;\;
      \begin{aligned}
      \xymatrix@C=14ex@R=8ex{
      (A_{\varphi})_{X}
      \ar[r]^{P^{\mathbf{d}^{\ast}(\mathbf{A})}}
      &
      A_{\varphi(s)}
      \ar@{=}[d]\\
      A_{\coprod_{\varphi}\!X}
      \ar[r]_{\eta^{\mathbf{d}}_{X,s}(P)^{\mathbf{A}}}
      \ar[u]^{\theta^{\varphi}_{X,A}}  &
      A_{\varphi(s)}
      }
      \end{aligned}
      $$

\item If $(\varphi,d,g)\colon (S,\Sigma,X')\mor(T,\Lambda,Y')$
      is another morphism in the category
      $\mathbf{MSet}\times_{\mathbf{Set}}\mathbf{Sig}$,
      $k$ an $S$-sorted mapping from $X$ to $X'$, and $\ell$ a
      $T$-sorted mapping from $Y$ to $Y'$ such that
      $\ell_{\varphi}\comp f= g\comp k$, then the following diagram
      commutes
      $$
      \xymatrix@C=9ex@R=5ex{
      &
      \mathbf{T}_{\mathbf{\Sigma}}(X')
      \ar[rr]|*+{\mathrm{Tr}^{X',\mathbf{d}^{\ast}(\mathbf{A})}}
      \ar[dd]|(.35)*+{g^{\mathbf{d}}} & &
      \mathbf{Ter}_{X'}(\mathbf{d}^{\ast}(\mathbf{A}))
      \ar[dd]|*+{\Upsilon^{\varphi,g}_{A}} \\
      \mathbf{T}_{\mathbf{\Sigma}}(X)
      \ar[rr]|(.63)*+{\mathrm{Tr}^{X,\mathbf{d}^{\ast}(\mathbf{A})}}
      \ar[dd]|*+{f^{\mathbf{d}}}
      \ar[ur]|*+{k^{@}} & &
      \mathbf{Ter}_{X}(\mathbf{d}^{\ast}(\mathbf{A}))
      \ar[dd]|(.35)*+{\Upsilon^{\varphi,f}_{A}}
      \ar[ur]|*+{\mathrm{Ter}_{k}(\mathbf{d}^{\ast}(\mathbf{A}))} \\
      &
      \mathbf{T}_{\mathbf{\Lambda}}(Y')_{\varphi}
      \ar[rr]|(.65)*+{\mathrm{Tr}^{Y',\mathbf{A}}_{\varphi}} &
      &
      \mathbf{Ter}_{Y'}(\mathbf{A})_{\varphi} \\
      \mathbf{T}_{\mathbf{\Lambda}}(Y)_{\varphi}
      \ar[rr]|*+{\mathrm{Tr}^{Y,\mathbf{A}}_{\varphi}}
      \ar[ur]|*+{\ell^{@}_{\varphi}} &
      &
      \mathbf{Ter}_{Y}(\mathbf{A})_{\varphi}
      \ar[ur]|*+{\mathrm{Ter}_{\ell}(\mathbf{A})_{\varphi}}
      }
      $$
\end{enumerate}
\end{corollary}

\begin{proof}
We restrict ourselves to prove the first part of the corollary.
The upper left-hand diagram commutes because, for a morphism
$(\varphi,d,f)$ from $(S,\Sigma,X)$ to $(T,\Lambda,Y)$ and a
$\mathbf{\Lambda}$-algebra $\mathbf{A}$, the $S$-sorted mapping
$\Upsilon^{\varphi,f}_{A}$ from $\Op_{X}(A_{\varphi})$
to $\Op_{Y}(A)_{\varphi}$ is actually a $\mathbf{\Sigma}$-homomorphism
from $\mathbf{Op}_{X}(\mathbf{d}^{\ast}(\mathbf{A}))$ to
$\mathbf{Op}_{Y}(\mathbf{A})_{\varphi}$ that restricts to
$\mathbf{Ter}_{X}(\mathbf{d}^{\ast}(\mathbf{A}))$ and
$\mathbf{Ter}_{Y}(\mathbf{A})_{\varphi}$.

The upper right-hand diagram commutes because, for the $T$-sorted set
$\coprod_{\varphi}X$ and the $S$-sorted mapping $\eta^{\varphi}_{X}$
from $X$ to $(\coprod_{\varphi}X)_{\varphi}$, we have that
$\vartheta^{\varphi,\eta^{\varphi}_{X}}_{A} = \theta^{\varphi}_{X,A}$.
\end{proof}



As is well-known, for a signature $\mathbf{\Sigma}$, the
conglomerate of terms for $\mathbf{\Sigma}$ is precisely the set
$\bigcup_{X\in \boldsymbol{\mathcal{U}}}\bigcup_{s\in
S}\mathrm{T}_{\mathbf{\Sigma}}(X)_{s}$, but such an amorphous set
is not adequate, because of its lack of structure, for some tasks,
as e.g., to explain the invariant character of the realization of
terms as term operations on algebras, under change of signature
(or to state a Completeness Theorem for finitary many-sorted
equational logic).

However, by conveniently generalizing the concept of term for a
signature $\mathbf{\Sigma}$ (as explained immediately below), it
is possible to endow, in a natural way, to the corresponding
generalized terms for $\mathbf{\Sigma}$, taken as
\emph{morphisms}, with a structure of category, that enables us to
give, in this paper, a category-theoretical explanation of the
existing relation between terms and algebras.  To this we add that
the use of the generalized terms and related notions, such as,
e.g., that of generalized equation (to be defined in the following
section), has allowed us, in~\cite{cs05}, to provide a purely
category-theoretical proof of the Completeness Theorem for monads
in categories of sorted sets. Moreover, such a proof, after
dualizing the generalized terms and equations, is also applicable
to get a corresponding Completeness Theorem for comonads in
categories of sorted sets.

Actually, we associate to every signature $\mathbf{\Sigma}$ the
category $\mathbf{Kl}(\mathbb{T}_{\mathbf{\Sigma}})^{\mathrm{op}}$, of
generalized terms for $\mathbf{\Sigma}$, that we denote, to shorten
notation, by $\mathbf{Ter}(\mathbf{\Sigma})$, i.e., the dual of the
Kleisli category for $\mathbb{T}_{\mathbf{\Sigma}} =
(\mathrm{T}_{\mathbf{\Sigma}},\eta,\mu)$, the standard monad derived
from the adjunction $\mathbf{T}_{\mathbf{\Sigma}}\dashv
\mathrm{G}_{\mathbf{\Sigma}}$ between the category
$\mathbf{Alg}(\mathbf{\Sigma})$ and the category $\mathbf{Set}^{S}$,
with $\mathrm{T}_{\mathbf{\Sigma}} =
\mathrm{G}_{\mathbf{\Sigma}}\comp \mathbf{T}_{\mathbf{\Sigma}}$.
Thus we will be working with a category,
$\mathbf{Ter}(\mathbf{\Sigma})$ (to be defined fully below), that has
as objects the $S$-sorted sets and as morphisms from an $S$-sorted set
$X$ into a like one $Y$ the $S$\nobreakdash-sorted mappings $P$ from $Y$ to
$\mathrm{T}_{\mathbf{\Sigma}}(X)$, i.e., the families $P =
(P_{s})_{s\in S}$, where, for every $s\in S$, $P_{s}$ is a mapping
from $Y_{s}$ to $\mathrm{T}_{\mathbf{\Sigma}}(X)_{s}$ which sends a
variable $y\in Y_{s}$ to the term $P_{s}(y)\in
\mathrm{T}_{\mathbf{\Sigma}}(X)_{s}$.  These morphisms in
$\mathbf{Ter}(\mathbf{\Sigma})$ from $X$ to $Y$ we will call
generalized terms for $\mathbf{\Sigma}$ of type $(X,Y)$, or, simply,
terms for $\mathbf{\Sigma}$ of type $(X,Y)$.

The construction of the category $\mathbf{Ter}(\mathbf{\Sigma})$ is a
natural one.  This is so, essentially, because it has been obtained by
applying a category-theoretical construction, concretely that of
Kleisli (in~\cite{kl65}).  However, to understand more plainly how
the category $\mathbf{Ter}(\mathbf{\Sigma})$ is obtained, or, more
precisely, from where the morphisms of $\mathbf{Ter}(\mathbf{\Sigma})$
arise, the following observation could be helpful.  For a signature
$\mathbf{\Sigma}$, an $S$-sorted set $X$, and a sort $s\in S$, an
ordinary term $P\in \mathrm{T}_{\Sigma}(X)_{s}$ for $\mathbf{\Sigma}$
of type $(X,s)$ is, essentially, an $S$-sorted mapping $P\colon
\delta^{s}\mor \mathrm{T}_{\Sigma}(X)$ where, for $s\in S$,
$\delta^{s} = (\delta^{s}_{t})_{t\in S}$, the delta of Kronecker in
$s$, is the $S$-sorted set such that $\delta^{s}_{t} = \vacio$ if
$s\neq t$ and $\delta^{s}_{s} = 1$.  But the just mentioned $S$-sorted
mappings do not constitute the morphisms of a category.  Therefore, in
order to get a category, it seems natural to replace the special
$S$-sorted sets that are the deltas of Kronecker, as domains of
morphisms, by arbitrary $S$\nobreakdash-sorted sets, thus obtaining
the generalized terms, that are the category-theoretical rendering of
the ordinary terms, since they are now $S$-sorted mappings from an
$S$-sorted set to the free $\mathbf{\Sigma}$-algebra on another
$S$-sorted set, i.e., morphisms in a category
$\mathbf{Ter}(\mathbf{\Sigma})$.

This category-theoretical perspective about terms, in its turn, will
allow us to get a functor $\mathrm{Tr}^{\mathbf{\Sigma}}$, of
realization of terms as term operations, from
$\mathbf{Alg}(\mathbf{\Sigma})\times \mathbf{Ter}(\mathbf{\Sigma})$ to
$\mathbf{Set}$, and therefore to define (in the next section) the
validation of equations, understood as ordered pairs of coterminal
terms in the corresponding generalized sense, in an algebra.

Since it will be fundamental in all that follows, we provide, for a
signature $\mathbf{\Sigma}$, the full definition of the category
$\mathbf{Ter}(\mathbf{\Sigma})$, as announced above, and also the
definition of the procedure of realization of the terms for
$\mathbf{\Sigma}$ as term operations on a given
$\mathbf{\Sigma}$-algebra.  Observe that we depart, in the definition
of the category $\mathbf{Ter}(\mathbf{\Sigma})$, but only for this
type of category, from the (non-Ehresmannian) tradition, in calling a
category by the name of its morphisms.

\begin{definition}
Let $\mathbf{\Sigma}$ be a signature and $\mathbf{A}$ a $\mathbf{\Sigma}$-algebra.
Then
\begin{enumerate}
\item The category of \emph{terms for} $\mathbf{\Sigma}$,
      $\mathbf{Ter}(\mathbf{\Sigma})$, is the dual of
      $\mathbf{Kl}(\mathbb{T}_{\mathbf{\Sigma}})$.  Therefore
      $\mathbf{Ter}(\mathbf{\Sigma})$ has
      \begin{enumerate}
      \item As objects the elements of $\boldsymbol{\mathcal{U}}^{S}$, i.e., the
            $S$-sorted sets,

      \item As morphisms from $X$ to $Y$, that we call \emph{terms
            for} $\mathbf{\Sigma}$ \emph{of type} $(X,Y)$, or, simply,
            \emph{terms} \emph{of type} $(X,Y)$, the $S$-sorted mappings
            from $Y$ to $\mathrm{T}_{\mathbf{\Sigma}}(X)$,

      \item As composition, denoted in $\mathbf{Ter}(\mathbf{\Sigma})$ and
            $\mathbf{Kl}(\mathbb{T}_{\mathbf{\Sigma}})$ by $\diamond$, the
            operation which sends terms $P\colon X\mor Y$ and $Q\colon Y\mor
            Z$ in $\mathbf{Ter}(\mathbf{\Sigma})$ to the term $Q\diamond P\colon
            X\mor Z$ in $\mathbf{Ter}(\mathbf{\Sigma})$ defined as
            $$
            Q\diamond P = \mu_{X}\comp P^{@}\comp Q,
            $$
            where $\mu_{X}$ is the value at $X$ of the multiplication
            $\mu$ of the monad $\mathbb{T}_{\mathbf{\Sigma}} =
            (\mathrm{T}_{\mathbf{\Sigma}},\eta,\mu)$ and $P^{@}$ the
            value of the functor $\mathbf{T}_{\mathbf{\Sigma}}$
            at the $S$-sorted mapping $P\colon Y\mor
            \mathrm{T}_{\mathbf{\Sigma}}(X)$, and

\item As identities the values of $\eta$, the unit of the
      monad $\mathbb{T}_{\mathbf{\Sigma}}$, in the $S$-sorted sets.
\end{enumerate}

\item If $P\colon X\mor Y$ is a term for $\mathbf{\Sigma}$ of
      type $(X,Y)$, then $P^{\mathbf{A}}$, the \emph{term operation
      on} $\mathbf{A}$ \emph{determined by} $P$, or the \emph{term
      realization of} $P$ \emph{on} $\mathbf{A}$, is the mapping from
      $A_{X}$ to $A_{Y}$ which assigns to a valuation $f$ of the
      variables $X$ in $A$ the valuation $f^{\sharp}\comp P$ of the
      variables $Y$ in $A$, i.e., the composition of $P\colon Y\mor
      \mathrm{T}_{\mathbf{\Sigma}}(X)$ and the underlying mapping of
      $f^{\sharp}\colon \mathbf{T}_{\mathbf{\Sigma}}(X)\mor
      \mathbf{A}$, the canonical extension of the valuation $f\colon
      X\mor A$.
\end{enumerate}

\end{definition}

\begin{remark}
For a term $P\colon X\mor Y$ for $\mathbf{\Sigma}$ of
type $(X,Y)$, the term operation $P^{\mathbf{A}}$ on $\mathbf{A}$
determined by $P$ is also the mapping from $A_{X}$ to $A_{Y}$
obtained from the family
$$
((\mathrm{Tr}^{X,\mathbf{A}}_{s}(P_{s}(y)))_{y\in Y_{s}})_{s\in S}\in
\textstyle\prod_{s\in S}\mathrm{Hom}(Y_{s},\mathrm{Hom}(A_{X},A_{s})),
$$
through the following natural isomorphisms:

\begin{align*}
    \textstyle\prod_{s\in S}\mathrm{Hom}(Y_{s},\mathrm{Hom}(A_{X},A_{s}))
    &\cong
    \textstyle\prod_{s\in S}\mathrm{Hom}(A_{X},\mathrm{Hom}(Y_{s},A_{s}))\\
    &\cong
    \mathrm{Hom}(A_{X},\textstyle\prod_{s\in S}\mathrm{Hom}(Y_{s},A_{s}))\\
    &\cong
    \mathrm{Hom}(A_{X},A_{Y}) .
\end{align*}
\end{remark}

After associating to every signature $\mathbf{\Sigma}$ the
corresponding category $\mathbf{Ter}(\mathbf{\Sigma})$ of terms, we
proceed to assign to every signature morphism $\mathbf{d}\colon
\mathbf{\Sigma}\mor \mathbf{\Lambda}$, where, we recall, $\mathbf{d}$
is a pair $(\varphi,d)$ with $\varphi$ a mapping from $S$ to $T$ and
$d$ a morphism in $\mathbf{Set}(S)$ from $\Sigma$ to
$\Lambda_{\varphi^{\star}\times \varphi}$, a corresponding functor
$\mathbf{d}_{\diamond}$ from $\mathbf{Ter}(\mathbf{\Sigma})$ to
$\mathbf{Ter}(\mathbf{\Lambda})$.

\begin{proposition}\label{functorMorfismoSignaturas}
Let $\mathbf{d}\colon \mathbf{\Sigma}\mor \mathbf{\Lambda}$ be a
signature morphism.  Then there exists a functor
$\mathbf{d}_{\diamond}$ from $\mathbf{Ter}(\mathbf{\Sigma})$ to
$\mathbf{Ter}(\mathbf{\Lambda})$ defined as follows
\begin{enumerate}
\item $\mathbf{d}_{\diamond}$ sends an $S$-sorted set $X$ to the $T$-sorted
      set $\mathbf{d}_{\diamond}(X) = \coprod_{\varphi}X$.

\item $\mathbf{d}_{\diamond}$ sends a morphism $P$ from $X$ to
      $Y$ in $\mathbf{Ter}(\mathbf{\Sigma})$ to the morphism
      $\mathbf{d}_{\diamond}(P) =
      (\theta^{\varphi})^{-1}(\eta^{\mathbf{d}}_{X}\comp P)$ from
      $\coprod_{\varphi}X$ to $\coprod_{\varphi}Y$ in
      $\mathbf{Ter}(\mathbf{\Lambda})$, where $\eta^{\mathbf{d}}_{X}$
      is the $\mathbf{\Sigma}$-homo\-morphism from
      $\mathbf{T}_{\mathbf{\Sigma}}(X)$ to
      $\mathbf{T}_{\mathbf{\Lambda}}(\coprod_{\varphi}X)_{\varphi}$
      that extends the $S$-sorted mapping
      $(\eta_{\coprod_{\varphi}X})_{\varphi}\comp \eta^{\varphi}_{X}$
      from $X$ to
      $\mathrm{T}_{\mathbf{\Lambda}}(\coprod_{\varphi}X)_{\varphi}$,
      i.e., for $\eta^{\mathbf{d}}_{X}$ we have that
      $$
      \eta^{\mathbf{d}}_{X} =
      \ext{((\eta_{\coprod_{\varphi}X})_{\varphi}\comp
      \eta^{\varphi}_{X})},
      $$
      $\theta^{\varphi}$ the natural isomorphism of the
      adjunction $\coprod_{\varphi}\ladj\Delta_{\varphi}$, and
      $\eta^{\varphi}$ the unit of the same adjunction.
\end{enumerate}

\end{proposition}

\begin{proof}
To begin with, we prove that $\mathbf{d}_{\diamond}$ preserves
identities.  If $X$ is an $S$-sorted set, then the following diagram
$$
\xymatrix@C=60pt@R=40pt{
X \ar[r]^{\eta_{X}}
  \ar[d]_{\eta^{\varphi}_{X}}
  \ar[rd]|{\theta^{\varphi}(\eta_{\coprod_{\varphi}X})} &
\mathrm{T}_{\mathbf{\Sigma}}(X)
  \ar[d]^{\eta^{\mathbf{d}}_{X} =
  \ext{(\theta^{\varphi}(\eta_{\coprod_{\varphi}X}))}} \\
(\coprod_{\varphi}X)_{\varphi}
  \ar[r]_{(\eta_{\coprod_{\varphi}X})_{\varphi}} &
\mathrm{T}_{\mathbf{\Lambda}}(\coprod_{\varphi}X)_{\varphi}
}
$$
commutes, therefore
$\mathbf{d}_{\diamond}(\eta_{X})=\eta_{\coprod_{\varphi}X}$.  Now we prove
that $\mathbf{d}_{\diamond}$ preserves compositions.  Let $P\colon X\mor Y$
and $Q\colon Y\mor Z$ be morphisms in ${\mathbf{Ter}(\mathbf{\Sigma})}$.
Then we have the following equations: %
\begin{align*}
    \mathbf{d}_{\diamond}(Q\dcomp P) &=
    (\theta^{\varphi})^{-1}(\eta^{\mathbf{d}}_{X}\comp
           P^{\sharp}\comp Q) \\
    &=
    (\theta^{\varphi})^{-1}(\eta^{\mathbf{d}}_{X})\comp
      \tcoprod_{\varphi}P^{\sharp}\comp \tcoprod_{\varphi}Q, \\
    \mathbf{d}_{\diamond}(Q)\dcomp \mathbf{d}_{\diamond}(P) &=
    \mathbf{d}_{\diamond}(P)^{\sharp} \comp \mathbf{d}_{\diamond}(Q)  \\
    &=
    \mathbf{d}_{\diamond}(P)^{\sharp}
      \comp
      (\theta^{\varphi})^{-1}(\eta^{\mathbf{d}}_{Y})\comp
      \tcoprod_{\varphi}Q,
\end{align*}
therefore, to prove that $\mathbf{d}_{\diamond}(Q\dcomp P) =
\mathbf{d}_{\diamond}(Q)\dcomp \mathbf{d}_{\diamond}(P)$ it is enough
to verify the following equation
\begin{equation*}
(\theta^{\varphi})^{-1}(\eta^{\mathbf{d}}_{X})\comp\tcoprod_{\varphi}
  P^{\sharp}
=
\mathbf{d}_{\diamond}(P)^{\sharp}
      \comp
      (\theta^{\varphi})^{-1}(\eta^{\mathbf{d}}_{Y}).
\end{equation*}
But for this, because of the commutativity of the following diagram
$$
\xymatrix@C=55pt@R=35pt{
\coprod_{\varphi}\mathrm{T}_{\mathbf{\Sigma}}(Y)
  \ar`u[rr]+/u30pt/`[rr]
  ^{(\theta^{\varphi})^{-1}(\eta^{\mathbf{d}}_{Y})}%
  [rr]
  \ar[r]^-{\coprod_{\varphi}\eta^{\mathbf{d}}_{Y}}
  \ar[d]|{\coprod_{\varphi}P^{\sharp}}
  &
\coprod_{\varphi}\mathrm{T}_{\mathbf{\Lambda}}(\coprod_{\varphi}Y)_{\varphi}
  \ar[r]^-{\varepsilon^{\varphi}_{\mathrm{T}_{\mathbf{\Lambda}}(\coprod_{\varphi}Y)}}
  \ar[d]|{\coprod_{\varphi}\mathbf{d}_{\diamond}(P)^{\sharp}_{\varphi}}
  &
\mathrm{T}_{\mathbf{\Lambda}}(\coprod_{\varphi}Y)
  \ar[d]|{\mathbf{d}_{\diamond}(P)^{\sharp}}
  \\
\coprod_{\varphi}\mathrm{T}_{\mathbf{\Lambda}}(X)
  \ar[r]_-{\coprod_{\varphi}\eta^{\mathbf{d}}_{X}}
  \ar`d[rr]+/d30pt/`[rr]
  _{(\theta^{\varphi})^{-1}(\eta^{\mathbf{d}}_{X})}%
  [rr]
  &
\coprod_{\varphi}\mathrm{T}_{\mathbf{\Lambda}}(\coprod_{\varphi}X)_{\varphi}
  \ar[r]_-{\varepsilon^{\varphi}_{\mathrm{T}_{\mathbf{\Lambda}}(\coprod_{\varphi}X)}}
  &
\mathrm{T}_{\mathbf{\Lambda}}(\coprod_{\varphi}X)
}
$$
it is enough to verify the following equation
\begin{equation}
    \eta^{\mathbf{d}}_{X}\comp P^{\sharp}
     =
    \mathbf{d}_{\diamond}(P)^{\sharp}_{\varphi}\comp\eta^{\mathbf{d}}_{Y}.
\tag{1}
\end{equation}
But equation (1) is valid because the restriction of both terms
to the generating $\mathrm{ms}$-set $Y$ coincide:
\begin{align*}
    \eta^{\mathbf{d}}_{X} \comp P^{\sharp}
      \comp \eta_{Y} &=
      \eta^{\mathbf{d}}_{X} \comp P, \\[2ex]
    \mathbf{d}_{\diamond}(P)^{\sharp}_{\varphi}
      \comp \eta^{\mathbf{d}}_{Y}
      \comp \eta_{Y}
    &=
    \mathbf{d}_{\diamond}(P)^{\sharp}_{\varphi}
      \comp (\eta_{\coprod_{\varphi}Y})_{\varphi}
      \comp \eta^{\varphi}_{Y} \\
    &=
    \mathbf{d}_{\diamond}(P)_{\varphi}\comp \eta^{\varphi}_{Y} \\
    &=
    (\theta^{\varphi})^{-1}(\eta^{\mathbf{d}}_{X}\comp P)_{\varphi}
      \comp\eta^{\varphi}_{Y} \\
    &=
    (\eta^{\mathbf{d}}_{X}\comp P)^{\sharp}\comp
    \eta_{Y}\\
    &=
    \eta^{\mathbf{d}}_{X} \comp P.
    \qedhere
\end{align*}
\end{proof}

\begin{remark}
For a term $P\colon X\mor Y$, the term $\mathbf{d}_{\diamond}(P)\colon
\coprod_{\varphi}X \mor \coprod_{\varphi}Y$ can be defined
alternative, but equivalently, as the composition of the morphisms in
the following diagram
$$
\xymatrix@C=9.5 ex{
\coprod_{\varphi}Y
\ar[r]^-{\coprod_{\varphi}P} &
*+!<1.89ex,0ex>{\coprod_{\varphi} \mathrm{T}_{\mathbf{\Sigma}}(X)}
\ar[r]^-{\coprod_{\varphi}\eta^{\mathbf{d}}_{X}} &
\coprod_{\varphi}\mathrm{T}_{\mathbf{\Lambda}}(\coprod_{\varphi}X)_{\varphi}
\ar[r]^-{\varepsilon^{\varphi}_{\mathrm{T}_{\mathbf{\Lambda}}(\coprod_{\varphi}X)}\,\,} &
\mathrm{T}_{\mathbf{\Lambda}}(\coprod_{\varphi}X),
}
$$
where, recalling that $\eta^{\mathbf{d}} =
\mathrm{G}_{\mathbf{\Sigma}}\ast \eta^{\mathbf{d}}$ is the second
natural transformation in Proposition~\ref{fundPdextranaturalbis}
and $\varepsilon^{\varphi}$ the counit of the adjunction
$\coprod_{\varphi}\ladj\Delta_{\varphi}$, we have that
\begin{enumerate}
\item The $T$-sorted mapping
      $\coprod_{\varphi}\eta^{\mathbf{d}}_{X}$ is the component at $X$
      of the natural transformation
      $\coprod_{\varphi}\ast\eta^{\mathbf{d}}$ from
      $\coprod_{\varphi}\comp \mathrm{T}_{\mathbf{\Sigma}}$ to
      $\coprod_{\varphi}\comp\Delta_{\varphi}\comp
      \mathrm{T}_{\mathbf{\Lambda}}\comp \coprod_{\varphi}$, and

\item The $T$-sorted mapping
      $\varepsilon^{\varphi}_{\mathrm{T}_{\mathbf{\Lambda}}(\coprod_{\varphi}X)}$
      is the component at $X$ of the natural transformation
      $\varepsilon^{\varphi}\ast
      (\mathrm{T}_{\mathbf{\Lambda}}\comp\coprod_{\varphi}X)$ from
      $\coprod_{\varphi}\comp\Delta_{\varphi}\comp
      \mathrm{T}_{\mathbf{\Lambda}}\comp \coprod_{\varphi}$ to
      $\mathrm{T}_{\mathbf{\Lambda}}\comp\coprod_{\varphi}X$.
\end{enumerate}
\end{remark}


We state now for the generalized terms the homologous of the
right-hand diagram in the first part of
Corollary~\ref{fundPdextranatural}, i.e., the invariant character
under signature change of the realization of terms as term operations
in arbitrary, but fixed, algebras.  We remark that from this fact we
will get, in the third section, the invariance of the relation of
satisfaction under signature change.

\begin{proposition}\label{invarianza}
Let $\mathbf{d}\colon\mathbf{\Sigma}\mor\mathbf{\Lambda}$ be a
signature morphism.  Then, for every $\mathbf{\Lambda}$\nobreakdash-algebra
$\mathbf{A}$ and term $P\colon X\mor Y$ for $\mathbf{\Sigma}$ of type
$(X,Y)$, the following diagram commutes
$$
\xymatrix@C=18ex@R=10ex{
(A_{\varphi})_{X}
  \ar[r]^{P^{\mathbf{d}^{\ast}(\mathbf{A})}} &
(A_{\varphi})_{Y} \\
A_{\coprod_{\varphi}\!X}
  \ar[r]_{\mathbf{d}_{\diamond}(P)^{\mathbf{A}}}
  \ar[u]^{\theta^{\varphi}_{X,A}} &
A_{\coprod_{\varphi}\!Y}
  \ar[u]_{\theta^{\varphi}_{Y,A}}
}
$$

\end{proposition}

\begin{proof}
Because the $S$-sorted set $Y$ is isomorphic to $\coprod_{ {s\in S,
y\in Y_{s}}}\delta^{s}$ and the functor $\coprod_{\varphi}$ preserves
colimits, since it has $\Delta_{\varphi}$ as a right adjoint,
$\coprod_{\varphi}Y$ is isomorphic to $\coprod_{s\in S,\, y\in
Y_{s}}\delta^{\varphi(s)}$.  But $\mathrm{Hom}(\coprod_{\varphi}Y,A)$
and $\prod_{s\in S, y\in Y_{s}}\mathrm{Hom}(\delta^{\varphi(s)},A)$
are isomorphic, therefore it is enough to prove the proposition for
the $S$-sorted sets of the type $\delta^{s}$, i.e., the deltas of
Kronecker, and this follows directly from
Corollary~\ref{fundPdextranatural}.
\end{proof}

Once defined the mappings that associate, respectively, to a signature
the corresponding category of terms, and to a signature morphism the
functor between the associated categories of terms, we state in the
following proposition that both mappings are actually the components
of a pseudo-functor from $\mathbf{Sig}$ to the $2$-category $\mathbf{Cat}$.

\begin{proposition}\label{pseudoPol}
There exists a pseudo-functor $\mathrm{Ter}$ from $\mathbf{Sig}$ to the
$2$-category $\mathbf{Cat}$ given by the following data
\begin{enumerate}
\item The object mapping of $\mathrm{Ter}$ is that which sends a
      signature $\mathbf{\Sigma}$ to the category
      $\mathrm{Ter}(\mathbf{\Sigma}) = \mathbf{Ter}(\mathbf{\Sigma})$.

\item The morphism mapping of $\mathrm{Ter}$ is that which sends
      a signature morphism $\mathbf{d}$ from $\mathbf{\Sigma}$ to
      $\mathbf{\Lambda}$ to the functor $\mathrm{Ter}(\mathbf{d}) =
      \mathbf{d}_{\diamond}$ from  $\mathbf{Ter}(\mathbf{\Sigma})$ to
      $\mathbf{Ter}(\mathbf{\Lambda})$.

\item For every $\mathbf{d}\colon \mathbf{\Sigma}\mor \mathbf{\Lambda}$ and
      $\mathbf{e}\colon \mathbf{\Lambda}\mor \mathbf{\Omega}$, the natural
      isomorphism $\gamma^{\mathbf{d},\mathbf{e}}$ from $\mathbf{e}_{\diamond} \comp
      \mathbf{d}_{\diamond}$ to $(\mathbf{e}\comp\mathbf{d})_{\diamond}$
      is that which is
      defined, for every $S$-sorted set $X$, as the isomorphism
      $\gamma^{\mathbf{d},\mathbf{e}}_{X}\colon\coprod_{\psi}\coprod_{\varphi}X
      \mor \coprod_{\psi \comp \varphi} X$ in $\mathbf{Ter}(\mathbf{\Omega})$
      that corresponds to the $U$-sorted mapping
      $$
      \xymatrix@C=60pt@R=40pt{
      \coprod_{\psi \comp \varphi}X
      \ar[r]^{(\gamma^{\varphi,\psi}_{X})^{-1}} &
      \coprod_{\psi}\coprod_{\varphi}X
      \ar[r]^-{\eta_{\coprod_{\psi}\coprod_{\varphi}X}}
      & \mathrm{T}_{\mathbf{\Omega}}(\coprod_{\psi}\coprod_{\varphi}X),
      }
      $$
      where $\gamma^{\varphi,\psi}_{X}$ is the component at $X$ of the
      natural isomorphism $\gamma^{\varphi,\psi}$ for the pseudo-functor
      $\mathrm{MSet}^{\ttcoprod}$.

\item For every signature $\mathbf{\Sigma}$, the
       natural isomorphism $\nu^{\mathbf{\Sigma}}$ from
       $\Id_{\mathbf{Ter}(\mathbf{\Sigma})}$ to $(\id_{\mathbf{\Sigma}})_{\diamond}$
       is that which is defined, for every $S$-sorted set $X$, as the
       isomorphism $\nu^{\mathbf{\Sigma}}_{X}\colon X\mor
       \coprod_{\id_{S}}X$ in $\mathbf{Ter}(\mathbf{\Sigma})$ that corresponds
       to the $S$-sorted mapping
       $$
       \xymatrix@C=60pt@R=40pt{
       \coprod_{\id_{S}}X
       \ar[r]^{\nu^{S}_{X}} & X \ar[r]^-{\eta_{X}} &
       \mathrm{T}_{\mathbf{\Omega}}(X),
       }
       $$
       where $\nu^{S}_{X}$ is the component at $X$ of the
       natural isomorphism  $\nu^{S}$ for the
       pseudo-functor $\mathrm{MSet}^{\ttcoprod}$.
\end{enumerate}
\end{proposition}

Our next goals are to prove that
\begin{enumerate}
\item For a signature $\mathbf{\Sigma}$, there exists a functor
      $\mathrm{Tr}^{\mathbf{\Sigma}}$ from the product category
      $\mathbf{Alg}(\mathbf{\Sigma})\times \mathbf{Ter}(\mathbf{\Sigma})$ to
      $\mathbf{Set}$, that formalizes simulta\-neous\-ly the procedure of
      realization of terms (as term operations on algebras), and its
      naturalness (by taking into account the variation of the
      algebras through the homomorphisms between them), and that

\item For a signature morphism
      $\mathbf{d}\colon\mathbf{\Sigma}\mor \mathbf{\Lambda}$, there
      exists a natural isomorphism $\theta^{\mathbf{d}}$ from the
      functor $\mathrm{Tr}^{\mathbf{\Lambda}}\comp
      (\mathrm{Id}_{\mathbf{Alg}(\mathbf{\Lambda})}\times
      \mathbf{d}_{\diamond})$ to the functor
      $\mathrm{Tr}^{\mathbf{\Sigma}}\comp
      (\mathbf{d}^{\ast}\times\mathrm{Id}_{\mathbf{Ter}(\mathbf{\Sigma})})$,
      that shows the invariant character of the procedure of
      realization of terms under signature change.
\end{enumerate}

To accomplish the first stated goal we begin by proving the following

\begin{lemma}\label{realizdiamondcompcomm}
Let $\mathbf{A}$ be a $\mathbf{\Sigma}$-algebra, $P$ a term of type
$(X,Y)$, and $Q$ a term of type $(Y,Z)$.  Then we have that
\begin{enumerate}
\item $(Q\diamond P)^{\mathbf{A}} = Q^{\mathbf{A}}\comp P^{\mathbf{A}}$.

\item For $\eta_{X}$, the identity morphism at $X$ in
      $\mathbf{Ter}(\mathbf{\Sigma})$, $\eta_{X}^{\mathbf{A}} =
      \mathrm{id}_{A_{X}}$.
\end{enumerate}
\end{lemma}

\begin{proof}
We restrict ourselves to prove the first part of the lemma.  Since
$(Q\diamond P)^{\mathbf{A}}$ is the mapping from $A_{X}$ to $A_{Z}$
which sends an $S$-sorted mapping $u\colon X\mor A$ to the $S$-sorted
mapping
$$
u^{\sharp}\comp (Q\diamond P) =
u^{\sharp}\comp \mu_{X}\comp P^{@}\comp Q \colon Z\mor A,
$$
where, we recall, $\mu_{X}$ is the value in $X$ of the multiplication
$\mu$ of the monad $\mathbb{T}_{\mathbf{\Sigma}} =
(\mathrm{T}_{\mathbf{\Sigma}},\eta,\mu)$ and $P^{@}$ the value in the
$S$-sorted mapping $P\colon Y\mor \mathrm{T}_{\mathbf{\Sigma}}(X)$ of
the functor $\mathbf{T}_{\mathbf{\Sigma}}$; and $Q^{\mathbf{A}}\comp
P^{\mathbf{A}}$ the mapping from $A_{X}$ to $A_{Z}$ which sends an
$S$-sorted mapping $u\colon X\mor A$ to the $S$-sorted mapping
$$
(u^{\sharp}\comp P)^{\sharp}\comp Q\colon Z\mor A,
$$
to show that $(Q\diamond P)^{\mathbf{A}} = Q^{\mathbf{A}}\comp
P^{\mathbf{A}}$ it is enough to prove that the
$\mathbf{\Sigma}$-homo\-morphisms $u^{\sharp}\comp \mu_{X}\comp P^{@}$
and $(u^{\sharp}\comp P)^{\sharp}$ from
$\mathbf{T}_{\mathbf{\Sigma}}(Y)$ to $\mathbf{A}$ are identical.  But
this follows from the equation
$$
u^{\sharp}\comp \mu_{X}\comp P^{@}\comp \eta_{Y} =
(u^{\sharp}\comp P)^{\sharp}\comp \eta_{Y},
$$
that, in its turn, is a consequence of the laws for the monad
$\mathbb{T}_{\mathbf{\Sigma}}$ and of the equation
$$
  P^{@}\comp \eta_{Y} = \eta_{\mathrm{T}_{\mathbf{\Sigma}}(X)}\comp P,
$$
where $\eta_{Y}$ is the canonical embedding of $Y$ into
$\mathrm{T}_{\mathbf{\Sigma}}(Y)$ and
$\eta_{\mathrm{T}_{\mathbf{\Sigma}}(X)}$ the canonical embedding of
$\mathrm{T}_{\mathbf{\Sigma}}(X)$ into
$\mathrm{T}_{\mathbf{\Sigma}}(\mathrm{T}_{\mathbf{\Sigma}}(X))$.
\end{proof}

This lemma has as an immediate consequence the following

\begin{corollary}
Let $\mathbf{\Sigma}$ be a signature and $\mathbf{A}$ a
$\mathbf{\Sigma}$-algebra.  Then there exists a functor
$\mathrm{Tr}^{\mathbf{\Sigma},\mathbf{A}}$ from
$\mathbf{Ter}(\mathbf{\Sigma})$ to $\mathbf{Set}$ which sends an $S$-sorted set
$X$ to the set $\mathrm{Tr}^{\mathbf{\Sigma},\mathbf{A}}(X) = A_{X}$
and a term $P\colon X\mor Y$ to
$\mathrm{Tr}^{\mathbf{\Sigma},\mathbf{A}}(P) = P^{\mathbf{A}}\colon
A_{X}\mor A_{Y}$, the term operation on $\mathbf{A}$ determined by
$P$.
\end{corollary}

Therefore, from the definition of the object and morphism mappings of
the functors of the type $\mathrm{Tr}^{\mathbf{\Sigma},\mathbf{A}}$,
we see that they encapsulate the procedure of realization of terms.
And, from the fact that they preserve identities and compositions in
$\mathbf{Ter}(\mathbf{\Sigma})$, we conclude that they formally
represent the two basic intuitions about the behaviour of the just
named procedure, i.e., that the realization of an identity term is an
identity term operation, and that the realization of a composite of
two terms is the composite of their respective realizations (in the
same order).

\begin{remark}
By identifying the $\mathbf{\Sigma}$-algebras with
the $\mathbb{T}_{\mathbf{\Sigma}}$-algebras, the just stated corollary
can be interpreted as meaning that every $\mathbf{\Sigma}$-algebra is
a functor from $\mathbf{Ter}(\mathbf{\Sigma}) =
\mathbf{Kl}(\mathbb{T}_{\mathbf{\Sigma}})^{\mathrm{op}}$ to
$\mathbf{Set}$.
\end{remark}

Before stating the following lemma we recall that, for an $S$-sorted
mapping $f$ from an $S$-sorted set $A$ into a like one $B$ and an
$S$-sorted set $X$, $f_{X}$ is the mapping from $A_{X}$ to $B_{X}$
which assigns to an $S$-sorted mapping $u\colon X\mor A$ the
$S$-sorted mapping $f\comp u$, i.e., $f_{X}$ is the value at $X$ of
the natural transformation $\mathrm{H}(\cdot,f)$ from the
contravariant functor $\mathrm{H}(\cdot,A)$ to the contravariant
functor $\mathrm{H}(\cdot,B)$, both from
$(\mathbf{Set}^{S})^{\mathrm{op}}$ to $\mathbf{Set}$.

\begin{lemma}
Let $f$ be a $\mathbf{\Sigma}$-homomorphism from $\mathbf{A}$ to
$\mathbf{B}$ and $P$ a term of type $(X,Y)$ in
$\mathbf{Ter}(\mathbf{\Sigma})$. Then the following diagram commutes
$$
\xymatrix@C=25pt@R=15pt{
{} &
{A_{X}}
\ar[dl]_{f_{X}}
\ar[dr]^{P^{\mathbf{A}}} & {} \\
{B_{X}}
\ar[dr]_{P^{\mathbf{B}}} & {} &
A_{Y}
\ar[dl]^{f_{Y}} \\
{} &
B_{Y} & {}
}
$$
We agree to denote by $f_{P}$ the diagonal mapping from $A_{X}$ to
$B_{Y}$ in the above commutative diagram.
\end{lemma}

\begin{proof}
Given an $S$-sorted mapping $u\colon X\mor A$, we have that $(f\comp
u)^{\sharp} = f\comp u^{\sharp}$, by the universal property of the
free $\mathbf{\Sigma}$-algebra on $X$ and taking into account that $f$ is
a $\mathbf{\Sigma}$-homo\-morphism from $\mathbf{A}$ to $\mathbf{B}$.
Therefore, since $P^{\mathbf{B}}\comp f_{X}(u) = (f\comp
u)^{\sharp}\comp P$, and $f_{Y}\comp P^{\mathbf{A}}(u) = f\comp
(u^{\sharp}\comp P)$, we have that $P^{\mathbf{B}}\comp f_{X}(u) =
f_{Y}\comp P^{\mathbf{A}}(u)$.  Thus $P^{\mathbf{B}}\comp f_{X} =
f_{Y}\comp P^{\mathbf{A}}$.
\end{proof}

This lemma has as an immediate consequence the following

\begin{corollary}
Let $\mathbf{\Sigma}$ be a signature and $f$ a
$\mathbf{\Sigma}$-homomorphism from $\mathbf{A}$ to $\mathbf{B}$.
Then there exists a natural transformation
$\mathrm{Tr}^{\mathbf{\Sigma},f}$ from  the functor
$\mathrm{Tr}^{\mathbf{\Sigma},\mathbf{A}}$ to the functor
$\mathrm{Tr}^{\mathbf{\Sigma},\mathbf{B}}$, as reflected in the diagram
$$
\xymatrix@C=20ex@R=8ex{
\mathbf{Ter}(\mathbf{\Sigma})
\ar@/^15pt/[r]^{\mathrm{Tr}^{\mathbf{\Sigma},\mathbf{A}}}="f"
  \ar@/^-15pt/[r]_{\mathrm{Tr}^{\mathbf{\Sigma},\mathbf{B}}}="g" &
\mathbf{Set}
\ar @{} "f";"g" |{\dir{=>}}^{\,\mathrm{Tr}^{\mathbf{\Sigma},f}}
}
$$
which sends an $S$-sorted set $X$ to the mapping
$\mathrm{Tr}^{\mathbf{\Sigma},f}_{X} = f_{X}$ from $A_{X}$ to
$B_{X}$. Besides, we have that
\begin{enumerate}
\item For $\mathrm{id}_{\mathbf{A}}$, the identity
      $\mathbf{\Sigma}$-homomorphism at $\mathbf{A}$, it is the case that
      $$
      \mathrm{Tr}^{\mathbf{\Sigma},\mathrm{id}_{\mathbf{A}}} =
      \mathrm{id}_{\mathrm{Tr}^{\mathbf{\Sigma},\mathbf{A}}}.
      $$

\item If $g\colon \mathbf{B}\mor\mathbf{C}$ is another
      $\mathbf{\Sigma}$-homomorphism, then
      $$
      \mathrm{Tr}^{\mathbf{\Sigma},g\comp f} =
      \mathrm{Tr}^{\mathbf{\Sigma},g}\comp
      \mathrm{Tr}^{\mathbf{\Sigma},f}.
      $$
\end{enumerate}
\end{corollary}

Therefore, the naturalness of the procedure of realization of terms as
term operations on the different algebras is embodied in the natural
transformations of the type $\mathrm{Tr}^{\mathbf{\Sigma},f}$.

\begin{remark}
By identifying the $\mathbf{\Sigma}$-homomorphisms with the
$\mathbb{T}_{\mathbf{\Sigma}}$-homomor\-phisms, the just stated
corollary can be interpreted as meaning that every
$\mathbf{\Sigma}$-homomorphism $f$ from $\mathbf{A}$ to $\mathbf{B}$
is a natural transformation from the functor
$\mathrm{Tr}^{\mathbf{\Sigma},\mathbf{A}}$ to the functor
$\mathrm{Tr}^{\mathbf{\Sigma},\mathbf{B}}$, both from
$\mathbf{Ter}(\mathbf{\Sigma}) =
\mathbf{Kl}(\mathbb{T}_{\mathbf{\Sigma}})^{\mathrm{op}}$ to
$\mathbf{Set}$.  Actually, each homomorphism $(\mathbf{d},f)$ from an
algebra $(\mathbf{\Sigma},\mathbf{A})$ into a like one
$(\mathbf{\Lambda},\mathbf{B})$ is identifiable to a morphism (in the
category $(\mathbf{Cat})_{/\!/\mathbf{Set}}$, see~\cite{Gro71}, p.
(sub) 186) from the object
$(\mathbf{Ter}(\mathbf{\Sigma}),\mathrm{Tr}^{\mathbf{\Sigma},\mathbf{A}})$
over $\mathbf{Set}$ to the object
$(\mathbf{Ter}(\mathbf{\Lambda}),\mathrm{Tr}^{\mathbf{\Lambda},\mathbf{B}})$
over $\mathbf{Set}$, concretely, to the morphism given by the pair
$(\mathbf{d}_{\diamond},(\theta^{\varphi}_{\cdot,B})^{-1}\comp
\mathrm{H}(\cdot,f))$, and represented by the following diagram
$$
\xymatrix@C= 25ex@R= 10ex{
\mathbf{Ter}(\mathbf{\Sigma})
  \ar[r]^{\mathbf{d}_{\diamond}}
  \ar[d]_{\mathrm{Tr}^{\mathbf{\Sigma},\mathbf{A}}}
  &
  \mathbf{Ter}(\mathbf{\Lambda})
  \ar[d]^{\mathrm{Tr}^{\mathbf{\Lambda},\mathbf{B}}}
  \\
\mathbf{Set}
  \ar[r]_{\mathrm{Id}}
  \ar@{}[ru]|{\dir{=>}}^-(.5){(\theta^{\varphi}_{\cdot,B})^{-1}\comp
  \mathrm{H}(\cdot,f)}&
\mathbf{Set}
}
$$
where $\mathrm{H}(\cdot,f)$ is the natural transformation from the
contravariant hom-functor $\mathrm{H}(\cdot,A)$ to the
contravariant hom-functor $\mathrm{H}(\cdot,B_{\varphi})$, and
$(\theta^{\varphi}_{\cdot,B})^{-1}$ the natural isomorphism from
$\mathrm{H}(\cdot,B_{\varphi})$ to
$\mathrm{H}(\coprod_{\varphi}(\cdot),B)$.  Observe that the
naturalness of $(\theta^{\varphi}_{\cdot,B})^{-1}\comp
\mathrm{H}(\cdot,f)$ means that, for every term $P$ for
$\mathbf{\Sigma}$ of type $(X,Y)$, the following diagram commutes
$$
\xymatrix@C= 10ex@R=10ex{
A_{X} = \mathrm{Hom}(X,A)
  \ar[r]^-{\mathrm{H}(X,f)}
  \ar[d]_{P^{\mathbf{A}}}
  &
\mathrm{Hom}(X,B_{\varphi})
  \ar[r]^-{(\theta^{\varphi}_{X,B})^{-1}} &
\mathrm{Hom}(\coprod_{\varphi}X,B) = B_{\coprod_{\varphi}X}
  \ar[d]^{\mathbf{d}_{\diamond}(P)^{\mathbf{B}}}
  \\
A_{Y} = \mathrm{Hom}(Y,A)
  \ar[r]_-{\mathrm{H}(Y,f)} &
\mathrm{Hom}(Y,B_{\varphi})
  \ar[r]_-{(\theta^{\varphi}_{Y,B})^{-1}} &
\mathrm{Hom}(\coprod_{\varphi}Y,B) = B_{\coprod_{\varphi}Y}
}
$$
From the identification of the homomorphisms between algebras in the
category $\mathbf{Alg}$ to some convenient morphisms between the
associated objects over $\mathbf{Set}$, we can conclude, e.g., that
the concept of homomorphism as defined by Bénabou in~\cite{jB68} (that
does not allow the variation of the signature and therefore it works
between algebras of the same signature (see~\cite{jB68}, p.  (sub) 16,
last paragraph)), corresponds itself, for a signature
$\mathbf{\Sigma}$ and a $\mathbf{\Sigma}$-homomorphism $f$ from
$\mathbf{A}$ to $\mathbf{B}$, to the (very special) case in which
$(\mathbf{d}_{\diamond},(\theta^{\varphi}_{\cdot,B})^{-1}\comp
\mathrm{H}(\cdot,f))$ is precisely
$$
(\mathbf{d}_{\diamond},(\theta^{\varphi}_{\cdot,B})^{-1}\comp
\mathrm{H}(\cdot,f)) = (\mathrm{Id}_{\mathbf{Ter}(\mathbf{\Sigma})},
\mathrm{H}(\cdot,f)),
$$
that we represent by the following diagram
$$
\xymatrix@C= 25ex@R= 10ex{
\mathbf{Ter}(\mathbf{\Sigma})
  \ar[r]^{\mathrm{Id}}
  \ar[d]_{\mathrm{Tr}^{\mathbf{\Sigma},\mathbf{A}}}
  &
\mathbf{Ter}(\mathbf{\Sigma})
  \ar[d]^{\mathrm{Tr}^{\mathbf{\Sigma},\mathbf{B}}}
  \\
\mathbf{Set}
  \ar[r]_{\mathrm{Id}}
  \ar@{}[ru]|{\dir{=>}}^-(.5){\mathrm{H}(\cdot,f)}&
\mathbf{Set}
}
$$
i.e., definitely, it corresponds to the natural transformation
$\mathrm{Tr}^{\mathbf{\Sigma},f}$ from the functor
$\mathrm{Tr}^{\mathbf{\Sigma},\mathbf{A}}$ to the functor
$\mathrm{Tr}^{\mathbf{\Sigma},\mathbf{B}}$.
\end{remark}

For an arbitrary, but fixed, signature $\mathbf{\Sigma}$ the family of
functors $(\mathrm{Tr}^{\mathbf{\Sigma},\mathbf{A}})_{\mathbf{A}\in
\mathbf{Alg}(\mathbf{\Sigma})}$ together with the family of natural
transformations $(\mathrm{Tr}^{\mathbf{\Sigma},f})_{f\in
\mathrm{Mor}(\mathbf{Alg}(\mathbf{\Sigma}))}$ actually constitute the
object and morphism mappings, respectively, of a functor
$\mathrm{Tr}^{\mathbf{\Sigma},(\cdot)}$ from the category
$\mathbf{Alg}(\mathbf{\Sigma})$ to the exponential category
$\mathbf{Set}^{\mathbf{Ter}(\mathbf{\Sigma})}$.  And it is precisely
the functor $\mathrm{Tr}^{\mathbf{\Sigma},(\cdot)}$ that will allow us
to prove, in the following proposition, the existence of a functor
$\mathrm{Tr}^{\mathbf{\Sigma}}$ from
$\mathbf{Alg}(\mathbf{\Sigma})\bprod\mathbf{Ter}(\mathbf{\Sigma})$ to
$\mathbf{Set}$ that formalizes the realization of terms as term
operations on algebras, but taking into account the variation of the
algebras through the homomorphisms between them.

\begin{proposition}\label{funcTr}
Let $\mathbf{\Sigma}$ be a signature.  Then there exists a functor
$\mathrm{Tr}^{\mathbf{\Sigma}}$ from
$\mathbf{Alg}(\mathbf{\Sigma})\bprod\mathbf{Ter}(\mathbf{\Sigma})$ to
$\mathbf{Set}$ defined as follows
\begin{enumerate}
\item $\mathrm{Tr}^{\mathbf{\Sigma}}$ sends a pair $(\mathbf{A},X)$,
      formed by a $\mathbf{\Sigma}$-algebra $\mathbf{A}$  and an
      $S$-sorted set $X$, to the set
      $\mathrm{Tr}^{\mathbf{\Sigma}}(\mathbf{A},X) =
      \mathrm{Tr}^{\mathbf{\Sigma},\mathbf{A}}(X) = A_{X}$ of the
      $S$-sorted mappings from $X$ to the underlying $S$-sorted set $A$
      of $\mathbf{A}$.

\item $\mathrm{Tr}^{\mathbf{\Sigma}}$ sends an arrow $(f,P)$
      from $(\mathbf{A},X)$ to $(\mathbf{B},Y)$ in
      $\mathbf{Alg}(\mathbf{\Sigma})\bprod\mathbf{Ter}(\mathbf{\Sigma})$,
      to the mapping $\mathrm{Tr}^{\mathbf{\Sigma}}(f,P) = f_{P}$ from
      $A_{X}$ to $B_{Y}$, i.e., to
      $\mathrm{Tr}^{\mathbf{\Sigma},\mathbf{B}}(P) \comp
      \mathrm{Tr}^{\mathbf{\Sigma},f}_{X} =
      \mathrm{Tr}^{\mathbf{\Sigma},f}_{Y}
      \comp\mathrm{Tr}^{\mathbf{\Sigma},\mathbf{A}}(P)$.
\end{enumerate}

\end{proposition}

\begin{proof}
We restrict ourselves to prove that $\mathrm{Tr}^{\mathbf{\Sigma}}$
preserves compositions.  Given two arrows $(f,P)$ from $(\mathbf{A},X)$
to $(\mathbf{B},Y)$ and $(g,Q)$ from $(\mathbf{B},Y)$ to
$(\mathbf{C},Z)$, we have that $(g\comp f)_{Q\diamond P}$ is the
diagonal mapping in the commutative diagram
$$
\xymatrix@C=25pt@R=15pt{
{} &
A_{X}
\ar[dd]|*+{(g\comp f)_{Q\diamond P}}
\ar[dl]_{(g\comp f)_{X}}
\ar[dr]^{(Q\diamond P)^{\mathbf{A}}} & {} \\
C_{X}
\ar[dr]_{(Q\diamond P)^{\mathbf{C}}} & {} &
A_{Z}
\ar[dl]^{(g\comp f)_{Z}} \\
{} &
C_{Z} & {}
}
$$
But it happens that the following diagram commutes
$$
\xymatrix@C=25pt@R=15pt{
{} & {} &
A_{X}
\ar[dd]|{f_{P}}
\ar[dl]_{f_{X}}
\ar[dr]^{P^{\mathbf{A}}} & {} & {} \\
{} &
B_{X}
\ar[dl]_{g_{X}}
\ar[dr]_{P^{\mathbf{B}}} & {} &
A_{Y}
\ar[dl]^{f_{Y}}
\ar[dr]^{Q^{\mathbf{A}}}& {} \\
C_{X}
\ar[dr]_{P^{\mathbf{C}}} & {} &
B_{Y}
\ar[dd]|{g_{Q}}
\ar[dl]_{g_{Y}}
\ar[dr]^{Q^{\mathbf{B}}}
& {} &
A_{Z}
\ar[dl]^{f_{Z}} \\
{} &
C_{Y}
\ar[dr]_{Q^{\mathbf{C}}} & {} &
B_{Z}
\ar[dl]^{g_{Z}} & {} \\
{} & {} &
C_{Z} & {} & {}
}
$$
Therefore, by Lemma~\ref{realizdiamondcompcomm}, we have that
$g_{Q}\comp f_{P} = (g\comp f)_{Q\diamond P}$.
\end{proof}

\begin{remark}
The functor $\mathrm{Tr}^{\mathbf{\Sigma}}$ can also be obtained from
the functor $\mathrm{Tr}^{\mathbf{\Sigma},(\cdot)}$ as the composite
of the functors $\mathrm{Tr}^{\mathbf{\Sigma},(\cdot)}\times
\mathrm{Id}_{\mathbf{Ter}(\mathbf{\Sigma})}$ and
$\mathrm{Ev}_{\mathbf{Ter}(\mathbf{\Sigma}),\mathbf{Set}}$, as in the
following diagram
$$
\xymatrix@C=95pt@R=50pt{
\mathbf{Alg}(\mathbf{\Sigma})\bprod \mathbf{Ter}(\mathbf{\Sigma})
\ar[d]_{\mathrm{Tr}^{\mathbf{\Sigma},(\cdot)}\times
        \mathrm{Id}_{\mathbf{Ter}(\mathbf{\Sigma})}}
\ar[dr]^{\mathrm{Tr}^{\mathbf{\Sigma}}}
& {} \\
\mathbf{Set}^{\mathbf{Ter}(\mathbf{\Sigma})}\times \mathbf{Ter}(\mathbf{\Sigma})
\ar[r]_-{\mathrm{Ev}_{\mathbf{Ter}(\mathbf{\Sigma}),\mathbf{Set}}}
 & \mathbf{Set}
}
$$
where $\mathrm{Ev}_{\mathbf{Ter}(\mathbf{\Sigma}),\mathbf{Set}}$ is the
evaluation functor, because of the natural isomorphism
$$
\mathrm{Hom}(\mathbf{Alg}(\mathbf{\Sigma}),
\mathbf{Set}^{\mathbf{Ter}(\mathbf{\Sigma})})
\cong \mathrm{Hom}(\mathbf{Alg}(\mathbf{\Sigma})\bprod
\mathbf{Ter}(\mathbf{\Sigma}),\mathbf{Set}).
$$
\end{remark}

To accomplish the earlier stated second goal, i.e., to show
the invariant character of the procedure of realization of terms under
signature change, we prove, for a morphism
$\mathbf{d}\colon \mathbf{\Sigma}\mor \mathbf{\Lambda}$, the existence
of a natural isomorphism between two functors from
$\mathbf{Alg}(\mathbf{\Lambda})\bprod\mathbf{Ter}(\mathbf{\Sigma})$ to
$\mathbf{Set}$, constructed from the functors
$\mathrm{Tr}^{\mathbf{\Lambda}}$, $\mathrm{Tr}^{\mathbf{\Sigma}}$,
$\mathbf{d}_{\diamond}$ and $\mathbf{d}^{\ast}$.

\begin{proposition}\label{lemaPdextranatural}
Let $\mathbf{d}\colon \mathbf{\Sigma}\mor \mathbf{\Lambda}$ be a signature
morphism.  Then the family $\theta^{\mathbf{d}} =
(\theta^{\mathbf{d}}_{\mathbf{A},X})_{(\mathbf{A},X)\in
\mathbf{Alg}(\mathbf{\Lambda})\bprod \mathbf{Ter}(\mathbf{\Sigma})}$, where
$\theta^{\mathbf{d}}_{\mathbf{A},X} = \theta^{\varphi}_{X,A}$, i.e., the
natural isomorphism of the adjunction
$\coprod_{\varphi}\ladj\Delta_{\varphi}$, is a natural isomorphism
as shown in the following diagram
$$
\xymatrix@C=20ex@R=10ex{
{\mathbf{Alg}(\mathbf{\Lambda})\bprod\mathbf{Ter}(\mathbf{\Sigma})}
  \ar[r]^-{\mathbf{d}^{\ast}\bprod\Id}
  \ar[d]_{\Id\bprod \mathbf{d}_{\diamond}} &
{\mathbf{Alg}(\mathbf{\Sigma})\bprod\mathbf{Ter}(\mathbf{\Sigma})} \xyn{d}
  \ar[d]^{\mathrm{Tr}^{\mathbf{\Sigma}}} \\
{\mathbf{Alg}(\mathbf{\Lambda})\bprod\mathbf{Ter}(\mathbf{\Lambda})} \xyn{c}
  \ar[r]_-{\mathrm{Tr}^{\mathbf{\Lambda}}} &
\mathbf{Set}
\ar@{} "c";"d"|{\dir{=>}}^{\theta^{\mathbf{d}}}
}
$$
\end{proposition}

\begin{proof}
Let $(f,P)\colon(\mathbf{A},X)\mor (\mathbf{B},Y)$ be a morphism in
${\mathbf{Alg}(\mathbf{\Lambda})\bprod\mathbf{Ter}(\mathbf{\Sigma})}$.  Then we have
the following situation
$$
\xy
0;<1ex,0ex>:<0ex,1ex>::
%
\POS 0+(0,0)
\xymatrix@C=20ex@R=8ex@!0{
(\mathbf{A},X) \ar[r]^{(f,P)} & (\mathbf{B},Y)
\save"1,1"."1,2"!C*\frm{}="a"\restore 
}
\POS 0+(20,-10)
\xymatrix@C=21ex@R=8ex@!0{
(\mathbf{d}^{\ast}(\mathbf{A}),X)
  \ar[r]^{(f_{\varphi},P)} &
(\mathbf{d}^{\ast}(\mathbf{B}),Y)
\save"1,1"."1,2"!C*+<5ex>\frm{}="b"\restore 
}
\POS 0+(-20,-10)
\xymatrix@C=21ex@R=8ex@!0{
(\mathbf{A},\coprod_{\varphi}X)
  \ar[r]^{(f,\mathbf{d}_{\diamond}(P))}&
(\mathbf{B},\coprod_{\varphi}Y)
\save"1,1"."1,2"!C*+<5ex>\frm{}="c"\restore 
}
\POS 0+(14,-20)
\xymatrix"d"@C=11ex@R=11ex@!0{
&
(A_{\varphi})_{X}
  \ar[rd]|*+{P^{\mathbf{d}^{\ast}(\mathbf{A})}}
  \ar[ld]|*+{(f_{\varphi})_{X}}
  \\
(B_{\varphi})_{X}
  \ar[rd]|*{P^{\mathbf{d}^{\ast}(\mathbf{B})}}
  &&
(A_{\varphi})_{Y}
  \ar[ld]|*+{(f_{\varphi})_{Y}}
  \\
&(B_{\varphi})_{Y}
\save "d1,2"."d3,2"!C*+<1ex>\frm{}="d"\restore 
}
\POS (-14,-25)
\xymatrix"e"@C=12ex@R=11ex@!0{
&
A_{\coprod_{\varphi}X}
  \ar[rd]|*+{\mathbf{d}_{\diamond}(P)^{\mathbf{A}}}
  \ar[ld]|*+{f_{\coprod_{\varphi}\!X}}
  \\
B_{\coprod_{\varphi}X}
  \ar[rd]|*+{\mathbf{d}_{\diamond}(P)^{\mathbf{B}}}
  &&
A_{\coprod_{\varphi}Y}
  \ar[ld]|*+{f_{\coprod_{\varphi}\!Y}}
  \\
&
B_{\coprod_{\varphi}Y}
\save "e1,2"."e3,2"!C*+<5ex>\frm{}="e"\restore 
}
\ar "e1,2";"d1,2"|*+{\theta^{\varphi}_{X,A}}
\ar "e2,1";"d2,1"|*+{\theta^{\varphi}_{X,B}}
\ar "e2,3";"d2,3"|*+{\theta^{\varphi}_{Y,A}}
\ar "e3,2";"d3,2"|*+{\theta^{\varphi}_{Y,B}}
\ar @{|->}"a";"b"
\ar @{|->}"a";"c"
\ar @{|->}"b";"d"
\ar @{|->}"c";"e"
\endxy
$$
But the bottom diagram in the above figure commutes, because of
Proposition~\ref{invarianza}, the naturalness of $\theta^{\varphi}$,
and the fact that $f$ is a $\mathbf{\Lambda}$-homomorphism.  Therefore
the following diagram also commutes
$$
\xymatrix@C=20ex@R=8ex{
{A_{\coprod_{\varphi}X}}
  \ar[r]^-{\theta^{\varphi}_{X,A}}
  \ar[d]_{f_{\mathbf{d}_{\diamond}(P)}} &
{(A_{\varphi})_{X}}
  \ar[d]^{(f_{\varphi})_{P}} \\
{B_{\coprod_{\varphi}Y}}
  \ar[r]_-{\theta^{\varphi}_{Y,B}} &
{(B_{\varphi})_{Y}}
}
$$
From this it follows that the family $\theta^{\mathbf{d}}$ is a natural
isomorphism from $\mathrm{Tr}^{\mathbf{\Lambda}}\comp (\Id\bprod
\mathbf{d}_{\diamond})$ to $\mathrm{Tr}^{\mathbf{\Sigma}}\comp (
\mathbf{d}^{\ast}\bprod\Id)$.
\end{proof}


If we just recapitulate about that which has been, essentially,
obtained up to this point, then we can summarize it by saying that
what we have at our disposal consists of the following:
\begin{enumerate}
\item The contravariant functor $\Alg$ from $\mathbf{Sig}$ to $\mathbf{Cat}$,
      giving the category of models of a given signature,

\item The pseudo-functor $\mathrm{Ter}$ from $\mathbf{Sig}$ to
      $\mathbf{Cat}$, giving the dual of the Kleisli category for the
      standard monad derived from the adjunction induced by a given
      signature,

\item The family of functors
      $\mathrm{Tr}=(\mathrm{Tr}^{\mathbf{\Sigma}})_{\mathbf{\Sigma}\in\mathbf{Sig}}$,
      where, for every signature $\mathbf{\Sigma}$,
      $\mathrm{Tr}^{\mathbf{\Sigma}}$ is a functor from
      $\mathbf{Alg}(\mathbf{\Sigma})\bprod\mathbf{Ter}(\mathbf{\Sigma})$
      to $\mathbf{Set}$ that formalize the realization of terms as
      term operations on algebras, and

\item The family of natural isomorphisms
      $\theta=(\theta^{\mathbf{d}})_{\mathbf{d}\in\Mor(\mathbf{Sig})}$,
      where, for every signature morphism $\mathbf{d}$,
      $\theta^{\mathbf{d}}$ is the natural isomorphism that explains
      the invariant character of the procedure of realization of terms
      under the variation of the signature.
\end{enumerate}

Our next goal is to construct the many-sorted term institution by
combining adequately the above components.  To attain the just stated
goal we need to recall beforehand some auxiliary category-theoretic
concepts.  In particular, we proceed to define next, among others, the
concept of pseudo-extranatural transformation in $2$-categories and
for pseudo-functors.  This generality is necessary because, later on
(e.g., in the sixth section), after adding a certain type of $2$-cells
(the transformations of Fujiwara) to the category
$\mathbf{Sig}_{\mathfrak{pd}}$, of signatures and polyderivors (defined in
the fifth section), that converts it into a $2$-category, still
denoted by the same symbol, we will get another $2$-category,
precisely $\mathbf{Sig}_{\mathfrak{pd}}^{\mathrm{op}}\times
\mathbf{Sig}_{\mathfrak{pd}}$, on which we will define a pseudo-functor
related by a pseudo-extranatural transformation to a definite functor
(also defined on the same derived $2$-category), from which we will
get the so-called term $2$-institution of Fujiwara.

\begin{definition}\label{defLaxDinatural}
Let $\mathbf{C}$ and $\mathbf{D}$ be two $2$-categories,
$F,G\colon\mathbf{C}^{\opp}\bprod \mathbf{C}\mor\mathbf{D}$ two
2-functors, and $(\alpha,\beta)$ a pair such that
\begin{enumerate}
\item For every 0-cell $c$ in $\mathbf{C}$, $\alpha_{c}\colon F(c,c)\mor
      G(c,c)$ is a 1-cell in $\mathbf{D}$.

\item For every 1-cell $f\colon c\mor c'$ in $\mathbf{C}$, $\beta_{f}$
      is a $2$-cell in $\mathbf{D}$
      from
      $G(1,f)\comp \alpha_{c}\comp F(f,1)$
      to
      $G(f,1)\comp \alpha_{c'}\comp F(1,f)$.
\end{enumerate}
Then we say that $(\alpha,\beta)$ is a
\begin{enumerate}
\item  \emph{Lax-dinatural transformation} from $F$ to $G$ if,
       for every $2$-cell $\xi\colon f\Rightarrow g$ in $\mathbf{C}$,
       we have that
       $$
       \beta_{g}\comp(G(1,\xi)\hcomp\alpha_{c}\hcomp F(\xi,1)) =
       (G(\xi,1)\hcomp\alpha_{c'}\hcomp F(1,\xi))\comp \beta_{f}.
       $$
\item \emph{Pseudo-dinatural transformation} from $F$ to $G$ if it
      is a lax-dinatural transformation and, for every
      $f\colon c\mor c'$ in $\mathbf{C}$, $\beta_{f}$ is an isomorphism.

\item $2$-\emph{dinatural transformation} from $F$ to $G$ if it is
      a lax-dinatural transformation and, for every
      $f\colon c\mor c'$ in $\mathbf{C}$, $\beta_{f}$ is an identity.
\end{enumerate}
\end{definition}

The dinatural transformations when $F$ and $G$ are pseudo-functors
will also be relevant for us.  In this case it is necessary to impose
additional conditions of compatibility with the natural isomorphisms
of the pseudo-functors.  The definition is as follows.

\begin{definition}
Let $\mathbf{C}$ and $\mathbf{D}$ be two $2$-categories,
$(F,\gamma^{F},\nu^{F})$ and $(G,\gamma^{G},\nu^{G})$ two
pseudo-functors from $\mathbf{C}^{\opp}\bprod \mathbf{C}$ to
$\mathbf{D}$, and $(\alpha,\beta)$ a pair such that
\begin{enumerate}
\item For every 0-cell $c$ in $\mathbf{C}$, $\alpha_{c}\colon F(c,c)\mor
      G(c,c)$ is a 1-cell in $\mathbf{D}$.

\item For every 1-cell $f\colon c\mor c'$ in $\mathbf{C}$, $\beta_{f}$
      is a $2$-cell in $\mathbf{D}$
      from
      $G(1,f)\comp \alpha_{c}\comp F(f,1)$
      to
      $G(f,1)\comp \alpha_{c'}\comp F(1,f)$.
\end{enumerate}
Then we say that $(\alpha,\beta)$ is a \emph{lax-dinatural
transformation} from $(F,\gamma^{F},\nu^{F})$ to
$(G,\gamma^{G},\nu^{G})$ if it satisfies the following compatibility
conditions:
\begin{enumerate}
\item  For every $2$-cell $\xi\colon f\Rightarrow g$ in $\mathbf{C}$,
       we have that
       $$
       \beta_{g}\comp(G(1,\xi)\hcomp\alpha_{c}\hcomp F(\xi,1)) =
       (G(\xi,1)\hcomp\alpha_{c'}\hcomp F(1,\xi))\comp \beta_{f}.
       $$
\item For every pair of 1-cells $f\colon c\mor c'$, $g\colon
      c'\mor c''$ in
      $\mathbf{C}$, we have that
      \begin{multline*}
      \gamma^{F}_{(1,f),(1,g)}\comp
      \bigl({G(f,1)\hcomp\beta_{g}\hcomp F(1,f)}\bigr)
      \comp
      \bigl({G(1,g)\hcomp\beta_{f}\hcomp F(g,1)}\bigr) \\
      =
      {\beta_{g\comp f}}
      \comp \bigl({\gamma^{G}_{(1,f),(1,g)}\hcomp\alpha_{c}\hcomp
          \gamma^{F}_{(g,1),(f,1)}}\bigr).
      \end{multline*}

\item  For every object $c$ in $\mathbf{C}$, we have that
       $$
       \alpha_{c}\hcomp \nu^{F}_{(c,c)} =
       \nu^{G}_{(c,c)}\hcomp \alpha_{c}.
       $$
\end{enumerate}
\end{definition}

If the pseudo-functor $G$ is independent of both variables, then we say that
the above transformations are \emph{lax-extranatural},
\emph{pseudo-extranatural} or \emph{extranatural}, respectively.  Then
the compatibility with the $2$-cells of $\mathbf{C}$ is equivalent to
$$
\beta_{g}\comp (\alpha_{c}\hcomp F(\xi,1)) =
(\alpha_{c'}\hcomp F(1,\xi))\comp \beta_{f},
$$
or geometrically:
$$
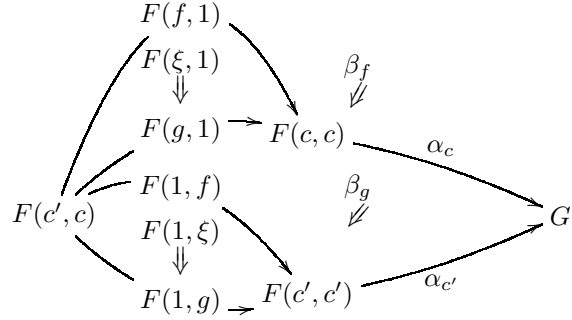
\xy
0;<22ex,-0ex>:<0ex,7ex>::
\POS 0+(0,0)*+{F(c',c)}="a",
\POS 0+(1,1)*+{F(c,c)}="b",
\POS 0+(2,0)*+{G}="c",
\POS 0+(1,-1)*+{F(c',c')}="b'",
\ar @/r^14ex/ "a";"b" |*+{F(f,1)}="sf1"
\ar @/r^4ex/  "a";"b" |*+{F(g,1)}="sg1"
\ar @/r^1ex/  "b";"c"  ^{\alpha_{c}}="x"
\ar @/r^6ex/ "a";"b'" |*+{F(1,f)}="s1f"
\ar @/r_4ex/ "a";"b'" |*+{F(1,g)}="s1g"
\ar @/r_1ex/ "b'";"c" _{\alpha_{c'}}="y"
\ar @{} "sf1";"sg1" |(.65){\dir{=>}}|(.4)*+{F(\xi,1)}
\ar @{} "s1f";"s1g" |(.65){\dir{=>}}|(.4)*+{F(1,\xi)}
\POS 0+(1.2,1.5)*:a(-30)@_{=>}
\POS 0+(1.2,1.85)*{\beta_{f}},
\POS 0+(1.2,0)*:a(-40)@_{=>}
\POS 0+(1.2,0.35)*{\beta_{g}}
\endxy
$$
and the compatibility of the composition of $1$-cells in $\mathbf{C}$ with
the natural isomorphisms of $F$ is equivalent to
$$
\gamma^{F}_{(1,f),(1,g)}\comp
(\beta_{g}\hcomp F(1,f))\comp (\beta_{f}\hcomp F(g,1)) =
\beta_{g\comp f}\comp (\alpha_{c}\hcomp
\gamma^{F}_{(g,1),(f,1)}),
$$
or geometrically:
{\Large
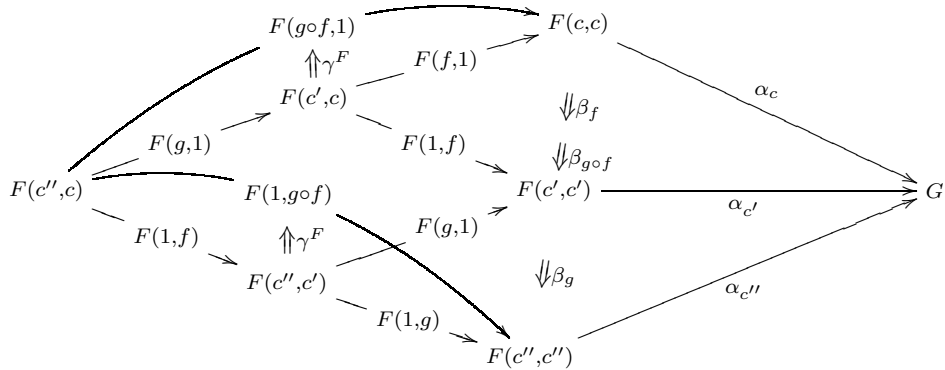
$$
\def\labelstyle{\scriptstyle}     
\def\twocellstyle{\scriptstyle}
\def\objectstyle{\scriptstyle}
\xy
0;<95pt,-0pt>:<5pt,35pt>::
\POS 0+(0,0)*+{F(c'',c)}="a",
\POS 0+(1,1)*+{F(c',c)}="a1",
\POS 0+(1,-1)*+{F(c'',c')}="a2",
\POS 0+(2,1.8)*+{F(c,c)}="b1",
\POS 0+(2,-1.8)*+{F(c'',c'')}="b2",
\POS 0+(2,0)*+{F(c',c')}="e",
\POS 0+(3.5,0)*+{G}="f",
\ar "a";"a1"  |*+{F(g,1)}=""
\ar "a1";"b1" |*+{F(f,1)}=""
\ar "b1";"f" ^{\alpha_{c}}="alphac"
\ar "a1";"e"  |*+{F(1,f)}=""
\ar "e";"f"   _{\alpha_{c'}} ="alphac'"
\ar "a";"a2"  |*+{F(1,f)}=""
\ar "a2";"e"  |(.6)*+{F(g,1)}=""
\ar "a2";"b2" |*+{F(1,g)}=""
\ar "b2";"f" _{\alpha_{c''}} ="alphac''"
\ar @{} "b1";"e" |(.5){\dir{=>}}^{\beta_{f}}
\ar @{} "e";"b2" |(.5){\dir{=>}}^{\beta_{g}}
\ar @{} "b1";"e" |(.8){\dir{=>}}^(.8){\beta_{g\comp f}}
%
\ar @/va(-0)^30pt/ "a";"b1" |*+{F(g\comp f,1)}="sgf1"
\ar @/va(-0)^30pt/ "a";"b2" |*+{F(1,g\comp f)}="s1gf"
\ar @{} "a1";"sgf1"|(.5){\dir{=>}}_{\gamma^{F}}
\ar @{} "a2";"s1gf"  |(.5){\dir{=>}}_{\gamma^{F}}
\endxy
$$
}

In the next proposition we construct a pseudo-functor
$\Alg(\farg)\bprod \mathrm{Ter}(\farg)$ from the product category
$\mathbf{Sig}^{\opp}\bprod \mathbf{Sig}$ to $\mathbf{Cat}$ (obtained from the
contravariant functor $\Alg$ and the pseudo-functor $\mathrm{Ter}$),
and prove that the family
$\mathrm{Tr}=(\mathrm{Tr}^{\mathbf{\Sigma}})_{\mathbf{\Sigma}\in\mathbf{Sig}}$,
together with the family
$\theta=(\theta^{\mathbf{d}})_{\mathbf{d}\in\Mor(\mathbf{Sig})}$ is a
pseudo-extranatural transformation from the pseudo-functor
$\Alg(\farg)\bprod\mathrm{Ter}(\farg)$ to the functor
$\mathrm{K}_{\mathbf{Set}}$ (from $\mathbf{Sig}^{\opp}\bprod \mathbf{Sig}$ to
$\mathbf{Cat}$) that is constantly $\mathbf{Set}$.

\begin{proposition}\label{PdExtranatural}
There exists a pseudo-functor $\Alg(\farg)\bprod \mathrm{Ter}(\farg)$
from  $\mathbf{Sig}^{\opp}\bprod \mathbf{Sig}$ to $\mathbf{Cat}$,
obtained from the contravariant functor $\Alg$ and the pseudo-functor
$\mathrm{Ter}$, which sends a pair of signatures
$(\mathbf{\Sigma},\mathbf{\Lambda})$ to the category
$\mathbf{Alg}(\mathbf{\Sigma})\bprod\mathbf{Ter}(\mathbf{\Lambda})$, and a pair of
signature morphisms $(\mathbf{d},\mathbf{e})$ from
$(\mathbf{\Sigma},\mathbf{\Lambda})$ to $(\mathbf{\Sigma}',\mathbf{\Lambda}')$ in
$\mathbf{Sig}^{\opp}\bprod\mathbf{Sig}$ to the functor $\mathbf{d}^{\ast}\bprod
\mathbf{e}_{\diamond}$ from
$\mathbf{Alg}(\mathbf{\Sigma})\bprod\mathbf{Ter}(\mathbf{\Lambda})$ to
$\mathbf{Alg}(\mathbf{\Sigma}')\bprod\mathbf{Ter}(\mathbf{\Lambda}')$.

Furthermore, the family of functors
$\mathrm{Tr}=(\mathrm{Tr}^{\mathbf{\Sigma}})_{\mathbf{\Sigma}\in\mathbf{Sig}}$,
together with the family
$\theta=(\theta^{\mathbf{d}})_{\mathbf{d}\in\Mor(\mathbf{Sig})}$,
where $\theta^{\mathbf{d}}$ is the natural isomorphism of
Proposition~\ref{lemaPdextranatural}, is a pseudo-extranatural
transformation from the pseudo-functor
$\Alg(\farg)\bprod\mathrm{Ter}(\farg)$ to the functor
$\mathrm{K}_{\mathbf{Set}}$ from $\mathbf{Sig}^{\opp}\bprod
\mathbf{Sig}$ to $\mathbf{Cat}$ that is constantly $\mathbf{Set}$.

\end{proposition}

\begin{proof}
Because the structure of $2$-category of $\mathbf{Sig}$ is, in this case,
trivial, we need only show the compatibility with the natural
isomorphisms of the pseudo-functor $\Alg(\farg)\bprod
\mathrm{Ter}(\farg)$.

We restrict our attention to show the compatibility of the composition
of $1$-cells in $\mathbf{Sig}$ with the natural isomorphisms of
$\Alg(\farg)\bprod \mathrm{Ter}(\farg)$.  But for this, it is enough
to verify that, for every $f\colon \mathbf{A}\mor \mathbf{B}$ in
$\mathbf{Alg}(\mathbf{\Omega})$ and $P\colon X\mor Y$ in
$\mathbf{Ter}(\mathbf{\Sigma})$, the following diagram commutes %
\begin{narrow}{-8pt}{-8pt}
$$
\xy
0;<1ex,0ex>:<0ex,1ex>::
%
\POS 0+(8,0)
\xymatrix"d"@C=30ex@R=19ex@!0{
(A_{\varphi})_{\coprod_{\varphi}X}
  \ar[r]|*+{\theta^{\varphi}_{X,A_{\psi}}}
  \ar[d]|*+{\mathbf{d}_{\diamond}(P)^{\mathbf{e}^{\ast}(\mathbf{A})}}
  &
({A_{\psi}}_{\varphi})_{X}
  \ar[d]|*+{P^{\mathbf{d}^{\ast}(\mathbf{e}^{\ast}(\mathbf{A}))}}
   \\
(A_{\varphi})_{\coprod_{\varphi}Y}
  \ar[r]|*+{\theta^{\varphi}_{Y,A_{\psi}}}
  \ar[d]|*+{(f_{\varphi})_{\coprod_{\varphi}Y}}
  &
({A_{\psi}}_{\varphi})_{Y}
  \ar[d]|*+{({f_{\psi}}_{\varphi})_{Y}}
   \\
(B_{\psi})_{\coprod_{\varphi}Y}
  \ar[r]|*+{\theta^{\varphi}_{Y,B_{\psi}}}
  &
({B_{\psi}}_{\varphi})_{Y}
\save "d1,2"."d3,1"!C*+<5ex>\frm{}="d"\restore 
}
\POS (-13,-6)
\xymatrix"e"@C=30ex@R=19ex@!0{
A_{\coprod_{\psi}\coprod_{\varphi}X}
  \ar[r]|*+{(\gamma^{\mathbf{d},\mathbf{e}}_{X})^{\mathbf{A}}}
  \ar[d]|*+{\mathbf{e}_{\diamond}\comp\mathbf{d}_{\diamond}(P)^{\mathbf{A}}}&
A_{\coprod_{\psi\comp \varphi}X}
  \ar[d]|*+{\mathbf{e}_{\diamond}(\mathbf{d}_{\diamond}(P))^{\mathbf{A}}}  \\
A_{\coprod_{\psi}\coprod_{\varphi}Y}
  \ar[r]|*+{(\gamma^{\mathbf{d},\mathbf{e}}_{Y})^{\mathbf{A}}}
  \ar[d]|*+{f_{\coprod_{\psi}\coprod_{\varphi}Y}}&
A_{\coprod_{\psi\comp \varphi}Y}
  \ar[d]|*+{f_{\coprod_{\psi\comp \varphi}Y}}  \\
B_{\coprod_{\psi}\coprod_{\varphi}Y}
  \ar[r]|*+{(\gamma^{\mathbf{d},\mathbf{e}}_{Y})^{\mathbf{B}}} &
B_{\coprod_{\psi\comp\varphi}Y}
\save"e1,2"."e3,1"!C*+<10ex>\frm{}="e"\restore 
}
\ar "e1,1";"d1,1"|*-{\theta^{\psi}_{\coprod_{\varphi}X,A}}
\ar "e1,2";"d1,2"|*+{\theta^{\psi\comp\varphi}_{X,A}}
\ar "e2,1";"d2,1"|*-{\theta^{\psi}_{\coprod_{\varphi}Y,A}}
\ar "e2,2";"d2,2"|*+{\theta^{\psi\comp\varphi}_{Y,A}}
\ar "e3,1";"d3,1"|*-{\theta^{\psi}_{\coprod_{\varphi}Y,B}}
\ar "e3,2";"d3,2"|*+{\theta^{\psi\comp\varphi}_{Y,B}}
\endxy
$$
\end{narrow}
And this is so in consequence of the definitions of the involved
entities.
\end{proof}


The preceding proposition can be reformulated in a more compact form,
taking into account the directing principles of the institutional
frame of Goguen and Burstall in~\cite{gb86}, as asserting the
existence of a certain institution on the category $\mathbf{Set}$.  By
doing so the conceptual and structural richness involved in the
proposition is fully and elegantly reflected in the institution
structure.  However, to actually realize the noted reformulation we
should begin by stating a concept of institution that generalizes,
even more, that one defined by Goguen and Burstall in~\cite{gb86},
owing to reasons explained below.

Before we define the adequate concept of institution and justify
the underlying reasons for positing it, as a perhaps helpful
historical remark, it seems appropriate to recall that direct
ancestors of the concept of institution, as defined by Goguen and
Burstall in~\cite{gb84}, are, at the very least, that of \emph{regular
model-theoretic language} defined by Feferman in~\cite{f74}, pp.
155--156,
as a system
$$
L = (\mathrm{Typ}_{L},\mathrm{Str}_{L},\mathrm{Stc}_{L},\models_{L})
$$
where $\mathrm{Typ}_{L}$ is a non-empty set of similarity types,
called the \emph{admitted types} of $L$, and $\mathrm{Str}_{L}$,
$\mathrm{Stc}_{L}$, $\models_{L}$ are functions with domain
$\mathrm{Typ}_{L}$ such that for each admitted types $\tau$, $\tau'$:
\begin{enumerate}
\item[(i)] $\mathrm{Str}_{L}(\tau)$ is a sub-collection of
      $\mathrm{Str}(\tau)$, called the \emph{admitted structures for}
      $L(\tau)$,

\item[(ii)] $\mathrm{Stc}_{L}(\tau)$ is a collection called the
      \emph{sentences} of $L(\tau)$,

\item[(iii)] $\models_{L,\tau}$ is a sub-relation of
      $\mathrm{Str}_{L}(\tau)\times \mathrm{Stc}_{L}(\tau)$, called
      the \emph{satisfaction} (or \emph{truth}) \emph{relation of}
      $L(\tau)$,

\item[(iv)] \textsl{\textbf{Expansion}}. $\tau\subseteq \tau'\Rightarrow
      \mathrm{Stc}_{L}(\tau)\subseteq \mathrm{Stc}_{L}(\tau')$;
      $\mathfrak{M}'\in \mathrm{Str}_{L}(\tau')\Rightarrow
      \mathfrak{M}'\upharpoonright\tau\in \mathrm{Str}_{L}(\tau)$
      \text{ and } $\varphi\in \mathrm{Stc}_{L}(\tau)\Rightarrow
      [\mathfrak{M}'\upharpoonright\tau\models \varphi \Leftrightarrow
      \mathfrak{M}'\models \varphi]$,

\item[(v)] \textsl{Renaming.} Each $\tau\equiv_{\gamma}\tau'$ induces
      a $1-1$ correspondence
      $$
      \overline{\gamma}\colon \mathrm{Stc}(\tau)\mor \mathrm{Stc}(\tau')
      $$
      such that if $\mathfrak{M}\in \mathrm{Str}_{L}(\tau)$ and
      $\mathfrak{M}'\in \mathrm{Str}(\tau')$ and
      $\mathfrak{M}\equiv_{\gamma}\mathfrak{M}'$, then
      $\mathfrak{M}'\in \mathrm{Str}_{L}(\tau')$ and
      $\mathfrak{M}\models \varphi \Leftrightarrow
      \mathfrak{M}'\models \overline{\gamma}(\varphi)$, and

\item [(vi)] \textsl{Isomorphism.} If $\mathfrak{M}\in
      \mathrm{Str}_{L}(\tau)$ and $\mathfrak{M}'\in
      \mathrm{Str}(\tau)$ and  $\mathfrak{M}\cong \mathfrak{M}'$, then
      $\mathfrak{M}'\in\mathrm{Str}_{L}(\tau)$ and $\mathfrak{M}\models \varphi
      \Leftrightarrow \mathfrak{M}'\models \varphi$;
\end{enumerate}
and that of a \emph{logic} $\mathfrak{L}^{\ast}$
defined by Barwise in~\cite{b74}, pp.  234--235,
where he says that it consists of a
syntax and a semantics which fit together nicely.  The \emph{syntax
of} $\mathfrak{L}^{\ast}$ is a functor $\ast$ on some category
$\mathbf{C}$ of languages to the category of classes.  The functor $\ast$
satisfies the following axiom:
\begin{enumerate}
\item[] \textsl{Occurrence Axiom.} For every
     $\mathfrak{L}^{\ast}$-sentence $\varphi$ there is a smallest
     (under $\subseteq$) language $L_{\varphi}$ in $\mathbf{C}$ such
     that $\varphi\in L_{\varphi}^{\ast}$. If $i\colon
     L_{\varphi}\subseteq K$ is an inclusion morphism, so is $i^{\ast}\colon
     L_{\varphi}^{\ast}\subseteq K^{\ast}$.
\end{enumerate}

The \emph{semantics of} $\mathfrak{L}^{\ast}$ is a relation $\models$
such that if $\mathfrak{M}\models \varphi$, then $\mathfrak{M}$ is
an $L$-structure for some $L$ in $\mathbf{C}$ and $\varphi\in L^{\ast}$.
It satisfies the following axiom:
\begin{enumerate}
\item[] \textsl{Isomorphism Axiom.} If $\mathfrak{M}\models \varphi$
and $\mathfrak{M}\cong \mathfrak{N}$, then $\mathfrak{N}\models \varphi$.
\end{enumerate}

The syntax and semantics of $\mathfrak{L}^{\ast}$ fit together
according to the final axiom.
\begin{enumerate}
\item[] \textsl{\textbf{Translation Axiom}}. For every $\mathfrak{L}^{\ast}$
        sentence $\varphi$, every $K$-structure $\mathfrak{M}$
        and every morphism $\alpha\colon L_{\varphi}\mor K$
        $$
        \mathfrak{M}\models \alpha^{\ast}(\varphi) \text{\,\, iff \,\,}
        \mathfrak{M}^{-\alpha} \text{\, is an \,}
        L_{\varphi}\!\!-\!\!\text{structure\,} \text{\,\,and\,\,}
        \mathfrak{M}^{-\alpha}\models \varphi.
        $$
\end{enumerate}

We recall that Goguen and Burstall in~\cite{gb84}, p. 229, define
an \emph{institution} as a category $\mathbf{Sign}$, of signatures, a
functor $\mathrm{Sen}$ from $\mathbf{Sign}$ to $\mathbf{Set}$, giving
the set of \emph{sentences} over a given signature, a functor
$\mathrm{Mod}$ from $\mathbf{Sign}$ to $\mathbf{Cat}^{\mathrm{op}}$,
giving the category of \emph{models} of a given signature, and, for
each $\Sigma\in \mathbf{Sign}$, a satisfaction relation
$\models_{\Sigma}\subseteq\bb{\mathbf{Mod}(\Sigma)}\times
\mathrm{Sen}(\Sigma)$, where $\bb{\cdot}$ is the endofunctor of
$\mathbf{Cat}$ which sends a category to the discrete category on its set
of objects, such that, for each morphism $\varphi\colon \Sigma\mor
\Sigma'$, the
\begin{enumerate}
\item[]\textsl{\textbf{Satisfaction Condition}}.
$
\mathbf{M'}\models_{\Sigma'} \varphi(e) \text{\,\, iff \,\,}
\varphi(\mathbf{M'})\models_{\Sigma} e,
$
\end{enumerate}
holds for each $\mathbf{M'}\in \bb{\mathbf{Mod}(\Sigma')}$ and each
$e\in \mathrm{Sen}(\Sigma)$.  Later on, in~\cite{gb86}, p. 316, they
define an \emph{institution} as a category $\mathbf{Sign}$, of
signatures, a functor $\mathrm{Sen}$ from $\mathbf{Sign}$ to
$\mathbf{Cat}$ (observe the large-scale change from $\mathbf{Set}$ to
$\mathbf{Cat}$), giving \emph{sentences} and \emph{proofs} over a
given signature, a functor $\mathrm{Mod}$ from $\mathbf{Sign}$ to
$\mathbf{Cat}^{\mathrm{op}}$, giving the category of \emph{models} of
a given signature, and a satisfaction relation
$\models_{\Sigma}\subseteq\bb{\mathbf{Mod}(\Sigma)}\times
\bb{\mathbf{Sen}(\Sigma)}$, for each $\Sigma\in \bb{\mathbf{Sign}}$,
such that
\begin{enumerate}
\item[]\textsl{\textbf{Satisfaction Condition}}:
$
\mathbf{M'}\models_{\Sigma'}
\mathrm{Sen}(\varphi)s \text{\, iff \,}
\mathrm{Mod}(\varphi)\mathbf{M'}\models_{\Sigma} s,
$
for each $\varphi\colon \Sigma\mor \Sigma'$ in $\mathbf{Sign}$,
$\mathbf{M'}\in \bb{\mathbf{Mod}(\Sigma')}$ and $s\in
\bb{\mathbf{Sen}(\Sigma)}$, and
\end{enumerate}

\begin{enumerate}
\item[]\textsl{Soundness Condition}:
$ \text{if } \mathbf{M}\models_{\Sigma}s,
\text{ then } \mathbf{M}\models_{\Sigma} s',
$
for each $\mathbf{M}\in\bb{\mathbf{Mod}(\Sigma)}$ and $s\mor s'\in
\mathbf{Sen}(\Sigma)$.
\end{enumerate}

Besides, the same authors, in~\cite{gb86}, p.  327, define, for a
category $\mathbf{V}$, a \emph{generalized}
$\mathbf{V}$-\emph{institution} as a pair of functors $\mathrm{Mod}$,
from $\mathbf{Sign}^{\mathrm{op}}$ to $\mathbf{Cat}$, and
$\mathrm{Sen}$, from $\mathbf{Sign}$ to $\mathbf{Cat}$, with an
extranatural transformation $\models$ from
$\bb{\mathrm{Mod}(\cdot)}\times \mathrm{Sen}(\cdot)$ to $\mathbf{V}$.
Observe that the second concept of institution falls under this last
one because, taking as $\mathbf{V}$ the category $\mathbf{2}$, with
two objects and just one morphism not the identity, the existence of
an extranatural transformation from $\bb{\mathrm{Mod}(\cdot)}\times
\mathrm{Sen}(\cdot)$ to $\mathbf{2}$ is equivalent to the above
satisfaction and soundness conditions.


But it happens that terms and algebras are not only compatible with
signature changes, but also with the category structure on the terms
and the algebras.  From this it follows that the restriction imposed
by Goguen and Burstall in~\cite{gb86}, p.  327, to the concept of
institution, concretely, that the domain of the extranatural
transformation is $\bb{\mathrm{Mod}(\cdot)}\times
\mathrm{Sen}(\cdot)$, is a real loss of generality, that prevents to
reflect faithfully the involved complexity.  Therefore such a
restriction, at least in this case, is unsound and should be left out.
Thus, under these circumstances, we propose the following definitions
of $2$-institution and institution, both relative to a given category.

\begin{definition}
Let $\mathbf{C}$ be a category.  Then a $2$-\emph{institution on}
$\mathbf{C}$ is a quadruple $(\mathbf{Sig},\Mod,\Sen,(\alpha,\beta))$,
where
\begin{enumerate}
\item $\mathbf{Sig}$ is a $2$-category.

\item $\Mod\colon\mathbf{Sig}^{\opp}\mor \mathbf{Cat}$ a pseudo-functor.

\item $\Sen\colon\mathbf{Sig}\mor\mathbf{Cat}$ a pseudo-functor.

\item $(\alpha,\beta)\colon\Mod(\farg)\bprod\Sen(\farg)\mor
             \mathrm{K}_{\mathbf{C}}$ a pseudo-extranatural transformation.
\end{enumerate}
If $\mathbf{Sig}$ is an ordinary category, instead of a $2$-category, then
we will speak of an \emph{institution on} $\mathbf{C}$.
\end{definition}

\begin{remark}
The concept of $2$-institution is defined relative to a category,
i.e., it has meaning for a $0$-cell $\mathbf{C}$ of the $2$-category
$\mathbf{Cat} = 1\!-\!\mathbf{Cat}$ of categories, functors, and natural
transformations between functors.  Therefore, if it were necessary for
some application, the concept of $3$-institution ought to be defined
relative to a $0$-cell $\mathbf{C}$ of the $3$-category
$2\!-\!\mathbf{Cat}$ of $2$-categories, $2$-functors, $2$-natural
transformations and modifications between transformations, and so
forth.
\end{remark}

Actually, $2$-institutions and institutions, if they are understood as
pseudo-extra\-natural transformations, go beyond both the classical
conception of semantical truth defined (mathematically for the first
time, through a recursive definition of satisfaction of a formula in
an arbitrary relational system by a valuation of the variables in the
system) in Tarski and Vaught~\cite{tv57}, p.  85, and the latest
conception of institution in Goguen and Burstall~\cite{gb86}, p.  327.

From the above it follows, immediately, the following

\begin{corollary}
The quadruple $\mathfrak{Tm} =
(\mathbf{Sig},\Alg,\mathrm{Ter},(\mathrm{Tr},\theta))$ is an institution
on the category $\mathbf{Set}$, the so-called \emph{many-sorted term
institution}, or, to abbreviate, the \emph{term institution}.

\end{corollary}

\section{Many-sorted specifications and morphisms.}

In this section we begin by defining, for a signature
$\mathbf{\Sigma}$, the concept of $\mathbf{\Sigma}$-equation, but for
the generalized terms defined in the preceding section, the binary
relation of satisfaction between $\mathbf{\Sigma}$-algebras and
$\mathbf{\Sigma}$-equations, and the semantical consequence operators
$\Cn_{\mathbf{\Sigma}}$.  Then, after extending the translation of
generalized terms up to generalized equations, we prove the
corresponding satisfaction condition, and define a pseudo-functor
$\mathrm{LEq}$ which assigns (among others) to a signature
$\mathbf{\Sigma}$, the discrete category associated to the set of all
labelled $\mathbf{\Sigma}$-equations, that enables us to get the
many-sorted equational institution $\mathfrak{LEq}$.

Following this, in order to show that the semantical consequence
operators $\Cn_{\mathbf{\Sigma}}$, associated to the different
signatures $\mathbf{\Sigma}$, are the components of a pseudo-functor,
$\mathrm{Cn}$, from the category of signatures to a convenient
$2$-category of monads, we define, by means of the construction of
Ehresmann-Grothendieck, the category of many-sorted closure spaces.
Then we prove that the pseudo-functor $\mathrm{Cn}$, in its turn, is
part of an entailment system, but for a concept of entailment system
that generalizes that defined by Meseguer in~ \cite{m89}, pp.
282--283.

After this we define, for the generalized terms, the concepts of
many-sorted specification and of many-sorted specification morphism,
from which we get the corresponding category, denoted by
$\mathbf{Spf}$.  Then by extending some of the notions and
constructions previously developed for the category $\mathbf{Sig}$ up
to the category $\mathbf{Spf}$, we get $\mathfrak{Spf}$, the
many-sorted specification institution on $\mathbf{Set}$.  Besides, we
prove that there exists a morphism from $\mathfrak{Spf}$ to
$\mathfrak{Tm}$, the many-sorted term institution on $\mathbf{Set}$,
which, together with the canonical embedding of $\mathfrak{Tm}$ into
$\mathfrak{Spf}$, makes of $\mathfrak{Tm}$ a retract of
$\mathfrak{Spf}$.  We point out that, conveniently generalized, the
many-sorted specification morphisms will be used, together with some
other concepts, in the last section, to prove the equivalence between
the many-sorted specifications of Hall and Bénabou.


We now define the equations over a given signature through the
morphisms of the category of terms for the signature, what it means
for an equation to be valid in an algebra, and the consequence
operator on the $\mathrm{ms}$-set of the equations.

\begin{definition}
Let $\mathbf{\Sigma}$ be a signature, $X$, $Y$ two $S$-sorted
sets and $\mathbf{A}$ a $\mathbf{\Sigma}$-algebra.
\begin{enumerate}
\item A $\mathbf{\Sigma}$-\emph{equation of type} $(X,Y)$ is a
      pair $(P,Q)\colon X\mor Y$ of parallel morphisms in
      $\mathbf{Ter}(\mathbf{\Sigma}) =
      \mathbf{Kl}(\mathbb{T}_{\mathbf{\Sigma}})^{\mathrm{op}}$, hence
      $(P,Q)\in \mathrm{Hom}(Y,\mathrm{T}_{\mathbf{\Sigma}}(X))^{2}$,
      and a $\mathbf{\Sigma}$-\emph{equation} is a
      $\mathbf{\Sigma}$-equation of type $(X,Y)$ for some $S$-sorted
      sets $X$, $Y$.  We denote by $\Eq(\mathbf{\Sigma})$ the
      $(\boldsymbol{\mathcal{U}}^{S})^{2}$-sorted set of the
      $\mathbf{\Sigma}$-equations, i.e., the family
      $(\mathrm{Hom}(Y,\mathrm{T}_{\mathbf{\Sigma}}(X))^{2})_{X,Y\in
      \boldsymbol{\mathcal{U}}^{S}}$. 

\item We say that a $\mathbf{\Sigma}$-equation $(P,Q)\colon X\mor Y$
      is \emph{valid} in $\mathbf{A}$ (or that $\mathbf{A}$ is a
      \emph{model} of $(P,Q)$, or also that $\mathbf{A}$
      \emph{satisfies} $(P,Q)$),
      $\mathbf{A}\models^{\mathbf{\Sigma}}_{X,Y}(P,Q)$, iff, for every
      $s\in S$ and $y\in Y_{s}$,
      $\mathbf{A}\models^{\mathbf{\Sigma}}_{X,s}(P_{s}(y),Q_{s}(y))$,
      i.e., $(P_{s}(y))^{\mathbf{A}} = (Q_{s}(y))^{\mathbf{A}}$.  We extend
      this satisfaction relation between $\mathbf{\Sigma}$-algebras and
      $\mathbf{\Sigma}$-equations to $\mathbf{\Sigma}$-algebras $\mathbf{A}$ and
      families $\mathcal{E}\subseteq\Eq(\mathbf{\Sigma})$ by agreeing that
      $\mathbf{A}\models^{\mathbf{\Sigma}}\mathcal{E}$ iff, for every $X,Y\in
      \boldsymbol{\mathcal{U}}^{S}$ and $(P,Q)\in \mathcal{E}_{X,Y}$, we have that
      $\mathbf{A}\models^{\mathbf{\Sigma}}_{X,Y}(P,Q)$.

\item We denote by $\Cn_{\mathbf{\Sigma}}$ the operator on the
      $(\boldsymbol{\mathcal{U}}^{S})^{2}$-sorted set
      $\Eq(\mathbf{\Sigma})$ which assigns
      to $\mathcal{E}\subseteq\Eq(\mathbf{\Sigma})$ the
      $(\boldsymbol{\mathcal{U}}^{S})^{2}$-sorted set
      $\Cn_{\mathbf{\Sigma}}(\mathcal{E})$,
      where, for every $X,Y\in \boldsymbol{\mathcal{U}}^{S}$ and $(P,Q)\in
      \Eq(\mathbf{\Sigma})_{X,Y}$,
      $(P,Q)\in\Cn_{\mathbf{\Sigma}}(\mathcal{E})_{X,Y}$ iff, for
      every $\mathbf{\Sigma}$-algebra $\mathbf{A}$, if
      $\mathbf{A}\models^{\mathbf{\Sigma}}\mathcal{E}$, then
      $\mathbf{A}\models^{\mathbf{\Sigma}}_{X,Y}(P,Q)$.  We call
      $\Cn_{\mathbf{\Sigma}}(\mathcal{E})$ the
      $(\boldsymbol{\mathcal{U}}^{S})^{2}$-sorted
      set of the \emph{semantical consequences} of $\mathcal{E}$.
\end{enumerate}

\end{definition}

If we keep in mind that for a term $P\colon X\mor Y$ for
$\mathbf{\Sigma}$ of type $(X,Y)$, $P^{\mathbf{A}}$, the term
operation on $\mathbf{A}$ determined by $P$, is the mapping from
$A_{X}$ to $A_{Y}$ which assigns to an $S$-sorted mapping $f\colon
X\mor A$ precisely $f^{\sharp}\comp P\colon Y\mor A$, then we get
the following convenient characterization of the relation
$\mathbf{A}\models^{\mathbf{\Sigma}}_{X,Y}(P,Q)$:
$$
\mathbf{A}\models^{\mathbf{\Sigma}}_{X,Y}(P,Q) \quad \text{ iff }\quad
P^{\mathbf{A}} = Q^{\mathbf{A}}.
$$
Besides, by the Completeness Theorem in~\cite{cs05}, for
$\mathcal{E}\subseteq\Eq(\mathbf{\Sigma})$, we have that
$\Cn_{\mathbf{\Sigma}}(\mathcal{E})$ is precisely
$\mathrm{Cg}^{\Pi}_{\mathbf{Ter}(\mathbf{\Sigma})}(\mathcal{E})$,
i.e., the smallest $\Pi$-compatible congruence on
$\mathbf{Ter}(\mathbf{\Sigma})$ that contains $\mathcal{E}$, where
the superscript $\Pi$ in the operator
$\mathrm{Cg}^{\Pi}_{\mathbf{Ter}(\mathbf{\Sigma})}$ abbreviates
\lq\lq product\rq\rq.  Therefore the operator
$\Cn_{\mathbf{\Sigma}}$ on the
$(\boldsymbol{\mathcal{U}}^{S})^{2}$-sorted set
$\Eq(\mathbf{\Sigma})$ is a closure operator.

\begin{remark}
It is true that, for a signature $\mathbf{\Sigma}$, in order to
equationally characterize the varieties (resp., the finitary
varieties) of $\mathbf{\Sigma}$-algebras it is enough to consider
the $S$-finite (resp., the finite) subsets of an arbitrary, but
fixed, $S$-sorted set $V^{S}$ with a countable infinity of
variables in each coordinate.  However, the generalized terms and
equations proposed in this paper, besides containing as particular
cases the ordinary terms and equations, respectively, have proved
their worth, e.g., in the proof of the Completeness Theorem for
monads in categories of sorted sets in~\cite{cs05}, and also to
attain a truly category-theoretical understanding of the subject
matter (through the theory of monads as sketched at the end of the
sixth section).  Moreover, the generalized terms and equations
have the advantage over the ordinary terms and equations of being
automatically dualizable, thus allowing the definition of the
generalized coterms and coequations, from which it is easily
obtainable, e.g., a Completeness Theorem for comonads in
categories of sorted sets.
\end{remark}

By recalling that every signature morphism $\mathbf{d}$ from
$\mathbf{\Sigma}$ to $\mathbf{\Lambda}$ determines a functor
$\mathbf{d}_{\diamond}$ from $\mathbf{Ter}(\mathbf{\Sigma})$ to
$\mathbf{Ter}(\mathbf{\Lambda})$, and taking into account the above
definition of the equations for a signature, we next formalize the
procedure of translation, by means of a signature morphism, of
equations for a signature into equations for another signature in the
following

\begin{definition}
Let $\mathbf{d}\colon \mathbf{\Sigma}\mor \mathbf{\Lambda}$ be a
signature morphism.  Then $\mathbf{d}$ induces an
$\mathrm{ms}$-mapping
$$
((\tcoprod_{\varphi})^{2},\mathbf{d}_{\diamond}^{2}) \colon
((\boldsymbol{\mathcal{U}}^{S})^{2},\mathrm{Eq}(\mathbf{\Sigma})) \mor
((\boldsymbol{\mathcal{U}}^{T})^{2},\mathrm{Eq}(\mathbf{\Lambda})),
$$
the so called \emph{translation of equations for} $\mathbf{\Sigma}$
\emph{into equations for} $\mathbf{\Lambda}$ \emph{relative to}
$\mathbf{d}$, where
\begin{enumerate}
\item $(\coprod_{\varphi})^{2}$ is the mapping from
      $(\boldsymbol{\mathcal{U}}^{S})^{2}$ to
      $(\boldsymbol{\mathcal{U}}^{T})^{2}$ which sends a pair of
      $S$-sorted sets $(X,Y)$ to the pair
      $(\coprod_{\varphi}X,\coprod_{\varphi}Y)$ of $T$-sorted sets,
      and

\item $\mathbf{d}_{\diamond}^{2}$ the
      $(\boldsymbol{\mathcal{U}}^{S})^{2}$-sorted mapping which to a
      $\mathbf{\Sigma}$-equation $(P,Q)$ of type $(X,Y)$ assigns the
      $\mathbf{\Lambda}$-equation
      $(\mathbf{d}_{\diamond}(P),\mathbf{d}_{\diamond}(Q))$ of type
      $(\coprod_{\varphi}X,\coprod_{\varphi}Y)$.
\end{enumerate}

\end{definition}

Once defined the translation of equations, we prove in the following
lemma the invariance of the relation of satisfaction under signature
change, also known, for those following the terminology coined by
Goguen and Burstall in~\cite{gb84}, p.  229, as the \emph{satisfaction
condition}.

\begin{lemma}\label{lemaSatisfaccion}
Let $\mathbf{d}\colon \mathbf{\Sigma}\mor \mathbf{\Lambda}$ be a signature
morphism, $(P,Q)$ a $\mathbf{\Sigma}$-equa\-tion of type $(X,Y)$ and
$\mathbf{A}$ a $\mathbf{\Lambda}$-algebra.  Then we have that
$$
\mathbf{d}^{\ast}(\mathbf{A})\models^{\mathbf{\Sigma}}_{X,Y} (P,Q)\text{ iff }
\mathbf{A}\models^{\mathbf{\Lambda}}_{\coprod_{\varphi}\!X,\coprod_{\varphi}\!Y}
(\mathbf{d}_{\diamond}(P),\mathbf{d}_{\diamond}(Q)).
$$
\end{lemma}

\begin{proof}
The condition
$\mathbf{d}^{\ast}(\mathbf{A})\models^{\mathbf{\Sigma}}_{X,Y} (P,Q)$
is equivalent to $P^{\mathbf{d}^{\ast}(\mathbf{A})} =
Q^{\mathbf{d}^{\ast}(\mathbf{A})}$ but this condition, by
Proposition~\ref{invarianza}, is equivalent to
$\mathbf{d}_{\diamond}(P)^{\mathbf{A}} =
\mathbf{d}_{\diamond}(Q)^{\mathbf{A}}$, therefore it is also
equivalent to the condition
$\mathbf{A}\models^{\mathbf{\Lambda}}_{\coprod_{\varphi}\!X,\coprod_{\varphi}\!Y}
(\mathbf{d}_{\diamond}(P),\mathbf{d}_{\diamond}(Q))$.
\end{proof}

Related to the quasi-triviality of the (short and conceptual) proof of
Lemma~\ref{lemaSatisfaccion} (as a consequence, essentially, of the
fact that it is ultimately rooted in
Proposition~\ref{invarianza}), perhaps, it would be convenient to
recall that Goguen and Burstall, in~\cite{gb84}, p.  228, have omitted
the corresponding proof because they qualify it as being not
entirely trivial.

Everything we have made up to this point suggest that the many-sorted
term institution introduced in the preceding section, if an
institution is understood as meaning a pseudo-extranatural
transformation, neither is a useless institution nor a trivial step in
a natural process of evolution of the concept of institution, by
adaptation to situations not noticed explicitly until now.  And this
can indicate that it is not unreasonable to give support to the view
that the many-sorted term institution $\mathfrak{Tm}$ is, because of
its primitivity and elementariness, in fact, more basic, or fundamental,
than the many-sorted equational institution, defined immediately
below.

To construct the many-sorted equational institution we now define a
pseudo-functor $\mathrm{LEq}$ on the category of signatures.  In order
to do so (and also to define, later on, another pseudo-functor that
will contribute to the construction of a certain entailment system
founded on the family of consequences), we need to assume, besides the
Grothendieck universe $\boldsymbol{\mathcal{U}}$, the existence of
another one $\boldsymbol{\mathcal{V}}$ such that
$\boldsymbol{\mathcal{U}}\in \boldsymbol{\mathcal{V}}$.  The new
Grothendieck universe $\boldsymbol{\mathcal{V}}$ will be used to
construct the appropriate $2$-categories where the just named
pseudo-functors will take its values.  Therefore, to exclude any
misunderstanding, we agree to denote those categories $\mathbf{C}$
properly depending on $\boldsymbol{\mathcal{V}}$ by
$\mathbf{C}_{\boldsymbol{\mathcal{V}}}$.  However, since the additional
assumption of a universe $\boldsymbol{\mathcal{V}}$ such that
$\boldsymbol{\mathcal{U}}\in\boldsymbol{\mathcal{V}}$, will be used,
almost, exclusively in this section, we do not label those categories
depending on $\boldsymbol{\mathcal{U}}$ with the subscript
$\boldsymbol{\mathcal{U}}$, such as has been done until now.

\begin{definition}
We denote by $\mathrm{LEq}$ the pseudo-functor from
$\mathbf{Sig}$ to $\mathbf{Cat}_{\boldsymbol{\mathcal{V}}}$ given by the
following data
\begin{enumerate}
\item The object mapping of $\mathrm{LEq}$ is that which sends a
      signature $\mathbf{\Sigma}$ to the discrete category
      $\mathbf{LEq}(\mathbf{\Sigma})$ canonically associated to the set
      $$
      \textstyle \bigcup_{X,Y\in \boldsymbol{\mathcal{U}}}
      (\mathrm{Hom}(Y,\mathrm{T}_{\mathbf{\Sigma}}(X))^{2}\times \{(X,Y)\})
      $$
      of \emph{labelled} $\mathbf{\Sigma}$-\emph{equations}, i.e., the set of all
      pairs $((P,Q),(X,Y))$ with $(P,Q)$ a $\mathbf{\Sigma}$-equation
      of type $(X,Y)$, for some $X,Y\in \boldsymbol{\mathcal{U}}$.

\item The morphism mapping of $\mathrm{LEq}$ is that which sends
      a signature morphism $\mathbf{d}$ from $\mathbf{\Sigma}$ to
      $\mathbf{\Lambda}$ to the functor $\mathrm{LEq}(\mathbf{d})$
      from $\mathbf{LEq}(\mathbf{\Sigma})$ to
      $\mathbf{LEq}(\mathbf{\Lambda})$ which assigns to the labelled
      equation $((P,Q),(X,Y))$ in $\mathbf{LEq}(\mathbf{\Sigma})$ the
      labelled equation
      $$\textstyle\mathrm{LEq}(\mathbf{d})((P,Q),(X,Y)) =
      ((\mathbf{d}_{\diamond}(P),\mathbf{d}_{\diamond}(Q)),
      (\coprod_{\varphi}X,\coprod_{\varphi}Y))
      $$
      in $\mathbf{LEq}(\mathbf{\Lambda})$.
\end{enumerate}

\end{definition}

\begin{corollary}
The quadruple $\mathfrak{LEq} =
(\mathbf{Sig},\Alg,\mathrm{LEq},(\models,\theta))$ is an institution
on $\mathbf{2}$, the so-called \emph{many-sorted equational
institution}, or, to abbreviate, the \emph{equational institution}.

\end{corollary}


Before examining the many-sorted specifications and the many-sorted
specification morphisms, it seems worthwhile to give a
category-theoretic look at the concept of equational consequence.
Concretely, what we want to establish now is the following:
\begin{enumerate}
\item[(1)] That the closure operators $\Cn_{\mathbf{\Sigma}}$
      are, essentially, the components of a suitable pseudo-functor
      $\Cn$ from the category $\mathbf{Sig}$, of signatures, to a
      convenient $2$-category
      $\mathbf{Mnd}_{\boldsymbol{\mathcal{V}},\mathrm{alg}}$, of
      monads, for a Grothendieck universe $\boldsymbol{\mathcal{V}}$
      such that $\boldsymbol{\mathcal{U}}\in
      \boldsymbol{\mathcal{V}}$, obtained by properly choosing the
      $2$-cells in the $2$-category
      $\mathbf{Mnd}_{\boldsymbol{\mathcal{V}}}$, of monads for
      $\boldsymbol{\mathcal{V}}$, and

\item[(2)] That the pseudo-functor $\Cn$ is in fact part of an entailment
     system (understood in a more general sense than that defined by
     Meseguer in~ \cite{m89}, pp. 282--283).
\end{enumerate}

However, in order to do what has been just enumerated, we should begin
by a thorough investigation of the building blocks that are involved
and constitute the basis for the understanding of the equational
consequence from the category-theoretical standpoint.  Explicitly,
this means that we should carry through the following:
\begin{enumerate}
\item To assign to each set of sorts $S$ the corresponding
      category $\mathbf{ClSp}(S)$, of $S$-sorted closure spaces, and
      to associate to an arbitrary mapping $\varphi\colon S\mor T$,
      from a set of sorts $S$ into a like one $T$, the corresponding
      functor $\Delta^{\mathrm{cl}}_{\varphi}$ from $\mathbf{ClSp}(T)$ to
      $\mathbf{ClSp}(S)$, and all in such a way that both procedures give
      rise to a contravariant functor $\Delta^{\mathrm{cl}}$ from
      $\mathbf{Sig}$ to $\mathbf{Cat}$, and

\item To get the category $\mathbf{MClSp}$, of many-sorted
      closure spaces and morphism between them, by applying the
      construction of Ehresmann-Grothen\-dieck to the contravariant
      functor $\Delta^{\mathrm{cl}}$.
\end{enumerate}

After having completed the above, the pseudo-functor $\Cn$ will
be obtained by using the fact that the category $\mathbf{MClSp}$ can be
identified to a subcategory of the $2$-category
$\mathbf{Mnd}_{\boldsymbol{\mathcal{V}},\mathrm{alg}}$.

\begin{definition}
Let $A$ be an $S$-sorted set.
\begin{enumerate}
\item An $S$-\emph{closure system on} $A$ is a subset
      $\mathcal{C}$ of $\Sub(A)$, the set of all $S$-sorted sets $X$
      such that, for every $s\in S$, $X_{s}\subseteq A_{s}$,
      abbreviated to $X\subseteq A$, that satisfies the following
      conditions
      \begin{enumerate}
      \item $A\in\mathcal{C}$.

      \item For every $\mathcal{D}\incl\mathcal{C}$, if $\mathcal{D}\neq\vacio$,
            then $\inter\mathcal{D} = (\inter_{D\in
            \mathcal{D}}D_{s})_{s\in S}\in\mathcal{C}$.
      \end{enumerate}
      We denote by $\mathrm{ClSy}(A)$ the set of the $S$-closure
      systems on $A$ and by $\mathbf{ClSy}(A)$ the same set but
      ordered by inclusion.  We call the pairs of the form
      $(A,\mathcal{C})$, with $\mathcal{C}\in \mathrm{ClSy}(A)$,
      $S$-\emph{closure system spaces}.

\item An $S$-\emph{closure operator on} $A$ is an operator $J$
      on $\mathrm{Sub}(A)$, i.e., a mapping from $\mathrm{Sub}(A)$ to
      $\mathrm{Sub}(A)$, such that, for every $X,Y\subseteq A$,
      satisfies the following conditions
      \begin{enumerate}
      \item $X\subseteq J(X)$, i.e., $J$ is extensive.

      \item If $X\subseteq Y$, then $J(X)\subseteq J(Y)$,
            i.e., $J$ is isotone.

      \item $J(J(X)) = J(X)$, i.e., $J$ is idempotent.

      \end{enumerate}
      We denote by $\mathrm{ClOp}(A)$ the set of the $S$-closure
      operators on $A$ and by $\mathbf{ClOp}(A)$ the same set but
      ordered by the relation $\leq$, where, for $J$ and $K$ in
      $\mathrm{ClOp}(A)$, we have that $J\leq K$ if, for every
      $X\subseteq A$, $J(X)\subseteq K(X)$.  We call the pairs of the
      form $(A,J)$, with $J\in \mathrm{ClOp}(A)$,
      $S$-\emph{closure operator spaces}.

\end{enumerate}
\end{definition}

\begin{example}
For a $\mathbf{\Sigma}$-algebra $\mathbf{A}$, the set
$\mathrm{Sub}(\mathbf{A})$, of subalgebras of $\mathbf{A}$, is an
(algebraic) $S$-closure system on the $S$-sorted set $A$, and the
operator $\mathrm{Sg}_{\mathbf{A}}$, of generated subalgebra for
$\mathbf{A}$, is an (algebraic) $S$-closure operator on $A$.
\end{example}

\begin{example}
For a $\mathbf{\Sigma}$-algebra $\mathbf{A}$, the set
$\mathrm{Cgr}(\mathbf{A})$, of congruences on $\mathbf{A}$, is an
(algebraic) $S$-closure system on the $S$-sorted set $A\times A =
(A_{s}\times A_{s})_{s\in S}$, and the operator
$\mathrm{Cg}_{\mathbf{A}}$, of generated congruence for $\mathbf{A}$,
is an (algebraic) $S$-closure operator on $A\times A$.
\end{example}

\begin{example}
For a signature $\mathbf{\Sigma}$, the operator
$\Cn_{\mathbf{\Sigma}}$, of semantical consequence for
$\mathbf{\Sigma}$, is a $(\boldsymbol{\mathcal{U}}^{S})^{2}$-closure
operator on $\Eq(\mathbf{\Sigma})$.
\end{example}

As in the single-sorted case, also in the many-sorted case, for a set
of sorts $S$, every $S$-closure system $\mathcal{C}$ on an $S$-sorted set
$A$, when ordered by inclusion, induces a complete lattice
$\boldsymbol{\mathcal{C}} = (\mathcal{C},\subseteq)$.  Moreover, the
ordered sets $\mathbf{ClOp}(A)$, of $S$-closure operators on $A$, and
$\mathbf{ClSy}(A)$, of $S$-closure systems on $A$, are complete
lattices and anti\-isomor\-phic through the mapping $\Fix$ from
$\mathbf{ClOp}(A)$ to $\mathbf{ClSy}(A)$ which sends an $S$-closure
operator $J$ on $A$ to the $S$-closure system $\Fix(J) = \{ X\subseteq
A\mid J(X) = X \}$ on $A$, of the \emph{fixed points of} $J$.

After having defined, for a set of sorts $S$, the objects of interest,
i.e., in this case the $S$-closure system spaces  and the
$S$-closure operator spaces, we proceed to define the morphisms that are
suitable for these mathematical constructs.

\begin{definition}\label{morclsp}
Let $S$ be a set of sorts, $A$, $B$ two $S$-sorted sets, $\mathcal{C}$ and
$\mathcal{D}$ closure systems on $A$ and $B$, respectively, and $J$ and $K$
closure operators on $A$ and $B$, respectively.
\begin{enumerate}
\item An $S$-\emph{continuous mapping} from the $S$-closure
      system space $(A,\mathcal{C})$ to the $S$-closure system space
      $(B,\mathcal{D})$ is a triple
      $((A,\mathcal{C}),f,(B,\mathcal{D}))$, denoted by $f\colon
      (A,\mathcal{C})\mor(B,\mathcal{D})$, where $f$ is an $S$-mapping
      from $A$ to $B$ such that, for every $D\in \mathcal{D}$, $
      f^{-1}[D]\in \mathcal{C}$.

\item An $S$-\emph{continuous mapping} from the $S$-closure
      operator space $(A,J)$ to the $S$-closure operator space $(B,K)$
      is a triple $((A,J),f,(B,K))$, denoted by $f\colon
      (A,J)\mor(B,K)$, where $f$ is an $S$-mapping from $A$ to $B$
      such that, for every $X\subseteq A$, $f[J(X)]\subseteq K(f[X])$.
\end{enumerate}
\end{definition}

%

\begin{example}
For a $\mathbf{\Sigma}$-homomorphism $f$ from $\mathbf{A}$ to
$\mathbf{B}$ and an $S$-sorted subset $X$ of $A$, we have that
$f[\mathrm{Sg}_{\mathbf{A}}(X)] = \mathrm{Sg}_{\mathbf{B}}(f[X])$.
Therefore the $\mathbf{\Sigma}$-homomorphism $f$ induces an
$S$-continuous (and closed) mapping from
$(A,\mathrm{Sg}_{\mathbf{A}})$ to $(B,\mathrm{Sg}_{\mathbf{B}})$.
\end{example}

\begin{example}
For a $\mathbf{\Sigma}$-homomorphism $f$ from $\mathbf{A}$ to
$\mathbf{B}$ and a congruence $\Psi$ on $\mathbf{B}$, we have that
$(f\times f)^{-1}[\Psi]\in \mathrm{Cgr}(\mathbf{A})$.  Therefore from
the $\mathbf{\Sigma}$-homomorphism $f$ we get the $S$-continuous
mapping $f\times f$ from $(A\times A,\mathrm{Cgr}(\mathbf{A}))$ to
$(B\times B,\mathrm{Cgr}(\mathbf{B}))$.
\end{example}

\begin{example}
Later on, after having defined the category
$\mathbf{MClSp}_{\boldsymbol{\mathcal{V}}}$, we will prove that every
signature morphism $\mathbf{d}$ from a signature $\mathbf{\Sigma}$ to
a signature $\mathbf{\Lambda}$ induces a
$(\boldsymbol{\mathcal{U}}^{S})^{2}$-continuous mapping from
$(\Eq(\mathbf{\Sigma}),\Cn_{\mathbf{\Sigma}})$ to
$(\Eq(\mathbf{\Lambda})_{(\coprod_{\varphi})^{2}},
(\Cn_{\mathbf{\Lambda}})_{(\coprod_{\varphi})^{2}})$.
\end{example}

Also as for the single-sorted case, for every set of sorts $S$, there
exists, up to a concrete isomorphism, a category of $S$-closure spaces,
with objects given by an $S$-sorted set and, alternative, but
equivalently, an $S$-closure system or an $S$-closure operator
on it.

\begin{proposition}
Let $S$ be a set of sorts. Then we have that
\begin{enumerate}
\item The $S$-closure system spaces together with the
      $S$-continuous mappings between them, as defined in the first
      part of Definition~\ref{morclsp}, constitute a category
      $\mathbf{ClSySp}(S)$.  Furthermore, from $\mathbf{ClSySp}(S)$ to
      $\mathbf{Set}^{S}$ the forgetful functor, which sends an
      $S$-continuous mapping $f$ from $(A,\mathcal{C})$ to
      $(B,\mathcal{D})$ to the $S$-sorted mapping $f$ from $A$ to $B$,
      is faithful.  Therefore $\mathbf{ClSySp}(S)$ is a concrete
      category on $\mathbf{Set}^{S}$.

\item The $S$-closure operator spaces together with the
      $S$-continuous mappings between them, as defined in the second
      part of Definition~\ref{morclsp}, constitute a category
      $\mathbf{ClOpSp}(S)$.  Furthermore, from $\mathbf{ClOpSp}(S)$ to
      $\mathbf{Set}^{S}$ the forgetful functor, which sends an
      $S$-continuous mapping $f$ from $(A,J)$ to $(B,K)$ to the
      $S$-sorted mapping $f$ from $A$ to $B$, is faithful.  Therefore
      $\mathbf{ClOpSp}(S)$ is a concrete category on
      $\mathbf{Set}^{S}$.

\item The categories $\mathbf{ClOpSp}(S)$ and $\mathbf{ClSySp}(S)$ are
      concretely isomorphic, through the functor which sends the
      $S$-continuous mapping $f$ from $(A,J)$ to $(B,K)$ to the
      $S$-continuous mapping $f$ from $(A,\Fix(J))$ to $(B,\Fix(K))$.
\end{enumerate}
\end{proposition}

From now on, and for a set of sorts $S$, by the category of
$S$-closure spaces, denoted by $\mathbf{ClSp}(S)$, we will refer,
indistinctly, to anyone of the categories $\mathbf{ClSySp}(S)$ or
$\mathbf{ClOpSp}(S)$.

We point out that, once more, as for the single-sorted case, for every
set of sorts $S$, the forgetful functor from the category
$\mathbf{ClSp}(S)$ to the  category $\mathbf{Set}^{S}$ has left and right
adjoints and constructs limits and colimits.

After associating to a set of sorts $S$ the category
$\mathbf{ClSp}(S)$ of $S$-closure spaces, we prove next that a mapping
$\varphi\colon S\mor T$ induces an adjunction
$\coprod^{\mathrm{cl}}_{\varphi}\ladj \Delta^{\mathrm{cl}}_{\varphi}$
from $\mathbf{ClSp}(S)$ to $\mathbf{ClSp}(T)$.

For the proof of the above assertion it will be shown to be useful
to introduce the following notational conventions.  Let
$\varphi\colon S\mor T$ be a mapping and $\tcoprod_{\varphi}\ladj
\Delta_{\varphi}$ the adjunction from $\mathbf{Set}^{S}$ to
$\mathbf{Set}^{T}$ induced by $\varphi$, then
\begin{enumerate}
\item For a $T$-sorted set $B$ and a subset $\mathcal{D}$ of
      $\mathrm{Sub}(B)$, $\Delta_{\varphi}[\mathcal{D}]$ denotes the
      subset $\{\,D_{\varphi}\mid D\in \mathcal{D}\,\}$ of
      $\mathrm{Sub}(B_{\varphi})$, and

\item For an $S$-sorted set $A$ and a subset $\mathcal{C}$ of
      $\mathrm{Sub}(A)$, $\coprod_{\varphi}[\mathcal{C}]$ denotes the
      subset $\{\,\tcoprod_{\varphi}C\mid C\in \mathcal{C}\,\}$ of
      $\mathrm{Sub}(\tcoprod_{\varphi}A)$.
\end{enumerate}

\begin{proposition}
Let $\varphi\colon S\mor T$ be a mapping.  Then from
$\mathbf{ClSp}(T)$ to $\mathbf{ClSp}(S)$ there exists a functor
$\Delta^{\mathrm{cl}}_{\varphi}$ defined as follows
\begin{enumerate}
\item $\Delta^{\mathrm{cl}}_{\varphi}$ sends $(B,\mathcal{D})$ in
      $\mathbf{ClSp}(T)$ to $(B_{\varphi},\Delta_{\varphi}[\mathcal{D}])$
      in $\mathbf{ClSp}(S)$.

\item $\Delta^{\mathrm{cl}}_{\varphi}$ sends a $T$-continu\-ous
      mapping $f\colon (B,\mathcal{D})\mor (B',\mathcal{D}')$ to the
      $S$-continu\-ous mapping $f_{\varphi}\colon
      (B_{\varphi},\Delta_{\varphi}[\mathcal{D}])\mor
      (B'_{\varphi},\Delta_{\varphi}[\mathcal{D}'])$.
\end{enumerate}

\end{proposition}

\begin{proof}
Let $\mathcal{D}$ be a $T$-closure system on $B$, then
$\Delta_{\varphi}[\mathcal{D}]$ is an $S$-closure system on
$B_{\varphi}$, because, for every family
$(Y^{i})_{i\in I}$ of $T$-sorted sets, we have that %
$$
(\textstyle{\bigcap}_{i\in I}Y^{i})_{\varphi} =
              \textstyle{\bigcap}_{i\in I}Y^{i}_{\varphi}.
$$
Besides, if $f\colon(B,\mathcal{D})\mor (B',\mathcal{D}')$ is a
$T$-continuous mapping and
$Y_{\varphi}'\in\Delta_{\varphi}[\mathcal{D}']$, then $Y'\in
\mathcal{D}'$ and $f^{-1}[Y']\in \mathcal{D}$, thus
$\Delta_{\varphi}(f^{-1}[Y'])\in\Delta_{\varphi}[\mathcal{D}]$.
But $\Delta_{\varphi}(f^{-1}[Y'])$ is identical to
$\Delta_{\varphi}(f)^{-1}[Y_{\varphi}']$, therefore $f_{\varphi}$
is an $S$-continuous mapping.
\end{proof}

\begin{proposition}
Let $\varphi\colon S\mor T$ be a mapping.  Then from $\mathbf{ClSp}(S)$ to
$\mathbf{ClSp}(T)$ there exists a functor
$\coprod^{\mathrm{cl}}_{\varphi}$ defined as follows
\begin{enumerate}
\item $\coprod^{\mathrm{cl}}_{\varphi}$ sends $(A,\mathcal{C})$ in $\mathbf{ClSp}(S)$
      to $(\coprod_{\varphi}A,\coprod_{\varphi}[\mathcal{C}])$ in $\mathbf{ClSp}(T)$.

\item $\coprod^{\mathrm{cl}}_{\varphi}$ sends an $S$-continu\-ous
      mapping $f\colon (A,\mathcal{C})\mor (A',\mathcal{C}')$ to the
      $T$-continu\-ous mapping $\coprod_{\varphi}f\colon
      (\coprod_{\varphi}A,\coprod_{\varphi}[\mathcal{C}])\mor
      (\coprod_{\varphi}A',\coprod_{\varphi}[\mathcal{C}'])$.
\end{enumerate}

\end{proposition}

\begin{proof}
Let $\mathcal{C}$ be an $S$-closure system on $A$, then
$\coprod_{\varphi}[\mathcal{C}]$ is a $T$-closure system on
$\coprod_{\varphi}A$, because, for every family
$(X^{i})_{i\in I}$ of $S$-sorted sets, we have that %
$$
\tcoprod_{\varphi}\inter_{i\in I}X^{i} =
\inter_{i\in I}\tcoprod_{\varphi}X^{i}.
$$
Besides, if $f\colon(A,\mathcal{C})\mor (A',\mathcal{C}')$ is an an
$S$-continuous mapping and
$\coprod_{\varphi}X'\in\coprod_{\varphi}[\mathcal{C}']$, then $X'$ belongs
to $\mathcal{C'}$ y $f^{-1}[X']\in \mathcal{C}$, thus
$\coprod_{\varphi}f^{-1}[X']\in\coprod_{\varphi}[\mathcal{C}]$.  But
$\coprod_{\varphi}f^{-1}[X']$ is identical to
$(\coprod_{\varphi}f)^{-1}[\coprod_{\varphi}X']$, therefore
$\coprod_{\varphi}f$ is a $T$-continuous mapping.
\end{proof}

\begin{proposition}
Let $\varphi\colon S\mor T$ be a mapping.  Then the functor
$\coprod^{\mathrm{cl}}_{\varphi}$ is left adjoint to the functor
$\Delta^{\mathrm{cl}}_{\varphi}$.
\end{proposition}

\begin{proof}
The  natural isomorphism $\theta^{\varphi}$ of the adjunction
$\coprod_{\varphi}\ladj \Delta_{\varphi}$ also happens to be a natural
isomorphism %
$$
\Hom((A,\mathcal{C}),(B_{\varphi},\Delta_{\varphi}[\mathcal{D}]))
\iso
\Hom((\tcoprod_{\varphi}A,\tcoprod_{\varphi}[\mathcal{C}]),(B,\mathcal{D})),
$$
for every $(A,\mathcal{C})$ in $\mathbf{ClSp}(S)$ and every $(B,\mathcal{D})$ in
$\mathbf{ClSp}(T)$.

Let $f$ be an $S$-continuous mapping from $(A,\mathcal{C})$ to
$(B_{\varphi},\Delta_{\varphi}[\mathcal{D}])$, and $Y\in \mathcal{D}$.
Since $Y_{\varphi}\in\Delta_{\varphi}[\mathcal{D}]$ and $f$ is
continuous, we have that $f^{-1}[Y_{\varphi}]\in \mathcal{C}$ and
$\coprod_{\varphi}f^{-1}[Y_{\varphi}]\in\coprod_{\varphi}[\mathcal{C}]$.
 But $\coprod_{\varphi}f^{-1}[Y_{\varphi}]$ is identical to
$((\theta^{\varphi})^{-1}(f))^{-1}[Y]$, because %
\begin{align*}
((\theta^{\varphi})^{-1}(f))^{-1}[Y]
&=
(
\{  (a,s)\in\tcoprod_{\varphi}(A)_{t} \mid a\in A_{s},\, \varphi(s)=t,\,
f_{s}(a)\in Y_{t}
\}
)_{t\in T}
\\
&=
(
\{  (a,s)\in\tcoprod_{\varphi}(A)_{t} \mid a\in f^{-1}[Y_{\varphi}]_{s},\,
\varphi(s)=t
\}
)_{t\in T}
\\
&=
(
\tcoprod_{s\in\varphi^{-1}[t]}f^{-1}[Y_{\varphi}]_{s}
)_{t\in T}
\\
&=
\tcoprod_{\varphi}f^{-1}[Y_{\varphi}],
\end{align*}
therefore
$(\theta^{\varphi})^{-1}(f)$ is a $T$-continuous mapping.

Reciprocally, let us suppose that $g$ is a $T$-continuous mapping from
$(\tcoprod_{\varphi}A,\tcoprod_{\varphi}[\mathcal{C}])$ to $(B,\mathcal{D})$. Let
$Y_{\varphi}\in\Delta_{\varphi}[\mathcal{D}]$ be, then $Y\in \mathcal{D}$ and
$g^{-1}[Y]\in\coprod_{\varphi}[\mathcal{C}]$. But we have that %
\begin{align*}
g^{-1}[Y]
&=
(\{
(a,s)\in\tcoprod_{\varphi}(A)_{t} \mid g_{t}(a,s)\in Y_{t}
\})_{t\in T}
\\
&=
(
\tcoprod_{s\in\varphi^{-1}[t]}
\{
a\in A_{s} \mid g_{\varphi(s)}(a,s)\in Y_{\varphi(s)}
\}
)_{t\in T}
\\
&=
\tcoprod_{\varphi}
(
(
\{  a\in A_{s} \mid g_{\varphi(s)}(a,s)\in Y_{\varphi(s)}  \}
)_{s\in S}
),
\end{align*}
and, additionally,
\begin{align*}
(
\{  a\in A_{s} \mid g_{\varphi(s)}(a,s)\in Y_{\varphi(s)}  \}
)_{s\in S}
&=
(
\{  a\in A_{s} \mid \theta^{\varphi}(g)_{s}(a)\in Y_{\varphi(s)}  \}
)_{s\in S}
\\
&=
(\theta^{\varphi}(g))^{-1}[Y_{\varphi}],
\end{align*}
thus
$g^{-1}[Y]=\coprod_{\varphi}(\theta^{\varphi}(g))^{-1}[Y_{\varphi}]$,
therefore $(\theta^{\varphi}(g))^{-1}[Y_{\varphi}]\in \mathcal{C}$ and
$\theta^{\varphi}(g)$ is an $S$\nobreakdash-continuous mapping.
\end{proof}

The functors $\Delta^{\mathrm{cl}}_{\varphi}$ and
$\coprod^{\mathrm{cl}}_{\varphi}$ can, obviously, also be defined for
$S$-closure operators.  Actually, the definition for the functor
$\coprod^{\mathrm{cl}}_{\varphi}$, as shown in the following
proposition, is immediate.

\begin{proposition}
Given a mapping $\varphi\colon S\mor T$ and an $S$-closure space
$(A,J)$, the pair $(\coprod_{\varphi}A,J_{\varphi})$ is a $T$-closure
space, where the operator $J_{\varphi}$ on $\coprod_{\varphi}A$ assigns
to $\coprod_{\varphi}X$, for $X\incl A$, the $T$-sorted set
$\coprod_{\varphi}J(X)$.
\end{proposition}

\begin{proof}
The definition of the operator $J_{\varphi}$ is sound, because
$\Sub(A)\iso \Sub(\coprod_{\varphi}A) = \coprod_{\varphi}[\Sub(A)]$.
\end{proof}

However, the corresponding definition for the functor
$\Delta^{\mathrm{cl}}_{\varphi}$ is more involved, because for a
$T$-sorted set $B$, we only have, in general, that
$\Delta_{\varphi}[\Sub(B)]\subseteq\Sub(B_{\varphi})$.

\begin{proposition}
Given a mapping $\varphi\colon S\mor T$ and a $T$-closure space
$(B,K)$, the pair $(B_{\varphi},K_{\varphi})$ is an $S$-closure space,
where the operator $K_{\varphi}$ on $B_{\varphi}$ is defined as
follows
$$
K_{\varphi}\nfunction
{\mathrm{Sub}(B_{\varphi})}{\mathrm{Sub}(B_{\varphi})}
{Y}{K((\textstyle{\bigcup}_{s\in \varphi^{-1}[t]}Y_{s})_{t\in T})_{\varphi}}
$$
\end{proposition}

\begin{proof}
The definition of the operator $K_{\varphi}$ as the composition of the
mappings in the diagram
$$
\xymatrix@C=70pt{
\Sub(B_{\varphi})
  \ar[r]^{K_{\varphi}}
  \ar[d]_{\textstyle{\union}_{\varphi,B}} &
\Sub(B_{\varphi})
    \\
\Sub(B)
  \ar[r]_{K} &
\Sub(B)
\ar[u]_{\Delta_{\varphi,B}}
}
$$
is sound, because the mapping $\union_{\varphi,B}$ from
$\Sub(B_{\varphi})$ to $\Sub(B)$, which sends a subset $Y$ of
$\Sub(B_{\varphi})$ to the subset
$(\union_{s\in\varphi^{-1}[t]}Y_{s})_{t\in T}$ of $B$, is isotone and
has, precisely, as right adjoint, to the mapping $\Delta_{\varphi,B}$
from $\Sub(B)$ to $\Sub(B_{\varphi})$, which sends a subset $X$ of $B$
to the subset $X_{\varphi}$ of $B_{\varphi}$.
\end{proof}

For a mapping $\varphi\colon S\mor T$, the functors
$\Delta^{\mathrm{cl}}_{\varphi}$, from $\mathbf{ClSp}(T)$ to
$\mathbf{ClSp}(S)$, and $\coprod^{\mathrm{cl}}_{\varphi}$, from
$\mathbf{ClSp}(S)$ to $\mathbf{ClSp}(T)$, are the components,
respectively, of a contravariant functor and of a pseudo-functor, from
$\mathbf{Set}$ to $\mathbf{Cat}$.  In particular, by applying the
construction of Ehresmann-Grothendieck to the contravariant
functor we will get a category of many-sorted closure spaces.

\begin{proposition}
There exists a contravariant functor $\Delta^{\mathrm{cl}}$ from
$\mathbf{Set}$ to $\mathbf{Cat}$ which sends a set $S$ to $\Delta^{\mathrm{cl}}(S)
= \mathbf{ClSp}(S)$, the category of $S$-closure spaces, and a mapping
$\varphi\colon S\mor T$ to the functor
$\Delta^{\mathrm{cl}}_{\varphi}\colon \mathbf{ClSp}(T)\mor \mathbf{ClSp}(S)$
defined as follows
\begin{enumerate}
\item $\Delta^{\mathrm{cl}}_{\varphi}$ assigns to a $T$-closure
      space $(B,\mathcal{D})$ the $S$-closure space
      $(B_{\varphi},\Delta_{\varphi}[\mathcal{D}])$.

\item $\Delta^{\mathrm{cl}}_{\varphi}$ assigns to a
      $T$-continuous mapping $f$ from $(B,\mathcal{D})$ to $(B',\mathcal{D}')$
      the $S$-continuous mapping $f_{\varphi}$ from
      $(B_{\varphi},\Delta_{\varphi}[\mathcal{D}])$ to
      $(B'_{\varphi},\Delta_{\varphi}[\mathcal{D}'])$.
\end{enumerate}

\end{proposition}

\begin{definition}
The category $\mathbf{MClSp}$ of \emph{many-sorted closure spaces and
continuous mappings}, obtained by applying the construction of
Ehresmann-Grothendieck to the contravariant functor
$\Delta^{\mathrm{cl}}$ from $\mathbf{Set}$ to $\mathbf{Cat}$, is $\mathbf{MClSp} =
\int^{\mathbf{Set}}\Delta^{\mathrm{cl}}$.
\end{definition}

Therefore $\mathbf{MClSp}$ has as objects the triples
$(S,A,\mathcal{C})$, where $S$ is a set and $(A,\mathcal{C})$ an
$S$-closure space, and as morphisms from $(S,A,\mathcal{C})$ to
$(T,B,\mathcal{D})$ the triples
$((S,A,\mathcal{C}),(\varphi,f),(T,B,\mathcal{D}))$, abbreviated to
$(\varphi,f)\colon (S,A,\mathcal{C})\mor (T,B,\mathcal{D})$, where
$(\varphi,f)$ is such that $\varphi\colon S\mor T$ is a mapping and
$f\colon (A,\mathcal{C})\mor
(B_{\varphi},\Delta_{\varphi}[\mathcal{D}])$ an $S$\nobreakdash-continuous
mapping.  From now on, to shorten terminology, we will say
\emph{closure space} and \emph{continuous mapping}, instead of
\emph{many-sorted closure space} and \emph{many-sorted continuous
mapping}, respectively, when this is unlikely to cause confusion.

The forgetful functor from the category $\mathbf{MClSp}$ to the
category $\mathbf{MSet}$ has left and right adjoints and constructs
limits and colimits, exactly as for the forgetful functor from the
category $\mathbf{ClSp}(S)$ to the category $\mathbf{Set}^{S}$.  These
results follow from the following two lemmas.

\begin{lemma}\label{oplift}
Let $(S,A)$ be a $\mathrm{ms}$-set, $(S_{i},A^{i},\mathcal{C}^{i})_{
i\in I}$ a family of closure spaces and $(\varphi_{I},f^{I}) =
(\varphi_{i},f^{i})_{i\in I}$ a family of $\mathrm{ms}$-mappings,
where, for every $i\in I$, $(\varphi_{i},f^{i})$ is a
$\mathrm{ms}$-mapping from $(S,A)$ to $(S_{i},A^{i})$, i.e.,
$\varphi_{i}$ is a mapping from $S$ to $S_{i}$ and $f^{i} =
(f^{i}_{s})_{s\in S}$ an $S$-sorted mapping from $A$ to
$A^{i}_{\varphi_{i}} = (A^{i}_{\varphi_{i}(s)})_{s\in S}$.  Then there
exists a uniquely determined closure system $\mathcal{C}$ on $A$,
denoted by
$\text{L}^{(\varphi_{I},f^{I})}(S_{i},A^{i},\mathcal{C}^{i})_{i\in I}$
and called the optimal lift of $(S_{i},A^{i},\mathcal{C}^{i})_{i\in
I}$ through $(\varphi_{I},f^{I})$, such that:
\begin{enumerate}
\item For every $i\in I$, $(\varphi_{i},f^{i})\colon
      (S,A,\text{L}^{(\varphi_{I},f^{I})}(S_{i},A^{i},\mathcal{C}^{i})_{
      i\in I})\mor(S_{i},A^{i},\mathcal{C}^{i})$ is a continuous
      mapping.

\item For every closure space $(T,B,\mathcal{D})$ and every
      $\mathrm{ms}$-mapping $(\psi,g)$ from $(T,B)$ to $(S,A)$, if,
      for every $i\in I$, $(\varphi_{i},f^{i})\comp (\psi,g)$ is a
      continuous mapping from $(T,B,\mathcal{D})$ to
      $(S_{i},A^{i},\mathcal{C}^{i})$, then
      $(\psi,g)$ is a continuous mapping from $(T,B,\mathcal{D})$ to
      $(S,A,\text{L}^{(\varphi_{I},f^{I})}(S_{i},A^{i},\mathcal{C}^{i})_{
      i\in I})$.
\end{enumerate}

Besides, we have that:
\begin{enumerate}
\item For every closure system $\mathcal{C}$ on $A$:
     $$
      \text{L}^{(\mathrm{id}_{S},\mathrm{id}_{A})}(S,A,\mathcal{C}) = \mathcal{C}.
     $$
\item If, for every $i\in I$,
      $(S_{i,m},A^{i,m},\mathcal{C}^{i,m})_{m\in M_{i}}$ is a
      family of closure spaces, $(\varphi_{i,M_{i}},g^{i,M_{i}}) =
      (\varphi_{i,m},g^{i,m})_{m\in M_{i}}$ a family of
      $\mathrm{ms}$-mappings, where, for every $m\in M_{i}$,
      $(\varphi_{i,m},g^{i,m})$ is a $\mathrm{ms}$-mapping from
      $(S_{i},A^{i})$ to $(S_{i,m},A^{i,m})$ and $\mathcal{C}^{i} =
      \text{L}^{(\varphi_{i,M_{i}},g^{i,M_{i}})}
      (S_{i,m},A^{i,m},\mathcal{C}^{i,m})_{m\in M_{i}}$, then
      $$
        \text{L}^{((\varphi_{i,M_{i}},g^{i,M_{i}})\comp
        (\varphi_{I},f^{I}))_{i\in I}}
        (S_{i,m},A^{i,m},\mathcal{C}^{i,m})_{
        (i,m)\in\coprod_{i\in I}M_{i}} =
        \text{L}^{(\varphi_{I},f^{I})}(S_{i},A^{i},\mathcal{C}^{i})_{i\in I}.
      $$
\end{enumerate}
\end{lemma}

\begin{proof}
To show that there exists the optimal lift of
$(S_{i},A^{i},\mathcal{C}^{i})_{i\in I}$ through the family
$(\varphi_{I},f^{I})$ of $\mathrm{ms}$-mappings, it is enough to take as
$L^{(\varphi_{I},f^{I})}(S_{i},A^{i},\mathcal{C}^{i})_{i\in I}$
the closure system on $A$ generated by $ \bigcup_{i\in
I}\{\,(f^{i})^{-1}[C]\mid C\in
\Delta_{\varphi_{i}}[\mathcal{C}^{i}]\,\}$.

The remaining parts are an obvious consequence of the first part.
\end{proof}

\begin{lemma}\label{co-oplift}
Let $(S,A)$ be $\mathrm{ms}$-set, $(S_{i},A^{i},\mathcal{C}^{i})_{
i\in I}$ a family of closure spaces and $(\varphi_{I},f^{I}) =
(\varphi_{i},f^{i})_{i\in I}$ a family of $\mathrm{ms}$-mappings,
where, for every $i\in I$, $(\varphi_{i},f^{i})$ is a
$\mathrm{ms}$-mapping from $(S_{i},A^{i})$ to $(S,A)$, i.e.,
$\varphi_{i}$ is a mapping from $S_{i}$ to $S$ and $f^{i} =
(f^{i}_{s})_{s\in S_{i}}$ an $S_{i}$-sorted mapping from $A^{i}$ to
$A_{\varphi_{i}} = (A_{\varphi_{i}(s)})_{s\in S_{i}}$.  Then there
exists a uniquely determined closure system $\mathcal{C}$ on $A$,
denoted by
$\text{L}_{(\varphi_{I},f^{I})}(S_{i},A^{i},\mathcal{C}^{i})_{i\in I}$
and called the co-optimal lift of $(S_{i},A^{i},\mathcal{C}^{i})_{i\in
I}$ through $(\varphi_{I},f^{I})$, such that:
\begin{enumerate}
\item For every $i\in I$, $(\varphi_{i},f^{i})\colon
      (S_{i},A^{i},\mathcal{C}^{i})\mor
      (S,A,\text{L}_{(\varphi_{I},f^{I})}(S_{i},A^{i},\mathcal{C}^{i})_{
      i\in I})$ is a continuous mapping.

\item For every closure space $(T,B,\mathcal{D})$ and every
      $\mathrm{ms}$-mapping $(\psi,g)$ from $(S,A)$ to $(T,B)$, if,
      for every $i\in I$, $(\psi,g)\comp (\varphi_{i},f^{i})$ is a
      continuous mapping from $(S_{i},A^{i},\mathcal{C}^{i})$ to
      $(T,B,\mathcal{D})$, then $(\psi,g)$ is a continuous mapping from
      $(S,A,
      \text{L}_{(\varphi_{I},f^{I})}(S_{i},A^{i},\mathcal{C}^{i})_{
      i\in I})$ to $(T,B,\mathcal{D})$.
\end{enumerate}

Besides, we have that:
\begin{enumerate}
\item For every closure system $\mathcal{C}$ on $A$:
     $$
      \text{L}_{(\mathrm{id}_{S},\mathrm{id}_{A})}(S,A,\mathcal{C}) = \mathcal{C}.
     $$

\item If, for every $i\in I$,
      $(S_{i,m},A^{i,m},\mathcal{C}^{i,m})_{m\in M_{i}}$ is a
      family of closure spaces, $(\varphi_{i,M_{i}},g^{i,M_{i}}) =
      (\varphi_{i,m},g^{i,m})_{m\in M_{i}}$ a family of
      $\mathrm{ms}$-mappings, where, for every $m\in M_{i}$,
      $(\varphi_{i,m},g^{i,m})$ is a $\mathrm{ms}$-mapping from
      $(S_{i,m},A^{i,m})$ to $(S_{i},A^{i})$
      and $\mathcal{C}^{i} = \text{L}_{(\varphi_{i,M_{i}},g^{i,M_{i}})}
      (S_{i,m},A^{i,m},\mathcal{C}^{i,m})_{m\in M_{i}}$, then
      $$
        \text{L}_{((\varphi_{I},f^{I})\comp
        (\varphi_{i,M_{i}},g^{i,M_{i}}))_{i\in I}}
        (S_{i,m},A^{i,m},\mathcal{C}^{i,m})_{
        (i,m)\in\coprod_{i\in I}M_{i}} =
        \text{L}_{(\varphi_{I},f^{I})}(S_{i},A^{i},\mathcal{C}^{i})_{i\in I}.
      $$
\end{enumerate}
\end{lemma}

\begin{proof}
To show that there exists the co-optimal lift of
$(S_{i},A^{i},\mathcal{C}^{i})_{i\in I}$ through the family
$(\varphi_{I},f^{I})$ of $\mathrm{ms}$-mappings, let $\Lambda$
be the set of all closure systems $\mathcal{L}$ on $A$ such that, for
every $i\in I$, $(\varphi_{i},f^{i})$ is a continuous mapping from
$(S_{i},A^{i},\mathcal{C}^{i})$ to $(S,A,\mathcal{L})$.  The set
$\Lambda$ is nonempty since the empty optimal lift is in $\Lambda$.
Next, let $\mathcal{C}$ be the optimal lift of the $\Lambda$-indexed
family $(\mathrm{id}_{S},\mathrm{id}_{A})\colon (S,A)\mor
(S,A,\mathcal{L})$, $(T,B,\mathcal{D})$ a closure space and $(\psi,g)$
a $\mathrm{ms}$-mapping from $(S,A)$ to $(T,B)$ such that, for every
$i\in I$, $(\psi,g)\comp (\varphi_{i},f^{i})$ is a continuous mapping
from $(S_{i},A^{i},\mathcal{C}^{i})$ to $(T,B,\mathcal{D})$.  Let
$\mathcal{L}$ be the optimal lift of $(T,B,\mathcal{D})$ through the
$\mathrm{ms}$-mapping $(\psi,g)$.  Since $\mathcal{L}$
is optimal, $\mathcal{L}\in\Lambda$ and $(\psi,g)$ is a
continuous mapping from $(S,A,\mathcal{C})$ to $(T,B,\mathcal{D})$
because it is the composition of $(\mathrm{id}_{S},\mathrm{id}_{A})$ and
$(\psi,g)$.

The remaining parts are an obvious consequence of the first part.
\end{proof}

From Lemmas~\ref{oplift} $\!\!\And\!\!$ \ref{co-oplift} it follows, as
announced above, immediately the following

\begin{corollary}
The forgetful functor from $\mathbf{MClSp}$ to $\ct{MSet}$ has left
and right adjoints and constructs limits and colimits.
\end{corollary}

We observe that the properties of the category $\mathbf{MClSp}$ could,
eventually, be useful in order to facilitate the construction of
logical systems dealing simultaneously with objects of two, or more,
types, under the hypothesis of the existence of some kind of
interaction between them (reflected at the model-theoretical level by
the existence of, e.g., adjoint situations).

\begin{remark}
All of the results stated by Feitosa and D'Ottaviano in~\cite{fd01}
(compare with those stated a long time ago by Brown in~\cite{djb69},
by Brown and Suszko in~\cite{bs73}, and by Porte in~\cite{jP65},
especially those in Chapter 12, pp.  83--96) that have to do with
closure spaces, continuous mappings, optimal and co-optimal lifts, and
completeness and co-completeness of the category of closure spaces,
fall, as a very particular case, under the just developed theory,
because what they call \emph{logics} are, simply, ordinary (not
many-sorted) closure spaces.  Besides, by defining the appropriate
subcategories of $\mathbf{MClSp}$, the many-sorted counterparts of the
remaining results in~\cite{fd01} are also, easily, provable from the
above generalized theory about many-sorted closure spaces.
\end{remark}

To make the family
$(\mathrm{Cn}_{\mathbf{\Sigma}})_{\mathbf{\Sigma}\in \mathbf{Sig}}$ of
closure operators the components of a pseudo-functor $\mathrm{Cn}$
from $\mathbf{Sig}$ to a convenient $2$-category of monads, we begin
by proving, for a Grothendieck universe $\boldsymbol{\mathcal{V}}$
such that $\boldsymbol{\mathcal{U}}\in \boldsymbol{\mathcal{V}}$, the
existence of two mappings from the sets of objects and morphisms of
$\mathbf{Sig}$ to the respective sets of objects and morphisms of the
category $\mathbf{MClSp}_{\boldsymbol{\mathcal{V}}}$, from which we
will get the pseudo-functor $\mathrm{Cn}$.

\begin{proposition}
Let $\boldsymbol{\mathcal{V}}$ be a Grothendieck universe such that
$\boldsymbol{\mathcal{U}}\in \boldsymbol{\mathcal{V}}$.  Then there
exists a pair of mappings, both denoted by $\Cn$, one from the set of
objects of $\mathbf{Sig}$ to the set of objects of the category
$\mathbf{MClSp}_{\boldsymbol{\mathcal{V}}}$, of closure spaces for
$\boldsymbol{\mathcal{V}}$, which sends $\mathbf{\Sigma}$ to the
closure space $\Cn(\mathbf{\Sigma}) =
((\boldsymbol{\mathcal{U}}^{S})^{2},\Eq(\mathbf{\Sigma}),\Cn_{\mathbf{\Sigma}})$,
and the other from the set of morphisms of $\mathbf{Sig}$ to the set
of morphism of $\mathbf{MClSp}_{\boldsymbol{\mathcal{V}}}$, which
sends $\mathbf{d}\colon \mathbf{\Sigma}\mor \mathbf{\Lambda}$ to the
continuous mapping
$$
\Cn(\mathbf{d}) =
((\textstyle{\coprod}_{\varphi})^{2},\mathbf{d}_{\diamond}^{2})\colon
((\boldsymbol{\mathcal{U}}^{S})^{2},\Eq(\mathbf{\Sigma}),\Cn_{\mathbf{\Sigma}})\mor
((\boldsymbol{\mathcal{U}}^{T})^{2},\Eq(\mathbf{\Lambda}),\Cn_{\mathbf{\Lambda}}).
$$
\end{proposition}

\begin{proof}
We restrict ourselves to prove that $\mathbf{d}_{\diamond}^{2}$ is a
$(\boldsymbol{\mathcal{U}}^{S})^{2}$-continuous mapping from
$(\Eq(\mathbf{\Sigma}),\Cn_{\mathbf{\Sigma}})$ to
$(\Eq(\mathbf{\Lambda})_{(\coprod_{\varphi})^{2}},
(\Cn_{\mathbf{\Lambda}})_{(\coprod_{\varphi})^{2}})$, i.e., that, for
every $\mathcal{E}\subseteq \Eq(\mathbf{\Sigma})$, $X, Y\in
\boldsymbol{\mathcal{U}}^{S}$, and $(P,Q)\in
\Eq(\mathbf{\Sigma})_{X,Y}$, we have that
$$
(P,Q)\in\Cn_{\mathbf{\Sigma}}(\mathcal{E})_{X,Y}\quad \text{only if}\quad
(\mathbf{d}_{\diamond}(P),\mathbf{d}_{\diamond}(Q))\in\Cn_{\mathbf{\Lambda}}
(\mathbf{d}_{\diamond}^{2}[\mathcal{E}])_{\coprod_{\varphi}X,\coprod_{\varphi}Y}.
$$
If $(P,Q)\in\Cn_{\mathbf{\Sigma}}(\mathcal{E})_{X,Y}$,  then
$(\mathbf{d}_{\diamond}(P),\mathbf{d}_{\diamond}(Q))\in\Cn_{\mathbf{\Lambda}}
(\mathbf{d}_{\diamond}^{2}[\mathcal{E}])_{\coprod_{\varphi}X,\coprod_{\varphi}Y}$,
because, for every $\mathbf{\Lambda}$-algebra $\mathbf{A}$, from
$\mathbf{A}\models^{\mathbf{\Lambda}}\mathbf{d}_{\diamond}^{2}[\mathcal{E}]$,
by Lemma~\ref{lemaSatisfaccion}, it follows that
$\mathbf{d}^{\ast}(\mathbf{A})\models^{\mathbf{\Sigma}}\mathcal{E}$,
hence
$\mathbf{d}^{\ast}(\mathbf{A})\models^{\mathbf{\Sigma}}_{X,Y}(P,Q)$,
therefore
$\mathbf{A}\models^{\mathbf{\Lambda}}_{\coprod_{\varphi}X,\coprod_{\varphi}Y}
(\mathbf{d}_{\diamond}(P),\mathbf{d}_{\diamond}(Q))$.
\end{proof}

But $\Cn$ does not determine a functor from $\mathbf{Sig}$ to
$\mathbf{MClSp}_{\boldsymbol{\mathcal{V}}}$, because, for example, for
two composable morphisms $\mathbf{d}\colon \mathbf{\Sigma}\mor
\mathbf{\Lambda}$ and $\mathbf{e}\colon \mathbf{\Lambda}\mor
\mathbf{\Omega}$, it is not true, generally, that $\Cn(\mathbf{e}\comp
\mathbf{d}) = \Cn(\mathbf{e})\comp \Cn(\mathbf{d})$.  However, by
defining the adequate $1$-cells and $2$-cells, we will get a
$2$-category $\mathbf{Mnd}_{\boldsymbol{\mathcal{V}},\mathrm{alg}}$,
defined below, that will act as the target $2$-category for a
pseudo-functor defined on $\mathbf{Sig}$ and itself obtained from
$\Cn$.


Since the target $2$-category we want to determine,
$\mathbf{Mnd}_{\boldsymbol{\mathcal{V}},\mathrm{alg}}$, will be obtained
from the similar $2$-category $\mathbf{Mnd}_{\mathrm{alg}}$, simply, by
changing the Grothendieck universe from $\boldsymbol{\mathcal{U}}$ to
$\boldsymbol{\mathcal{V}}$, we proceed next to define this last
$2$-category.

We begin by recalling the concept of adjoint square and one of the
fundamental facts about it, i.e., that the adjoint squares are endowed
with a structure of double category (more details about this subject
matter can be found in~\cite{jwG66}, \cite{jmm65}, and \cite{pP71}),
since to define $\mathbf{Mnd}_{\mathrm{alg}}$ it will be required.

\begin{definition}(Cf.,\cite{jwG66}, pp. 144--145)
An \emph{adjoint square} is a triple
$$
(F\ladj G,(J,\lambda,H),F'\ladj G'),
$$
where the adjoints $F\ladj G$ and $F'\ladj G'$ and the functors $J$
and $H$ are related as in the diagram
$$
\xymatrix@C=50pt@R=45pt{
*++{\mathbf{C}}\xyn{1} &
*++{\mathbf{D}}\xyn{2} \\
*++{\mathbf{C}'}\xyn{3} &
*++{\mathbf{D}'}\xyn{4}
\ar@<+1.5ex>@{<- }"1";"2"^{G}
\ar@{}"1";"2"|{\uadj}
\ar@<-1.5ex>@{ ->}"1";"2"_{F}
\ar@{ ->}"1";"3"_{J}
\ar@{ ->}"2";"4"^{H}
\ar@<+1.5ex>@{<- }"3";"4"^{G'}
\ar@{}"3";"4"|{\uadj}
\ar@<-1.5ex>@{ ->}"3";"4"_{F'}
}
$$
and $\lambda$ is a matrix
$$
\lambda=
  \left(
  \begin{matrix}
  \lambda_{0}\colon F'J\cel HF &\lambda_{1}\colon J\cel G'HF \\
  \lambda_{2}\colon F'JG\cel H &\lambda_{3}\colon JG\cel G'H
  \end{matrix}
\right)
$$
of compatible $2$-cells, i.e., a matrix of natural transformations as
indicated such that
$$
\begin{aligned}
\lambda_{0} &= (\lambda_{2}F)(F'J\eta)
             = (\varepsilon'HF)(F'\lambda_{1})
             = (\varepsilon'HF)(F'\lambda_{3}F)(F'J\eta),  \\
\lambda_{1} &= (G'\lambda_{0})(\eta'J)
             = (G'\lambda_{2}F)(\eta'J\eta)
             = (\lambda_{3}F)(J\eta),  \\
\lambda_{2} &= (H\varepsilon)(\lambda_{0}G)
             = (\varepsilon'H\varepsilon)(F'\lambda_{1}G)
             = (\varepsilon'H)(F'\lambda_{3}),  \\
\lambda_{3} &= (G'H\varepsilon)(G'\lambda_{0}G)(\eta'JG)
             = (G'\lambda_{2})(\eta'JG)
             = (G'H\varepsilon)(\lambda_{1}G),
\end{aligned}
$$
where $\eta\colon 1\cel GF$ and $\varepsilon\colon FG\cel 1$ are the
unit and counit of $F\ladj G$, and $\eta'\colon 1\cel G'F'$ and
$\varepsilon'\colon F'G'\cel 1$ the unit and counit of $F'\ladj G'$.
\end{definition}

In the following proposition it is stated that the adjoint squares
form a double category. We do not give a proof of this proposition,
since one by Gray can be found in~\cite{jwG66}, pp.  146--149.

However, following the proposition we will give explicit details about
the defining data of the double category of adjoint squares, to
obviate the search in the original sources (\cite{jwG66},
\cite{jmm65}, and \cite{pP71}), and because some of them will be
needful below (when defining the algebraic morphisms between monads
and the algebraic transformations from an algebraic morphism into a
like one).

\begin{proposition}
The adjoint squares constitute a double category, denoted by
$\mathbf{AdFun}$.
\end{proposition}

As announced above, we recall the definition of the data
that occur in the double category $\mathbf{AdFun}$.

Given an adjoint square $(F\ladj G,(J,\lambda,H),F'\ladj G')$ its
Ad-\emph{domain} and Ad-\emph{codomain} in $\mathbf{AdFun}$ are $F\ladj G$
and $F'\ladj G'$, respectively, and its Fun-\emph{domain} and
Fun-\emph{codomain} in $\mathbf{AdFun}$ are $J$ and $H$, respectively.

The Ad-\emph{identities} and Fun-\emph{identities} are represented by
the following adjoint squares
$$
\xymatrix@C=50pt@R=45pt{
*++{\mathbf{C}}\xyn{1} &
*++{\mathbf{C}}\xyn{2} \\
*++{\mathbf{C}'}\xyn{3} &
*++{\mathbf{C}'}\xyn{4}
\ar@<+1.5ex>@{<- }"1";"2"^{1}
\ar@{}"1";"2"|{\uadj}
\ar@<-1.5ex>@{ ->}"1";"2"_{1}
\ar@{ ->}"1";"3"_{J}
\ar@{ ->}"2";"4"^{J}
\ar@<+1.5ex>@{<- }"3";"4"^{1}
\ar@{}"3";"4"|{\uadj}
\ar@<-1.5ex>@{ ->}"3";"4"_{1}
\ar @{} "1";"4"
 |{
   \left(
   \begin{matrix}
   J & J \\ J & J
   \end{matrix}
   \right)
  }
}
\qquad\text{and}\qquad
\xymatrix@C=50pt@R=45pt{
*++{\mathbf{C}}\xyn{1} &
*++{\mathbf{D}}\xyn{2} \\
*++{\mathbf{C}}\xyn{3} &
*++{\mathbf{D}}\xyn{4}
\ar@<+1.5ex>@{<- }"1";"2"^{G}
\ar@{}"1";"2"|{\uadj}
\ar@<-1.5ex>@{ ->}"1";"2"_{F}
\ar@{ ->}"1";"3"_{1}
\ar@{ ->}"2";"4"^{1}
\ar@<+1.5ex>@{<- }"3";"4"^{G}
\ar@{}"3";"4"|{\uadj}
\ar@<-1.5ex>@{ ->}"3";"4"_{F}
\ar @{} "1";"4"
 |{
   \left(
   \begin{matrix}
   F & \eta \\ \varepsilon & G
   \end{matrix}
   \right)
  }
}
$$
The Ad-\emph{composition} of two adjoint squares
$$
\xymatrix@C=50pt@R=45pt{
*++{\mathbf{C}}\xyn{1} &
*++{\mathbf{D}}\xyn{2} &
*++{\mathbf{E}}\xyn{3}\\
*++{\mathbf{C}'}\xyn{4} &
*++{\mathbf{D}'}\xyn{5} &
*++{\mathbf{E}'}\xyn{6}
\ar@<+1.5ex>@{<- }"1";"2"^{G}
\ar@{}"1";"2"|{\uadj}
\ar@<-1.5ex>@{ ->}"1";"2"_{F}
\ar@<+1.5ex>@{<- }"2";"3"^{R}
\ar@{}"2";"3"|{\uadj}
\ar@<-1.5ex>@{ ->}"2";"3"_{L}
\ar@{ ->}"1";"4"_{J}
\ar@{ ->}"2";"5"^{H}
\ar@{ ->}"3";"6"^{M}
\ar@<+1.5ex>@{<- }"4";"5"^{G'}
\ar@{}"4";"5"|{\uadj}
\ar@<-1.5ex>@{ ->}"4";"5"_{F'}
\ar@<+1.5ex>@{<- }"5";"6"^{R'}
\ar@{}"5";"6"|{\uadj}
\ar@<-1.5ex>@{ ->}"5";"6"_{L'}
\ar @{} "1";"5"|{\lambda}
\ar @{} "2";"6"|{\delta}
}
$$
is the adjoint square $(LF\ladj
GR,(J,\delta\adcomp\lambda,M),L'F'\ladj G'R')$ where
$\delta\adcomp\lambda$ is the matrix
$$
\delta\adcomp\lambda=
  \left(
  \begin{matrix}
  (\delta_{0}F)(L'\lambda_{0}) &(G'\delta_{1}F)\lambda_{1} \\
  \delta_{2}(L'\lambda_{2}R) &(G'\delta_{3})(\lambda_{3}R)
  \end{matrix}
\right).
$$
And the Fun-\emph{composition} of two adjoint squares
$$
\xymatrix@C=50pt@R=45pt{
*++{\mathbf{C}}\xyn{1} &
*++{\mathbf{D}}\xyn{2} \\
*++{\mathbf{C}'}\xyn{3} &
*++{\mathbf{D}'}\xyn{4}\\
*++{\mathbf{C}''}\xyn{5} &
*++{\mathbf{D}''    }\xyn{6}
\ar@<+1.5ex>@{<- }"1";"2"^{G}
\ar@{}"1";"2"|{\uadj}
\ar@<-1.5ex>@{ ->}"1";"2"_{F}
\ar@{ ->}"1";"3"_{J}
\ar@{ ->}"2";"4"^{H}
\ar@<+1.5ex>@{<- }"3";"4"^{G'}
\ar@{}"3";"4"|{\uadj}
\ar@<-1.5ex>@{ ->}"3";"4"_{F'}
\ar@{ ->}"3";"5"_{J'}
\ar@{ ->}"4";"6"^{H'}
\ar@<+1.5ex>@{<- }"5";"6"^{G''}
\ar@{}"5";"6"|{\uadj}
\ar@<-1.5ex>@{ ->}"5";"6"_{F''}
\ar @{} "1";"4"|{\lambda}
\ar @{} "3";"6"|{\lambda'}
}
$$
is the adjoint square $(F\ladj
G,(J'J,\lambda'\funcomp\lambda,H'H),F''\ladj G'')$ where
$\lambda'\funcomp\lambda$ is the matrix
$$
\lambda'\funcomp\lambda=
  \left(
  \begin{matrix}
  (H'\lambda_{0})(\lambda'_{0}J)
  &(G'H'\varepsilon' HF)(\lambda'_{1}\lambda_{1}) \\
  (\lambda'_{2}\lambda_{2})(F''J'\eta'JG)
  &(\lambda'_{3}H)(J'\lambda_{3})
  \end{matrix}
\right).
$$

To attain our aim of defining the $2$-category
$\mathbf{Mnd}_{\mathrm{alg}}$, we continue by defining the concept of
monad and by stating, for a pair of monads and an adjunction between
the underlying categories of the monads, the existence of a
commutative square of bijections between four sets of natural
transformations obtained from the monads and the adjunction, as well
as conditions of compatibility on the matrices of natural
transformations arranged in the pattern of the just named commutative
square of bijections.

\begin{definition}
By a \emph{monad} we understand a pair $(\mathbf{C},\mathbb{T})$, with
$\mathbf{C}$ a category and $\mathbb{T} = (T,\eta,\mu)$ a monad in
$\mathbf{C}$.
\end{definition}

\begin{proposition}\label{AlgAdSq}
Let $(\mathbf{C},\mathbb{T})$ and $(\mathbf{C}',\mathbb{T}')$ be two monads
and $(J, K,\maol{\eta},\maol{\varepsilon})\colon \mathbf{C}\mor \mathbf{C}'$
an adjunction.  Then for the following diagram
$$
\begin{aligned}
\xymatrix@C=60pt@R=40pt{
*++{\mathbf{C}}\xyn{1} &
*++{\mathbf{C}}\xyn{2} \\
*++{\mathbf{C}'}\xyn{3} &
*++{\mathbf{C}'}\xyn{4}
\ar@{ ->}"1";"2"^{T}
\ar@{ ->}"3";"4"_{T'}
\ar@<+1.5ex>@{<- }"1";"3"^{K}
\ar@{}"1";"3"|{\ladj}
\ar@<-1.5ex>@{ ->}"1";"3"_{J}
\ar@<+1.5ex>@{<- }"2";"4"^{K}
\ar@{}"2";"4"|{\ladj}
\ar@<-1.5ex>@{ ->}"2";"4"_{J}
}
\end{aligned}
$$
there exists, by Corollary I,6.6 stated by Gray in~\cite{jwG66}, p.
143, a commutative square of bijections
$$
\xymatrix@C=30pt@R=30pt{
\Nat(JT,T'J)
\ar[r]^{\iso}
\ar[d]_{\iso} &
\Nat(T,KT'J)
\ar[d]^{\iso}  \\
\Nat(JTK,T')
\ar[r]_{\iso} &
\Nat(TK,KT')  }
$$
Furthermore, the following conditions on the natural transformations in the
matrix
$$
\lambda=
  \left(
  \begin{matrix}
  \lambda_{0}\colon JT\cel T'J &\lambda_{1}\colon T\cel KT'J \\
  \lambda_{2}\colon JTK\cel T' &\lambda_{3}\colon TK\cel KT'
  \end{matrix}
\right)
$$
are compatible with the above bijections:
\Smmallmatrix
\begin{enumerate}
\xymatrixcolsep = {37pt}
\xymatrixrowsep = {33pt}
\item The natural transformations $\lambda_{0}\colon  JT\cel T' J$
such that %
$$
\xymatrix{
{}
  \ar@`{{**{}?(1)+<10pt,35pt>},{?(0)+<-10pt,35pt>}}%
       "1,1";"1,2"^{1}|{}="1"
  \ar[r]|*+{T}="2"
  \ar[d]|*+{J}
  &
{}
  \ar[d]|*+{J}
  \ar@{}[ld]|{\dir{=>}}|{\hspace{4ex}\lambda_{0}}
  \\
{}
  \ar[r]_{T'} &
{}
  \ar @{} "1";"2" |{\dir{=>}}|{\hspace{3ex}\eta}
}
\hspace{2pt}\xymatrix{{}\ar@{}[d]|(.5){=} \\ {}}\hspace{2pt}
\xymatrix{
{}
  \ar[r]^{1}|{}="1"
  \ar[d]|*+{J} &
{}\xyn{b}
  \ar[d]|*+{J}
  \\
{}\xyn{c}
  \ar[r]|*+{1}="2"
  \ar@`{{**{} ?(1)+<10pt,-35pt>}, {?(0)+<-10pt,-35pt>}}%
      "2,1";"2,2"_{T'}|{}="3" &
{}
  \ar @{} "b";"c" |{\dir{=>}}|{\hspace{3ex}J}
  \ar @{} "2";"3" |{\dir{=>}}|{\hspace{3ex}\eta'}
}
\hspace{12pt}
\xymatrix{
{}
  \ar[r]^{T}
  \ar[d]|*+{J} &
{}
  \ar[r]^{T}
  \ar[d]|*+{J}
  \ar@{}[ld]|{\dir{=>}}|{\hspace{4ex}\lambda_{0}} &
{}
  \ar[d]|*+{J}
  \ar@{}[ld]|{\dir{=>}}|{\hspace{4ex}\lambda_{0}}
  \\
{}
  \ar[r]|*+{T'}
  \ar@`{{**{} ?(1)+<10pt,-35pt>}, {?(0)+<-10pt,-35pt>}}%
      [rr]_{T'}|{}="a" &
{}
  \ar[r]|*+{T'} &
{}
\ar @{} "2,2";"a" |{\dir{=>}}|{\hspace{3ex}\mu'}
}
\hspace{2pt}\xymatrix{{}\ar@{}[d]|(.5){=} \\ {}}\hspace{2pt}
\xymatrix{
{}
  \ar@`{{**{}?(1)+<10pt,35pt>},{?(0)+<-10pt,35pt>}}%
    [r]^{TT}|{}="1"
  \ar[r]|*+{T}="2"
  \ar[d]|*+{J} &
{}
  \ar[d]|*+{J}
  \ar@{}[ld]|{\dir{=>}}|{\hspace{4ex}\lambda_{0}}
  \\
{}
  \ar[r]_{T'} &
{}
\ar @{} "1";"2" |{\dir{=>}}|{\hspace{3ex}\mu}
}
$$

\item The natural transformations $\lambda_{1}\colon T\cel KT' J$
such that %
$$
\xymatrix{
{}
  \ar@`{{**{}?(1)+<10pt,35pt>},{?(0)+<-10pt,35pt>}}%
     [r]^{1}|{}="1"
  \ar[r]|*+{T}="2"
  \ar[d]|*+{J} &
{}
  \\
{}
  \ar[r]_{T'}|{}="3" &
{}
  \ar[u]|*+{K}
\ar @{} "1";"2" |{\dir{=>}}|{\hspace{3ex}\eta}
\ar @{} "2";"3" |{\dir{=>}}|{\hspace{4ex}\lambda_{1}}
}
\hspace{2pt}\xymatrix{{}\ar@{}[d]|(.5){=} \\ {}}\hspace{2pt}
\xymatrix{
{}
  \ar[r]^{1}|{}="1"
  \ar[d]|*+{J} &
{}
  \\
{}
  \ar[r]|*+{1}="2"
  \ar@`{{**{} ?(1)+<10pt,-35pt>}, {?(0)+<-10pt,-35pt>}}%
    [r]_{T'}|{}="3" &
{}
  \ar[u]|*+{K}
\ar @{} "1";"2" |{\dir{=>}}|{\hspace{4ex}\maol{\eta}}
\ar @{} "2";"3" |{\dir{=>}}|{\hspace{4ex}\eta'}
}
\hspace{12pt}
\xymatrix{
{}
  \ar[r]^{T}|{}="1"
  \ar[d]|*+{J} &
{}
  \ar[r]^{1}|{}="3" &
{}
  \ar[r]^{T}|{}="5"
  \ar[d]|*+{ J} &
{}
  \\
{}
  \ar[r]_{T'}|{}="2"
  \ar@`{{**{} ?(1)+<10pt,-45pt>}, {?(0)+<-10pt,-45pt>}}%
      [rrr]_{T'}="z"&
{}
  \ar[u]|*+{ K}
  \ar[r]_{1}|{}="4" &
{}
  \ar[r]_{T'}|{}="6" &
{}
  \ar[u]|*+{K}
\ar @{} "1";"2" |{\dir{=>}}|{\hspace{4ex}\lambda_{1}}
\ar @{} "3";"4" |{\dir{=>}}|{\hspace{3ex}\maol{\varepsilon}}
\ar @{} "5";"6" |{\dir{=>}}|{\hspace{4ex}\lambda_{1}}
\ar @{} "4";"z" |{\dir{=>}}|{\hspace{3ex}\mu'}
}
\hspace{2pt}\xymatrix{{}\ar@{}[d]|(.5){=} \\ {}}\hspace{2pt}
\xymatrix{
{}
  \ar@`{{**{}?(1)+<10pt,35pt>},{?(0)+<-10pt,35pt>}}%
    [r]^{TT}|{}="1"
  \ar[r]|*+{T}="2"
  \ar[d]|*+{J} &
{}
  \\
{}
  \ar[r]_{T'}|{}="3" &
{}
  \ar[u]|*+{K}
\ar @{} "1";"2" |{\dir{=>}}|{\hspace{3ex}\mu}
\ar @{} "2";"3" |{\dir{=>}}|{\hspace{4ex}\lambda_{1}}
}
$$

\item The natural transformations $\lambda_{2}\colon  JT K\cel T'$
such that  %
$$
\xymatrix{
{}
  \ar@`{{**{}?(1)+<10pt,35pt>},{?(0)+<-10pt,35pt>}}%
     [r]^{1}|{}="1"
  \ar[r]|*+{T}="2"
  &
{}
  \ar[d]|*+{J}\\
{}
  \ar[r]_{T'}|{}="3"
  \ar[u]|*+{K} &
{}
\ar @{} "1";"2" |{\dir{=>}}|{\hspace{3ex}\eta}
\ar @{} "2";"3" |{\dir{=>}}|{\hspace{4ex}\lambda_{2}}
}
\hspace{2pt}\xymatrix{{}\ar@{}[d]|(.5){=} \\ {}}\hspace{2pt}
\xymatrix{
{}
  \ar[r]^{1}|{}="1" &
{}
  \ar[d]|*+{J}
  \\
{}
  \ar[u]|*+{K}
  \ar[r]^{1}="2"
  \ar@`{{**{} ?(1)+<10pt,-35pt>}, {?(0)+<-10pt,-35pt>}}%
    [r]_{T'}="3" &
{}
\ar @{} "1";"2" |{\dir{=>}}|{\hspace{3ex}\maol{\varepsilon}}
\ar @{} "2";"3" |{\dir{=>}}|{\hspace{3.5ex}\eta'}
}
\hspace{12pt}
\xymatrix{
{}
  \ar[r]^{T}|{}="1" &
{}
  \ar[r]^{1}|{}="3"
  \ar[d]|*+{J} &
{}
  \ar[r]^{T}|{}="5" &
{}
  \ar[d]|*+{J}
  \\
{}
  \ar[u]|*+{K}
  \ar[r]_{T'}|{}="2"
  \ar@`{{**{} ?(1)+<10pt,-45pt>}, {?(0)+<-10pt,-45pt>}}%
      [rrr]_{T'}="z"&
{}
  \ar[r]_{1}|{}="4" &
{}
  \ar[u]|*+{K}
  \ar[r]_{T'}|{}="6" &
{}
\ar @{} "1";"2" |{\dir{=>}}|{\hspace{4ex}\lambda_{2}}
\ar @{} "3";"4" |{\dir{=>}}|{\hspace{3ex}\maol{\eta}}
\ar @{} "5";"6" |{\dir{=>}}|{\hspace{4ex}\lambda_{2}}
\ar @{} "4";"z" |{\dir{=>}}|{\hspace{3ex}\mu'}
}
\hspace{2pt}\xymatrix{{}\ar@{}[d]|(.5){=} \\ {}}\hspace{2pt}
\xymatrix{
{}
  \ar@`{{**{}?(1)+<10pt,35pt>},{?(0)+<-10pt,35pt>}}%
    [r]^{TT}|{}="1"
  \ar[r]|*+{T}="2" &
{}
  \ar[d]|*+{J}
  \\
{}
  \ar[u]|*+{K}
  \ar[r]_{T'}|{}="3" &
{}
\ar @{} "1";"2" |{\dir{=>}}|{\hspace{3ex}\mu}
\ar @{} "2";"3" |{\dir{=>}}|{\hspace{4ex}\lambda_{2}}
}
$$

\item The natural transformations $\lambda_{3}\colon TK\cel KT'$
such that %
$$
\xymatrix{
{}
  \ar@`{{**{}?(1)+<10pt,35pt>},{?(0)+<-10pt,35pt>}}%
       "1,1";"1,2"^{1}|{}="1"
  \ar[r]|*+{T}="2"
  \ar@{}[rd]|{\dir{=>}}|{\hspace{4ex}\lambda_{3}} &
{}
  \\
{}
  \ar[u]|*+{K}
  \ar[r]_{T'} &
{}
  \ar[u]|*+{K}
\ar @{} "1";"2" |{\dir{=>}}|{\hspace{3ex}\eta}
}
\hspace{2pt}\xymatrix{{}\ar@{}[d]|(.5){=} \\ {}}\hspace{2pt}
\xymatrix{
{}
  \ar[r]^{1}|{}="1"
  &
{}
  \\
{}
  \ar[u]|*+{K}
  \ar[r]|*+{1}="2"
  \ar@`{{**{} ?(1)+<10pt,-35pt>}, {?(0)+<-10pt,-35pt>}}%
      "2,1";"2,2"_{T'}|{}="3"&
{}
  \ar[u]|*+{K}
\ar @{} "1,1";"2,2" |{\dir{=>}}|{\hspace{4ex}K}
\ar @{} "2";"3" |{\dir{=>}}|{\hspace{4ex}\eta'}
}
\hspace{12pt}
\xymatrix{
{}
  \ar[r]^{T}
  \ar@{}[rd]|{\dir{=>}}|{\hspace{4ex}\lambda_{3}} &
{}
  \ar[r]^{T}
  \ar@{}[rd]|{\dir{=>}}|{\hspace{4ex}\lambda_{3}} &
{}
  \\
{}
  \ar[u]|*+{K}
  \ar[r]_{T'}
  \ar@`{{**{} ?(1)+<10pt,-35pt>}, {?(0)+<-10pt,-35pt>}}%
      [rr]_{T'}|{}="a"&
{}
  \ar[u]|*+{K}
  \ar[r]_{T'} &
{}
  \ar[u]|*+{K}
\ar @{} "2,2";"a" |{\dir{=>}}^{\mu'}
}
\hspace{2pt}\xymatrix{{}\ar@{}[d]|(.5){=} \\ {}}\hspace{2pt}
\xymatrix{
{}
  \ar@`{{**{}?(1)+<10pt,35pt>},{?(0)+<-10pt,35pt>}}%
    [r]^{TT}|{}="1"
  \ar@{}[rd]|{\dir{=>}}|{\hspace{4ex}\lambda_{3}}
  \ar[r]|*+{T}="2" &
{}
  \\
{}
  \ar[u]|*+{K}
  \ar[r]_{T'} &
{}
  \ar[u]|*+{K}
\ar @{} "1";"2" |{\dir{=>}}|{\hspace{4ex}\mu}
}
$$
\end{enumerate}
\end{proposition}

Next we proceed to define the concept of algebraic morphism between
monads, the composition of algebraic morphisms, and the identity
algebraic morphisms at the monads, from which we will get the
\emph{horizontal} component of the $2$-category under consideration.

\begin{definition}
Let $(\mathbf{C},\mathbb{T})$ and $(\mathbf{C}',\mathbb{T}')$ be two
monads.  An \emph{algebraic morphism}, or, to shorten terminology, an
\emph{alg-morphism} from $(\mathbf{C},\mathbb{T})$ to
$(\mathbf{C}',\mathbb{T}')$ is an adjoint square $(J\ladj
K,(T,\lambda,T'),J\ladj K)$, also denoted by $(J\ladj K,\lambda)$, or,
geometrically, by
$$
\begin{aligned}
\xymatrix@C=60pt@R=40pt{
*++{\mathbf{C}}\xyn{1} &
*++{\mathbf{C}}\xyn{2} \\
*++{\mathbf{C}'}\xyn{3} &
*++{\mathbf{C}'}\xyn{4}
\ar@{ ->}"1";"2"^{T}
\ar@{ ->}"3";"4"_{T'}
\ar@<+1.5ex>@{<- }"1";"3"^{K}
\ar@{}"1";"3"|{\ladj}
\ar@<-1.5ex>@{ ->}"1";"3"_{J}
\ar@<+1.5ex>@{<- }"2";"4"^{K}
\ar@{}"2";"4"|{\ladj}
\ar@<-1.5ex>@{ ->}"2";"4"_{J}
\ar @{} "1";"4"|{\lambda}
}
\end{aligned}
$$
such that the components of the matrix
$\lambda = \left(
\begin{smallmatrix}
\lambda_{0} & \lambda_{1} \\ \lambda_{2} & \lambda_{3}
\end{smallmatrix}
\right)$ are compatible as in
the last proposition.  Then taking as objects the monads
$(\mathbf{C},\mathbb{T})$ such that $\mathbf{C}$ is in
$\boldsymbol{\mathcal{U}}$, as morphisms between monads the
alg-morphisms, as the identity at a monad $(\mathbf{C},\mathbb{T})$
the pair $\left(1\ladj 1, \left(
\begin{smallmatrix}
T & T \\ T & T
\end{smallmatrix}
\right)\right)$, and as composition of two alg-morphisms precisely
their Ad-composition as adjoint squares, we get a category, that
we denote by $\mathbf{Mnd}_{\mathrm{alg}}$.
\end{definition}

We now define for two alg-morphisms $(J\ladj K,\lambda)$ and $(J'\ladj
K',\lambda')$ from $(\mathbf{C},\mathbb{T})$ to $(\mathbf{C}',\mathbb{T}')$
the concept of algebraic transformation from $(J\ladj K,\lambda)$ to
$(J'\ladj K',\lambda')$ as well as the vertical and horizontal
composition of algebraic transformations. This will constitute
the \emph{vertical} component of the $2$-category under consideration.

\begin{definition}
Let
$$
\begin{aligned}
\xymatrix@C=60pt@R=40pt{
*++{\mathbf{C}}\xyn{1} &
*++{\mathbf{C}}\xyn{2} \\
*++{\mathbf{C}'}\xyn{3} &
*++{\mathbf{C}'}\xyn{4}
\ar@{ ->}"1";"2"^{T}
\ar@{ ->}"3";"4"_{T'}
\ar@<+1.5ex>@{<- }"1";"3"^{K}
\ar@{}"1";"3"|{\ladj}
\ar@<-1.5ex>@{ ->}"1";"3"_{J}
\ar@<+1.5ex>@{<- }"2";"4"^{K}
\ar@{}"2";"4"|{\ladj}
\ar@<-1.5ex>@{ ->}"2";"4"_{J}
\ar @{} "1";"4"|{\lambda}
}
\end{aligned}
\qquad \text{and}\qquad
\begin{aligned}
\xymatrix@C=60pt@R=40pt{
*++{\mathbf{C}}\xyn{1} &
*++{\mathbf{C}}\xyn{2} \\
*++{\mathbf{C}'}\xyn{3} &
*++{\mathbf{C}'}\xyn{4}
\ar@{ ->}"1";"2"^{T}
\ar@{ ->}"3";"4"_{T'}
\ar@<+1.5ex>@{<- }"1";"3"^{K'}
\ar@{}"1";"3"|{\ladj}
\ar@<-1.5ex>@{ ->}"1";"3"_{J'}
\ar@<+1.5ex>@{<- }"2";"4"^{K'}
\ar@{}"2";"4"|{\ladj}
\ar@<-1.5ex>@{ ->}"2";"4"_{J'}
\ar @{} "1";"4"|{\lambda'}
}
\end{aligned}
$$
be two alg-morphisms from $(\mathbf{C},\mathbb{T})$ to
$(\mathbf{C}',\mathbb{T}')$.  Then an \emph{algebraic transformation}, or,
to shorten terminology, an \emph{alg-transformation}, from $(J\ladj
K,\lambda)$
to $(J'\ladj K',\lambda')$ is an adjoint square %
$$
\begin{aligned}
\xymatrix@C=60pt@R=40pt{
*++{\mathbf{C}}\xyn{1} &
*++{\mathbf{C}}\xyn{2} \\
*++{\mathbf{C}'}\xyn{3} &
*++{\mathbf{C}'}\xyn{4}
\ar@{ ->}"1";"2"^{1}
\ar@{ ->}"3";"4"_{T'}
\ar@<+1.5ex>@{<- }"1";"3"^{K}
\ar@{}"1";"3"|{\ladj}
\ar@<-1.5ex>@{ ->}"1";"3"_{J}
\ar@<+1.5ex>@{<- }"2";"4"^{K'}
\ar@{}"2";"4"|{\ladj}
\ar@<-1.5ex>@{ ->}"2";"4"_{J'}
\ar @{} "1";"4"|{\xi}
}
\end{aligned}
$$
such that
$$
\begin{aligned}
\xymatrix@C=42pt@R=40pt{
*+{\mathbf{C}}\xyn{1} &
*+{\mathbf{C}}\xyn{2} &
*+{\mathbf{C}}\xyn{3} \\
*+{\mathbf{C}'}\xyn{4} &
*+{\mathbf{C}'}\xyn{5} &
*+{\mathbf{C}'}\xyn{6} \\
*+{\mathbf{C}'}\xyn{7} &&
*+{\mathbf{C}'}\xyn{8}
\ar@{ ->}"1";"2"^{T}
\ar@{ ->}"2";"3"^{1}
\ar@{ ->}"4";"5"|*+{T'}
\ar@{ ->}"5";"6"|*+{T'}
\ar@{ ->}"7";"8"_{T'}
\ar@<-6pt>@{ ->}"1";"4"_{J}
\ar@{}"1";"4"|{\ladj}
\ar@<+6pt>@{<- }"1";"4"^{K}
\ar@<-6pt>@{ ->}"2";"5"_{J}
\ar@{}"2";"5"|{\ladj}
\ar@<+6pt>@{<- }"2";"5"^{K}
\ar@<-6pt>@{ ->}"3";"6"_{J'}
\ar@{}"3";"6"|{\ladj}
\ar@<+6pt>@{<- }"3";"6"^{K'}
\ar@<-6pt>@{ ->}"4";"7"_{1}
\ar@{}"4";"7"|{\ladj}
\ar@<+6pt>@{<- }"4";"7"^{1}
\ar@<-6pt>@{ ->}"6";"8"_{1}
\ar@{}"6";"8"|{\ladj}
\ar@<+6pt>@{<- }"6";"8"^{1}
\ar @{} "1";"5"|{\lambda}
\ar @{} "2";"6"|{\xi}
\ar @{} "4";"8"|{\mu'}
}
\end{aligned}
\hspace{3pt} = \hspace{3pt}
\begin{aligned}
\xymatrix@C=42pt@R=40pt{
*+{\mathbf{C}}\xyn{1} &
*+{\mathbf{C}}\xyn{2} &
*+{\mathbf{C}}\xyn{3} \\
*+{\mathbf{C}'}\xyn{4} &
*+{\mathbf{C}'}\xyn{5} &
*+{\mathbf{C}'}\xyn{6} \\
*+{\mathbf{C}'}\xyn{7} &&
*+{\mathbf{C}'}\xyn{8}
\ar@{ ->}"1";"2"^{1}
\ar@{ ->}"2";"3"^{T}
\ar@{ ->}"4";"5"|*+{T'}
\ar@{ ->}"5";"6"|*+{T'}
\ar@{ ->}"7";"8"_{T'}
\ar@<-6pt>@{ ->}"1";"4"_{J}
\ar@{}"1";"4"|{\ladj}
\ar@<+6pt>@{<- }"1";"4"^{K}
\ar@<-6pt>@{ ->}"2";"5"_{J'}
\ar@{}"2";"5"|{\ladj}
\ar@<+6pt>@{<- }"2";"5"^{K'}
\ar@<-6pt>@{ ->}"3";"6"_{J'}
\ar@{}"3";"6"|{\ladj}
\ar@<+6pt>@{<- }"3";"6"^{K'}
\ar@<-6pt>@{ ->}"4";"7"_{1}
\ar@{}"4";"7"|{\ladj}
\ar@<+6pt>@{<- }"4";"7"^{1}
\ar@<-6pt>@{ ->}"6";"8"_{1}
\ar@{}"6";"8"|{\ladj}
\ar@<+6pt>@{<- }"6";"8"^{1}
\ar @{} "1";"5"|{\xi}
\ar @{} "2";"6"|{\lambda'}
\ar @{} "4";"8"|{\mu'}
}
\end{aligned}
$$
i.e., such that $\mu'\adcomp (\xi \funcomp \lambda ) = \mu' \adcomp
(\lambda' \funcomp \xi)$.

For every alg-morphism $(J\ladj K,\lambda)\colon
(\mathbf{C},\mathbb{T})\mor (\mathbf{C}',\mathbb{T}')$, the
\emph{identity} at $(J\ladj K,\lambda)$ is the adjoint square
determined by the matrix %
$$
  \left(
  \begin{matrix}
  \eta'J & K\eta'J \eta^{J\ladj K} \\ \eta'\varepsilon^{J\ladj K} & K\eta'
  \end{matrix}
\right),
$$
where $\eta^{J\ladj K}$ is the unit and $\varepsilon^{J\ladj K}$ the
counit of the adjunction $J\ladj K$.

The \emph{vertical} composition of two alg-tranformations as in the
following diagram%
$$
\xymatrix@C=90pt{
(\mathbf{C},\mathbb{T})
  \ar @/^25pt/ [r]^{(J \ladj K ,\lambda )}|{}="0"
  \ar         [r]|(.5){(J'\ladj K',\lambda' )}="1"
  \ar @/_25pt/  [r]_{(J''\ladj K'',\lambda'')}|{}="2"
&
(\mathbf{C}',\mathbb{T}'),
\ar @{}"0";"1"|{\dir{~>}}^{\,\xi }
\ar @{}"1";"2"|{\dir{~>}}^{\,\xi'}
}
$$
denoted by $\xi'\tcomp\xi $,  is the adjoint square $\mu'\adcomp
(\xi'\funcomp \xi )$. %
%
%
%
%

The \emph{horizontal} composition of two alg-transformations as in the
following diagram%
$$
\xymatrix@C=85pt{
(\mathbf{C},\mathbb{T})
  \ar @/^14pt/ [r]^{(J \ladj K ,\lambda )}|{}="0"
  \ar @/_14pt/  [r]_{(J'\ladj K',\lambda' )}|{}="1"
&
(\mathbf{C}',\mathbb{T}')
  \ar @/^14pt/ [r]^{(J''\ladj K'',\lambda'')}|{}="2"
  \ar @/_14pt/  [r]_{(J'''\ladj K''',\lambda''' )}|{}="3"
&
(\mathbf{C}'',\mathbb{T}''),
\ar @{}"0";"1"|{\dir{~>}}^{\,\xi }
\ar @{}"2";"3"|{\dir{~>}}^{\,\xi'}
}
$$
denoted by $\xi'\thcomp\xi $, is the adjoint square
$$
\mu'\adcomp
(\lambda''\funcomp \xi')\adcomp \xi  = \mu'\adcomp
(\xi'\funcomp \lambda''' )\adcomp \xi.
$$
%
%
%
%
\end{definition}

From the horizontal and vertical components just
defined, it follows immediately the following

\begin{proposition}
The monads whose underlying category is in $\boldsymbol{\mathcal{U}}$,
together with the alg-morphisms between monads, and the
alg-trans\-for\-ma\-tions from an alg-morphism into a like one
determine a 2-category, denoted by $\mathbf{Mnd}_{\mathrm{alg}}$.
\end{proposition}


\begin{definition}
We denote by $\mathbf{Mnd}_{\boldsymbol{\mathcal{V}},\mathrm{alg}}$ the
$2$-category with objects the monads $(\mathbf{C},\mathbb{T})$ such that
$\mathbf{C}$ is in $\boldsymbol{\mathcal{V}}$, $1$-cells the alg-morphisms
and $2$-cells the alg-trans\-for\-ma\-tions between alg-morphisms.
\end{definition}

\begin{proposition}
The category $\mathbf{MClSp}_{\boldsymbol{\mathcal{V}}}$ can be
identified to a subcategory of (the underlying category of) the
$2$-category $\mathbf{Mnd}_{\boldsymbol{\mathcal{V}},\mathrm{alg}}$.
\end{proposition}

\begin{proof}
It is enough to assign to a closure space $(S,A,J)$ the monad
$(\mathbf{Sub}(A),\mathbb{T}_{J})$, where $\mathbf{Sub}(A)$ is the
category determined by the ordered set $(\mathrm{Sub}(A),\subseteq)$
and $\mathbb{T}_{J}$ the monad on $\mathbf{Sub}(A)$ obtained from the
$S$-sorted closure operator $J$ on $A$, as in~\cite{sM98}, p.  139;
and to a continuous mapping $(\varphi,f)$ from $(S,A,J)$ to $(T,B,K)$
the $\mathrm{alg}$-morphism
$$
\xymatrix@C=100pt@R=50pt{
*++{\mathbf{Sub}(A)}\xyn{1} &
*++{\mathbf{Sub}(A)}\xyn{2} \\
*++{\mathbf{Sub}(B)}\xyn{3} &
*++{\mathbf{Sub}(B)}\xyn{4}
\ar@{ ->}"1";"2"^{T_{J}}
\ar@{ ->}"3";"4"_{T_{K}}
\ar@<+1.5ex>@{<- }"1";"3"^{f^{-1}[\cdot]\comp\Delta_{\varphi,B}}
\ar@{}"1";"3"|{\ladj}
\ar@<-1.5ex>@{ ->}"1";"3"_{\union_{\varphi,B}\comp f[\cdot]}
\ar@<+1.5ex>@{<- }"2";"4"^{f^{-1}[\cdot]\comp\Delta_{\varphi,B}}
\ar@{}"2";"4"|{\ladj}
\ar@<-1.5ex>@{ ->}"2";"4"_{\union_{\varphi,B}\comp f[\cdot]}
%
}
$$
from $(\mathbf{Sub}(A),\mathbb{T}_{J})$ to
$(\mathbf{Sub}(B),\mathbb{T}_{K})$.  Observe that we have not written
the matrix $\lambda$ of the alg-morphism because, in this case, it is
trivial.
\end{proof}

\begin{remark}
Taking into account the just stated proposition, we can induce, in a
derived way, for two continuous mappings $(\varphi,f)$ and $(\psi,g)$
from a closure space $(S,A,J)$ into a like one $(T,B,K)$, a notion of
alg-transformation from $(\varphi,f)$ to $(\psi,g)$.  But, because the
underlying categories of the monads associated to the given closure
spaces are complete lattices, hence preorders, there will be at most
an alg-trans\-for\-ma\-tion from the first continuous mapping to the
second one.  Actually, there will be an alg-transformation from
$(\varphi,f)$ to $(\psi,g)$ exactly if, for every $X\subseteq A$, we
have that $\bigcup_{\psi,B}g[X]\subseteq K(\bigcup_{\varphi,B}f[X])$.
\end{remark}

\begin{proposition}
Let $\boldsymbol{\mathcal{V}}$ be a Grothendieck universe such that
$\boldsymbol{\mathcal{U}}\in \boldsymbol{\mathcal{V}}$.  Then there
exists a pseudo-functor, also denoted by $\Cn$, from $\mathbf{Sig}$ to
$\mathbf{Mnd}_{\boldsymbol{\mathcal{V}},\mathrm{alg}}$ that has as
components, essentially, the consequence operators
$\mathrm{Cn}_{\mathbf{\Sigma}}$, for the different signatures
$\mathbf{\Sigma}$.
\end{proposition}

\begin{proof}
We restrict ourselves to define the object and morphism mappings of
the pseudo-functor $\Cn$.  The object mapping of $\Cn$ is that which
sends a signature $\mathbf{\Sigma}$ to the monad
$(\mathbf{Sub}(\Eq(\mathbf{\Sigma})),\mathbb{T}_{\Cn_{\mathbf{\Sigma}}})$,
and the morphism mapping of $\Cn$ is that which sends a signature
morphism $\mathbf{d}\colon \mathbf{\Sigma}\mor \mathbf{\Lambda}$ to
the $\mathrm{alg}$-morphism
$$
(\textstyle{\bigcup}_{(\coprod_{\varphi})^{2},\Eq(\mathbf{\Lambda})}
\comp \mathbf{d}_{\diamond}^{2}[\cdot] \ladj
(\mathbf{d}_{\diamond}^{2})^{-1}[\cdot] \comp
\Delta_{(\coprod_{\varphi})^{2},\Eq(\mathbf{\Lambda})})
$$
from the monad
$(\mathbf{Sub}(\Eq(\mathbf{\Sigma})),\mathbb{T}_{\Cn_{\mathbf{\Sigma}}})$
to the monad
$(\mathbf{Sub}(\Eq(\mathbf{\Lambda})),\mathbb{T}_{\Cn_{\mathbf{\Lambda}}})$.

So defined, it is obvious that $\Cn$ is a pseudo-functor from
$\mathbf{Sig}$ to $\mathbf{Mnd}_{\boldsymbol{\mathcal{V}},\mathrm{alg}}$.
\end{proof}

Relying on the results just stated we propose the following notion of
entailment system, that generalizes that one by Meseguer
in~\cite{m89}, pp.  282--283.

\begin{definition}
An \emph{entailment system} is a quintuple $\mathcal{E} =
(\mathbf{Sig},\mathrm{T},\mathrm{L},\mathrm{M},\Cn)$ with $\mathbf{Sig}$ a
$\boldsymbol{\mathcal{U}}$-category whose objects are called
\emph{signatures}, $\mathrm{T}$ a pseudo-functor from $\mathbf{Sig}$ to
$\mathbf{Cat}$, $\mathrm{L}$ a functor from $\mathbf{Cat}$ to
$\mathbf{Set}_{\boldsymbol{\mathcal{V}}}$, $\mathrm{M}$ a functor from
$\mathbf{Set}_{\boldsymbol{\mathcal{V}}}$ to
$\mathbf{Set}_{\boldsymbol{\mathcal{V}}}$, and $\Cn$ a pseudo-functor from
$\mathbf{Sig}$ to $\mathbf{Mnd}_{\boldsymbol{\mathcal{V}},\mathrm{alg}}$ such
that the following diagram commutes up to isomorphism
$$
\xymatrix@C=16ex@R=12ex{
\mathbf{Sig}
\ar[rr]^{\Cn}
\ar[d]_{\mathrm{T}}
\ar[rd]^{\mathrm{Sen}}& {} &
\mathbf{Mnd}_{\boldsymbol{\mathcal{V}},\mathrm{alg}}
\ar[d]^{\pi_{0}} \\
\mathbf{Cat}
\ar[r]_{L} &
\mathbf{Set}_{\boldsymbol{\mathcal{V}}}
\ar[r]_{\mathrm{M}} &
\mathbf{Set}_{\boldsymbol{\mathcal{V}}}
}
$$
where $\mathrm{Sen} = \mathrm{L}\comp\mathrm{T}$ and $\pi_{0}$ the
functor from $\mathbf{Mnd}_{\boldsymbol{\mathcal{V}},\mathrm{alg}}$ to
$\mathbf{Set}_{\boldsymbol{\mathcal{V}}}$ extracting the underlying set of
the set of objects of the first component of its pairs.
\end{definition}

Taking into account that, for every signature $\mathbf{\Sigma}$, we
have that $\mathrm{Sub}(\Eq(\mathbf{\Sigma}))$ and
$\mathrm{Sub}(\coprod\Eq(\mathbf{\Sigma}))$ are isomorphic, we get the
following

\begin{corollary}
The quintuple $\mathfrak{Eqcn} =
(\mathbf{Sig},\mathrm{Ter},\mathrm{L},\mathrm{M},\Cn)$ with
$\mathbf{Sig}$ the category of signatures, $\mathrm{L}$ the functor
from $\mathbf{Cat}$ to $\mathbf{Set}_{\boldsymbol{\mathcal{V}}}$ which
sends a $\boldsymbol{\mathcal{U}}$-category $\mathbf{C}$ to the
$\boldsymbol{\mathcal{V}}$-small set $\coprod_{x,y\in
\mathbf{C}}\mathrm{Hom}(x,y)^{2}$, and $\mathrm{M}$ the covariant
power set endofunctor of $\mathbf{Set}_{\boldsymbol{\mathcal{V}}}$, is
an entailment system, the so-called \emph{many-sorted equational
consequence} entailment system, or \emph{equational consequence}
entailment system.
\end{corollary}

Observe that the many-sorted equational consequence entailment system
$\mathfrak{Eqcn}$ embodies the essentials of the syntactical many-sorted
equational deduction.


After clarifying category-theoretically the concept of equational
consequence, we proceed to define the concept of many-sorted
specification and that of many-sorted specification morphism.

\begin{definition}
A \emph{many-sorted specification} is a pair
$(\mathbf{\Sigma},\ec{E})$, where $\mathbf{\Sigma}$ is a signature
while $\ec{E}\incl\Eq(\mathbf{\Sigma})$.  A \emph{many-sorted
specification morphism} from $(\mathbf{\Sigma},\ec{E})$ to
$(\mathbf{\Lambda},\ec{H})$ is a signature morphism
$\mathbf{d}\colon\mathbf{\Sigma}\mor\mathbf{\Lambda}$ such that
$\mathbf{d}_{\diamond}^{2}[\ec{E}]\incl
\Cn_{\mathbf{\Lambda}}(\ec{H})$.  From now on, to shorten
terminology, we will say \emph{specification} and
\emph{specification morphism} instead of \emph{many-sorted
specification} and \emph{many-sorted specification morphism},
respectively.  Besides, if in a specification
$(\mathbf{\Sigma},\ec{E})$ the set $\ec{E}$ of equations is
closed, i.e., $\Cn_{\mathbf{\Sigma}}(\ec{E}) = \ec{E}$, then we
call $(\mathbf{\Sigma},\ec{E})$ a \emph{theory}.  To abbreviate,
we write, sometimes, $\cl{\ec{E}}$ instead of
$\Cn_{\mathbf{\Sigma}}(\ec{E})$.
\end{definition}

\begin{proposition}
The specifications and the specification morphisms determine a
category denoted as $\mathbf{Spf}$.
\end{proposition}

\begin{proof}
We restrict ourselves to prove that the composition of specification
morphisms is a specification morphism.

Before proving this, let us remark that if
$\mathbf{d}\colon\mathbf{\Sigma}\mor\mathbf{\Lambda}$ and
$\mathbf{e}\colon\mathbf{\Lambda}\mor \mathbf{\Omega}$ are signature morphisms,
$(P,Q)$ a $\mathbf{\Sigma}$-equation of type $(X,Y)$ and $\mathbf{C}$ a
$\mathbf{\Omega}$-algebra, then
$\mathbf{e}_{\diamond}(\mathbf{d}_{\diamond}(P))^{\mathbf{C}} =
\mathbf{e}_{\diamond}(\mathbf{d}_{\diamond}(Q))^{\mathbf{C}}$ iff
$(\mathbf{e}\comp\mathbf{d})_{\diamond}(P)^{\mathbf{C}} =
(\mathbf{e}\comp\mathbf{d})_{\diamond}(Q)^{\mathbf{C}}$.  Therefore, for every
family of $\mathbf{\Sigma}$-equations \ec{E}, we have that
$\Cn_{\mathbf{\Omega}}(\mathbf{e}_{\diamond}^{2}[\mathbf{d}_{\diamond}^{2}[\ec{E}]])=
\Cn_{\mathbf{\Omega}}((\mathbf{e}\comp\mathbf{d})_{\diamond}^{2}[\ec{E}])$.  Now, if
$\mathbf{d}\colon (\mathbf{\Sigma},\ec{E})\mor(\mathbf{\Lambda},\ec{H})$ and
$\mathbf{e}\colon (\mathbf{\Lambda},\ec{H})\mor (\mathbf{\Omega},\ec{F})$ are
specification morphisms, then
$$
\mathbf{e}_{\diamond}^{2}[\mathbf{d}_{\diamond}^{2}[\ec{E}]] \incl
\mathbf{e}_{\diamond}^{2}[\Cn_{\mathbf{\Lambda}}(\ec{H})] \incl
\Cn_{\mathbf{\Omega}}(\mathbf{e}_{\diamond}^{2}[\ec{H}]) \incl
\Cn_{\mathbf{\Omega}}(\ec{F}),
$$
from which the proposition follows.
\end{proof}

\begin{remark}
The category $\mathbf{Th}_{\mathrm{b}}$ with objects the theories and
morphisms from one theory to another the, so-called by Bénabou
in~\cite{jB68}, p. (sub) 27, \emph{banal} morphisms (also known as
\emph{axiom-preserving} morphisms), is
$$
\mathbf{Th}_{\mathrm{b}} = \int^{\mathbf{Sig}}\mathrm{Fix}\comp \mathrm{Cn},
$$
where $\mathrm{Fix}$ is the contravariant functor from
$\mathbf{Mnd}_{\boldsymbol{\mathcal{V}},\mathrm{alg}}$ to
$\mathbf{Cat}_{\boldsymbol{\mathcal{V}}}$ which sends a monad
$(\mathbf{C},\mathbb{T})$ for $\boldsymbol{\mathcal{V}}$, to the
preordered set $\mathbf{Fix}(\mathbb{T}) =
(\mathrm{Fix}(\mathbb{T}),\preccurlyeq)$, of the fixed points of
$\mathbb{T}$, being $\mathrm{Fix}(\mathbb{T})$ the set of all
$\mathbb{T}$-algebras $(A,\delta)$ such that the structural morphism
$\delta$ from $\mathrm{T}(A)$ to $A$ is an isomorphism, and
$\preccurlyeq$ the preorder on $\mathrm{Fix}(\mathbb{T})$ defined by
imposing that $(A,\delta) \preccurlyeq (A',\delta')$ iff there exists
a $\mathbb{T}$-homomorphism from $(A,\delta)$ to $(A',\delta')$.
Therefore, informally speaking, we can say that the world of theories,
$\mathbf{Th}_{\mathrm{b}}$, is the totalization over $\mathbf{Sig}$ of
the fixed points of the consequences.
\end{remark}

We state next some, obvious, relations between the categories
$\mathbf{Sig}$ and $\mathbf{Spf}$ that, notwithstanding, will show to
be useful shortly afterwards.  Every signature $\mathbf{\Sigma}$
determines the specification $(\mathbf{\Sigma},\vacio)$, the so-called
\emph{indiscrete specification}, from which we get an inclusion
functor
$$
\mathrm{sp}_{\mathrm{i}}\colon\mathbf{Sig}\mor\mathbf{Spf}
$$
that, in its turn, is left adjoint to the forgetful functor
$$
\mathrm{sig}\colon\mathbf{Spf}\mor\mathbf{Sig}
$$
which sends an specifi\-cation $(\mathbf{\Sigma},\ec{E})$ to the
underlying signature $\mathbf{\Sigma}$.  Besides, $\mathbf{Sig}$ is a
retract of $\mathbf{Spf}$, i.e.,
$\mathrm{sig}\comp\mathrm{sp}_{\mathrm{i}} =
\mathrm{Id}_{\mathbf{Sig}}$.

The functor $\mathrm{sig}$, on its part, has a right adjoint
$$
\mathrm{sp}_{\mathrm{d}}\colon\mathbf{Sig}\mor\mathbf{Spf}
$$
which sends a signature $\mathbf{\Sigma}$ to
$(\mathbf{\Sigma},\Eq(\mathbf{\Sigma}))$, the so-called \emph{discrete
specification}.

What we want now is to lift the contravariant functor $\Alg$ from
$\mathbf{Sig}$ to $\mathbf{Cat}$ to the category $\mathbf{Spf}$,
by assigning, in particular, to a specification
$(\mathbf{\Sigma},\ec{E})$ the category
$\mathbf{Alg}(\mathbf{\Sigma},\ec{E})$ of its models.

\begin{proposition}
There exists a contravariant functor $\Alg^{\mathrm{sp}}$ from $\mathbf{Spf}$
to $\mathbf{Cat}$ defined as follows
\begin{enumerate}
\item $\Alg^{\mathrm{sp}}$ sends a specification
      $(\mathbf{\Sigma},\ec{E})$ to the category
      $\mathrm{Alg}^{\mathrm{sp}}(\mathbf{\Sigma},\ec{E}) =
      \mathbf{Alg}(\mathbf{\Sigma},\ec{E})$ of its models, i.e., the full
      subcategory of $\mathbf{Alg}(\mathbf{\Sigma})$ determined by those
      $\mathbf{\Sigma}$\nobreakdash-algebras which satisfy all the equations in $\ec{E}$.

\item $\Alg^{\mathrm{sp}}$ sends a specification
      morphism $\mathbf{d}$ from $(\mathbf{\Sigma},\ec{E})$ to
      $(\mathbf{\Lambda},\ec{H})$ to the functor
      $\Alg^{\mathrm{sp}}(\mathbf{d}) = \mathbf{d}^{\ast}$ from
      $\mathbf{Alg}(\mathbf{\Lambda},\ec{H})$ to
      $\mathbf{Alg}(\mathbf{\Sigma},\ec{E})$, obtained from the functor
      $\mathbf{d}^{\ast}$ from $\mathbf{Alg}(\mathbf{\Lambda})$ to
      $\mathbf{Alg}(\mathbf{\Sigma})$ by bi-restriction.
\end{enumerate}

\end{proposition}

\begin{proof}
Let $\mathbf{B}$ be a $\mathbf{\Lambda}$-algebra such that
$\mathbf{B}\models^{\mathbf{\Lambda}} \ec{H}$.  Then
$\mathbf{B}\models^{\mathbf{\Lambda}} \Cn_{\mathbf{\Lambda}}(\ec{H})$,
therefore $\mathbf{B}\models^{\mathbf{\Lambda}}
\mathbf{d}_{\diamond}^{2}[\ec{E}]$ hence, by
Lemma~\ref{lemaSatisfaccion},
$\mathbf{d}^{\ast}(\mathbf{B})\models^{\mathbf{\Sigma}} \ec{E}$.
\end{proof}

\begin{remark}
By applying the construction of Ehresmann-Grothendieck to the
contravariant functor $\Alg^{\mathrm{sp}}$ from $\mathbf{Spf}$ to
$\mathbf{Cat}$ we get the category $\mathbf{Alg}^{\mathrm{sp}} =
\int^{\mathbf{Spf}}\Alg^{\mathrm{sp}}$ into which is embedded the
category $\mathbf{Alg}$ as a retract (because $\mathbf{Sig}$ is a
retract of $\mathbf{Spf}$).
\end{remark}

On the other hand, taking care of the Completeness Theorem
in~\cite{cs05}, every family of equations
$\ec{E}\subseteq\Eq(\mathbf{\Sigma})$ determines a congruence on
the category $\mathbf{Ter}(\mathbf{\Sigma})$, hence a quotient
category $\mathbf{Ter}(\mathbf{\Sigma})/\cl{\ec{E}}$.  Besides,
this procedure can be completed, as stated in the following
proposition, up to a pseudo-functor $\mathrm{Ter}^{\mathrm{sp}}$
from $\mathbf{Spf}$ to $\mathbf{Cat}$, and the restriction of
$\mathrm{Ter}^{\mathrm{sp}}$ to $\mathbf{Sig}$ is precisely the
pseudo-functor $\mathrm{Ter}$.

\begin{proposition}
There exists a pseudo-functor $\mathrm{Ter}^{\mathrm{sp}}$ from
$\mathbf{Spf}$ to $\mathbf{Cat}$ defined as follows
\begin{enumerate}
\item $\mathrm{Ter}^{\mathrm{sp}}$ sends a
      specification $(\mathbf{\Sigma},\ec{E})$ to the
      category $\Ter^{\mathrm{sp}}(\mathbf{\Sigma},\ec{E}) =
      \mathbf{Ter}(\mathbf{\Sigma},\ec{E})$, where
      $\mathbf{Ter}(\mathbf{\Sigma},\ec{E})$
      is the quotient category
      $\mathbf{Ter}(\mathbf{\Sigma})/\cl{\ec{E}}$.

\item $\mathrm{Ter}^{\mathrm{sp}}$ sends a specification
      morphism $\mathbf{d}$ from $(\mathbf{\Sigma},\ec{E})$ to
      $(\mathbf{\Lambda},\ec{H})$ to the functor
      $\Ter^{\mathrm{sp}}(\mathbf{d})$, also occasionally denoted by
      $\mathbf{d}_{\diamond}$, from the quotient category
      $\mathbf{Ter}(\mathbf{\Sigma},\ec{E}) =
      \mathbf{Ter}(\mathbf{\Sigma})/\cl{\ec{E}}$ to the quotient
      category $\mathbf{Ter}(\mathbf{\Lambda},\ec{H}) =
      \mathbf{Ter}(\mathbf{\Lambda})/\cl{\ec{H}}$, which assigns to a
      morphism $[P]_{\cl{\ec{E}}}$ from $X$ to $Y$ in
      $\mathbf{Ter}(\mathbf{\Sigma},\ec{E})$ the morphism
      $$
      \Ter^{\mathrm{sp}}(\mathbf{d})([P]_{\cl{\ec{E}}}) =
      [\mathbf{d}_{\diamond}(P)]_{\cl{\ec{H}}}\colon \tcoprod_{\varphi}X\mor
      \tcoprod_{\varphi}Y
      $$
      in $\mathbf{Ter}(\mathbf{\Lambda},\ec{H})$.

\end{enumerate}

\end{proposition}

\begin{proof}
Everything follows, essentially, from the fact that the action of
$\mathrm{Ter}^{\mathrm{sp}}(\mathbf{d})$ on $[P]_{\cl{\ec{E}}}$ is
well defined because $\mathcal{E}\subseteq
\mathrm{Ker}(\mathrm{Pr}_{\cl{\ec{H}}}\comp \mathbf{d}_{\diamond})$,
where $\mathrm{Pr}_{\cl{\ec{H}}}$ is the projection functor from
$\mathbf{Ter}(\mathbf{\Lambda})$ to the quotient category
$\mathbf{Ter}(\mathbf{\Lambda})/\cl{\ec{H}}$.
\end{proof}

After this we prove that the family of functors $\mathrm{Tr} =
(\mathrm{Tr}^{\mathbf{\Sigma}})_{\mathbf{\Sigma}\in \mathbf{Sig}}$,
defined in Proposition~\ref{funcTr}, can be lifted to the family
of functors $\mathrm{Tr}^{\mathrm{sp}} =
(\mathrm{Tr}^{\mathrm{sp},(\mathbf{\Sigma},\mathcal{E})})
_{(\mathbf{\Sigma},\mathcal{E})\in \mathbf{Spf}}$.

\begin{proposition}
Let $(\mathbf{\Sigma},\mathcal{E})$ be a specification.  Then from
the product category
$\mathbf{Alg}(\mathbf{\Sigma},\mathcal{E})\bprod
\mathbf{Ter}(\mathbf{\Sigma},\mathcal{E})$ to the category
$\mathbf{Set}$ there exists a functor
$\mathrm{Tr}^{\mathrm{sp},(\mathbf{\Sigma},\mathcal{E})}$ defined
as follows
\begin{enumerate}
\item $\mathrm{Tr}^{\mathrm{sp},(\mathbf{\Sigma},\mathcal{E})}$
      sends a pair $(\mathbf{A},X)$, formed by a
      $\mathbf{\Sigma}$-algebra $\mathbf{A}$ which satisfies
      $\mathcal{E}$ and an $S$-sorted set $X$, to the set
      $\mathrm{Tr}^{\mathrm{sp},(\mathbf{\Sigma},\mathcal{E})}(\mathbf{A},X)
      = A_{X}$ of the $S$-sorted mappings from $X$ to the underlying
      $S$-sorted set $A$ of $\mathbf{A}$,

\item $\mathrm{Tr}^{\mathrm{sp},(\mathbf{\Sigma},\mathcal{E})}$
      sends an arrow $(f,[P]_{\cl{\ec{E}}})$ from $(\mathbf{A},X)$ to
      $(\mathbf{B},Y)$ to the mapping
      $\mathrm{Tr}^{\mathrm{sp},(\mathbf{\Sigma},\mathcal{E})}(f,[P]_{\cl{\ec{E}}})
      = f_{P}$ from $A_{X}$ to $B_{Y}$.
       \end{enumerate}
\end{proposition}

\begin{proof}
Everything follows from the fact that the action of
$\mathrm{Tr}^{\mathrm{sp},(\mathbf{\Sigma},\mathcal{E})}$ on
$(f,[P]_{\cl{\ec{E}}})$ is well defined because from
$[P]_{\cl{\ec{E}}} = [Q]_{\cl{\ec{E}}}$ it follows that, for every
$\mathbf{\Sigma}$-algebra $\mathbf{C}$ which satisfies $\mathcal{E}$,
$P^{\mathbf{C}} = Q^{\mathbf{C}}$.
\end{proof}

Next we state that the family of natural isomorphisms $\theta =
(\theta^{\mathbf{d}})_{\mathbf{d}\in \mathrm{Mor}(\mathbf{Sig})}$,
defined in Proposition~\ref{lemaPdextranatural}, can
be lifted to the family of natural isomorphisms $\theta^{\mathrm{sp}}
= (\theta^{\mathrm{sp},\mathbf{d}})_{\mathbf{d}\in
\mathrm{Mor}(\mathbf{Spf})}$.

\begin{proposition}
Let $\mathbf{d}$ be a specification morphism from
$(\mathbf{\Sigma},\ec{E})$ to $(\mathbf{\Lambda},\ec{H})$.  Then
there exists a natural isomorphism
$\theta^{\mathrm{sp},\mathbf{d}} =
(\theta^{\mathrm{sp},\mathbf{d}}_{\mathbf{A},X})_{(\mathbf{A},X)\in
\mathbf{Alg}(\mathbf{\Lambda},\ec{H})\bprod
\mathbf{Ter}(\mathbf{\Sigma},\ec{E})}$ as shown in the following
diagram
$$
\xymatrix@C=20ex@R=10ex{
{\mathbf{Alg}(\mathbf{\Lambda},\ec{H})\bprod
\mathbf{Ter}(\mathbf{\Sigma},\ec{E})}
\ar[r]^-{\mathbf{d}^{\ast}\bprod\Id}
\ar[d]_{\Id\bprod \mathbf{d}_{\diamond}} &
{\mathbf{Alg}(\mathbf{\Sigma},\ec{E})\bprod
\mathbf{Ter}(\mathbf{\Sigma},\ec{E})} \xyn{d}
\ar[d]^{\mathrm{Tr}^{\mathrm{sp},(\mathbf{\Sigma},\ec{E})}} \\
{\mathbf{Alg}(\mathbf{\Lambda},\ec{H})\bprod
\mathbf{Ter}(\mathbf{\Lambda},\ec{H})} \xyn{c}
\ar[r]_-{\mathrm{Tr}^{\mathrm{sp},(\mathbf{\Lambda},\ec{H})}} &
\mathbf{Set}
\ar@{} "c";"d"|{\dir{=>}}^{\theta^{\mathrm{sp},\mathbf{d}}}
}
$$
where, for every $(\mathbf{A},X)\in
\mathbf{Alg}(\mathbf{\Lambda},\ec{H})\bprod
\mathbf{Ter}(\mathbf{\Sigma},\ec{E})$,
$\theta^{\mathrm{sp},\mathbf{d}}_{\mathbf{A},X}$ is
$\theta^{\varphi}_{X,A}$, i.e., the value at $(X,A)$ of the natural
isomorphism of the adjunction
$\coprod_{\varphi}\ladj\Delta_{\varphi}$.
\end{proposition}

From these two last propositions it follows immediately the following

\begin{corollary}
The quadruple $\mathfrak{Spf} =
(\mathbf{Spf},\Alg^{\mathrm{sp}},\mathrm{Ter}^{\mathrm{sp}},
(\mathrm{Tr}^{\mathrm{sp}},\theta^{\mathrm{sp}}))$ is an institution
on the category $\mathbf{Set}$, the so-called \emph{ many-sorted specification
institution}, or, to abbreviate, the \emph{specification
institution}.
\end{corollary}

From the contravariant functor $\Alg^{\mathrm{sp}}$, from
$\mathbf{Spf}$ to $\mathbf{Cat}$, to the contravariant functor
$\Alg\comp \mathrm{sig}^{\mathrm{op}}$, between the same
categories, there exists a natural transformation, $\mathrm{In}$,
which sends a specification $(\mathbf{\Sigma},\ec{E})$ to the full
embedding $\mathrm{In}_{(\mathbf{\Sigma},\ec{E})}$ of
$\mathbf{Alg}(\mathbf{\Sigma},\ec{E})$ into
$\mathbf{Alg}(\mathbf{\Sigma})$. Besides, from the pseudo-functor
$\mathrm{Ter}\comp \mathrm{sig}$, from $\mathbf{Spf}$ to
$\mathbf{Cat}$, to the pseudo-functor
$\mathrm{Ter}^{\mathrm{sp}}$, between the same categories, there
exists a (strict) pseudo-natural transformation, $\mathrm{Pr}$,
given by the following data
\begin{enumerate}
\item For each specification
      $(\mathbf{\Sigma},\ec{E})$, the projection functor
      $\mathrm{Pr}_{\cl{\ec{E}}}$ from
      $\mathbf{Ter}(\mathbf{\Sigma})$ to the quotient category
      $\mathbf{Ter}(\mathbf{\Sigma})/\cl{\ec{E}}$.

\item For each specification morphism $\mathbf{d}$ from
      $(\mathbf{\Sigma},\ec{E})$ to $(\mathbf{\Lambda},\ec{H})$, the
      identity natural transformation, denoted in this case by
      $\mathrm{Pr}_{\mathbf{d}}$, from the functor
      $\mathrm{Pr}_{\cl{\ec{H}}}\comp (\mathrm{Ter}\comp
      \mathrm{sig})(\mathbf{d})$ to the functor
      $\mathrm{Ter}^{\mathrm{sp}}(\mathbf{d})\comp
      \mathrm{Pr}_{\cl{\ec{E}}}$, both from
      $\mathbf{Ter}(\mathbf{\Sigma})$ to
      $\mathbf{Ter}(\mathbf{\Lambda})/\cl{\ec{H}}$.
\end{enumerate}

Therefore we have obtained the following

\begin{corollary}
The pair $(\mathrm{sig},(\mathrm{In},\mathrm{Pr}))$ is a morphism of
institutions from the many-sorted specification institution
$\mathfrak{Spf}$ to the many-sorted term institution $\mathfrak{Tm}$.
\end{corollary}

\begin{remark}
Since, obviously, the many-sorted term institution $\mathfrak{Tm}$ is
embedded in the many-sorted specification institution
$\mathfrak{Spf}$, taking into account the just stated corollary, we
can assert that $\mathfrak{Tm}$ is a retract of $\mathfrak{Spf}$.
\end{remark}


\section{Hall and Bénabou algebras.}

The concept of many-sorted clone, that generalizes both that of
single-sorted clone axiomatized by P. Hall as a single-sorted
partial algebra subject to satisfy some laws (see e.g.,
\cite{pC81}, pp.  127 and 132, or \cite{mmt87}, pp.  136 and 143)
and by M. Lazard as a compositor (see~\cite{mL55}, p.  327), and
that of Boolean clone investigated, among others, by E. Post (see
e.g., \cite{eP21} and \cite{eP41}), was axiomatically defined by
Goguen and Meseguer (in~\cite{gm85}, pp.  318--319) as any
many-sorted algebra (of the appropriate signature) that satisfies
a system of many-sorted equational laws.  The corresponding
categories of many-sorted algebras, called categories of Hall
algebras, are the algebraic rendering of the categories of
finitary many-sorted algebraic theories of Bénabou, i.e., both
types of categories, as is well-known, are equivalent.

Our main aim in this section is to define, for each set of sorts,
through a system of many-sorted equational laws the, so-called,
Bénabou algebras as those many-sorted algebras that satisfy them, and
to prove that the corresponding category of Bénabou algebras, for a
given set of sorts, is \emph{isomorphic} to the category of finitary
many-sorted algebraic theories of Bénabou, for the same set of sorts.
Besides, we prove, directly, that the Hall and Bénabou algebras,
even having different specifi\-cation, are models of the essential
properties of the clones for the many-sorted operations, i.e., that
the respective categories of Hall and Bénabou algebras are
\emph{equivalent}.

The homomorphisms between Bénabou algebras, as we will show later on
(in the fifth section), are also adequate to define the composition of
the morphisms of Fujiwara between signatures, that are a strict
generalization of both the standard morphisms and the derivors
(defined in the fifth section) between signatures.  Informally, we can
say that the Bénabou algebras are to the composition of morphisms of
Fujiwara between signatures as the Hall algebras are to the
composition of derivors between signatures.


Before we define the Hall algebras as the models of a
specification, we agree that for a set of sorts $U$, a word
$x\in\fmon{U}$ and a standard $U$-sorted set of variables $V^{U} =
(\{\,v^{u}_{n}\mid n\in \mathbb{N}\,\})_{u\in U}$, $\vs{x}$ is the
$U$-sorted subset of $V^{U}$ defined, for every $u\in U$ as
$(\vs{x})_{u} = \{\,v^{u}_{i}\mid i\in x^{-1}[u]\,\}$, this will
apply, in particular, when $U = \fmon{S}\bprod S$ or $U =
\fmon{S}\bprod \fmon{S}$.

\begin{definition}
Let $S$ be a set of sorts and $V^{\Hall_{S}}$ the $\fmon{S}\times
S$-sorted set of variables $(V_{w,s})_{(w,s)\in\fmon{S}\times S}$
where, for every $(w,s)\in\fmon{S}\times S$, $V_{w,s} =
\{\,v^{w,s}_{n}\mid n\in \mathbb{N}\,\}$.  A \emph{Hall algebra for}
$S$ is a $\Hall_{S} = (\fmon{S}\bprod
S,\Sigma^{\Hall_{S}},\ec{E}^{\Hall_{S}})$-algebra, where
$\Sigma^{\Hall_{S}}$ is the $\fmon{S}\bprod S$-sorted signature, i.e.,
the $\fmon{(\fmon{S}\bprod S)}\bprod (\fmon{S}\bprod S)$-sorted set,
defined as follows:
\begin{enumerate}
\item[$\mathrm{HS}_{1}$.] For every $w\in \fmon{S}$ and $i\in\bb{w}$,
      $$
      \pi^{w}_{i}\colon\lambda\mor(w,w_{i}),
      $$
      where $\bb{w}$ is the \emph{length} of the word $w$ and $\lambda$ the
      \emph{empty word} in $\fmon{(\fmon{S}\bprod S)}$.

\item[$\mathrm{HS}_{2}$.] For every $u$, $w\in\fmon{S}$ and $s\in S$,
      $$
      \xi_{u,w,s}\colon ((w,s),(u,w_{0}),\ldots,(u,w_{\bb{w}-1}))\mor(u,s);
      $$
\end{enumerate}
while $\ec{E}^{\Hall_{S}}$ is the $\mathrm{ms}$-subset of
$\Eq(\Sigma^{\Hall_{S}}) =
(\mathrm{T}_{\Sigma^{\Hall_{S}}}(\vs{\ol{w}})_{(u,s)}^{2})_{(\ol{w},(u,s))\in
\fmon{(\fmon{S}\bprod S)}\bprod (\fmon{S}\bprod S)}$ defined as follows:
\begin{enumerate}
\item[$\mathrm{H}_{1}$.] \emph{Projection}.
      For every $u$, $w\in\fmon{S}$ and $i\in\bb{w}$, the equation

      $$
      \xi_{u,w,w_{i}}(\pi^{w}_{i},v^{u,w_{0}}_{0},\ldots,
           v^{u,w_{\bb{w}-1}}_{\bb{w}-1})=v^{u,w_{i}}_{i}
      $$
      of type
      $(((u,w_{0}),\ldots,(u,w_{\bb{w}-1})),(u,w_{i})).$
\item[$\mathrm{H}_{2}$.] \emph{Identity}.
      For every $u\in\fmon{S}$ and $j\in \bb{u}$, the equation
      $$
      \xi_{u,u,u_{j}}(v^{u,u_{j}}_{j},\pi^{u}_{0},\ldots,\pi^{u}_{\bb{u}-1})=
           v^{u,u_{j}}_{j}
      $$
      of type
      $(((u,u_{j})),(u,u_{j})).$
\item[$\mathrm{H}_{3}$.] \emph{Associativity}.
      For every $u$, $v$, $w\in\fmon{S}$ and $s\in S$, the
      equation
     \begin{align*}
      \xi_{u,v,s}(
      \xi_{v,w,s}(v^{w,s}_{0},v^{v,w_{0}}_{1},\ldots,
          v^{v,w_{\bb{w}-1}}_{\bb{w}}),
          v^{u,v_{0}}_{\bb{w}+1},
          \ldots,v^{u,v_{\bb{v}-1}}_{\bb{w}+\bb{v}}) = \\
      \begin{aligned}
       \xi_{u,w,s}(v^{w,s}_{0},
          &\xi_{u,v,w_{0}}(v^{v,w_{0}}_{1},v^{u,v_{0}}_{\bb{w}+1},
                   \ldots,v^{u,v_{\bb{v}-1}}_{\bb{w}+\bb{v}}),
          \ldots, \\
       &\xi_{u,v,w_{\bb{w}-1}}(v^{v,w_{\bb{w}-1}}_{\bb{w}},
                   v^{u,v_{0}}_{\bb{w}+1},
       \ldots,v^{u,v_{\bb{v}-1}}_{\bb{w}+\bb{v}}))
      \end{aligned}
     \end{align*}
     of type
     $(((w,s),(v,w_{0}),\ldots,(v,w_{\bb{w}-1}),
            (u,v_{0}),\ldots,(u,v_{\bb{v}-1})),(u,s)).$
\end{enumerate}

\begin{remark}
From $\mathrm{H}_{3}$, for $w=\lambda$, we get the invariance of
constant functions axiom in~\cite{gm85}: For every $u$, $v\in\fmon{S}$
and $s\in S$, we have the equation
$$
\xi_{u,v,s}(\xi_{v,\lambda,s}(v^{\lambda,s}_{0}),v^{u,v_{0}}_{1},
\ldots,v^{u,v_{\bb{v}-1}}_{\bb{v}})=
\xi_{u,\lambda,s}(v^{\lambda,s}_{0})
$$
of type $(((\lambda,s),(u,v_{0}),\ldots,(u,v_{\bb{v}-1})),(u,s))$.
\end{remark}

We call the formal constants $\pi^{w}_{i}$ \emph{projections}, and the
formal operations $\xi_{u,w,s}$ \emph{substitution operators}.
Furthermore, we denote by $\mathbf{Alg}(\mathrm{H}_{S})$ the category of
Hall algebras for $S$ and homomorphisms between Hall algebras.  Since
$\mathbf{Alg}(\mathrm{H}_{S})$ is a variety, the forgetful functor
$\mathrm{G}_{\Hall_{S}}$ from $\mathbf{Alg}(\mathrm{H}_{S})$ to
$\mathbf{Set}^{\fmon{S}\bprod S}$ has a left adjoint
$\mathbf{T}_{\Hall_{S}}$
$$
\xymatrix@=5pc{
\mathbf{Alg}(\mathrm{H}_{S}) \ar@<1.5ex>[r]^{\mathrm{G}_{\Hall_{S}}}
\ar@{}[r]|{\uadj} & \mathbf{Set}^{\fmon{S}\bprod S}
\ar@<1.5ex>[l]^{\mathbf{T}_{\Hall_{S}}}
}
$$
which assigns to an $\fmon{S}\bprod S$-sorted set $\Sigma$ the
corresponding free Hall algebra $\mathbf{T}_{\Hall_{S}}(\Sigma)$.
\end{definition}

For every $S$-sorted set $A$, $\mathrm{HOp}_{S}(A) =
(\mathrm{Hom}(A_{w},A_{s}))_{(w,s)\in\fmon{S}\bprod S}$, the
$\fmon{S}\times S$-sorted set of operation for $A$, is naturally
endowed with a structure of Hall algebra, as stated in the following
proposition, if we realize the projections as the true projections and
the substitution operators as the generalized composition of mappings.

\begin{proposition}
Let $A$ be an $S$-sorted set and $\mathbf{HOp}_{S}(A)$ the
$\Sigma^{\Hall_{S}}$-algebra with underlying $\mathrm{ms}$-set
$\mathrm{HOp}_{S}(A)$ and algebraic structure defined as follows
\begin{enumerate}
\item For every $w\in\fmon{S}$ and $i\in\bb{w}$,
      $(\pi^{w}_{i})^{\mathbf{HOp}_{S}(A)} =
      \pr^{A}_{w,i}\colon A_{w}\mor A_{w_{i}}$.

\item For every $u,w\in\fmon{S}$ and $s\in S$,
      $\xi_{u,w,s}^{\mathbf{HOp}_{S}(A)}$ is defined, for every $f\in
      A _{s}^{ A_{w} }$ and $g\in A^{A_{u}}_{w}$, as
      $\xi_{u,w,s}^{\mathbf{HOp}_{S}(A)}(f,g_{0},\ldots,g_{\bb{w}-1})
      = f\comp\tp{ g_{i}}_{i\in\bb{w} }$, where $\tp{
      g_{i}}_{i\in\bb{w}}$ is the unique mapping from $A_{u}$ to
      $A_{w}$ such that, for every $i\in \bb{w}$, $\pr^{A}_{w,i}\comp
      \tp{ g_{i}}_{i\in\bb{w}} = g_{i}$.
\end{enumerate}
Then $\mathbf{HOp}_{S}(A)$ is a Hall algebra,  the \emph{Hall
algebra for} $(S,A)$.
\end{proposition}

\begin{remark}
The closed sets of the Hall algebra $\mathbf{HOp}_{S}(A)$ for $(S,A)$
are precisely the clones of (many-sorted) operations on the $S$-sorted
set $A$.
\end{remark}

For every $S$-sorted signature $\Sigma$, $\mathrm{HTer}_{S}(\Sigma) =
(\mathrm{T}_{\Sigma}(\vs{w})_{s})_{(w,s)\in \fmon{S}\bprod S} $ is also
endowed with a structure of Hall algebra that formalizes the concept
of substitution as stated in the following

\begin{proposition}
Let $\Sigma$ be an $S$-sorted signature and
$\mathbf{HTer}_{S}(\Sigma)$ the $\Sigma^{\Hall_{S}}$-algebra with
underlying $\mathrm{ms}$-set $\mathrm{HTer}_{S}(\Sigma)$ and algebraic
structure defined as follows
\begin{enumerate}
\item For every $w\in\fmon{S}$ and $i\in\bb{w}$,
      $(\pi^{w}_{i})^{\mathbf{HTer}_{S}(\Sigma)}$
      is the image under $\eta_{\vs{w},w_{i}}$ of
      the variable $v_{i}^{w_{i}}$, where $\eta_{\vs{w}} =
      (\eta_{\vs{w},s})_{s\in S}$ is the canonical embedding
      of $\vs{w}$ into $\mathrm{T}_{\Sigma}(\vs{w})$.

\item For every $u,w\in\fmon{S}$ and $s\in S$,
      $\xi_{u,w,s}^{\mathbf{HTer}_{S}(\Sigma)}$
      is the mapping
      $$\xi_{u,w,s}^{\mathbf{HTer}_{S}(\Sigma)}\nfunction
      {\mathrm{T}_{\Sigma}(\vs{w})_{s} \bprod
      \mathrm{T}_{\Sigma}(\vs{u})_{w_{0}} \bprod \cdots \bprod
      \mathrm{T}_{\Sigma}(\vs{u})_{w_{\bb{w}-1}}}
      {\mathrm{T}_{\Sigma}(\vs{u})_{s}}
      {(P,(Q_{i}\mid i\in\bb{w}))}
      {\mathcal{Q}^{\sharp}_{s}(P)}
     $$
     where, for $\mathcal{Q}$ the $S$-sorted mapping from $\vs{w}$
     to $\mathrm{T}_{\Sigma}(\vs{u})$ canonically associated to the family
     $(Q_{i}\mid i\in\bb{w})$, $\mathcal{Q}^{\sharp}$ is the
     unique homomorphism from $\mathbf{T}_{\Sigma}(\vs{w})$ into
     $\mathbf{T}_{\Sigma}(\vs{u})$ such that
     $\mathcal{Q}^{\sharp}\comp \eta_{\vs{w}} = \mathcal{Q}$.
\end{enumerate}

Then $\mathbf{HTer}_{S}(\Sigma)$ is a Hall algebra, the \emph{Hall
algebra for} $(S,\Sigma)$.
\end{proposition}

Our next goal is to prove that, for every $\fmon{S}\bprod S$-sorted
set $\Sigma$, $\mathbf{T}_{\Hall_{S}}(\Sigma)$, the free Hall algebra
on $\Sigma$, is isomorphic to $\mathbf{HTer}_{S}(\Sigma)$.  We remark
that the existence of this isomorphism is interesting because it
enables us, on the one hand, to get a more tractable description of
the terms in $\mathbf{T}_{\Hall_{S}}(\Sigma)$, and, on the other hand,
to give, in the fifth section, an alternative, but equivalent,
definition of the concept of derivor (defined by Goguen, Thatcher
and Wagner in~\cite{gtw76}, p.  86) between signatures.

To attain the goal just stated we define, for a Hall algebra
$\mathbf{A}$, an $S$-sorted signature $\Sigma$, an $\fmon{S}\bprod
S$-mapping $f\colon\Sigma\mor A$, and a word $u\in\fmon{S}$, the
concept of derived $\Sigma$-algebra of $\mathbf{A}$ for $(f,u)$, since
it will be used afterwards in the proof of the isomorphism between
$\mathbf{T}_{\Hall_{S}}(\Sigma)$ and $\mathbf{HTer}_{S}(\Sigma)$.

\begin{definition}
Let $\mathbf{A}$ be a Hall algebra and $\Sigma$ an $S$-sorted signature.
Then, for every $f\colon\Sigma\mor A$ and $u\in\fmon{S}$,
$\mathbf{A}^{f,u}$, the \emph{derived} $\Sigma$-\emph{algebra}
\emph{of} $\mathbf{A}$ \emph{for} $(f,u)$, is the $\Sigma$-algebra with
underlying $S$-sorted set $A^{f,u} = (A_{u,s})_{s\in S}$ and
algebraic structure $F^{f,u}$, defined, for every
$(w,s)\in\fmon{S}\bprod S$, as
$$
F^{f,u}_{w,s} \nfunction
{\Sigma_{w,s}}{\mathrm{HOp}_{w}(A^{f,u})_{s}}
{\sigma}{\nfunction
    {\prod_{i\in \bb{w}}A_{u,w_{i}}}
    {A_{u,s}}
    {(a_{0},\ldots,a_{\bb{w}-1})}
    {\xi_{u,w,s}^{\mathbf{A}}(f_{(w,s)}(\sigma),a_{0},\ldots,a_{\bb{w}-1})}
}
$$
where $\mathrm{HOp}_{w}(A^{f,u})_{s} =
A_{u,s}^{\prod_{i\in \bb{w}}A_{u,w_{i}}}$.  Furthermore, we denote
by $p^{u}$ the $S$-sorted mapping from $\vs{u}$ to $A^{f,u}$ defined, for
every $s\in S$ and $i\in \bb{u}$, as
$p^{u}_{s}(v^{s}_{i}) = (\pi^{u}_{i})^{\mathbf{A}}$, and by $\ext{(p^{u})}$
the unique homomorphism from $\mathbf{T}_{\Sigma}(\vs{u})$ to
$\mathbf{A}^{f,u}$ such that $\ext{(p^{u})}\comp \eta_{\vs{u}} = p^{u}$.
\end{definition}


\begin{remark}
For a $\Sigma$-algebra $\mathbf{B}=(B,G)$, we have that
$G\colon\Sigma\mor\mathrm{HOp}_{S}(B)$ and
$\mathbf{B}\iso\mathbf{HOp}_{S}(B)^{G,\lambda}$, where $\lambda$ is
the empty word on $S$.  Besides, for every $u\in\fmon{S}$, we have
that $\mathbf{Op}_{\vs{u}}(\mathbf{B}) \iso
\mathbf{HOp}_{S}(B)^{G,u}$.
\end{remark}

\begin{lemma}\label{L:aux}
Let $\Sigma$ be an $S$-sorted signature, $\mathbf{A}$ a Hall algebra,
$f\colon\Sigma\mor A$ and $u\in\fmon{S}$.  Then, for every
$(w,s)\in \fmon{S}\times S$, $P\in \mathrm{T}_{\Sigma}(\vs{w})_{s}$ and $a\in
\prod_{i\in \bb{w}}A_{u,w_{i}}$, we have that
$$
P^{\mathbf{A}^{f,u}}(a_{0},\ldots,a_{\bb{w}-1}) =
\xi_{u,w,s}^{\mathbf{A}}(\ext{(p^{w})}_{s}(P),a_{0},\ldots,a_{\bb{w}-1}).
$$
\end{lemma}

\begin{proof}
By algebraic induction on the complexity of $P$.  If $P$ is a variable
$v_{i}^{s}$, with $i\in\bb{w}$, then
\begin{align*}
    v_{i}^{s,\mathbf{A}^{f,u}}(a_{0},\ldots,a_{\bb{w}-1})\ &=
    \ext{a}_{w_{i}}(v^{s}_{i})\\
    &=
    a_{i} \\
    &=
    \xi_{u,w,s}^{\mathbf{A}}((\pi^{w}_{i})^{\mathbf{A}},a_{0},\ldots,a_{\bb{w}-1} )
    \quad\text{($\mathrm{H}_{1}$)} \\
    &=
    \xi_{u,w,s}^{\mathbf{A}}(\ext{(p^{w})}_{s}(v^{s}_{i}),a_{0},\ldots,a_{\bb{w}-1} ).
\end{align*}
Let us assume that $P = \sigma(Q_{0},\ldots,Q_{\bb{x}-1})$, with
$\sigma\colon x\mor s$ and that, for every $j\in\bb{x}$,
$Q_{j}\in \mathrm{T}_{\Sigma}(\vs{w})_{x_{j}}$ fulfills the induction
hypothesis.  Then we have that
\begin{align*}
    &(\sigma(Q_{0},\ldots,Q_{\bb{x}-1}))^{\mathbf{A}^{f,u}}
      (a_{0},\ldots,a_{\bb{w}-1}) \\
    &=
    \sigma^{\mathbf{A}^{f,u}}
       (
       Q_{0}^{\mathbf{A}^{f,u}}(a_{0},\ldots,a_{\bb{w}-1}),
       \ldots,
       Q_{\bb{x}-1}^{\mathbf{A}^{f,u}}(a_{0},\ldots,a_{\bb{w}-1})
       )\\
    &=
    \xi_{u,x,s}^{\mathbf{A}}
       (
       f(\sigma),
       Q_{0}^{\mathbf{A}^{f,u}}(a_{0},\ldots,a_{\bb{w}-1}),
       \ldots,
       Q_{\bb{x}-1}^{\mathbf{A}^{f,u}}(a_{0},\ldots,a_{\bb{w}-1})
       )\\
   &=
   \begin{aligned}[t]
    \xi_{u,x,s}^{\mathbf{A}}
       (
       f(\sigma),
      &\xi_{u,w,x_{0}}^{\mathbf{A}}
            (\ext{(p^{w})}_{x_{0}}(Q_{0}),a_{0},\ldots,a_{\bb{w}-1}),
       \ldots, \\
      &\xi_{u,w,x_{\bb{x}-1}}^{\mathbf{A}}
            (\ext{(p^{w})}_{x_{\bb{x}-1}}(Q_{\bb{x}-1}),a_{0},\ldots,a_{\bb{w}-1})
       ) \quad\text{(Ind. hyp.)}
  \end{aligned}  \\
    &=
    \xi_{u,w,s}^{\mathbf{A}}
       (
         \xi_{w,x,s}^{\mathbf{A}}
           (f(\sigma),
            \ext{(p^{w})}_{x_{0}}(Q_{0}),
            \ldots,
            \ext{(p^{w})}_{x_{\bb{x}-1}}(Q_{\bb{x}-1})
           ),
         a_{0},
         \ldots,
         a_{\bb{w}-1}
       ) \text{($\mathrm{H}_{3}$)}   \\
    &=
    \xi_{u,w,s}^{\mathbf{A}}
       (
       \sigma^{\mathbf{A}_{w}}
          (
          \ext{(p^{w})}_{x_{0}}(Q_{0}),
          \ldots,
          \ext{(p^{w})}_{x_{\bb{x}-1}}(Q_{\bb{x}-1})
          ),
       a_{0},
       \ldots,
       a_{\bb{w}-1}
       ) \\
    &=
    \xi_{u,w,s}^{\mathbf{A}}
       (
       \ext{(p^{w})}_{s}(\sigma,Q_{0},\ldots,Q_{\bb{x}-1}),
       a_{0},
       \ldots,
       a_{\bb{w}-1}
       ) \\
    &=
    \xi_{u,w,s}^{\mathbf{A}}
       (
       \ext{(p^{w})}_{s}(P),
       a_{0},
       \ldots,
       a_{\bb{w}-1}
       ).
       \qedhere
\end{align*}
\end{proof}

Next we prove, as announced above, that, for every $\fmon{S}\bprod S$-sorted
set $\Sigma$, the Hall algebra for $(S,\Sigma)$ is isomorphic to the
free Hall algebra on $\Sigma$.

\begin{proposition}\label{iso:FrH-TerH}
Let $\Sigma$ be an $S$-sorted signature, i.e., an $\fmon{S}\bprod
S$-sorted set.  Then the Hall algebra $\mathbf{HTer}_{S}(\Sigma)$ is
isomorphic to $\mathbf{T}_{\Hall_{S}}(\Sigma)$.
\end{proposition}

\begin{proof}
It is enough to prove that $\mathbf{HTer}_{S}(\Sigma)$ has the
universal property of the free Hall algebra on $\Sigma$.  Therefore we
have to specify an $\fmon{S}\bprod S$-sorted mapping $h$ from $\Sigma$
to $\mathrm{HTer}_{S}(\Sigma)$ such that, for every Hall algebra
$\mathbf{A}$ and $\fmon{S}\bprod S$-sorted mapping $f$ from $\Sigma$
to $A$, there is a unique homomorphism $\mawh{f}$ from
$\mathbf{HTer}_{S}(\Sigma)$ to $\mathbf{A}$ such that $\mawh{f}\comp h =
f$.  Let $h$ be the $\fmon{S}\bprod S$-sorted mapping defined, for
every
$(w,s)\in\fmon{S}\bprod S$, as%
$$
h_{w,s}\nfunction
{\Sigma_{w,s}}{\mathrm{T}_{\Sigma}(\vs{w})_{s}}
{\sigma}{\sigma(v^{s}_{0},\ldots,v^{s}_{\bb{w}-1})}
$$
Let $\mathbf{A}$ be a Hall algebra, $f\colon\Sigma\mor A$ an
$\fmon{S}\bprod S$-sorted mapping and $\mawh{f}$ the $\fmon{S}\bprod
S$-sorted mapping from $\mathrm{HTer}_{S}(\Sigma)$ to $A$ defined, for
every $(w,s)\in\fmon{S}\bprod S$, as $\mawh{f}_{w,s} =
\ext{(p^{w})}_{s}$, where, we recall, \ext{(p^{w})} is the unique
homomorphism from $\mathbf{T}_{\Sigma}(\vs{w})$ to
$\mathbf{A}^{f,w}$ such that the following diagram commutes
$$
\xymatrix{
\vs{w}
  \ar[r]^-{\eta_{\vs{w}}}
  \ar[rd]_{p^{w}} &
\mathrm{T}_{\Sigma}(\vs{w})
  \ar[d]^{\ext{(p^{w})}} \\
 &
(A_{w,s})_{s\in S}
}
$$
Then $\mawh{f}$ is a homomorphism of Hall algebras, because, on the
one hand, for
$w\in\fmon{S}$ and $i\in\bb{w}$ we have that%
\begin{align*}
    \mawh{f}_{w,w_{i}}((\pi^{w}_{i})^{\mathbf{HTer}_{S}(\Sigma)})
    &=
    \mawh{f}_{w,w_{i}}(v^{s}_{i}) \\
    &=
    p^{w}_{w_{i}}(v^{s}_{i}) \\
    &=
    (\pi^{w}_{i})^{\mathbf{A}},
\end{align*}
and, on the other hand, for $P\in \mathrm{T}_{\Sigma}(\vs{w})_{s}$ and
$(Q_{i}\mid i\in \bb{w})\in \mathrm{T}_{\Sigma}(\vs{u})_{w}$ we have that %
\begin{align*}
    &\mawh{f}_{u,s}(
      \xi_{u,w,s}^{\mathbf{HTer}_{S}(\Sigma)}
        (P,Q_{0},\ldots,Q_{\bb{w}-1})
                   ) \\
    &=
    \ext{(p^{u})}_{s}(
      \mathcal{Q}^{\sharp}_{s}(P)
                     ) \\
    &=
    \ext{(\ext{(p^{u})}\comp \mathcal{Q})}_{s} (P)
    \qquad(because\, \ext{(p^{u})}\comp\mathcal{Q}^{\sharp} =
    (\ext{(p^{u})}\comp \mathcal{Q})^{\sharp})
    \\
    &=
    P^{\mathbf{A}^{f,u}}
      ( \ext{(p^{u})}_{w_{0}}(Q_{0}),\ldots,
        \ext{(p^{u})}_{w_{\bb{w}-1}}(Q_{\bb{w}-1})
      ) \\
    &=
    \xi_{u,w,s}^{\mathbf{A}}
      ( \ext{(p^{w})}_{s}(P),
        \ext{(p^{u})}_{w_{0}}(Q_{0}),\ldots,
        \ext{(p^{u})}_{w_{\bb{w}-1}}(Q_{\bb{w}-1})
      )
    \qquad(by\, \mathrm{Lemma}~\ref{L:aux})
      \\
    &=
    \xi_{u,w,s}^{\mathbf{A}}
      ( \mawh{f}_{w,s}(P),
        \mawh{f}_{u,w_{0}}(Q_{0}),\ldots,
        \mawh{f}_{u,w_{\bb{w}-1}}(Q_{\bb{w}-1})
      ).
\end{align*}
Therefore the $\fmon{S}\bprod S$-sorted mapping $\mawh{f}$ is a
homomorphism. Furthermore, $\mawh{f}\comp h = f$,
because, for every $w\in\fmon{S}$, $s\in S$, and
$\sigma\in \Sigma_{w,s}$, we have that%
\begin{align*}
    \mawh{f}_{w,s}(h_{w,s}(\sigma)) &=
    \ext{(p^{w})}_{s}(\sigma(v^{s}_{0},\ldots,v^{s}_{\bb{w}-1})) \\
    &=
    \sigma^{\mathbf{A}_{w}}
      (p^{w}_{{w}_{0}}(v^{s}_{0}),
        \ldots,
       p^{w}_{{w}_{\bb{w}-1}}(v^{s}_{\bb{w}-1})
      ) \\
    &=
    \xi_{w,w,s}^{\mathbf{A}}
      ( f_{(w,s)}(\sigma),
        (\pi^{w}_{0})^{\mathbf{A}},
        \ldots,
        (\pi^{w}_{\bb{w}-1})^{\mathbf{A}}
      ) \\
    &=
    f_{w,s}(\sigma) \quad\text{($\mathrm{H}_{2}$)}.
\end{align*}
It is obvious that $\mawh{f}$ is the unique homomorphism such that
$\mawh{f}\comp h = f$.  Henceforth $\mathbf{HTer}_{S}(\Sigma)$ is
isomorphic to $\mathbf{T}_{\Hall_{S}}(\Sigma)$.
\end{proof}

This isomorphism together with the adjunction
$\mathbf{T}_{\Hall_{S}}\ladj\mathrm{G}_{\Hall_{S}}$ has as a
consequence that, for every $S$-sorted set $A$ and $S$-sorted
signature $\Sigma$, the sets
$\mathrm{Hom}({\Sigma},\mathrm{HOp}_{S}(A))$ and
$\mathrm{Hom}(\mathbf{HTer}_{S}(\Sigma),\mathbf{HOp}_{S}(A))$ are
naturally isomorphic.  Actually, the isomorphism sends, for an
$S$-sorted set $A$, a structure of $\mathbf{\Sigma}$-algebra $F$ on
$A$, i.e., an $\mathrm{ms}$-mapping $F$ from $\Sigma$ to
$\mathrm{HOp}_{S}(A)$, to the homomorphism of Hall algebras
$\mathrm{HTr}^{(A,F)}_{S} = (\mathrm{Tr}^{\vs{w},(A,F)}_{s})_{(w,s)\in
\fmon{S}\times S}$, where, for every $w \in \fmon{S}$,
$\mathrm{Tr}^{\vs{w},(A,F)} = (\mathrm{Tr}^{\vs{w},(A,F)}_{s})_{s\in
S}$ is the unique homomorphism from $\mathbf{T}_{\Sigma}(\vs{w})$ to
$\mathbf{Op}_{\vs{w}}(A,F)\cong (A,F)^{A_{w}}$
such that  the following diagram commutes%
$$
\xymatrix{
\vs{w}
  \ar[r]^-{\eta_{\vs{w}}}
  \ar[rd]_{\p^{A}_{\vs{w}}} &
\mathrm{T}_{\Sigma}(\vs{w})
  \ar[d]^{\mathrm{Tr}^{\vs{w},(A,F)}}\\
 &
\mathrm{Op}_{\vs{w}}(A,F)\\
}
$$
where $\p^{A}_{\vs{w}}$ is the $S$-sorted mapping defined, for every
$s\in S$ and $v^{s}_{i}\in(\vs{w})_{s}$, as
$\p^{A}_{\vs{w},s}(v^{s}_{i}) = \pr^{A}_{w,i}$; while the inverse
isomorphism sends an homomorphism $h$ from $\mathbf{HTer}_{S}(\Sigma)$
to $\mathbf{HOp}_{S}(A)$ to the algebraic structure
$\mathrm{G}_{\Hall_{S}}(h)\comp \eta_{\Sigma}$ on $A$.



For a set of sorts $S$, the fundamental objects in the approach to
the many-sorted completeness theorem in~\cite{gm85}, i.e., the
Hall algebras for $S$, have an alternative, but equivalent,
description in terms of, what we call, B\'{e}nabou algebras for
$S$, that, as we will show below are more strongly linked to the
finitary many-theories algebraic theories than are the Hall
algebras.  Besides, the B\'{e}nabou algebras will be shown to be
more adequate in order to work with morphisms between signatures
more general than the standard ones.  Actually there exists an
equivalence between the category $\mathbf{Alg}(\mathrm{H}_{S})$,
of Hall algebras for $S$, and the category
$\mathbf{Alg}(\mathrm{B}_{S})$, of B\'{e}nabou algebras for $S$,
i.e., the category defined as follows.

\begin{definition}
Let $S$ be a set of sorts and $V^{\mathrm{B}_{S}}$ the
$(S^{\star})^{2}$-sorted set of variables $(V_{u,w})_{(u,w)\in
(S^{\star})^{2}}$ where, for every $(u,w)\in (S^{\star})^{2}$,
$V_{u,w} = \{\,v^{u,w}_{n}\mid n\in \mathbb{N}\,\}$.  A \emph{Bénabou
algebra for} $S$ is a $\mathrm{B}_{S} =
((S^{\star})^{2},\Sigma^{\mathrm{B}_{S}},\mathcal{E}^{\mathrm{B}_{S}})$-algebra,
where $\Sigma^{\mathrm{B}_{S}}$ is the $(S^{\star})^{2}$\nobreakdash-sorted
signature defined as follows:
\begin{enumerate}
\item[$\mathrm{BS}_{1}$.]  For the empty word $\lambda\in S^{\star}$,
      every $w\in S^{\star}$ and $i\in\lvert w \rvert$, where $\lvert w \rvert$
      is the domain of the word $w$, the formal operation of \emph{projection}:
      $$
      \pi^{w}_{i}\colon\lambda\mor(w,(w_{i})).
      $$

\item[$\mathrm{BS}_{2}$.]  For every $u$, $w\in S^{\star}$, the formal
      operation of \emph{finite tupling}:
      $$
      \langle\,\rangle_{u,w}\colon ((u,(w_{0})),\ldots,
      (u,(w_{\lvert w \rvert-1})))\mor (u,w).
      $$

\item[$\mathrm{BS}_{3}$.]  For every $u$, $x$, $w \in S^{\star}$, the formal
      operation of \emph{substitution}:
      $$
      \circ_{u,x,w}\colon ((u,x),(x,w)) \mor (u,w);
      $$
\end{enumerate}
while $\mathcal{E}^{\mathrm{B}_{S}}$ is the $\mathrm{ms}$-subset of
$\mathrm{Eq}(\Sigma^{\mathrm{B}_{S}}) = (\mathrm{T}_{\Sigma^{\mathrm{B}_{S}}}
(\mathbin{\downarrow}{\overline{w}})_{(u,x)}^{2})_{(\overline{w},(u,x))\in
((S^{\star})^{2})^{\star}\times (S^{\star})^{2}}$ defined as follows:
\begin{enumerate}
\item[$\mathrm{B}_{1}$.]  For every $u$, $w\in S^{\star}$ and
     $i\in\lvert w \rvert$, the equation:
   $$
   \pi^{w}_{i} \circ_{u,w,(w_{i})}
   \langle v^{u,(w_{0})}_{0},\ldots,
   v^{u,(w_{\lvert w \rvert-1})}_{\lvert w \rvert-1}\rangle_{u,w} =
   v^{u,(w_{i})}_{i},
   $$
of type $(((u,(w_{0})),\ldots,(u,(w_{\lvert w \rvert-1}))),(u,(w_{i})))$.

\item[$\mathrm{B}_{2}$.]  For every $u$, $w \in S^{\star}$, the
     equation:
   $$
   v^{u,w}_{0} \circ_{u,u,w} \langle \pi^{u}_{0}, \ldots,
   \pi^{u}_{\lvert u \rvert-1}\rangle_{u,u} = v^{u,w}_{0},
   $$
of type $(((u,w)),(u,w))$.

\item[$\mathrm{B}_{3}$.]  For every $u$, $w\in S^{\star}$, the
   equation:
   $$
   \langle \pi^{w}_{0}\circ_{u,w,w_{0}} v^{u,w}_{0}, \ldots,
      \pi^{w}_{\lvert w \rvert-1} \circ_{u,w,w_{\lvert w \rvert-1}}
      v^{u,w}_{0}\rangle_{u,w} = v^{u,w}_{0},
    $$
   of type $(((u,w)),(u,w))$.

\item[$\mathrm{B}_{4}$.]
   For every $w \in S^{\star}$, the equation:
   $$\langle \pi^{w}_{0}\rangle_{w,(w_{0})} = \pi^{w}_{0},$$
   of type $(((w,(w_{0}))),(w,(w_{0})))$.

\item[$\mathrm{B}_{5}$.]  For every $u$, $x$, $w$, $y \in S^{\star}$, the
   equation:
   $$
   v^{w,y}_{0} \circ_{u,w,y}
   ( v^{x,w}_{1} \circ_{u,x,w} v^{u,x}_{2} ) =
   ( v^{w,y}_{0} \circ_{x,w,y} v^{x,w}_{1} ) \circ_{u,x,y}
   v^{u,x}_{2},
   $$
   of type $(((w,y),(x,w),(u,x)),(u,y))$,
\end{enumerate}
where $v^{u,w}_{n}$ is the $n$-th variable of type $(u,w)$, $Q
\circ_{u,x,w} P$ is $\circ_{u,x,w} (P,Q)$, and
$\langle P_{0},\ldots,P_{\lvert w \rvert-1}\rangle_{u,w}$ is
$\langle\,\rangle_{u,w}(P_{0},\ldots,P_{\lvert w \rvert-1})$.

Since $\mathbf{Alg}(\mathrm{B}_{S})$ is a variety, the forgetful
functor $\mathrm{G}_{\Ben_{S}}$ from $\mathbf{Alg}(\mathrm{B}_{S})$ to
$\mathbf{Set}^{\fmon{S}\bprod \fmon{S}}$ has a left adjoint
$\mathbf{T}_{\Ben_{S}}$%
$$
\xymatrix@=5pc{
\mathbf{Alg}(\mathrm{B}_{S})
  \ar@<1.5ex>[r]^{\mathrm{G}_{\Ben_{S}}}
  \ar@{}[r]|{\uadj} &
\mathbf{Set}^{\fmon{S}\bprod \fmon{S}}
  \ar@<1.5ex>[l]^{\mathbf{T}_{\Ben_{S}}}
}
$$
which assigns to an $\fmon{S}\bprod \fmon{S}$-sorted set the
corresponding free Bénabou algebra.
\end{definition}

For every $S$-sorted set $A$, $\mathrm{BOp}_{S}(A) =
(\mathrm{Hom}(A_{w},A_{u}))_{(w,u)\in\fmon{S}\bprod \fmon{S}}$ is
endowed with a structure of Bénabou algebra as stated in the following

\begin{proposition}
Let $A$ be an $S$-sorted set and $\mathbf{BOp}_{S}(A)$ the
$\Sigma^{\Ben_{S}}$-algebra with underlying many-sorted set
$\mathrm{BOp}_{S}(A)$ and algebraic structure defined as follows
\begin{enumerate}
\item For every $w\in\fmon{S}$ and $i\in\bb{w}$,
      $(\pi^{w}_{i})^{\mathbf{BOp}_{S}(A)} =
      \pr^{A}_{w,i}\colon A_{w}\mor A_{(w_{i})}$.

\item For every $u,w\in\fmon{S}$,
      $\langle\,\rangle_{u,w}^{\mathbf{BOp}_{S}(A)}$ is defined, for every
      $(f_{0},\ldots,f_{\bb{w}-1})$ in
      $\prod_{i\in \bb{w}}\mathrm{Hom}(A_{w},A_{(w_{i})})$, as
      $\langle\,\rangle_{u,w}^{\mathbf{BOp}_{S}(A)}(f_{0},\ldots,f_{\bb{w}-1})
      = \tp{ f_{i}}_{i\in\bb{w} }$, where $\tp{
      f_{i}}_{i\in\bb{w}}$ is the unique mapping from $A_{u}$ to
      $A_{w}$ such that, for every $i\in \bb{w}$, $\pr^{A}_{w,i}\comp
      \tp{ f_{i}}_{i\in\bb{w}} = f_{i}$.

\item For every $u$, $x$, $w \in S^{\star}$,
      $\circ_{u,x,w}^{\mathbf{BOp}_{S}(A)}$ is defined as the
      composition of mappings.
\end{enumerate}

Then $\mathbf{BOp}_{S}(A)$ is a Bénabou algebra,  the \emph{Bénabou
algebra for} $(S,A)$.
\end{proposition}

For every $S$-sorted signature $\Sigma$, $\mathrm{BTer}_{S}(\Sigma) =
(\mathrm{T}_{\mathbf{\Sigma}}(\vs{w})_{u})_{(w,u)\in \fmon{S}\bprod
\fmon{S}}$, that is naturally isomorphic to
$(\mathrm{Hom}(\vs{u},\mathrm{T}_{\mathbf{\Sigma}}(\vs{w}))_{(w,u)\in
\fmon{S}\bprod \fmon{S}}$, is endowed with a structure of Bénabou
algebra as stated in the following

\begin{proposition}
Let $\Sigma$ be an $S$-sorted signature and
$\mathbf{BTer}_{S}(\Sigma)$ the $\Sigma^{\Ben_{S}}$-algebra with
underlying many-sorted set $\mathrm{BTer}_{S}(\Sigma)$ and algebraic
structure that obtained, by transport of structure, from the algebraic
structure defined on the $\fmon{S}\bprod \fmon{S}$-sorted set
$(\mathrm{Hom}(\vs{u},\mathrm{T}_{\mathbf{\Sigma}}(\vs{w}))_{(w,u)\in
\fmon{S}\bprod \fmon{S}}$ as follows
\begin{enumerate}
\item For every $w\in\fmon{S}$ and $i\in\bb{w}$,
      $(\pi^{w}_{i})^{\mathbf{BTer}_{S}(\Sigma)}$ is the composition of
      the canonical embedding from $\vs{(w_{i})}$ to $\vs{w}$ and the
      canonical embedding from $\vs{w}$ to $\mathrm{T}_{\Sigma}(\vs{w})$.

\item For every $u,w\in\fmon{S}$,
      $\langle\,\rangle_{u,w}^{\mathbf{BTer}_{S}(\Sigma)}$ is the canonical
      isomorphism from the cartesian product $\prod_{i\in
      \bb{w}}\mathrm{Hom}(\vs{(w_{i})},\mathrm{T}_{\mathbf{\Sigma}}(\vs{u}))$ to
      $\mathrm{Hom}(\vs{w},\mathrm{T}_{\mathbf{\Sigma}}(\vs{u}))$.

\item For every $u$, $x$, $w \in S^{\star}$,
      $\comp_{u,x,w}^{\mathbf{BTer}_{S}(A)}$ is defined as the
      mapping which sends a pair
      $\mathcal{P}\in \mathrm{Hom}(\vs{x},\mathrm{T}_{\mathbf{\Sigma}}(\vs{u}))$ and
      $\mathcal{Q}\in \mathrm{Hom}(\vs{w},\mathrm{T}_{\mathbf{\Sigma}}(\vs{x}))$ to
      $\mathcal{P}^{\sharp}\comp\mathcal{Q}$.
\end{enumerate}

Then $\mathbf{BTer}_{S}(\Sigma)$ is a Bénabou algebra, the \emph{Bénabou
algebra for} $(S,\Sigma)$.
\end{proposition}


Next, after defining the category $\mathbf{BTh}_{\mathrm{f}}(S)$, of
finitary many-sorted algebraic theories of B\'{e}nabou (defined for
the first time in~\cite{jB68}), that generalize the finitary
single-sorted algebraic theories of Lawvere, we prove that there
exists an isomorphism between the category
$\mathbf{BTh}_{\mathrm{f}}(S)$ and the category
$\mathbf{Alg}(\mathrm{B}_{S})$.

\begin{definition}
We denote by $\mathbf{BTh}_{\mathrm{f}}(S)$ the category with objects
pairs $\boldsymbol{\mathcal{B}} = (\mathbf{B},p)$, where $\mathbf{B}$
is a category that has as objects the words on $S$ and $p$ a family
$(p^{w})_{w\in S^{\star}}$ such that, for every word $w\in S^{\star}$,
$p^{w}$ is a family $(p^{w}_{i}\colon w\mor (w_{i}))_{i\in\lvert w
\rvert}$ of morphisms in $\mathbf{B}$, the \emph{projections} for $w$,
where $(w_{i})$ is the word of length 1 on $S$ whose only letter is
$w_{i}$, such that $(w,p^{w})$ is a product in $\mathbf{B}$ of the
family of words $((w_{i}))_{i\in\lvert w \rvert}$, and as morphisms
from $\boldsymbol{\mathcal{B}}$ to $\boldsymbol{\mathcal{B}}'$
functors $F$ from $\mathbf{B}$ to $\mathbf{B'}$ such that the object
mapping of $F$ is the identity and the morphism mapping of $F$
preserves the projections, i.e., for every $w\in S^{\star}$ and
$i\in\lvert w \rvert$, $F(p^{w,\boldsymbol{\mathcal{B}}}_{i}) =
p^{w,\boldsymbol{\mathcal{B}}'}_{i}$.

\end{definition}

\begin{proposition}\label{isoBalgBth}
There exists an isomorphism from the category
$\mathbf{Alg}(\mathrm{B}_{S})$ to the category
$\mathbf{BTh}_{\mathrm{f}}(S)$.
\end{proposition}

\begin{proof}
The isomorphism from $\mathbf{Alg}(\mathrm{B}_{S})$ to
$\mathbf{BTh}_{\mathrm{f}}(S)$ is the functor $B_{a,t}$ which to a
B\'{e}nabou algebra $\mathbf{B}$ assigns the B\'{e}nabou theory
$B_{a,t}(\mathbf{B})$ which has as underlying category that given by
the following data
\begin{enumerate}
\item The set of objects is $S^{\star}$ and, for $u,w\in
      S^{\star}$, $\mathrm{Hom}(u,w) = B_{u,w}$,

\item For every $w\in S^{\star}$,
      $\mathrm{id}_{w}=\langle(\pi^{w}_{i})^ {\mathbf{B}}\mid
      i\in\lvert w \rvert\rangle_{w,w}$,

\item If $P\colon u\mor x$, $Q\colon x\mor w$, then the composition
      of $P$ and  $Q$ is $\circ^{\mathbf{B}}_{u,x,w}(P,Q)$,
\end{enumerate}
and as underlying family of projections that given, for every $w\in
S^{\star}$, as $\pi^{w}=((\pi^{w}_{i})^{\mathbf{B}})_{i\in\lvert w
\rvert}$; and which to a morphism of B\'{e}nabou algebras $f\colon
\mathbf{B}\mor \mathbf{B'}$ assigns the morphism of B\'{e}nabou
theories $B_{a,t}(f)$ that to $P\colon w\mor u$ associates
$f_{w,u}(P)\colon w\mor u$.

The inverse of $B_{a,t}$ is the functor $B_{t,a}$ which to a
B\'{e}nabou theory $\boldsymbol{\mathcal{B}} = (\mathbf{B},p)$ assigns
the B\'{e}nabou algebra $B_{t,a}(\boldsymbol{\mathcal{B}})$ that has
\begin{enumerate}
\item As underlying $(S^{\star})^{2}$-sorted set the family
      $(\mathrm{Hom}_{\mathbf{B}}(w,u))_{(w,u)\in
      (S^{\star})^{2}}$, and

\item As structure of B\'{e}nabou algebra on
      $(\mathrm{Hom}_{\mathbf{B}}(w,u))_{(w,u)\in (S^{\star})^{2}}$
      that obtained by interpreting, for every $w\in S^{\star}$ and
      $i\in \lvert w \rvert$, $\pi^{w}_{i}$ as $p^{w}_{i}$, for every
      $u,w\in S^{\star}$, $\langle\,\rangle_{u,w}$ as the canonical
      mapping from $\prod_{i\in\lvert w
      \rvert}\mathrm{Hom}_{\mathbf{B}}(u,(w_{i}))$ to
      $\mathrm{Hom}_{\mathbf{B}}(u,w)$ obtained by the universal
      property of the product for $w$, and, for every $u,x,w\in
      S^{\star}$, $\circ_{u,x,w}$ as the composition in $\mathbf{B}$;
\end{enumerate}
and which to a morphism of B\'{e}nabou theories $F\colon
\boldsymbol{\mathcal{B}}\mor \boldsymbol{\mathcal{B}}'$ assigns the
morphism of B\'{e}nabou algebras $B_{t,a}(F)$, that for every $u,w\in
S^{\star}$, is the bi-restriction of $F$ to the corresponding hom-sets
$\mathrm{Hom}(u,w)$ and $\mathrm{Hom}(u,w)$.
\end{proof}

\begin{remark}
The isomorphism between $\mathbf{BTh}_{\mathrm{f}}(S)$ and
$\mathbf{Alg}(\mathrm{B}_{S})$ can be interpreted as meaning, and this can be
algebraically reassuring, that the category of finitary many-sorted
algebraic theories of B\'{e}nabou, a purely formal entity, has the form of
a category of models for a finitary many-sorted equational
presentation, a semantical, or substantial, entity, therefore
confirming, once more, that apparently \emph{form is substance}.
Moreover, the isomorphism shows that the B\'{e}nabou algebras are more
closely related to the finitary many-sorted algebraic theories of
B\'{e}nabou than are the Hall algebras.
\end{remark}


Next we prove, directly, that the categories
$\mathbf{Alg}(\mathrm{H}_{S})$ and $\mathbf{Alg}(\mathrm{B}_{S})$ of
Hall and Bénabou algebras, respectively, are equivalent.  Later on,
once we have defined the morphisms and transformations of Fujiwara, we
will get such an equivalence as a consequence of the existence of
both an equivalence between the specifications of Hall and Bénabou and
a pseudo-functor from the $2$-category of specifications to the
$2$-category $\mathbf{Cat}$.

\begin{proposition}\label{EquivCatHallBenS}
For every set of sorts $S$, the categories $\mathbf{Alg}(\mathrm{H}_{S})$ and
$\mathbf{Alg}(\mathrm{B}_{S})$ are equivalent.
\end{proposition}

\begin{proof}
The equivalence from $\mathbf{Alg}(\mathrm{H}_{S})$ to
$\mathbf{Alg}(\mathrm{B}_{S})$ is the functor $F_{h,b}$ which to a Hall
algebra $\mathbf{A}$ assigns the B\'{e}nabou algebra
$F_{h,b}(\mathbf{A})$ that has
\begin{enumerate}
\item As underlying $(S^{\star})^{2}$-sorted set
      $((A_{w})_{u})_{(w,u)\in(S^{\star})^{2}}$ where
      $A_{w}=(A_{w,s})_{s\in S}$ and $(A_{w})_{u} = \prod_{i\in \lvert
      u \rvert}A_{w,u_{i}}$, and

\item As structure of B\'{e}nabou algebra on
     $((A_{w})_{u})_{(w,u)\in(S^{\star})^{2}}$ that defined as
    \begin{align*}
    (\pi^{w}_{i})^{F_{h,b}(\mathbf{A})}=(&(\pi^{w}_{i})^{\mathbf{A}}),\\
    \langle (a_{0}),\ldots,(a_{\lvert w \rvert-1})\rangle_{u,w}^{F_{h,b}(\mathbf{A})} =
    (
    &\xi_{u,w,w_{0}}^{\mathbf{A}}(\pi^{w}_{0},a_{0},\ldots,a_{\lvert w \rvert-1}),
    \ldots, \\
    &\xi_{u,w,w_{\lvert w \rvert-1}}^{\mathbf{A}}(\pi^{w}_
    {\lvert w \rvert-1},a_{0},\ldots,
    a_{\lvert w \rvert-1})
    ), \\
    \circ^{F_{h,b}(\mathbf{A})}_{u,x,w}(a,b) =
      (&\xi^{\mathbf{A}}_{u,x,w_{0}}(b_{0},a_{0},\ldots,a_{\lvert x \rvert-1}),
      \ldots, \\
      &\xi^{\mathbf{A}}_{u,x,w_{\lvert w \rvert-1}}(b_{\lvert w \rvert-1},a_{0},\ldots,
      a_{\lvert x \rvert-1}));
    \end{align*}
\end{enumerate}
and which to a morphism $f\colon \mathbf{A}\mor \mathbf{B}$ of Hall
algebras assigns the morphism $F_{h,b}(f)=((f_{w})_{u})_{
(w,u)\in(S^{\star})^{2}}$ from $F_{h,b}(\mathbf{A})$ to
$F_{h,b}(\mathbf{B})$ defined, for $(a_{0},\ldots,a_{\lvert u
\rvert-1})$ in $(A_{w})_{u}$, as
$$
(a_{0},\ldots,a_{\lvert u \rvert-1})\longmapsto
(f_{w,u_{0}}(a_{0}),\ldots,f_{w,u_{\lvert u \rvert-1}}(a_{\lvert u \rvert-1}))).
$$

The quasi-inverse equivalence from $\mathbf{Alg}(\mathrm{B}_{S})$ to
$\mathbf{Alg}(\mathrm{H}_{S})$ is the functor $F_{b,h}$ which to a
B\'{e}nabou algebra $\mathbf{A}$ assigns the Hall algebra
$F_{b,h}(\mathbf{A})$ that has
\begin{enumerate}
    \item  As underlying $S^{\star}\times S$-sorted set
          $(A_{w,(s)})_{ (w,s)\in S^{\star}\times S}$ , and

    \item As structure of Hall algebra on $(A_{w,(s)})_{ (w,s)\in
    S^{\star}\times S}$ that defined as
    \begin{align*}
    (\pi^{w}_{i})^{F_{b,h}(\mathbf{A})}&=(\pi^{w}_{i})^{\mathbf{A}}, \\
    \xi_{u,w,s}^{F_{b,h}(\mathbf{A})}(a,a_{0},\ldots,a_{\lvert w \rvert-1})
    &=
    a\circ_{u,w,s}\langle a_{0},\ldots,a_{\lvert w \rvert-1}\rangle_{u,w};
    \end{align*}
\end{enumerate}
and which to a homomorphism $f\colon{\mathbf{A}}\mor \mathbf{B}$ of
B\'{e}nabou algebras assigns the bi-restriction of $f$ to
$F_{b,h}(\mathbf{A})$ and $F_{b,h}(\mathbf{B})$.

Next, for a Bénabou algebra $\mathbf{A}$, we prove that $\mathbf{A}$ and
$F_{h,b}(F_{b,h}(\mathbf{A}))$ are isomorphic.  Let
$f\colon\mathbf{A}\mor F_{h,b}(F_{b,h}(\mathbf{A}))$ be the $\fmon{S}\times
\fmon{S}$-sorted mapping defined, for $(u,w)\in\fmon{S}\times
\fmon{S}$ and $a\in A_{u,w}$, as
$$
a \mapsto
(
(\pi^{w}_{0})^{\mathbf{A}}\comp a, \ldots,
(\pi^{w}_{\bb{w}-1})^{\mathbf{A}}\comp a
).
$$
The definition is sound because, for $a\in A_{u,w}$, we have that
$(\pi^{w}_{i})^{\mathbf{A}}\comp a \in F_{b,h}(A)_{u,w_{i}}$, hence
$((\pi^{w}_{0})^{\mathbf{A}}\comp a, \ldots,
(\pi^{w}_{\bb{w}-1})^{\mathbf{A}}\comp a ) \in F_{h,b}(F_{b,h}(A))_{u,w}$.
Thus defined $f$ is a homomorphism, since we have, on the one hand, that
\begin{align*}
f((\pi^{w}_{i})^{\mathbf{A}})
&=
(\pi^{(w_{i})}_{0}\comp \pi^{w}_{i}) \\
&=
(\tp{\pi^{(w_{i})}_{0}}_{(w_{i}),(w_{i})}\comp \pi^{w}_{i})
&&\text{(by $\Ben_{4}$)}\\
&=
(\tp{\pi^{(w_{i})}_{0}\comp (\tp{\pi^{(w_{i})}_{0}}\comp
\pi^{w}_{i})}_{w,(w_{i})})
&&\text{(by $\Ben_{3}$)}\\
&=
(\tp{\pi^{(w_{i})}_{0}\comp \pi^{w}_{i}}_{w,(w_{i})})
&&\text{(by $\Ben_{2}$ and $\Ben_{5}$)}\\
&=
(\pi^{w}_{i})
&&\text{(by $\Ben_{3}$)}\\
&=
(\pi^{w}_{i})^{F_{h,b}(F_{b,h}(\mathbf{A}))},
\end{align*}
on the other hand, that%
\begin{align*}
f(\tp{a_{0},\ldots,a_{\bb{w}-1}}^{\mathbf{A}}_{u,w})
=
(
&(\pi^{w}_{0})^{\mathbf{A}}\comp \tp{a_{0},\ldots,a_{\bb{w}-1}}^{\mathbf{A}}_{u,w}
,\ldots,\\
&(\pi^{w}_{\bb{w}-1})^{\mathbf{A}}\comp \tp{a_{0},\ldots,a_{\bb{w}-1}}^{\mathbf{A}}
         _{u,w}
)\\
=
(&\xi^{F_{b,h}(\mathbf{A})}(
     (\pi^{w}_{0})^{F_{b,h}(\mathbf{A})},
     a_{0},\ldots,a_{\bb{w}-1}
    )
 ,\ldots,\\
 &\xi^{F_{b,h}(\mathbf{A})}(
     (\pi^{w}_{\bb{w}-1})^{F_{b,h}(\mathbf{A})},
     a_{0},\ldots,a_{\bb{w}-1}
    )
)\\
=
&\tp{(a_{0}),\ldots,(a_{\bb{w}-1})}
       ^{F_{h,b}(F_{b,h}(\mathbf{A}))}_{u,w}\\
=
&\tp{f(a_{0}),\ldots,f(a_{\bb{w}-1})}
       ^{F_{h,b}(F_{b,h}(\mathbf{A}))}_{u,w},
\end{align*}
and, lastly, that%
\begin{align*}
f(b\comp^{\mathbf{A}} a)
=
(
 &(\pi^{w}_{0})^{\mathbf{A}}\comp (b\comp a),\ldots,
 (\pi^{w}_{\bb{w}-1})^{\mathbf{A}}\comp (b\comp a)
)\\
=
(
 &(\pi^{w}_{0})^{\mathbf{A}}\comp b\comp \tp{a_{0},\ldots,a_{\bb{w}-1}}
 ,\ldots,\\
 &(\pi^{w}_{\bb{w}-1})^{\mathbf{A}}\comp b\comp \tp{a_{0},\ldots,a_{\bb{w}-1}}
)\\
=
(
 &f(b_{0})\comp \tp{a_{0},\ldots,a_{\bb{w}-1}}
 ,\ldots,\\
 &b(b_{\bb{w}-1})\comp \tp{a_{0},\ldots,a_{\bb{w}-1}}
)\\
=
(
 &\xi^{F_{b,h}(\mathbf{A})}(f(b_{0}),f(a_{0}),\ldots,f(a_{\bb{w}-1}))
 ,\ldots,\\
 &\xi^{F_{b,h}(\mathbf{A})}(f(b_{\bb{w}-1}),f(a_{0}),\ldots,f(a_{\bb{w}-1}))
)\\
=
\phantom{(} 
&f(b)\comp^{F_{h,b}(F_{b,h}(\mathbf{A}))}f(a).
\end{align*}

Reciprocally, let $g\colon F_{h,b}(F_{b,h}(\mathbf{A}))\mor
\mathbf{A}$ be the $\fmon{S}\times \fmon{S}$-sorted mapping defined, for
$(u,w)\in\fmon{S}\times \fmon{S}$ and $b\in F_{h,b}(F_{b,h}(A))$, as %
$$
b \mapsto
\tp{b_{0},\ldots,b_{\bb{w}-1}}^{\mathbf{A}}_{u,w}.
$$
The definition is sound because, for $b =
(b_{0},\ldots,b_{\bb{w}-1})\in F_{h,b}(F_{b,h}(A))$, we have that
$b_{i}\in F_{b,h}(A)_{u,w_{i}}$, hence $b_{i}\in A_{u,(w_{i})}$,
therefore $\tp{b_{0},\ldots,b_{\bb{w}-1}}^{\mathbf{A}}\in A_{u,w}$.  Thus
defined it is easy to prove that $g$ is a homomorphism.

Now we prove that the homomorphisms $f$ and $g$ are such that $g\comp
f = \mathrm{id}_{\mathbf{A}}$ and $f\comp g =
\mathrm{id}_{F_{h,b}(F_{b,h}(\mathbf{A}))}$.  On the one hand, if
$a\in A_{u,w}$, then, by $\mathrm{B}_{3}$, we have that
$$
\tp{(\pi^{w}_{0})^{\mathbf{A}}\comp a, \ldots,
(\pi^{w}_{\bb{w}-1})^{\mathbf{A}}\comp a } = a,
$$
hence $g\comp f = \mathrm{id}_{\mathbf{A}}$.  On the other hand, if
$b\in F_{h,b}(F_{b,h}(A))$, then $g_{u,w}$ sends $b$ to
$\tp{b_{0},\ldots,b_{\bb{w}-1}}^{\mathbf{A}}_{u,w}$, and $f_{u,w}$
sends $\tp{b_{0},\ldots,b_{\bb{w}-1}}^{\mathbf{A}}_{u,w}$ to
$$
((\pi^{w}_{0})^{F_{h,b}(F_{b,h}(\mathbf{A}))}\comp
    \tp{b_{0},\ldots,b_{\bb{w}-1}}^{\mathbf{A}}_{u,w}
   ,\ldots, \\
    (\pi^{w}_{\bb{w}-1})^{F_{h,b}(F_{b,h}(\mathbf{A}))}\comp
    \tp{b_{0},\ldots,b_{\bb{w}-1}}^{\mathbf{A}}_{u,w}),
$$
but this last coincides with
$$
((\pi^{w}_{0})^{F_{h,b}(\mathbf{A})}\comp
    \tp{b_{0},\ldots,b_{\bb{w}-1}}^{\mathbf{A}}_{u,w}
   ,\ldots, \\
    (\pi^{w}_{\bb{w}-1})^{F_{h,b}(\mathbf{A})}\comp
    \tp{b_{0},\ldots,b_{\bb{w}-1}}^{\mathbf{A}}_{u,w}),
$$
thus, by the axiom $\mathrm{B}_{1}$, we have that this, in its turn,
coincides with
$$
\tp{b_{0},\ldots,b_{\bb{w}-1}}^{\mathbf{A}}_{u,w},
$$
therefore $f_{u,w}\comp g_{u,w}(b) = b$.  From which we can assert
that $f\comp g = \mathrm{id}_{F_{h,b}(F_{b,h}(\mathbf{A}))}$.


Finally, for a Hall algebra $\mathbf{A}$ we have that $\mathbf{A}$ and
$F_{b,h}(F_{h,b}(\mathbf{A}))$ are identical, because $a\in
A_{w,s}$ iff $a\in F_{h,b}(A)_{w,(s)}$ iff $a\in
F_{b,h}F_{h,b}(A)_{w,s}$.
\end{proof}

In the following proposition, for a set of sorts $S$, we state some
relations among the just proved equivalence between the categories
$\mathbf{Alg}(\mathrm{H}_{S})$ and $\mathbf{Alg}(\mathrm{B}_{S})$, the
adjunctions $\mathbf{T}_{\Hall_{S}}\ladj \mathrm{G}_{\Hall_{S}}$ and
$\mathbf{T}_{\Ben_{S}}\ladj \mathrm{G}_{\Ben_{S}}$, and the adjunction
$\coprod_{1\bprod\between_{S}}\ladj \Delta_{1\bprod\between_{S}}$
determined by the mapping $1\bprod\between_{S}$ from $\fmon{S}\bprod
S$ to $\fmon{S}\bprod \fmon{S}$ which sends a pair $(w,s)$ in
$\fmon{S}\bprod S$ to the pair $(w,(s))$ in $\fmon{S}\bprod \fmon{S}$.
From these relations we will get as an easy, but interesting,
corollary, that, for every $\fmon{S}\bprod S$-sorted set $\Sigma$,
$\mathbf{T}_{\Ben_{S}}(\coprod_{1\bprod\between_{S}}\Sigma)$, the free
Bénabou algebra on $\coprod_{1\bprod\between_{S}}\Sigma$, is
isomorphic to $\mathbf{BTer}_{S}(\Sigma)$.

\begin{proposition}
Let $S$ be a set of sorts.  Then for the diagram
$$
\xymatrix@C=70pt@R=45pt{
*++{\mathbf{Set}^{\fmon{S}\bprod S} } \xyn{1} &
*++{\mathbf{Alg}(\mathrm{H}_{S})}\xyn{2} \\
*++{\mathbf{Set}^{\fmon{S}\bprod\fmon{S}}} \xyn{3} &
*++{\mathbf{Alg}(\mathrm{B}_{S})}\xyn{4}
\ar@<+1.5ex>@{<- }"1";"2"^{\mathrm{G}_{\Hall_{S}}}
\ar         @{}   "1";"2"|{\uadj}
\ar@<-1.5ex>@{ ->}"1";"2"_{\mathbf{T}_{\Hall_{S}}}
\ar@<+1.5ex>@{<- }"1";"3"^{\Delta_{1\bprod\between_{S}}}
\ar         @{}   "1";"3"|{\ladj}
\ar@<-1.5ex>@{ ->}"1";"3"_{\coprod_{1\bprod\between_{S}}}
\ar@<+1.5ex>@{<- }"2";"4"^{F_{b,h}}
\ar         @{}   "2";"4"|{\equiv}
\ar@<-1.5ex>@{ ->}"2";"4"_{F_{h,b}}
\ar@<+1.5ex>@{<- }"3";"4"^{\mathrm{G}_{\Ben_{S}}}
\ar         @{}   "3";"4"|{\uadj}
\ar@<-1.5ex>@{ ->}"3";"4"_{\mathbf{T}_{\Ben_{S}}}
}
$$
we have that
$
\Delta_{1\bprod\between_{S}}\comp \mathrm{G}_{\Ben_{S}} =
\mathrm{G}_{\Hall_{S}}\comp\, F_{b,h}
\text{\, and \,}
\mathbf{T}_{\Ben_{S}}\comp \tcoprod_{1\bprod\between_{S}} \iso F_{h,b} \comp
\mathbf{T}_{\Hall_{S}}.
$
\end{proposition}

\begin{proof}
The equality $\Delta_{1\bprod\between_{S}}\comp \mathrm{G}_{\Ben_{S}} =
\mathrm{G}_{\Hall_{S}}\comp\, F_{b,h}$ follows from the definitions of the
functors involved.  Then, being
$\mathbf{T}_{\Ben_{S}}\comp \tcoprod_{1\bprod\between_{S}}$ and
$F_{h,b} \comp \mathbf{T}_{\Hall_{S}}$ left adjoints to the same
functor, we can assert that
$\mathbf{T}_{\Ben_{S}}\comp \tcoprod_{1\bprod\between_{S}} \iso
F_{h,b} \comp \mathbf{T}_{\Hall_{S}}$.
\end{proof}

\begin{corollary}\label{Beniso}
Let $\Sigma$ be an $S$-sorted signature.  Then the free Bénabou
algebra $\mathbf{T}_{\Ben_{S}}(\coprod_{1\bprod\between_{S}}\Sigma)$
on $\coprod_{1\bprod\between_{S}}\Sigma$ is isomorphic to the Bénabou
algebra $\mathbf{BTer}_{S}(\Sigma)$ for $(S,\Sigma)$.
\end{corollary}

\begin{proof}
It follows after $\mathbf{BTer}_{S}(\Sigma) =
F_{h,b}(\mathbf{HTer}_{S}(\Sigma))$.
\end{proof}

This corollary is interesting because it enables us, on the one hand,
to get a more tractable description of the terms in
$\mathbf{T}_{\Ben_{S}}(\coprod_{1\bprod\between_{S}}\Sigma)$, and, on
the other hand, to give, in the fifth section, an alternative, but
equivalent, definition of the concept of morphism of Fujiwara between
signatures.

\section{Morphisms of Fujiwara.}


In Mathematics it is standard to compare pairs of objects by means of
homomorphisms, i.e., mappings from one of them to the other which
relate, in a predetermined way, the primitive operations on the source
object to the corresponding primitive operations on the target object.
But there are natural examples of comparisons between objects, e.g.,
the derivations in ring theory (see~\cite{nJ68}, pp.  169--172), which
can only be stated by relating the primitive operations on the source
object to corresponding (\emph{families} of) \emph{derived} operations
on the target object, thus showing the necessity of conveniently
generalizing the ordinary concept of homomorphism.  Related to this,
Fujiwara, in~\cite{tF59}, proposed a theory of mappings between
algebraic systems, of not necessarily the same similarity type, under
which falls the classical concept of homomorphism, but also the above
mentioned derivations in ring theory, among others.

Before we outline, briefly, the theory developed by Fujiwara
in~\cite{tF59}, we agree that for a natural number $n\in \mathbb{N}$
and a standard infinite countable set of variables $V = \{\,v_{n}\mid
n\in \mathbb{N}\,\}$, $\vs{v_{n}}$ is the set $\{\,v_{i}\mid i\in
n\,\}$ of the first $n$ variables in $V$.

Fujiwara defines (in~\cite{tF59}, p.  155) for two single-sorted
signatures $\Sigma = (\Sigma_{n})_{n\in\mathbb{N}}$ and $\Lambda =
(\Lambda_{n})_{n\in\mathbb{N}}$, a morphism (that he calls
\emph{family of basic mapping-formulas}) from $\Sigma$ to $\Lambda$
as, essentially, a pair $(\Phi,P)$, where $\Phi =
\{\,\varphi_{\mu}\mid \mu\in p\,\}$ is a set of mapping variables, to
be replaced by mappings from a $\Sigma$-algebra to another
$\Sigma$-algebra derived from a $\Lambda$-algebra, and $P =
(P^{n})_{n\in \mathbb{N}}$ a family of mappings such that, for every
natural number $n\in \mathbb{N}$, it is the case that
$$
  P^{n}\nfunction
  {\Phi\times \Sigma_{n}}{\mathrm{T}_{\Lambda}(\Phi\times\vs{v_{n}})}
  {(\varphi_{\mu},\sigma)}{P^{n}_{\varphi_{\mu},\sigma}}
$$
i.e., $P^{n}$ sends a pair $(\varphi_{\mu},\sigma)$, with
$\varphi_{\mu}$ a mapping variable and $\sigma$ a formal $n$-ary
operation, to a term $P^{n}_{\varphi_{\mu},\sigma}$ for $\Lambda$ on
the set of variables $\Phi\times \vs{v_{n}}$.  To shorten the notation
we identify the variables $(\varphi_{\mu},v_{i})$ in $\Phi\times
\vs{v_{n}}$ to the expressions $\varphi_{\mu}(v_{i})$, for $\mu\in p$
and $i\in n$.  We remark that for Fujiwara the cardinal number of a
set of mapping variables is not necessarily finite.  However, for us,
leaving out some examples shown below, and in order not to complicate
still more the notation, they will be finite sets.

Let $(\Phi,P)$ be a morphism from $\Sigma$ to $\Lambda$, $\mathbf{B}$
a $\Lambda$-algebra, $n\in \mathbb{N}$ and $\sigma\in \Sigma_{n}$.
Then Fujiwara (in~\cite{tF59}, p.  155) assigns to the formal
$n$\nobreakdash-ary operation $\sigma$ the $n$-ary operation
$F_{\sigma}^{\mathbf{B}^{\Phi}}$ from $(B^{\Phi})^{n}$ to $B^{\Phi}$
obtained as the composition of the vertical mappings in the following
diagram
$$
\xymatrix@C=80pt@R=40pt{
(B^{\Phi})^{n}
  \ar[d]
  \ar`l[dd]+/l4pc/ `[dd]_{F_{\sigma}^{\mathbf{B}^{\Phi}}} [dd]
& {} \\
B^{\Phi\times \vs{v_{n}}}
  \ar[d]|{\tp{P^{n,\mathbf{B}}_{\varphi_{\mu},\sigma}}_{\mu\in p}}
  \ar[rd]^{P^{n,\mathbf{B}}_{\varphi_{\mu},\sigma}} \\
B^{\Phi}
  \ar[r]_{\pr_{\mu}}
 &
B \\
}$$
where we have that
\begin{enumerate}
\item For every $\mu\in p$, $\pr_{\mu}$ is the canonical
      projection from $B^{\Phi}$ to $B$,

\item For every $\mu\in p$,
      $P^{n,\mathbf{B}}_{\varphi_{\mu},\sigma}$ is the term operation
      on the $\Lambda$-algebra $\mathbf{B}$ determined by the term
      $P^{n}_{\varphi_{\mu},\sigma}\in
      \mathrm{T}_{\Lambda}(\Phi\times\vs{v_{n}})$,

\item The vertical mapping from $(B^{\Phi})^{n}$ to $B^{\Phi\times \vs{v_{n}}}$
      is the canonical isomorphism between both, and

\item $\tp{P^{n,\mathbf{B}}_{\varphi_{\mu},\sigma}}_{\mu\in p}$
      is the unique mapping from $B^{\Phi\times \vs{v_{n}}}$ to
      $B^{\Phi}$ such that, for every $\mu\in p$,
      $\pr_{\mu}\comp
      \tp{P^{n,\mathbf{B}}_{\varphi_{\mu},\sigma}}_{\mu\in p}
      = P^{n,\mathbf{B}}_{\varphi_{\mu},\sigma}$.
\end{enumerate}

In this way, Fujiwara (in~\cite{tF59}, p.  155) associates to every
$\Lambda$-algebra $\mathbf{B}$ a corresponding $\Sigma$-algebra
$P(\mathbf{B})$ with $B^{\Phi}$ as underlying set and
$(F_{\sigma}^{\mathbf{B}^{\Phi}})_{\sigma\in \bigcup_{n\in
\mathbb{N}}\Sigma_{n}}$ as family of structural operations.  However
he does not extend this association up to a functor from the category
$\mathbf{Alg}(\Lambda)$ to the category $\mathbf{Alg}(\Sigma)$
(probably reflecting that the time was not still ripe for a
generalized acceptance of the category-theoretical way of thinking in
mathematics).

For a morphism $(\Phi,P)$ from $\Sigma$ to $\Lambda$, a
$\Sigma$-algebra $\mathbf{A}$, and a $\Lambda$-algebra $\mathbf{B}$,
the association just recalled led Fujiwara to take (in~\cite{tF59},
pp.  155--156, through the so-called \emph{algebraic Taylor's
expansion theorem}), as $(\Phi,P)$-mappings from $\mathbf{A}$ to
$\mathbf{B}$ those families of mappings $(f_{\mu})_{\mu\in p}\in
\prod_{\mu\in p}\mathrm{Hom}(A,B)$ such that the mapping
$\tp{f_{\mu}}_{\mu\in p}$ from $A$ to $B^{\Phi}$ is a
$\Sigma$-homomorphism from the $\Sigma$-algebra $\mathbf{A}$ to the
$\Sigma$-algebra $P(\mathbf{B})$, or, what is equivalent, those
families $(f_{\mu})_{\mu\in p}$ such that, for every $n\in
\mathbb{N}$, every $\sigma\in \Sigma_{n}$ and any elements
$a_{0},\ldots,a_{n-1}$ in $A$, all the identities obtained
by substituting $(f_{\mu})_{\mu\in p}$,
$F_{\sigma}^{\mathbf{A}}$, $P_{\varphi_{\mu},\sigma}^{n,\mathbf{B}}$, and
$a_{0},\ldots,a_{n-1}$ for $(\varphi_{\mu})_{\mu\in p}$,
$\sigma$, $P_{\varphi_{\mu},\sigma}^{n}$, and
$v_{0},\ldots,v_{n-1}$, respectively, in all formal equations
$$
\varphi_{\mu}(\sigma(v_{0},\ldots,v_{n-1})) =
P_{\varphi_{\mu},\sigma}^{n}
\left(
\begin{matrix}
\varphi_{0}(v_{0}), & \cdots & ,\varphi_{0}(v_{n-1})\\
\hdotsfor{3} \\
\varphi_{\mu}(v_{0}), & \cdots & ,\varphi_{\mu}(v_{n-1})\\
\hdotsfor{3} \\
\varphi_{p-1}(v_{0}), & \cdots & ,\varphi_{p-1}(v_{n-1})
\end{matrix}
\right)
$$
are true in $\mathbf{B}$.

\begin{example}
Let $\Sigma = (\Sigma_{n})_{n\in\mathbb{N}}$ be a single-sorted
signature, $\Phi = \{\,\varphi_{\mu}\mid \mu\in m\,\}$, with $m\in
\mathbb{N}$, and $P = (P^{n})_{n\in \mathbb{N}}$ the family defined,
for every natural number $n\in \mathbb{N}$, as follows
$$
  P^{n}\nfunction
  {\Phi\times \Sigma_{n}}{\mathrm{T}_{\Sigma}(\Phi\times\vs{v_{n}})}
  {(\varphi_{\mu},\sigma)}
  {\sigma(\varphi_{\mu}(v_{0}),\ldots,\varphi_{\mu}(v_{n-1}))}
$$
Then, for a $\Sigma$-algebra $\mathbf{B}$, $P(\mathbf{B})$ is
$\mathbf{B}^{m}$, the direct $m$-power of $\mathbf{B}$.
\end{example}

\begin{example}
Let $\Sigma = (\Sigma_{n})_{n\in\mathbb{N}}$ be a single-sorted
signature such that $\Sigma_{2} = \{\,+,-,\cdot\,\}$ and $\Sigma_{n} =
\vacio$, if $n\neq 2$, $\Phi = \{\,\varphi_{\mu}\mid \mu\in
\mathbb{N}\,\}$, and $P = (P^{n})_{n\in \mathbb{N}}$ the family
defined, for  $n \neq 2$, as the unique mapping from $\vacio$ to
$\mathrm{T}_{\Sigma}(\Phi\times\vs{v_{n}})$, and, for $n = 2$, as follows
\begin{enumerate}
\item $P^{2}_{\varphi_{\mu},+} =
      \varphi_{\mu}(v_{0})+\varphi_{\mu}(v_{1})$.

\item $P^{2}_{\varphi_{\mu},-} =
      \varphi_{\mu}(v_{0})-\varphi_{\mu}(v_{1})$.

\item $P^{2}_{\varphi_{\mu},\cdot} =
      \sum_{i= 0}^{\mu}\varphi_{\mu-i}(v_{0})\cdot\varphi_{\mu}(v_{1})$.
\end{enumerate}
Then, for a ring $\mathbf{B}$, $P(\mathbf{B})$ can be considered as
the ring of formal power series over $\mathbf{B}$.
\end{example}

\begin{example}
Let $\Sigma = (\Sigma_{n})_{n\in\mathbb{N}}$ be a single-sorted
signature such that $\Sigma_{2} = \{\,+,-,\cdot\,\}$ and $\Sigma_{n} =
\vacio$, if $n\neq 2$, $\Phi = \{\,\varphi_{\mu,\nu}\mid (\mu,\nu)\in
m\times m\,\}$, with $m\in \mathbb{N}$, and $P = (P^{n})_{n\in
\mathbb{N}}$ the family defined, for $n \neq 2$, as the unique mapping
from $\vacio$ to $\mathrm{T}_{\Sigma}(\Phi\times\vs{v_{n}})$, and, for
$n = 2$, as follows
\begin{enumerate}
\item $P^{2}_{\varphi_{\mu,\nu},+} =
      \varphi_{\mu,\nu}(v_{0})+\varphi_{\mu,\nu}(v_{1})$.

\item $P^{2}_{\varphi_{\mu,\nu},-} =
      \varphi_{\mu,\nu}(v_{0})-\varphi_{\mu,\nu}(v_{1})$.

\item $P^{2}_{\varphi_{\mu,\nu},\cdot} =
      \sum_{\lambda=
      0}^{m-1}\varphi_{\mu,\lambda}(v_{0})\cdot\varphi_{\lambda,\nu}(v_{1})$.
\end{enumerate}
Then, for a ring $\mathbf{B}$, $P(\mathbf{B})$ can be considered as
the matrix ring of degree $m$ over $\mathbf{B}$.

We point out that under this example falls the concept of derivation
from a ring into a like one.  We recall that, for a pair of rings
$\mathbf{A}$ and $\mathbf{B}$ and two ring homomorphisms $f$ and $g$ from
$\mathbf{A}$ to $\mathbf{B}$, a mapping $d$ from $A$ to $B$ is called
an $(f,g)$-derivation from $\mathbf{A}$ to $\mathbf{B}$ iff, for every
$x,y\in A$, we have that
\begin{enumerate}
\item $d(x+y) = d(x)+ d(y)$.

\item $d(xy) = d(x)f(y)+ g(x)d(y)$.
\end{enumerate}
Thus defined $d$ is not a ring homomorphism from $\mathbf{A}$ to
$\mathbf{B}$, however, if $d$ is any $(f,g)$-derivation from
$\mathbf{A}$ to $\mathbf{B}$, then the matrix
$\left(
\begin{smallmatrix}
f & \kappa_{0} \\
d & g
\end{smallmatrix}
\right) $, where $\kappa_{0}$ is the mapping from $A$ to $B$ that is
constantly $0$, defines a ring homomorphism from the ring $\mathbf{A}$
to $P(\mathbf{B}) = \mathbf{B}^{2\times 2}$, the matrix ring of degree
$2$ over $\mathbf{B}$, i.e., the matrix is a $(\Phi,P)$-mapping from
$\mathbf{A}$ to $\mathbf{B}$.
\end{example}

\begin{example}
Let $\Sigma = (\Sigma_{n})_{n\in\mathbb{N}}$ be a single-sorted
signature such that $\Sigma_{2} = \{\,+,-,\cdot\,\}$ and $\Sigma_{n} =
\vacio$, if $n\neq 2$, $\Lambda = (\Lambda_{n})_{n\in\mathbb{N}}$ a
single-sorted signature such that $\Lambda_{2} = \{\,+',-',\cdot'\,\}$
and $\Lambda_{n} = \vacio$, if $n\neq 2$, $\Phi = \{\,\varphi\,\}$,
and $P = (P^{n})_{n\in \mathbb{N}}$ the family defined, for $n \neq
2$, as the unique mapping from $\vacio$ to
$\mathrm{T}_{\Lambda}(\Phi\times\vs{v_{n}})$, and, for $n = 2$, as follows
\begin{enumerate}
\item $P^{2}_{\varphi,+} =
      \varphi(v_{0})+'\varphi(v_{1})$.

\item $P^{2}_{\varphi,-} =
      \varphi(v_{0})-'\varphi(v_{1})$.

\item $P^{2}_{\varphi,\cdot} =
      \varphi(v_{0})\cdot'\varphi(v_{1})-'\varphi(v_{1})\cdot'\varphi(v_{0})$.
\end{enumerate}
Then, for a ring $\mathbf{B}$, $P(\mathbf{B})$ is a Lie ring.
\end{example}

After this Fujiwara proceeds, among other things, to define the
composition of morphisms between single-sorted signatures.  But for
this he introduces (in~\cite{tF59}, pp.  157--160) a certain notation
that we explain as follows.  Given a morphism $(\Phi,P)\colon
\Sigma\mor \Lambda$, an index $\mu\in p$, and a set $X$, we denote by
$F^{(\Phi,P),X}_{\varphi_{\mu}}$ the $\Sigma$-homomorphism from
$\mathbf{T}_{\Sigma}(X)$ to $\mathbf{T}_{\Lambda}(\Phi\times X)$
obtained as the composition of the vertical mappings in the following
diagram
$$
\xymatrix@C=80pt@R=40pt{
X
  \ar[r]^{\eta_{X}}
  \ar[rd]_{\eta^{\flat}_{\Phi\times X}} &
\mathrm{T}_{\Sigma}(X)
  \ar[d]|{(\eta^{\flat}_{\Phi\times X})^{\sharp}}
  \ar`r[dd]+/r4pc/ `[dd]^{F^{(\Phi,P),X}_{\varphi_{\mu}}} [dd] \\
{} &
\mathrm{T}_{\Lambda}(\Phi\times X)^{\Phi}
  \ar[d]^{\pr_{\mu}} \\
{} &
\mathrm{T}_{\Lambda}(\Phi\times X)
}
$$
where $\eta^{\flat}_{\Phi\times X}\colon X\mor
\mathrm{T}_{\Lambda}(\Phi\times X)^{\Phi}$ is the adjunct to
$\eta_{\Phi\times X}\colon \Phi\times X\mor
\mathrm{T}_{\Lambda}(\Phi\times X)$, and $\mathrm{pr}_{\mu}$ the
$\mu$-th projection from $\mathbf{T}_{\Lambda}(\Phi\times X)^{\Phi}$ to
$\mathbf{T}_{\Lambda}(\Phi\times X)$.

Now given $(\Phi,P)\colon \Sigma\mor \Lambda$ and $(\Psi,Q)\colon
\Lambda\mor \Omega$, with $\Phi = \{\,\varphi_{\mu}\mid \mu\in p\,\}$
and $\Psi = \{\,\psi_{\nu}\mid \nu\in q\,\}$, Fujiwara defines
(in~\cite{tF59}, p.  164) the composition of $(\Phi,P)$ and $(\Psi,Q)$
as the pair $(\Psi\times \Phi, Q\comp P)$, where, for every $n\in
\mathbb{N}$, $(Q\comp P)^{n}$ is the mapping
$$
  (Q\comp P)^{n}\nfunction
  {(\Psi\times\Phi)\times \Sigma_{n}}
  {\mathrm{T}_{\Omega}((\Psi\times\Phi)\times\vs{v_{n}})}
  {((\psi_{\nu},\varphi_{\mu}),\sigma)}
  {(Q\comp P)^{n}_{(\psi_{\nu},\varphi_{\mu}),\sigma} =
  F_{\psi_{\nu}}^{(\Psi,Q),\Phi\times \vs{v_{n}}}(P^{n}_{\varphi_{\mu},\sigma})}
$$
from this he obtains (in~\cite{tF59}, pp.  165--166), for every
$\Omega$-algebra $\mathbf{C}$, an isomorphism between the
$\Sigma$-algebras $Q(P(\mathbf{C}))$ and $Q\comp P(\mathbf{C})$.
However he does not consider in~\cite{tF59} neither the corresponding
category with objects the single-sorted signatures, morphisms between
single-sorted signatures the families of basic mapping-formulas, and
composition the just defined composition of morphisms, nor, obviously,
the contravariant pseudo-functor defined on such a category of
signatures, although the essential components of the subject matter
are there (and patiently waiting).


In this section, following the work by Fujiwara in~\cite{tF59}, we
generalize the morphisms in $\mathbf{Sig}$ in such a way that a
signature morphism from a signature into a like one, to be called from
now on a \emph{morphism of Fujiwara}, or more briefly, a
\emph{polyderivor}, will consist of two suitably related mappings: On the
one hand, a mapping which relates the sets of sorts of the signatures
and assigns to each sort in the first, a derived sort in the second,
i.e., a word on the set of sorts of the second, and, on the other
hand, a mapping which assigns to each formal operation in the first, a
family of terms in the second, all in such a way that both
transformations are compatible.  This type of signature morphism, from
which we will get a category $\mathbf{Sig}_{\mathfrak{pd}}$, with
the same objects that $\mathbf{Sig}$, will allow us to generalize,
concordantly, the morphisms between algebras.

We will also prove that the category $\mathbf{Sig}_{\mathfrak{pd}}$
of signatures and polyderivors is isomorphic to the Kleisli category
for a monad in $\mathbf{Sig}$, and that fact will confirm, to some
extend, the naturalness of the concept of polyderivor.  Furthermore,
the contravariant functor $\Alg\colon \mathbf{Sig}\mor
\mathbf{Cat}$ will be lifted up to a contravariant pseudo-functor
$\Alg_{\mathfrak{pd}}\colon \mathbf{Sig}_{\mathfrak{pd}}\mor
\mathbf{Cat}$ and, by applying, once more, the construction of
Ehresmann-Grothendieck, we will get the category
$\mathbf{Alg}_{\mathfrak{pd}}$ of algebras and algebra morphisms
that will have the polyderivors as a component.  In the same way, the
pseudo-functor $\Ter\colon \mathbf{Sig}\mor \mathbf{Cat}$ will be
lifted up to a pseudo-functor $\mathrm{Ter}_{\mathfrak{pd}} \colon
\mathbf{Sig_{\mathfrak{pd}}} \mor \mathbf{Cat}$.

On the other hand, to account exactly for the invariant character of
the realization of terms in algebras under the polyderivors, we prove the
existence of a pseudo-extranatural transformation
$(\mathrm{Tr},\theta)$ from a pseudo-functor
$\Alg_{\mathfrak{pd}}(\farg)\bprod\mathrm{Ter}_{\mathfrak{pd}}(\farg)$,
on $\mathbf{Sig}^{\mathrm{op}}_{\mathfrak{pd}}\times
\mathbf{Sig}_{\mathfrak{pd}}$ to $\mathbf{Cat}$, to the functor
$\mathrm{K}_{\mathbf{Set}}$, between the same categories, that is
constantly $\mathbf{Set}$.


Before we define the polyderivors, we agree that
$\mathbb{T}_{\star}=(\mathrm{T}_{\star},\between,\cncat)$ is the
standard monad in $\mathbf{Set}$ for the monoid specification, where,
for every set $S$, $\mathrm{T}_{\star}(S) = \fmon{S}$ is the set
$\bigcup_{n\in \mathbb{N}}S^{n}$, $\between_{S}\colon S\mor \fmon{S}$
the inclusion of $S$ into $\fmon{S}$, and $\cncat_{S}\colon
\ffmon{S}\mor\fmon{S}$ the merging of strings of words to words.  To
simplify the notation, we will write $(s)$ instead of
$\between_{S}\!\!(s)$.  Furthermore, if $\varphi\colon S\mor \fmon{T}$,
then $\ext{\varphi}\colon \fmon{S}\mor\fmon{T}$ is the underlying
mapping of the canonical extension of $\varphi$ to the free monoid
$\mathbf{T}_{\star}(S)$ on $S$ and $\fmon{\varphi}$ the unique monoid
homomorphism from $\mathbf{T}_{\star}(S)$ to
$\mathbf{T}_{\star}(T^{\star})$, the free monoid on the underlying set
of the free monoid on $T$, such that $\fmon{\varphi}\comp \between_{S}
= \between_{\fmon{T}}\comp \varphi$.

\begin{definition}
Let $\mathbf{\Sigma}$ and $\mathbf{\Lambda}$ be signatures.  A
\emph{polyderivor} from $\mathbf{\Sigma}$ to $\mathbf{\Lambda}$ is a pair $\mathbf{d}
= (\varphi,d)$, where $\varphi\colon S\mor\fmon{T} $ while $d\colon
\Sigma\mor
\mathrm{BTer}_{T}(\Lambda)_{\ext{\varphi}\bprod\varphi}$.
\end{definition}

Therefore, if $\mathbf{d}\colon\mathbf{\Sigma}\mor \mathbf{\Lambda}$
is a polyderivor, then, for every $(w,s)\in\fmon{S}\bprod S$, we have that
$$
d_{w,s}\colon\Sigma_{w,s}\mor
\mathrm{BTer}_{T}(\Lambda)_{\ext{\varphi}(w),\varphi(s)}=
\mathrm{T}_{\mathbf{\Lambda}}(\vs{\ext{\varphi}(w)})_{\varphi(s)},
$$
and, given that $\Delta_{\ext{\varphi}\bprod\varphi}=\Delta_{1 \bprod
\between_{S}} \comp \Delta_{\ext{\varphi}\bprod\ext{\varphi}}$ and the
functor $\tcoprod_{1 \bprod \between_{S}}$ is left adjoint to the
functor $\Delta_{1 \bprod \between_{S}}$, $d$ is, essentially, an
$\fmon{S}\bprod\fmon{S}$-sorted mapping
$$
\theta^{1\bprod \between_{S}}(d)\colon \tcoprod_{1 \bprod
\between_{S}}\Sigma\mor
\mathrm{BTer}_{T}(\Lambda)_{\ext{\varphi}\bprod\ext{\varphi}}.
$$
From now on, for every polyderivor $\mathbf{d}$, we identify $d$ and
$\theta^{1\bprod \between_{S}}(d)$.

\begin{remark}
For every word $w$ in $\fmon{S}$, $\ext{\varphi}(w) =
\varphi(w_{0})\cncat\cdots \cncat\varphi(w_{\bb{w}-1})$ is a word on
$T$ and if we agree that, for every $\alpha\in \bb{w}$, $p_{\alpha} =
\bb{\varphi(w_{\alpha})}$, then, for every $\alpha\in \bb{w}$ and
$i\in \sum_{\alpha\in \bb{w}}p_{\alpha}$, we have that
$$
\ext{\varphi}(w)_{i} = \varphi(w_{\alpha})_{i-\sum_{\beta\in
\alpha}p_{\beta}}\quad \text{iff} \quad \textstyle
\sum_{\beta\in \alpha}p_{\beta}\leq i\leq
\sum_{\beta\in \alpha+1}p_{\beta}-1.
$$
Hence $\ext{\varphi}(w)$ has the configuration
\begin{multline*}
\ext{\varphi}(w)=
  (
  \overbrace{
    \ext{\varphi}(w)_{0},\ldots,\ext{\varphi}(w)_{p_{0}-1}
    }^{\varphi(w_{0})}
  ,\ldots  \\
  \ldots ,
  \overbrace{
  \ext{\varphi}(w)_{\sum_{\beta\in \alpha}p_{\beta}}, \ldots,
  \ext{\varphi}(w)_{\sum_{\beta\in \alpha+1}p_{\beta}-1}
    }^{\varphi(w_{\alpha})}
  , \ldots \\
  \ldots ,
  \overbrace{
    \ext{\varphi}(w)_{\sum_{\beta\in \bb{w}-1}p_{\beta}},
    \ldots, \ext{\varphi}(w)_{\sum_{\beta\in \bb{w}}p_{\beta}-1}
    }^{\varphi(w_{\bb{w}-1})}
    ).
\end{multline*}
Then, for a formal operation $\sigma\colon w\mor s$ in $\Sigma$, it is
useful to think about the family of terms $d_{w,s}(\sigma)\in
\mathrm{T}_{\mathbf{\Lambda}}(\vs{\ext{\varphi}(w)})_{\varphi(s)}$,
denoted by $d(\sigma)\colon\ext{\varphi}(w)\mor\varphi(s)$, as being
graphi\-cal\-ly represented by
$$
d(\sigma):
\begin{pmatrix}
 \ext{\varphi}(w)_{0} & \cdots & \ext{\varphi}(w)_{p_{0}-1}  \\
 \vdots & \ddots & \vdots \\
 \ext{\varphi}(w)_{\sum_{\beta\in \bb{w}-1}p_{\beta}} & \cdots &
 \ext{\varphi}(w)_{\sum_{\beta\in \bb{w}}p_{\beta}-1}
\end{pmatrix}
\mor
(\varphi(s)_{0},\ldots,\varphi(s)_{\bb{\varphi(s)}-1}).
$$
\end{remark}

For every signature $\mathbf{\Lambda}$, $\mathrm{BTer}_{T}(\Lambda)$
is the underlying $\mathrm{ms}$-set of $\mathbf{BTer}_{T}(\Lambda)$,
the Bénabou algebra for $(T,\Lambda)$, and
$\mathbf{BTer}_{T}(\Lambda)$ is isomorphic to
$\mathbf{T}_{\Ben_{T}}(\coprod_{1\bprod\between_{T}}\Lambda)$, by
Corollary~\ref{Beniso}, hence the polyderivors can also be defined as
pairs $\mathbf{d} = (\varphi,d)$, where $\varphi\colon S\mor \fmon{T}$
while $d$ is an $\fmon{S}\bprod S$-sorted mapping from $\Sigma$ to
$\mathrm{T}_{\Ben_{T}}(\coprod_{1\bprod\between_{T}}\Lambda)_{\ext{\varphi}\bprod
\varphi}$, or, equivalently, an $\fmon{S}\bprod \fmon{S}$-sorted
mapping from $\coprod_{1\bprod\between_{S}}\Sigma$ to
$\mathrm{T}_{\Ben_{T}}(\coprod_{1\bprod\between_{T}}\Lambda)_{\ext{\varphi}\bprod
\ext{\varphi}}$.


\begin{example}
Let $\mathbf{\Sigma}$ be a signature and $p\in \mathbb{N}$.  Then
taking
\begin{enumerate}
\item As $\varphi\colon S\mor \fmon{S}$ the mapping which sends
      $s\in S$ to the word $\cncat_{\mu\in p}(s)$ and,

\item For $(w,s)\in \fmon{S}\times S$, as $d_{w,s}$ the mapping
      from $\Sigma_{w,s}$ to
      $\mathrm{T}_{\mathbf{\Sigma}}(\vs{\ext{\varphi}(w)})_{s}^{p}$
      (because, in this case,
      $\mathrm{T}_{\mathbf{\Sigma}}(\vs{\ext{\varphi}(w)})_{\varphi(s)} =
      \mathrm{T}_{\mathbf{\Sigma}}(\vs{\ext{\varphi}(w)})_{s}^{p}$),
      which sends $\sigma\in\Sigma_{w,s}$ to
      $$
        (\sigma(v_{0}^{w_{0}}, v_{p}^{w_{1}},\ldots,
                v_{(\bb{w}-1)p}^{w_{\bb{w}-1}}),\ldots,
         \sigma(v_{p-1}^{w_{0}}, v_{2p-1}^{w_{1}},\ldots,
                v_{\bb{w}p-1}^{w_{\bb{w}-1}})),
      $$
      in $\mathrm{T}_{\mathbf{\Sigma}}(\vs{\ext{\varphi}(w)})_{s}^{p}$,
\end{enumerate}
we have that $\mathbf{d} = (\varphi,d)$ is a polyderivor from
$\mathbf{\Sigma}$ into itself.  Later on, after having defined the
category $\mathbf{Sig}_{\mathfrak{pd}}$, of signatures and polyderivors,
and the pseudo-functor $\mathrm{Alg}_{\mathfrak{pd}}$ from
$\mathbf{Sig}_{\mathfrak{pd}}$ to $\mathbf{Cat}$, we will see that, for
a signature $\mathbf{\Sigma}$, a natural number $p\in \mathbb{N}$, and
a $\mathbf{\Sigma}$-algebra $\mathbf{B}$, the result of the action of
$\mathrm{Alg}_{\mathfrak{pd}}$ on $\mathbf{B}$ is $\mathbf{B}^{p}$.

For more examples of polyderivors we refer to the last section of this
paper, where we consider polyderivors between the (many-sorted) signatures
of Hall and Bénabou.
\end{example}

\begin{example}
Let $\Sigma = (\Sigma_{n})_{n\in\mathbb{N}}$ and $\Lambda =
(\Lambda_{n})_{n\in\mathbb{N}}$ be two single-sorted signatures and
$(\Phi,P)$, with $\Phi = \{\,\varphi_{\mu}\mid \mu\in p\,\}$, a
family of basic mapping-formulas from $\Sigma$ to $\Lambda$ as defined
by Fujiwara in~\cite{tF59}, p.  155.  Then by associating
\begin{enumerate}
\item To the single-sorted signatures $\Sigma$ and $\Lambda$,
      the signatures $(1,(\Sigma_{n,0})_{(n,0)\in \fmon{1}\times 1})$
      and $(1,(\Lambda_{n,0})_{(n,0)\in \fmon{1}\times 1})$,
      respectively, where, for every $n\in \fmon{1}\cong \mathbb{N}$,
      $\Sigma_{n,0} = \Sigma_{n}$ and $\Lambda_{n,0} = \Lambda_{n}$,
      and

\item To the morphism $(\Phi,P)$ the pair $(\kappa_{p},d)$,
      where $\kappa_{p}$ is the mapping from $1$ to $\fmon{1}$ which
      sends $0\in 1$ to $p\in \fmon{1}$ and $d$ the $\fmon{1}\times
      1$-sorted mapping from $(\Sigma_{n,0})_{(n,0)\in \fmon{1}\times
      1}$ to
      $(\mathrm{T}_{\mathbf{\Lambda}}(\vs{\ext{\kappa_{p}}(n)})_{\kappa_{p}(0)})
      _{(n,0)\in \fmon{1}\times 1}\cong
      (\mathrm{T}_{\mathbf{\Lambda}}(\Phi\times \vs{v_{n}})^{p})_{n\in
      \mathbb{N}}$ which, for $n\in \fmon{1}$, sends $\sigma\in
      \Sigma_{n}$ to $d_{n,0}(\sigma) =
      (P^{n}_{\varphi_{0},\sigma},\ldots,P^{n}_{\varphi_{p-1},\sigma})$,
    \end{enumerate}
we have that the families of basic mapping-formulas defined by
Fujiwara for the single-sorted case fall, obviously, under the concept
of polyderivor.  Consequently, all the examples we have put in this
section before the definition of polyderivor, once reformulated as just
said, are also examples of polyderivors.
\end{example}

\begin{example}
A standard signature morphism from a signature $(S,\Sigma)$ into a
like one $(T,\Lambda)$, as defined in the first section, is a pair of
mappings $(\varphi,d)$, where $\varphi\colon S\mor T$ is a mapping in
$\mathbf{Set}$, which sends sort symbols to sort symbols, while
$d\colon \Sigma\mor \Lambda_{\fmon{\varphi}\bprod \varphi}$ is a
morphism in $\mathbf{Sig}(S)$, which sends formal operations to formal
operations, respecting the assignment of sorts.  But from
$\varphi\colon S\mor T$ we get $\between_{T}\comp\varphi\colon
S\mor\fmon{T}$, and from $d\colon \Sigma\mor
\Lambda_{\fmon{\varphi}\bprod \varphi}$, because there exists a
canonical embedding from $\Lambda_{\fmon{\varphi}\bprod \varphi}$ into
$(\coprod_{1\bprod\between_{T}}\Lambda)_{(\between_{T}\comp\varphi)^{\sharp}
\times(\between_{T}\comp\varphi)}$, we get the composite mapping
$$
\xymatrix@C=20pt{
\Sigma
\ar[r]^-{d} &
\Lambda_{\fmon{\varphi}\bprod \varphi}
\ar[r]^{} &
(\coprod_{1\bprod\between_{T}}\Lambda)_{(\between_{T}\comp\varphi)^{\sharp}
\times(\between_{T}\comp\varphi)}
\ar[r]^{} &
\mathrm{T}_{\Ben_{T}}
(\coprod_{1\bprod\between_{T}}\Lambda)_{(\between_{T}\comp\varphi)^{\sharp}
\times(\between_{T}\comp\varphi)}.
}
$$
In this way we have assigned to every standard signature morphism a
corresponding polyderivor from the source to the target signature of the
signature morphism.  Thus the standard signature morphisms fall under
the concept of polyderivor.
\end{example}

Our next goal is to define the composition of polyderivors in order to get
the category $\mathbf{Sig}_{\mathfrak{pd}}$, of signatures and
polyderivors.  To attain the just stated goal we need to recall beforehand
the concept of derivor from a signature into a like one as defined by
Goguen, Thatcher and Wagner in~\cite{gtw76}, p.  86, and to set out
some of its properties.

\begin{definition}
Let $\mathbf{\Sigma}$ and $\mathbf{\Lambda}$ be signatures.  Then a
\emph{derivor from} $\mathbf{\Sigma}$ \emph{to} $\mathbf{\Lambda}$ is
a pair $\mathbf{d} = (\varphi,d)$, with $\varphi\colon S\mor T$ and
$d\colon \Sigma\mor \mathrm{HTer}_{T}(\Lambda)_{\fmon{\varphi}\bprod
\varphi}$.
\end{definition}

Therefore, if $\mathbf{d}\colon\mathbf{\Sigma}\mor\mathbf{\Lambda}$ is
a derivor, then, for every $(w,s)\in\fmon{S}\bprod S$, we have that
$$
d_{w,s}\colon \Sigma_{w,s}\mor
\mathrm{HTer}_{T}(\Lambda)_{\fmon{\varphi}(w),\varphi(s)}=
\mathrm{T}_{\Lambda}(\vs{\fmon{\varphi}(w)})_{\varphi(s)}
$$
sends a formal operation $\sigma\colon w\mor s$ to a term
$d_{w,s}(\sigma)\colon\fmon{\varphi}(w)\mor \varphi(s)$, i.e., a term
for $\Lambda$ of type $(\vs{\fmon{\varphi}(w)},\varphi(s))$, and all in
such a way that the arities and coarities are preserved, modulus the
correspondence between the sorts given by the mapping $\varphi$.

For a signature $\mathbf{\Lambda}$, we have that
$\mathrm{HTer}_{T}(\Lambda)$ is the underlying $\mathrm{ms}$-set of
$\mathbf{HTer}_{T}(\Lambda)$, the Hall algebra for $(T,\Lambda)$.  But
$\mathbf{HTer}_{T}(\Lambda)$ is isomorphic to
$\mathbf{T}_{\Hall_{T}}(\Lambda)$, the free
$\Hall_{T}$-algebra on $\Lambda$, by
Proposition~\ref{iso:FrH-TerH}.  Consequently the derivors can be
defined, alternative, but equivalently, as pairs $\mathbf{d} =
(\varphi,d)$ with $\varphi\colon S\mor T$ and $d\colon\Sigma\mor
\mathrm{T}_{\Hall_{T}}(\Lambda)$.  Thus, taking into account the
equivalence between the categories $\mathbf{Alg}(\Hall_{T})$ and
$\mathbf{Alg}(\Ben_{T})$, we can state the following

\begin{corollary}
Every derivor is a polyderivor (although, obviously, not every polyderivor is
a derivor).
\end{corollary}

\begin{remark}
For a \emph{single-sorted} signature $\Sigma$, a derivor
from $\Sigma$ to $\Sigma$, i.e., an endoderivor of $\Sigma$, is,
essentially, a morphism $d$ in $\mathbf{Set}^{\mathbb{N}}$ from
$\Sigma$ to $(\mathrm{T}_{\Sigma}(\vs{v_{n}}))_{n\in \mathbb{N}}$.
Therefore, for every $n\in \mathbb{N}$, we have that
$$
d_{n}\colon \Sigma_{n}\mor \mathrm{T}_{\Sigma}(\vs{v_{n}})
$$
sends a formal operation $\sigma\in \Sigma_{n}$ to a term
$d(\sigma)\in \mathrm{T}_{\Sigma}(\vs{v_{n}})$.  From this it follows
that the concept of \emph{hypersubstitution} as defined, e.g., by Denecke
and Wismath in~\cite{dw00}, p.  13 (and itself the main tool from
which it is constructed the theory of \emph{hyperidentities}), is a
particular case of the concept of derivor and, consequently, also of
that of polyderivor.
\end{remark}

The following two are examples of derivors that originate in the work
by Higman and B.H. Neumann in group theory (see~\cite{hn52}).

\begin{example}
Let $\Sigma = (\Sigma_{n})_{n\in\mathbb{N}}$ be a single-sorted
signature such that $\Sigma_{0} = \{\,1\,\}$, $\Sigma_{2} = \{\,/\,\}$
and $\Sigma_{n} = \vacio$, if $n\neq 0, 2$, $\Lambda =
(\Lambda_{n})_{n\in\mathbb{N}}$ a single-sorted signature such that
$\Lambda_{0} = \{\,1\,\}$, $\Lambda_{1} = \{\,{}^{-1}\,\}$,
$\Lambda_{2} = \{\,\cdot\,\}$ and $\Lambda_{n} = \vacio$, if $n\neq
0,1,2$, and $d = (d_{n})_{n\in \mathbb{N}}$ the mapping from $\Sigma$
to $\Lambda$ whose $n$-th coordinate mapping is, for $n \neq 0, 2$,
the unique mapping from $\vacio$ to
$\mathrm{T}_{\Lambda}(\vs{v_{n}})$, and, for $n = 0, 2$, that defined
as follows
\begin{enumerate}
\item $d_{0}(1) = 1$;

\item $d_{2}(/)  = v_{0}\cdot v_{1}^{-1}$.
\end{enumerate}
Then $d$ is a derivor from $\Sigma$ to $\Lambda$.  This derivor
defines $/$, the \lq\lq division\rq\rq, in terms of $\cdot$, the
\lq\lq multiplication\rq\rq, and ${}^{-1}$, the \lq\lq inverse\rq\rq.
\end{example}

\begin{example}
For the same two single-sorted signatures as in the preceding example
let $e = (e_{n})_{n\in \mathbb{N}}$ be the mapping from $\Lambda$ to
$\Sigma$ whose $n$-th coordinate mapping is, for $n \neq 0,1, 2$, the
unique mapping from $\vacio$ to $\mathrm{T}_{\Sigma}(\vs{v_{n}})$,
and, for $n = 0,1, 2$, that defined as follows
\begin{enumerate}
\item $e_{0}(1) = 1$;

\item $e_{1}({}^{-1})  = 1/ v_{0}$;

\item $e_{2}(\cdot)  = v_{0}/(1/ v_{1})$.
\end{enumerate}
Then $e$ is a derivor from $\Lambda$ to $\Sigma$.  This derivor
defines $\cdot$, the \lq\lq multiplication\rq\rq, and ${}^{-1}$, the
\lq\lq inverse\rq\rq, in terms of $/$, the \lq\lq division\rq\rq.
\end{example}

The next two examples of derivors come from the proof by M. H. Stone
of the subsumption of the theory of Boolean algebras under the theory
of rings (see~\cite{mst36}, pp. 43--48).

\begin{example}
Let $\Sigma = (\Sigma_{n})_{n\in\mathbb{N}}$ be a single-sorted
signature such that $\Sigma_{0} = \{\,0,1\,\}$, $\Sigma_{1} =
\{\,'\,\}$, $\Sigma_{2} = \{\,\wedge,\vee\,\}$ and $\Sigma_{n} =
\vacio$, if $n\neq 0,1, 2$, $\Lambda = (\Lambda_{n})_{n\in\mathbb{N}}$
a single-sorted signature such that $\Lambda_{0} = \{\,0,1\,\}$,
$\Lambda_{1} = \{\,-\,\}$, $\Lambda_{2} = \{\,\cdot,+\,\}$ and
$\Lambda_{n} = \vacio$, if $n\neq 0,1,2$, and $d = (d_{n})_{n\in
\mathbb{N}}$ the mapping from $\Sigma$ to $\Lambda$ whose $n$-th
coordinate mapping is, for $n \neq 0,1, 2$, the unique mapping from
$\vacio$ to $\mathrm{T}_{\Lambda}(\vs{v_{n}})$, and, for $n = 0,1, 2$,
that defined as follows
\begin{enumerate}
\item $d_{0}(0) = 0$ and $d_{0}(1) = 1$;

\item $d_{1}(')  = 1+ v_{0}$;

\item $d_{2}(\wedge)  = v_{0}\cdot v_{1}$ and
      $d_{2}(\vee)  = v_{0}+ v_{1} + v_{0}\cdot v_{1}$.
\end{enumerate}
Then $d$ is a derivor from $\Sigma$ to $\Lambda$.  This derivor
defines the \lq\lq Boolean algebra\rq\rq op\-era\-tions in
terms of the \lq\lq Boolean ring\rq\rq op\-era\-tions.
\end{example}

\begin{example}
For the same two single-sorted signatures as in the preceding example let
$e = (e_{n})_{n\in \mathbb{N}}$ be the mapping from $\Lambda$ to $\Sigma$
whose $n$-th coordinate mapping is, for $n \neq 0,1, 2$, the unique
mapping from $\vacio$ to $\mathrm{T}_{\Sigma}(\vs{v_{n}})$, and, for
$n = 0,1, 2$, that defined as follows
\begin{enumerate}
\item $e_{0}(0) = 0$ and $e_{0}(1) = 1$;

\item $e_{1}(-)  = v_{0}$;

\item $e_{2}(\cdot)  = v_{0}\wedge v_{1}$ and
      $e_{2}(+)  = (v_{0}\wedge v_{1}') \vee (v_{0}'\wedge v_{1})$.
\end{enumerate}
Then $e$ is a derivor from $\Lambda$ to $\Sigma$.  This derivor
defines the \lq\lq Boolean ring\rq\rq op\-era\-tions in terms of the
\lq\lq Boolean algebra\rq\rq op\-era\-tions.
\end{example}

Another example of derivor is that provided by Gödel (see~\cite{kG33}) in
his work about an interpretation of the intuitionistic propositional
logic into a modal extension of the classical propositional logic.

\begin{example}
Let $\Sigma = (\Sigma_{n})_{n\in\mathbb{N}}$ be a single-sorted
signature such that $\Sigma_{1} = \{\,\neg_{\mathrm{i}}\,\}$,
$\Sigma_{2} =
\{\,\wedge_{\mathrm{i}},\vee_{\mathrm{i}},\to_{\mathrm{i}}\,\}$ and
$\Sigma_{n} = \vacio$, if $n\neq 1, 2$, $\Lambda =
(\Lambda_{n})_{n\in\mathbb{N}}$ a single-sorted signature such that
$\Lambda_{1} = \{\,\neg_{\mathrm{c}},\Box\,\}$, $\Lambda_{2} =
\{\,\wedge_{\mathrm{c}},\vee_{\mathrm{c}},\to_{\mathrm{c}}\,\}$, and
$\Lambda_{n} = \vacio$, if $n\neq 1,2$, and $g = (g_{n})_{n\in
\mathbb{N}}$ the family defined, for $n \neq 1, 2$, as the unique
mapping from $\vacio$ to $\mathrm{T}_{\Lambda}(\vs{v_{n}})$, and, for
$n = 1, 2$, as follows
\begin{enumerate}
\item $g_{1}(\neg_{\mathrm{i}})  = \neg_{\mathrm{c}}\Box v_{0}$.

\item $g_{2}(\wedge_{\mathrm{i}})  = v_{0}\wedge_{\mathrm{c}} v_{1}$.

\item $g_{2}(\vee_{\mathrm{i}})  = \Box v_{0}\vee_{\mathrm{c}} \Box v_{1}$.

\item $g_{2}(\to_{\mathrm{i}})  = \Box v_{0}\to_{\mathrm{c}} \Box v_{1}$.
\end{enumerate}
Then $d$ is a derivor from $\Sigma$ to $\Lambda$.  This derivor
defines the \lq\lq intuitionistic\rq\rq connectives in terms of the
\lq\lq classical\rq\rq connectives together with $\Box$, the operator
of \lq\lq necessity\rq\rq.
\end{example}

Next we proceed to define the composition of derivors and to prove
that the corresponding category of signatures and derivors,
denoted by $\mathbf{Sig}_{\mathfrak{d}}$, is isomorphic to the
Kleisli category for a monad $\mathbb{T}_{\mathfrak{d}}$ in
$\mathbf{Sig}$.  This last result means, in other words, that the
derivors are indiscernible from the morphisms of the Kleisli
category for the monad $\mathbb{T}_{\mathfrak{d}}$ in
$\mathbf{Sig}$, thus confirming its mathematical naturalness.
By proceeding in this way we, on the one hand, move one step forward,
from the standpoint of category theory, in the investigation of some
of the most notable positive properties of the category
$\mathbf{Sig}_{\mathfrak{d}}$, with regard to what has been done
in~\cite{gtw76}, and, on the other hand, get a model for the
subsequent work we have to do with the polyderivors.

\begin{remark}
The derivors, that set up relations between signatures, and the
\emph{simple} entailment system morphisms (see~\cite{m89}, p.  297,
for the definition of this concept), that do a similar, although more
sophisticated, task, but for entailment systems, are related through
the following proportion: derivors are to standard signature morphisms
as simple entailment system morphisms are to standard (\emph{plain})
entailment system morphisms.  Besides, the proportionality arrives, as
a matter of fact, to the point that the simple entailment system
morphisms are also, essentially, the morphisms of a Kleisli category
for a monad in $\mathbf{Ent}$, the category of entailment systems and
standard (plain) entailment system morphisms.
\end{remark}

We point out that the definition of the composition of derivors, in
strong contrast with that of polyderivors below, is based on the standard
specification morphisms between Hall specifications.  Actually, if
instead of starting from a mapping $\varphi\colon S\mor \fmon{T}$, as
is the case for the polyderivors, we start from an ordinary mapping
$\varphi\colon S\mor T$, then, as we state next, we get a functor
$(\fmon{\varphi}\bprod \varphi,h^{\varphi})^{\ast}$ from the category
$\mathbf{Alg}(\mathrm{H}_{T})$, of Hall algebras for $T$, to the
category $\mathbf{Alg}(\mathrm{H}_{S})$, of Hall algebras for $S$ (and
the existence of such a functor will follow from that of a
specification morphism from $(\fmon{S}\times
S,\Sigma^{\Hall_{S}},\ec{E}^{\Hall_{S}})$ to $(\fmon{T}\times
T,\Sigma^{\Hall_{T}},\ec{E}^{\Hall_{T}})$).  This functor, in its
turn, will allows us to endow with a structure of Hall algebra for $S$
to the many-sorted set
$\mathrm{HTer}_{T}(\Lambda)_{\fmon{\varphi}\bprod \varphi}$, from
which the composition of derivors will be defined explicitly.

\begin{proposition}
Let $\varphi\colon S\mor T$ be a mapping.  Then the $\fmon{S}\bprod
S$-sorted mapping
$$
h^{\varphi} \colon \Sigma^{\Hall_{S}} \mor
\Sigma^{\Hall_{T}}_{\fmon{\varphi}\bprod \varphi}
$$
defined as follows
\begin{enumerate}
\item For every $w\in \fmon{S}$ and $i\in\bb{w}$,
      $h^{\varphi}(\pi^{w}_{i}) = \pi^{\fmon{\varphi}(w)}_{i}$,

\item For every $u,\,w\in\fmon{S}$ and $s\in S$,
       $h^{\varphi}(\xi_{u,w,s}) =
       \xi_{\fmon{\varphi}(u),\fmon{\varphi}(w),\varphi(s)}$,
\end{enumerate}
is such that $(\fmon{\varphi}\bprod \varphi,h^{\varphi}) \colon
(\fmon{S}\bprod S,\Sigma^{\Hall_{S}},\ec{E}^{\Hall_{S}}) \mor
(\fmon{T}\bprod T,\Sigma^{\Hall_{T}},\ec{E}^{\Hall_{T}})$ is a
specification morphism.  Therefore $\varphi\colon S\mor T$ induces a
functor $(\fmon{\varphi}\bprod \varphi,h^{\varphi})^{\ast}$ from
$\mathbf{Alg}(\mathrm{H}_{T})$ to $\mathbf{Alg}(\mathrm{H}_{S})$ which sends
$\mathbf{HTer}_{T}(\Lambda)$, the free Hall algebra on a $T$-sorted
signature $\Lambda$, to a Hall algebra for $S$, with
$\mathrm{HTer}_{T}(\Lambda)_{\fmon{\varphi}\bprod\varphi}$ as
underlying $\fmon{S}\bprod S$-sorted set.
\end{proposition}

For a derivor $\mathbf{d}\colon \mathbf{\Sigma}\mor
\mathbf{\Lambda}$, the $\mathrm{ms}$-mapping $d$ from $\Sigma$ to
$\mathrm{HTer}_{T}(\Lambda)_{\fmon{\varphi}\bprod\varphi}$ can be
lifted to a homomorphism of Hall algebras $\ext{d}$ from
$\mathbf{HTer}_{S}(\Sigma)$ to
$\mathbf{HTer}_{T}(\Lambda)_{\fmon{\varphi}\bprod\varphi}$, whose
underlying $\fmon{S}\bprod S$-sorted mapping determines a translation
of terms for $\mathbf{\Sigma}$ to terms for $\mathbf{\Lambda}$.  In
particular, for every $(w,s)\in\fmon{S}\bprod S$, $\ext{d}_{w,s}$
assigns to terms in $\mathrm{T}_{\mathbf{\Sigma}}(\vs{w})_{s}$ terms in
$\mathrm{T}_{\mathbf{\Lambda}}(\vs{\ext{\varphi}(w)})_{\varphi(s)}$.

Before we define immediately below the composition of derivors we
recall that $\mathbf{\Sigma}$, $\mathbf{\Lambda}$, $\mathbf{\Omega}$,
and $\mathbf{\Xi}$ denote the signatures $(S,\Sigma)$, $(T,\Lambda)$,
$(U,\Omega)$, and $(X,\Xi)$, respectively, and $\mathbf{d}$,
$\mathbf{e}$, and $\mathbf{h}$ denote the derivors $(\varphi,d)$,
$(\psi,e)$, and $(\gamma,h)$, respectively.

\begin{definition}
Let $\mathbf{d}\colon \mathbf{\Sigma}\mor \mathbf{\Lambda}$ and  $\mathbf{e}\colon
\mathbf{\Lambda}\mor \mathbf{\Omega}$ be derivors.  Then
$\mathbf{e}\comp\mathbf{d}$, the composition of $\mathbf{d}$ and
$\mathbf{e}$, is the derivor $(\psi\comp\varphi, \ext{e}_{\fmon{\varphi}\bprod
\varphi}\comp d)$, where $\ext{e}_{\fmon{\varphi}\bprod \varphi}\comp d$ is
obtained from
$$
\begin{aligned}
\xymatrix@C=40pt@R=30pt{
\Lambda
  \ar[r]^-{\eta_{\Lambda}^{\Hall_{T}}}
  \ar[rd]_-{e} &
\mathrm{HTer}_{T}(\Lambda) \ar[d]^{\ext{e}}
  \\
& \mathrm{HTer}_{U}(\Omega)_{\fmon{\psi}\bprod \psi}
}
\end{aligned}
\quad \text{as} \quad\;\;
\begin{aligned}
\xymatrix@C=40pt@R=30pt{
\mathrm{HTer}_{T}(\Lambda)_{\fmon{\varphi}\bprod \varphi}
\ar[d]^{\ext{e}_{\fmon{\varphi}\bprod \varphi}} &
\Sigma \ar[l]_-{d} \\
{\mathrm{HTer}_{U}(\Omega)_{\fmon{\psi}\bprod\psi}}_{\fmon{\varphi}\bprod\varphi}
}
\end{aligned}
$$
where $\ext{e}$ is the extension of $e$ to the free Hall algebra
on $\Lambda$.

For every signature $\mathbf{\Sigma}$, the identity at $\mathbf{\Sigma}$ is
$(\id_{S},\eta_{\Sigma}^{\Hall_{S}})$.
\end{definition}

The preceding definition allows us to get a corresponding, and
explicit, category of signatures and derivors.

\begin{proposition}
The signatures together with the derivors determine a category, that
we denote by $\mathbf{Sig}_{\mathfrak{d}}$.
\end{proposition}

\begin{proof}
We restrict ourselves to prove that the composition of derivors is a
derivor and that the composition is associative.

We have that %
\begin{align*}
 {\mathrm{HTer}_{U}(\Omega)_{\fmon{\psi}\bprod\psi}}_{\fmon{\varphi}\bprod\varphi}
&=
 \mathrm{HTer}_{U}(\Omega)_{(\fmon{\psi}\bprod\psi)
 \comp(\fmon{\varphi}\bprod\varphi)} \\
&=
 \mathrm{HTer}_{U}(\Omega)_{(\fmon{\psi}
 \comp\fmon{\varphi})\bprod(\psi\comp\varphi)} \\
&=
 \mathrm{HTer}_{U}(\Omega)_{\fmon{(\psi\comp\varphi)}\bprod (\psi\comp\varphi)},
\end{align*}
hence the composition of derivors is a derivor.

Given the situation described by the following diagram
$$
\xymatrix{
\mathbf{\Sigma} \ar[r]^{\mathbf{d}} &
\mathbf{\Lambda} \ar[r]^{\mathbf{e}} &
\mathbf{\Omega} \ar[r]^{\mathbf{h}} &
\mathbf{\Xi},
}$$
we have that
\begin{align*}
 \mathbf{h}\comp(\mathbf{e}\comp\mathbf{d})
 &=
 \mathbf{h}\comp(\psi\comp\varphi,
   \ext{e}_{\fmon{\varphi}\bprod\varphi}\comp d) \\
 &=
 (\gamma\comp(\psi\comp\varphi),
  \ext{h}_{\fmon{(\psi\comp\varphi)}\bprod
               {(\psi\comp\varphi)}}
  \comp(\ext{e}_{\fmon{\varphi}\bprod\varphi}\comp d)) \\
 &=
 ((\gamma\comp\psi)\comp\varphi,
 {\ext{h}_{\fmon{\psi}\bprod\psi}}
         _{\fmon{\varphi}\bprod\varphi}
  \comp(\ext{e}_{\fmon{\varphi}\bprod\varphi}\comp d)) \\
 &=
 ((\gamma\comp\psi)\comp\varphi,
 ({\ext{h}_{\fmon{\psi}\bprod\psi}}
          _{\fmon{\varphi}\bprod\varphi}
  \comp \ext{e}_{\fmon{\varphi}\bprod\varphi})\comp d) \\
 &=
 ((\gamma\comp\psi)\comp\varphi,
 ({\ext{h}_{\fmon{\psi}\bprod\psi}}
  \comp \ext{e})_{\fmon{\varphi}\bprod\varphi}\comp d) \\
 &=
 ((\gamma\comp\psi)\comp\varphi,
 (\ext{{\ext{h}_{\fmon{\psi}\bprod\psi}}
  \comp e)}_{\fmon{\varphi}\bprod\varphi}\comp d) \\
 &=
 (\gamma\comp\psi,\ext{h}_{\fmon{\psi}\bprod\psi}\comp e)
 \comp \mathbf{d} \\
 &=
 (\mathbf{h}\comp\mathbf{e})\comp\mathbf{d},
\end{align*}
therefore the composition of derivors is associative.
\end{proof}


We point out that the category $\mathbf{Sig}_{\mathfrak{d}}$ of signatures
and derivors can be obtained, naturally, as an isomorphic copy of the
Kleisli category for a monad in $\mathbf{Sig}$.  This is founded on the fact
that, for every set of sorts $S$, we have the adjunction
$\mathbf{T}_{\Hall_{S}}\ladj\mathrm{G}_{\Hall_{S}}$, from which we get the
monad $\mathbb{T}_{\Hall_{S}} =
(\mathrm{T}_{\Hall_{S}},\eta^{\Hall_{S}},\mu^{\Hall_{S}})$ in
$\mathbf{Set}^{\fmon{S}\times S}$, that canonically induces the monad in
$\mathbf{Sig}$ at issue.

\begin{proposition}
The triple $\mathbb{T}_{\mathfrak{d}} =
(\mathfrak{d},\eta^{\mathfrak{d}},\mu^{\mathfrak{d}})$, where
\begin{enumerate}
\item $\mathfrak{d}$ is the functor which sends a signature
      $\mathbf{\Sigma}$ to the signature $(S,\mathrm{T}_{\Hall_{S}}(\Sigma))$, and
      a signature morphism $\mathbf{d}$ from $\mathbf{\Sigma}$ to
      $\mathbf{\Lambda}$ to the signature morphism $(\varphi,\ext{d})$
      from $(S,\mathrm{T}_{\Hall_{S}}(\Sigma))$ to
      $(T,\mathrm{T}_{\Hall_{T}}(\Lambda))$,

\item $\eta^{\mathfrak{d}}_{\mathbf{\Sigma}} =
      (\mathrm{id}_{S},\eta^{\Hall_{S}}_{\Sigma})$, with
      $\eta^{\Hall_{S}}_{\Sigma}$ the value at $\Sigma$ of the unit
      $\eta^{\Hall_{S}}$ of the monad $\mathbb{T}_{\Hall_{S}}$, and

\item $\mu^{\mathfrak{d}}_{\mathbf{\Sigma}} =
      (\mathrm{id}_{S},\mu^{\Hall_{S}}_{\Sigma})$, with
      $\mu^{\Hall_{S}}_{\Sigma}$ the value at $\Sigma$ of the
      multiplication
      $\mu^{\Hall_{S}}$ of the monad $\mathbb{T}_{\Hall_{S}}$,
\end{enumerate}
is a monad in $\mathbf{Sig}$ and the categories
$\mathbf{Sig}_{\mathfrak{d}}$ and
$\mathbf{Kl}(\mathbb{T}_{\mathfrak{d}})$ are isomorphic.

\end{proposition}

\begin{remark}
Almost all the results about the categories $\mathbf{Sig}$,
$\mathbf{Alg}$ and $\mathbf{Spf}$ established in the second and third
section, suitably extended, are also valid for the corresponding
categories $\mathbf{Sig}_{\mathfrak{d}}$,
$\mathbf{Alg}_{\mathfrak{d}}$ and $\mathbf{Spf}_{\mathfrak{d}}$.  But
being the derivors a particular case of the polyderivors, we restrict
ourselves to unfold those results only for the polyderivors.
\end{remark}

We state now a lemma from which the existence of coproducts in
$\mathbf{Sig}_{\mathfrak{d}}$ follows immediately.

\begin{lemma}\label{KlCoprod}
Let $\mathbb{T}$ be a monad in a category $\mathbf{C}$.  If
$\mathbf{C}$ has coproducts, then $\mathbf{Kl}(\mathbb{T})$ has
coproducts.
\end{lemma}

\begin{proof}
Let $(X_{i})_{i\in I}$ be a family of objects in
$\mathbf{Kl}(\mathbb{T})$.  Then $\coprod_{i\in I}X_{i}$, together with
the family of morphisms $(\eta_{\coprod_{i\in
I}X_{i}}\circ\mathrm{in}_{i})_{i\in I}$ where, for $i\in I$, $\mathrm{in}_{i}$ is
the structural morphism from $X_{i}$ into $\coprod_{i\in I}X_{i}$, is a
coproduct in $\mathbf{Kl}(\mathbb{T})$ of $(X_{i})_{i\in I}$.

Let $(f_{i}\colon X_{i}\mor Y)_{i\in I}$ be a family of morphisms in
$\mathbf{Kl}(\mathbb{T})$.  Then, by the universal property of the
coproduct, there exists a unique morphism $[f_{i}]_{i\in I}$, in
$\mathbf{C}$, from $\coprod_{i\in I}X_{i}$ into $T(Y)$ such that, for every
$i\in I$, $[f_{i}]_{i\in I}\circ \mathrm{in}_{i} = f_{i}$.  Furthermore, we have
that
\begin{alignat*}{2}
[f_{i}]_{i\in I} &= [f_{i}]_{i\in I}\diamond \eta_{\coprod_{i\in
                     I}X_{i}} & &\quad \text{(since $\eta_{\coprod_{i\in
                     I}X_{i}}$ is neutral for $\diamond$)}\\
                 &= \mu_{Y}\circ T([f_{i}]_{i\in I})
                    \circ \eta_{\coprod_{i\in I}X_{i}} &
                    &\quad \text{(by definition
                    of $\diamond$)}.
\end{alignat*}
Therefore we get the commutative diagram
$$
\xymatrix{
X_{i}
  \ar[r]^-{\mathrm{in}_{i}}
  \ar[rd]_{f_{i}}  &
\coprod_{i\in I}X_{i}
  \ar[r]^{\eta_{\coprod_{i\in I}X_{i}}}
  \ar[d]^{[f_{i}]_{i\in I}}  &
T(\coprod_{i\in I}X_{i})
  \ar[d]^{T([f_{i}]_{i\in I})} \\
&
T(Y) &
T(T(Y))
  \ar[l]^{\mu_{Y}}
}
$$
and we can assert that, for every $i\in I$,
$[f_{i}]_{i\in I}\diamond(\eta_{\coprod_{i\in I}X_{i}}\circ
\mathrm{in}_{i}) = f_{i}$. To prove the uniqueness, let $g$ be
a morphism, in $\mathbf{Kl}(\mathbb{T})$, from $\coprod_{i\in I}X_{i}$
into $Y$ such that, for every $i\in I$,
$g\diamond(\eta_{\coprod_{i\in I}X_{i}}\circ
\mathrm{in}_{i}) = f_{i}$. Then, for $i\in I$, we have that
\begin{alignat*}{2}
f_{i} &= g\diamond (\eta_{\coprod_{i\in I}X_{i}}\circ \mathrm{in}_{i}) \\
      &= \mu_{Y}\circ T(g)\circ (\eta_{\coprod_{i\in I}X_{i}}\circ  \mathrm{in}_{i})
         & &\quad \text{(by definition of $\diamond$)} \\
      &= (\mu_{Y}\circ T(g)\circ\eta_{\coprod_{i\in I}X_{i}})\circ
          \mathrm{in}_{i} \\
      &= (g\diamond \eta_{\coprod_{i\in I}X_{i}})\circ \mathrm{in}_{i}
         & &\quad \text{(by definition of $\diamond$)}\\
      &= g\circ \mathrm{in}_{i}
         & &\quad \text{(since $\eta_{\coprod_{i\in I}X_{i}}$
         is neutral for $\diamond$)},
\end{alignat*}
thus $g = [f_{i}]_{i\in I}$.  Therefore $[f_{i}]_{i\in I}\colon
\coprod_{i\in I}X_{i}\mor Y$ in $\mathbf{Kl}(\mathbb{T})$ satisfies
the universal property.
\end{proof}

\begin{corollary}
The category $\mathbf{Sig}_{\mathfrak{d}}$ has coproducts.
\end{corollary}

\begin{proof}
Because $\mathbf{Sig}$ has coproducts, the category
$\mathbf{Kl}(\mathbb{T}_{\mathfrak{d}})$, by Lemma~\ref{KlCoprod}, has
also coproducts.  Therefore, since $\mathbf{Sig}_{\mathfrak{d}}$ and
$\mathbf{Kl}(\mathbb{T}_{\mathfrak{d}})$ are isomorphic, the category
$\mathbf{Sig}_{\mathfrak{d}}$ has coproducts.
\end{proof}


After having stated the above facts about the derivors, we are in the
position to introduce the definition of the composition of two
polyderivors, in order to get the corresponding category
$\mathbf{Sig}_{\mathfrak{pd}}$ of signatures and polyderivors.  To this end
we begin by stating that every mapping $\varphi\colon S\mor \fmon{T}$
gives rise to a functor
$(\ext{\varphi}\bprod\ext{\varphi},b^{\varphi})^{\ast}$ from
$\mathbf{Alg}(\mathrm{B}_{T})$, the category of Bénabou algebras for
$T$, to $\mathbf{Alg}(\mathrm{B}_{S})$, the category of Bénabou
algebras for $S$ (observe that such a functor is induced not by a
standard specification morphism between Bénabou specifications, but by
a \emph{derivor} $b^{\varphi}$ between the corresponding Bénabou
signatures).  This functor, in its turn, will allow us to endow the
many-sorted set $\mathrm{BTer}_{T}(\Lambda)_{\ext{\varphi}\bprod
\ext{\varphi}}$ with a structure of Bénabou algebra for $S$, from
which the definition of the composition of polyderivors will follow.

\begin{proposition}\label{derivorFuji}
Let $\varphi$ be a mapping from $S$ to $\fmon{T}$.  Then the
$\fmon{({(\fmon{S})}^{2})}\bprod {(\fmon{S})}^{2}$-sorted mapping
$$
b^{\varphi} \colon \Sigma^{\Ben_{S}} \mor
\mathrm{HTer}_{\fmon{T}\times \fmon{T}}(\Sigma^{\Ben_{T}})
_{\fmon{(\ext{\varphi}\bprod\ext{\varphi})}\bprod
(\ext{\varphi}\bprod\ext{\varphi})}
$$
defined as follows
\begin{enumerate}
\item For every $w\in \fmon{S}$ and $\alpha\in\bb{w}$,
      $b^{\varphi}(\pi^{w}_{\alpha})$ is the
      $\mathbf{\Sigma}^{\Ben_{T}}$-term %
      $$
      \tp{\pi^{\ext{\varphi}(w)}_{\sum_{\beta\in
      \alpha}p_{\beta}},\ldots, \pi^{\ext{\varphi}(w)}_{\sum_{\beta\in
      \alpha+1}p_{\beta}-1}}_{\ext{\varphi}(w),\varphi(w_{\alpha})}
      $$
      of type %
      $\lambda\mor(\ext{\varphi}(w),(\varphi(w_{\alpha})))$,

\item For every $u,\,w\in\fmon{S}$,
      $b^{\varphi}(\tp{\,}_{u,w})$ is the
      $\mathbf{\Sigma}^{\Ben_{T}}$-term %
      \begin{align*}
      \tp{
      &\pi^{\varphi(w_{0})}_{0} \comp
      v^{(\ext{\varphi}(u),\varphi(w_{0}))}_{0},
      \ldots,
      \pi^{\varphi(w_{0})}_{\bb{\varphi(w_{0})}-1} \comp
      v^{(\ext{\varphi}(u),\varphi(w_{0}))}_{0},
      \ldots,\\
      &\pi^{\varphi(w_{i})}_{0} \comp
      v^{(\ext{\varphi}(u),\varphi(w_{i}))}_{i},
      \ldots,
      \pi^{\varphi(w_{i})}_{\bb{\varphi(w_{i})}-1} \comp
      v^{(\ext{\varphi}(u),\varphi(w_{i}))}_{i},
      \ldots,\\
      &\pi^{\varphi(w_{\bb{w}-1})}_{0} \comp
      v^{(\ext{\varphi}(u),\varphi(w_{\bb{w}-1}))}_{\bb{w}-1},
      \ldots,
      \pi^{\varphi(w_{\bb{w}-1})}_{\bb{\varphi(w_{\bb{w}-1})}-1} \comp
      v^{(\ext{\varphi}(u),\varphi(w_{\bb{w}-1}))}_{\bb{w}-1}
      }
      \end{align*}
      of type $((\ext{\varphi}(u),\varphi(w_{0})), \ldots,
      (\ext{\varphi}(u),\varphi(w_{\bb{w}-1}))) \mor
      (\ext{\varphi}(u),\ext{\varphi}(w))$,

\item  For every $u,\,x,\,w \in \fmon{S}$,
       $b^{\varphi}(\comp_{u,x,w})$ is the
       $\mathbf{\Sigma}^{\Ben_{T}}$-term %
       $$
       \comp_{\ext{\varphi}(u),\ext{\varphi}(x),\ext{\varphi}(w)}
       (v^{(\ext{\varphi}(u),\ext{\varphi}(x))}_{0},
       v^{(\ext{\varphi}(x),\ext{\varphi}(w))}_{1})
       $$
       of type
       $((\ext{\varphi}(u),\ext{\varphi}(x)),
       (\ext{\varphi}(x),\ext{\varphi}(w)))
       \mor (\ext{\varphi}(u),\ext{\varphi}(w))$,
\end{enumerate}
is such that $(\ext{\varphi}\bprod\ext{\varphi},b^{\varphi})\colon
(\fmon{S}\bprod \fmon{S},\Sigma^{\Ben_{S}}, \mathcal{E}^{\Ben_{S}})
\mor (\fmon{T}\bprod\fmon{T},\Sigma^{\Ben_{T}},
\mathcal{E}^{\Ben_{T}})$ is a specification morphism.
Therefore $\varphi\colon S\mor \fmon{T}$ induces a functor
$(\ext{\varphi}\bprod\ext{\varphi},b^{\varphi})^{\ast}$ from
$\mathbf{Alg}(\mathrm{B}_{T})$ to $\mathbf{Alg}(\mathrm{B}_{S})$ which sends
$\mathbf{BTer}_{T}(\Lambda)$, the free Bénabou algebra on the
$T$\nobreakdash-sorted signature $\Lambda$, to a Bénabou algebra for $S$, with
$\mathrm{BTer}_{T}(\Lambda)_{\ext{\varphi}\bprod\ext{\varphi}}$ as
underlying $\fmon{S}\bprod \fmon{S}$-sorted set.
\end{proposition}

For a polyderivor $\mathbf{d}\colon \mathbf{\Sigma}\mor \mathbf{\Lambda}$, we can
extend the $\fmon{S}\bprod \fmon{S}$-sorted mapping $d$ from
$\coprod_{1\bprod\between_{S}}\Sigma$ to
$\mathrm{BTer}_{T}(\Lambda)_{\ext{\varphi}\bprod\ext{\varphi}}$ to a
homomorphism of Bénabou algebras $\ext{d}$ from
$\mathrm{BTer}_{S}(\Sigma)$ to
$\mathrm{BTer}_{T}(\Lambda)_{\ext{\varphi}\bprod\ext{\varphi}}$, whose
underlying $\fmon{S}\bprod\fmon{S}$-mapping determines a translation
of terms for $\mathbf{\Sigma}$ into terms for $\mathbf{\Lambda}$.  In
particular, for every $(w,s)\in\fmon{S}\bprod S$, $\ext{d}_{w,(s)}$
assigns to terms in $\mathrm{T}_{\mathbf{\Sigma}}(\vs{w})_{s}$ terms in
$\mathrm{T}_{\mathbf{\Lambda}}(\vs{\ext{\varphi}(w)})_{\varphi(s)}$, in such a way
that to a variable $v^{s}_{\alpha}$ in $\vs{w}$ associates the family
of variables
$$
(v^{\ext{\varphi}(w)_{\sum_{\beta\in
\alpha}p_{\beta}}}_{\sum_{\beta\in \alpha}p_{\beta}}, \ldots,
v^{\ext{\varphi}(w)_{\sum_{\beta\in
\alpha+1}p_{\beta}-1}}_{\sum_{\beta\in \alpha+1}p_{\beta}-1}).
$$

Before we define next the composition of polyderivors we recall that
$\mathbf{\Sigma}$, $\mathbf{\Lambda}$, $\mathbf{\Omega}$, and
$\mathbf{\Xi}$ denote the signatures $(S,\Sigma)$, $(T,\Lambda)$,
$(U,\Omega)$, and $(X,\Xi)$, respectively, and $\mathbf{d}$,
$\mathbf{e}$, and $\mathbf{h}$ denote the polyderivors $(\varphi,d)$,
$(\psi,e)$, and $(\gamma,h)$, respectively.

\begin{definition}
Let $\mathbf{d}\colon \mathbf{\Sigma}\mor \mathbf{\Lambda}$ and
$\mathbf{e}\colon \mathbf{\Lambda}\mor \mathbf{\Omega}$ be polyderivors.
Then the \emph{composition} of $\mathbf{d}$ and $\mathbf{e}$, denoted
by $\mathbf{e}\comp\mathbf{d}$, is the morphism
$(\ext{\psi}\comp\varphi,
\ext{e}_{\ext{\varphi}\bprod\ext{\varphi}}\comp d)$, where the first
component $\ext{\psi}\comp\varphi$ is a mapping from $S$ to $\fmon{U}$
and $\ext{e}_{\ext{\varphi}\bprod\ext{\varphi}}\comp d$ is obtained
from
$$
\begin{aligned}
\xymatrix@C=40pt@R=30pt{
\coprod_{1\bprod\between_{T}}\Lambda
  \ar[r]^-{\eta_{\coprod_{1\bprod\between_{T}}\Lambda}^{\Ben_{T}}}
  \ar[rd]_-{e} &
\mathrm{BTer}_{T}(\Lambda) \ar[d]^{\ext{e}}
  \\
& \mathrm{BTer}_{U}(\Omega)_{\ext{\psi}\bprod \ext{\psi}}
}
\end{aligned}
\,\, \text{as} \,\,
\begin{aligned}
\xymatrix@C=40pt@R=30pt{
\mathrm{BTer}_{T}(\Lambda)_{\ext{\varphi}\bprod \ext{\varphi}}
\ar[d]^{\ext{e}_{\ext{\varphi}\bprod \ext{\varphi}}}&
\coprod_{1\bprod\between_{S}}\Sigma
\ar[l]_-{d}  \\
{\mathrm{BTer}_{U}(\Omega)_{\ext{\psi}\bprod\ext{\psi}}}_
{\ext{\varphi}\bprod\ext{\varphi}}
}
\end{aligned}
$$
and, for every signature $\mathbf{\Sigma}$, the \emph{identity} at
$\mathbf{\Sigma}$ is the polyderivor
$(\between_{S},\eta^{\Ben_{S}}_{\Sigma})$.
\end{definition}

From this definition we get the corresponding category of signatures
and polyderivors.

\begin{proposition}
The signatures together with the polyderivors determine a category, that
we denote by $\mathbf{Sig}_{\mathfrak{pd}}$.
\end{proposition}

\begin{proof}
To begin with we prove that the composition of polyderivors is a polyderivor.
\begin{align*}
(\mathrm{BTer}_{U}(\Omega)_{\ext{\psi}\bprod\ext{\psi}})_
{\ext{\varphi}\bprod\ext{\varphi}}
&=
\mathrm{BTer}_{U}(\Omega)_{(\ext{\psi}\bprod\ext{\psi})\comp
            (\ext{\varphi}\bprod\ext{\varphi})} \\
&=
\mathrm{BTer}_{U}(\Omega)_{(\ext{\psi}\comp\ext{\varphi})\bprod
            (\ext{\psi}\comp\ext{\varphi})}  \\
&=
\mathrm{BTer}_{U}(\Omega)_{\ext{(\ext{\psi}\comp\varphi)}\bprod
            \ext{(\ext{\psi}\comp\varphi)}}.
\end{align*}
Next we prove that the identities are true identities.
\begin{align*}
\mathbf{d}\comp(\between_{S},\eta^{\Ben_{S}}_{\Sigma})
  &=
(\ext{\varphi}\comp\between_{S},\ext{d}_{\ext{\between_{S}}\bprod\ext{\between_{S}}}
    \comp\eta_{\Sigma}^{\Ben_{S}}) \\
  &=
\mathbf{d}, \\[5pt]
(\between_{T},\eta_{\Lambda}^{\Ben_{T}})\comp \mathbf{d}
  &=
(\ext{\between_{T}}\comp\varphi,{\ext{(\eta_{\Lambda}^{\Ben_{T}})}}
      _{\ext{\varphi}\bprod\ext{\varphi}}
    \comp d) \\
  &=
\mathbf{d}.
\end{align*}
Finally we prove that the composition is associative.  Given the
morphisms
$$
\xymatrix@C=50pt@R=40pt{
\mathbf{\Sigma}
\ar[r]^{\mathbf{d}} &
\mathbf{\Lambda}
\ar[r]^{\mathbf{e}} &
\mathbf{\Omega}
\ar[r]^{\mathbf{h}} &
\mathbf{\Xi},
}
$$
we have that
\begin{align*}
 \mathbf{h}\comp(\mathbf{e}\comp\mathbf{d})
 &=
 \mathbf{h}\comp(\ext{\psi}\comp\varphi,
   \ext{e}_{\ext{\varphi}\bprod\ext{\varphi}}\comp d) \\
 &=
 (\ext{\gamma}\comp(\ext{\psi}\comp\varphi),
  \ext{h}_{\ext{(\ext{\psi}\comp\varphi)}\bprod
           \ext{(\ext{\psi}\comp\varphi)}}
  \comp(\ext{e}_{\ext{\varphi}\bprod\ext{\varphi}}\comp d)) \\
 &=
 (\ext{(\ext{\gamma}\comp\psi)}\comp\varphi,
 {\ext{h}_{\ext{\psi}\bprod\ext{\psi}}}
         _{\ext{\varphi}\bprod\ext{\varphi}}
  \comp(\ext{e}_{\ext{\varphi}\bprod\ext{\varphi}}\comp d)) \\
 &=
 (\ext{(\ext{\gamma}\comp\psi)}\comp\varphi,
 ({\ext{h}_{\ext{\psi}\bprod\ext{\psi}}}
         _{\ext{\varphi}\bprod\ext{\varphi}}
  \comp\ext{e}_{\ext{\varphi}\bprod\ext{\varphi}})\comp d) \\
 &=
 (\ext{(\ext{\gamma}\comp\psi)}\comp\varphi,
 ({\ext{h}_{\ext{\psi}\bprod\ext{\psi}}}
         _{\ext{\varphi}\bprod\ext{\varphi}}
  \comp\ext{e})_{\ext{\varphi}\bprod\ext{\varphi}}\comp d) \\
 &=
 (\ext{(\ext{\gamma}\comp\psi)}\comp\varphi,
 (\ext{{\ext{h}_{\ext{\psi}\bprod\ext{\psi}}}
         _{\ext{\varphi}\bprod\ext{\varphi}}
  \comp e)}_{\ext{\varphi}\bprod\ext{\varphi}}\comp d) \\
 &=
 (\ext{\gamma}\comp\psi,\ext{h}_{\ext{\psi}\bprod\ext{\psi}}\comp e)
 \comp \mathbf{d} \\
 &=
 (\mathbf{h}\comp\mathbf{e})\comp\mathbf{d}.
 \qedhere
\end{align*}
\end{proof}

\begin{remark}
From the fact that $\mathbf{Sig}_{\mathfrak{pd}}$ is a category it
follows at once that, for every signature $\mathbf{\Sigma}$ in
$\mathbf{Sig}_{\mathfrak{pd}}$, the set of all endopolyderivors of
$\mathbf{\Sigma}$, $\mathrm{End}_{\mathfrak{pd}}(\mathbf{\Sigma})$, is
the underlying set of a monoid, denoted by
$\mathbf{End}_{\mathfrak{pd}}(\mathbf{\Sigma})$.  Since the monoid
$\mathbf{End}_{\mathfrak{d}}(\mathbf{\Sigma})$, of hypersubstitutions
of $\mathbf{\Sigma}$, i.e, the monoid of endoderivors of
$\mathbf{\Sigma}$, is embedded (in general, strictly) in the monoid
$\mathbf{End}_{\mathfrak{pd}}(\mathbf{\Sigma})$, we conclude that
$\mathbf{End}_{\mathfrak{pd}}(\mathbf{\Sigma})$ can serve as a basis to
develop a doubly generalized (because of the use of many-sorted
signatures and endopolyderivors, instead of single-sorted signatures and
endoderivors,) theory of hyperidentities.  But we leave this task for
another occasion.
\end{remark}

%
%


Having shown above that the concept of derivor, because of its
reducibility to that of morphism of a Kleisli category for a monad
in $\mathbf{Sig}$, is mathematically natural, one could also
expect to show the mathematical naturalness of the notion of
polyderivor by proving that the category $\mathbf{Sig}_{\mathfrak{pd}}$
is obtainable as an isomorphic copy of the Kleisli category for
some monad in $\mathbf{Sig}$.  This is actually true, however the
procedure we should follow to determine such a monad is more
involved than the one, relatively simple, we have followed for the
derivors.  This is due to the fact that, for a signature
$\mathbf{\Sigma} = (S,\Sigma)$, the pair
$(\fmon{S}\bprod\fmon{S},\mathrm{BTer}_{S}(\Sigma))$ is not a
signature, because $\mathrm{BTer}_{S}(\Sigma)$ is an
$\fmon{S}\bprod\fmon{S}$-sorted set, but not an
$\fmon{({(\fmon{S})}^{2})}\bprod {(\fmon{S})}^{2}$-sorted set.

The approach we offer to prove the existence of the monad in
$\mathbf{Sig}$ whose Kleisli category is isomorphic to
$\mathbf{Sig}_{\mathfrak{pd}}$ will be based, on the one hand, on the
functor
$$
\Delta_{\cncat_{S}\bprod 1}\colon
\mathbf{Set}^{\fmon{S}\bprod\fmon{S}}\mor
\mathbf{Set}^{\ffmon{S}\bprod\fmon{S}}
$$
which sends $\fmon{S}\bprod\fmon{S}$-sorted sets to
$\fmon{S}$-signatures, therefore, for an $S$-sorted signature
$\Sigma$, we will have that $\Delta_{\cncat_{S}\bprod
1}(\mathrm{BTer}_{S}(\Sigma))$ is a $\fmon{S}$-signature, and, on the
other hand, of the fact that, for every set of sorts $S$, the
adjunction $\mathbf{T}_{\Ben_{S}}\ladj\mathrm{G}_{\Ben_{S}}$, determines a
monad on $\mathbf{Set}^{\fmon{S}\bprod\fmon{S}}$ denoted as
$\mathbb{T}_{\Ben_{S}} = (\mathrm{T}_{\Ben_{S}},\eta^{\Ben_{S}},\mu^{\Ben_{S}})$.

\begin{proposition}\label{isoSigfujKlfuj}
There exists a monad
$\mathbb{T}_{\mathfrak{pd}} =
(\mathfrak{pd},\eta^{\mathfrak{pd}},\mu^{\mathfrak{pd}})$
in $\mathbf{Sig}$ such that the categories $\mathbf{Sig}_{\mathfrak{pd}}$ and
$\mathbf{Kl}(\mathbb{T}_{\mathfrak{pd}})$ are isomorphic.
\end{proposition}

\begin{proof}
Let $\mathfrak{pd}$ be the endofunctor of $\mathbf{Sig}$ defined as follows
\begin{enumerate}
\item $\mathfrak{pd}$ sends a signature $\mathbf{\Sigma}$ to
      $$
      (\fmon{S},\mathrm{T}_{\Ben_{S}}(\textstyle{\coprod}_{1\bprod\between_{S}}\Sigma)
      _{\cncat_{S}\bprod 1})).
      $$
\item $\mathfrak{pd}$ sends a signature morphism $\mathbf{d}$ from
      $\mathbf{\Sigma}$ to $\mathbf{\Lambda}$ to
      $$
      (\fmon{\varphi},(\ext{d})_{\cncat_{S}\bprod
      1})\colon
      (\fmon{S},\mathrm{T}_{\Ben_{S}}(\textstyle{\coprod}_{1\bprod\between_{S}}\Sigma)
      _{\cncat_{S}\bprod 1}))\mor
     (\fmon{T},\mathrm{T}_{\Ben_{T}}(\textstyle{\coprod}_{1\bprod\between_{T}}\Lambda)
     _{\cncat_{T}\bprod 1})),
     $$
     where
     $\mathrm{T}_{\Ben_{S}}(\coprod_{1\bprod\between_{S}}\Sigma)_{\cncat_{S}\bprod
     1}$ is the value in $\Sigma$ of the functor %
     $$
     \xymatrix@C=40pt@R=30pt{
     \mathbf{Set}^{\fmon{S}\bprod S}
     \ar[r]^{\coprod_{1\bprod\between_{S}}}&
     \mathbf{Set}^{\fmon{S}\bprod \fmon{S}}
     \ar[r]^{\mathrm{T}_{\Ben_{S}}}&
     \mathbf{Set}^{\fmon{S}\bprod \fmon{S}}
     \ar[r]^{\Delta_{\cncat_{S}\bprod 1}}&
     \mathbf{Set}^{\ffmon{S}\bprod \fmon{S}}.
     }
     $$
\end{enumerate}

After having defined the endofunctor $\mathfrak{pd}$ of $\mathbf{Sig}$, we
proceed to define the unit $\eta^{\mathfrak{pd}}$ and multiplication
$\mu^{\mathfrak{pd}}$ of the monad $\mathbb{T}_{\mathfrak{pd}}$.

Let $\mathbf{\Sigma}$ be a signature.  Then we have that
$\eta^{\mathfrak{pd}}_{\mathbf{\Sigma}}$, the component of the unit
$\eta^{\mathfrak{pd}}$ of the purported monad
$\mathbb{T}_{\mathfrak{pd}}$ in $\mathbf{\Sigma}$, is the signature
morphism $(\between_{S},\eta^{\Ben_{S}}_{\Sigma})$, i.e., the
value in $\Sigma$ of the unit of the monad
$\mathbb{T}_{\Ben_{S}}=(\mathrm{T}_{\Ben_{S}},\eta^{\Ben_{S}},\mu^{\Ben_{S}})$
in $\mathbf{Set}^{\fmon{S}\bprod\fmon{S}}$, obtained from the adjunction
$\mathbf{T}_{\Ben_{S}}\ladj\mathrm{G}_{\Ben_{S}}$.  On the other hand, we want
$\mu^{\mathfrak{pd}}_{\mathbf{\Sigma}}$, the component of the
multiplication $\mu^{\mathfrak{pd}}$ of the purported monad
$\mathbb{T}_{\mathfrak{pd}}$ in $\mathbf{\Sigma}$, to be a morphism
as in the following diagram
$$
\xymatrix{
  (\ffmon{S},\mathrm{T}_{\Ben_{\fmon{S}}}(\tcoprod_{1\bprod\between_{\fmon{S}}}
  (\mathrm{T}_{\Ben_{S}}(\tcoprod_{1\bprod\between_{S}}\Sigma)_{\cncat_{S}\bprod 1})
  )_{\cncat_{\fmon{S}}\bprod 1})
  \ar[d]_{\mu^{\mathfrak{pd}}_{\mathbf{\Sigma}}}\\
  (\fmon{S},\mathrm{T}_{\Ben_{S}}(\tcoprod_{1\bprod\between_{S}}\Sigma)_{\cncat_{S}\bprod 1})
}
$$
The first coordinate of $\mu^{\mathfrak{pd}}_{\mathbf{\Sigma}}$ is $\cncat_{S}$,
the multiplication of the monad $\mathbb{T}_{\star}$.  To get the second
coordinate of $\mu^{\mathfrak{pd}}_{\mathbf{\Sigma}}$ we have to define a natural
transformation $\alpha$ as in the following diagram
$$
\xymatrix@C=19ex@R=8ex{
&\xymn{\mathbf{Set}^{\ffmon{S}\bprod \ffmon{S}}}{5}&
\xymn{\mathbf{Set}^{\ffmon{S}\bprod \ffmon{S}}}{6}&
\xymn{\mathbf{Set}^{\fffmon{S}\bprod \ffmon{S}}}{7}\\
&\xymn{\mathbf{Set}^{\ffmon{S}\bprod \fmon{S}}}{4}\\
\xymn{\mathbf{Set}^{\fmon{S}\bprod \fmon{S}}}{2}&
\xymn{\mathbf{Set}^{\fmon{S}\bprod \fmon{S}}}{3}&
\xymn{\mathbf{Set}^{\fmon{S}\bprod \fmon{S}}}{8}&
\xymn{\mathbf{Set}^{\ffmon{S}\bprod \fmon{S}}}{9}\\
\xymn{\mathbf{Set}^{\fmon{S}\bprod S}}{1}
\ar "1";"2" ^{\coprod_{1\bprod\between_{S}}}
\ar "2";"3" ^{\mathrm{T}_{\Ben_{S}}}
\ar "3";"4" ^{\Delta_{\cncat_{S}\bprod 1}}
\ar "3";"8" ^{\mathrm{T}_{\Ben_{S}}}
\ar "4";"5" ^{\coprod_{1\bprod\between_{\fmon{S}}}}
\ar "5";"6" ^{\mathrm{T}_{\Ben_{\fmon{S}}}}
\ar "6";"7" ^{\Delta_{\cncat_{\fmon{S}}\bprod 1}}
\ar "8";"9" ^{\Delta_{\cncat_{S}\bprod 1}}
\ar "8";"6" |*+{\Delta_{\cncat_{S}\bprod \cncat_{S}}}
\ar "9";"7" |*_{\Delta_{\cncat_{\fmon{S}}\bprod \cncat_{S}}}
\ar @/_30pt/ "2";"8" |*+{\mathrm{T}_{\Ben_{S}}}="b"
\ar @{} "5";"8" |{\dir{=>}}^{\alpha}
\ar @{} "6";"9" |{=}
\ar @{} "3";"b" |{\dir{=>}}^{\mu^{\Ben_{S}}}
}
$$
Let $\Theta$ be an $\fmon{S}\bprod\fmon{S}$-sorted set.  Then
$\mathrm{T}_{\Ben_{S}}(\Theta)_{\cncat_{S}\bprod \cncat_{S}}$ has a natural
structure of $\mathbf{\Sigma}^{\Ben_{\fmon{S}}}$-algebra, obtained
from the $\fmon{(\ffmon{S}\bprod\ffmon{S})}\bprod
(\ffmon{S}\bprod\ffmon{S})$-sorted mapping
$b^{\cncat_{S}} \colon
  \Sigma^{\Ben_{\fmon{S}}}
  \mor
  \mathrm{Ter}_{\fmon{S}\times\fmon{S}}(\Sigma^{\Ben_{S}})
    _{\fmon{(\cncat_{S}\bprod\cncat_{S})}\bprod
    (\cncat_{S}\bprod\cncat_{S})}
$
by applying Proposition~\ref{derivorFuji} to the mapping
$\cncat_{S}\colon \ffmon{S}\mor \fmon{S}$.


On the other hand, for every $\fmon{S}\bprod\fmon{S}$-sorted set
$\Theta$, we have an $\ffmon{S}\bprod\ffmon{S}$-sorted mapping
$f_{\Theta}$ from
$\coprod_{1\bprod\between_{\fmon{S}}}(\Delta_{\cncat_{S}\bprod
1}(\Theta))$ to $\Delta_{\cncat_{S}\bprod
\cncat_{S}}(\mathrm{T}_{\Ben_{S}}(\Theta))$ which, for every\
$(\ol{u},\ol{w})\in \ffmon{S}\bprod\ffmon{S}$, assigns to an element
$P$, the image of $P$ under the inclusion $\eta^{\Ben_{S}}_{\Theta}$
of $\Theta$ into $\mathrm{T}_{\Ben_{S}}(\Theta)$.  The definition is sound
because, in this case, $\ol{w}$ has the form $(w)$, $P$ is in
$\Theta_{\cncat_{S}u,w}$ and $\eta^{\Ben_{S}}_{\Theta}(P)$ belongs to
$\Delta_{\cncat_{S}\bprod \cncat_{S}}(\mathrm{T}_{\Ben_{S}}(\Theta))$.  Then
the extension $\ext{f_{\Theta}}$ of $f_{\Theta}$ to
$\mathrm{T}_{\Ben_{\fmon{S}}}(\coprod_{1\bprod\between_{\fmon{S}}}
(\Delta_{\cncat_{S}\bprod 1}(\Theta)))$ is the component at $\Theta$
of the natural transformation $\alpha$.

Therefore, the second coordinate of $\mu^{\mathfrak{pd}}_{\mathbf{\Sigma}}$
is the value at $\Sigma$ of the natural transformation%
$$
(
\Delta_{\cncat_{\fmon{S}}\bprod \cncat_{S}} \hcomp
\Delta_{\cncat_{S} \bprod 1} \hcomp
\mu^{\Ben_{S}} \hcomp
\tcoprod_{1 \bprod \between_{S}}
)\comp (
\Delta_{\cncat_{\fmon{S}}\bprod 1} \hcomp
\alpha \hcomp
\mathrm{T}_{\Ben_{S}} \hcomp
\tcoprod_{1 \bprod \between_{S}}
).
$$

Finally we prove that the categories $\mathbf{Sig}_{\mathfrak{pd}}$ and
$\mathbf{Kl}(\mathbb{T}_{\mathfrak{pd}})$ are isomorphic.

A morphism $\mathbf{d}\colon \mathbf{\Sigma}\mor \mathbf{\Lambda}$ in
$\mathbf{Kl}(\mathbb{T}_{\mathfrak{pd}})$ is a morphism  $\mathbf{d}\colon
\mathbf{\Sigma}\mor \mathfrak{pd}(\mathbf{\Lambda})$ in $\mathbf{Sig}$, hence
$\varphi\colon S\mor \fmon{T}$ and
\begin{align*}
d\colon \Sigma\mor
& \Delta_{\fmon{\varphi}\bprod\varphi}(\mathrm{T}_{\Ben_{T}}
         (\tcoprod_{1\bprod\between_{T}}\Lambda)
    _{\cncat_{T}\bprod 1})  \\
&= \Delta_{\ext{\varphi}\bprod \varphi}
     (\mathrm{T}_{\Ben_{T}}
         (\tcoprod_{1\bprod\between_{T}}\Lambda)) \\
&\iso
\Delta_{\ext{\varphi}\bprod \varphi}
     (\mathrm{BTer}_{T}
         (\Lambda)),
\end{align*}
that is exactly the definition of polyderivor in $\mathbf{Sig}_{\mathfrak{pd}}$.
\end{proof}


\begin{remark}
From the existence of the monad $\mathbb{T}_{\mathfrak{pd}}$ in
$\mathbf{Sig}$ it follows the existence of an adjunction
$F_{\mathfrak{pd}}\ladj G_{\mathfrak{pd}}$ from $\mathbf{Sig}$ to
$\mathbf{Kl}(\mathbb{T}_{\mathfrak{pd}})$ which, in its turn, defines
in $\mathbf{Sig}$ exactly the monad $\mathbb{T}_{\mathfrak{pd}}$
(recall that the functor $F_{\mathfrak{pd}}$ from $\mathbf{Sig}$ to
$\mathbf{Kl}(\mathbb{T}_{\mathfrak{pd}})$ sends a signature morphism
$\mathbf{d}$ from $\mathbf{\Sigma}$ to $\mathbf{\Lambda}$ in
$\mathbf{Sig}$ to the composite signature morphism
$$
\xymatrix{
\mathbf{\Sigma}
\ar[r]^-{\mathbf{d}} &
\mathbf{\Lambda}
\ar[r]^-{\eta^{\mathfrak{pd}}_{\mathbf{\Lambda}}} &
\mathfrak{pd}(\Lambda)
}
$$
in $\mathbf{Kl}(\mathbb{T}_{\mathfrak{pd}})$; and that the functor
$G_{\mathfrak{pd}}$ from $\mathbf{Kl}(\mathbb{T}_{\mathfrak{pd}})$ to
$\mathbf{Sig}$ sends a morphism $\mathbf{d}$ from $\mathbf{\Sigma}$ to
$\mathbf{\Lambda}$ in $\mathbf{Kl}(\mathbb{T}_{\mathfrak{pd}})$ to the
composite signature morphism
$$
\xymatrix{
\mathfrak{pd}(\mathbf{\Sigma})
\ar[r]^-{\mathfrak{pd}(\mathbf{d})} &
\mathfrak{pd}(\mathfrak{pd}(\mathbf{\Lambda}))
\ar[r]^-{\mu^{\mathfrak{pd}}_{\mathbf{\Lambda}}} &
\mathfrak{pd}(\Lambda)
}
$$
in $\mathbf{Sig}$).  Therefore, from the functor $F_{\mathfrak{pd}}$
and taking into account the isomorphism between
$\mathbf{Kl}(\mathbb{T}_{\mathfrak{pd}})$ and
$\mathbf{Sig}_{\mathfrak{pd}}$, we get, automatically, a result
stated in a laborious way in a previous example: that the concept of
standard signature morphism is a particular case of that of polyderivor.
\end{remark}

\begin{corollary}
The category $\mathbf{Sig}_{\mathfrak{pd}}$ has coproducts.
\end{corollary}

\begin{proof}
Because $\mathbf{Sig}$ has coproducts, the category
$\mathbf{Kl}(\mathbb{T}_{\mathfrak{pd}})$, by Lemma~\ref{KlCoprod}, has
also coproducts.  Therefore, since $\mathbf{Sig}_{\mathfrak{pd}}$ and
$\mathbf{Kl}(\mathbb{T}_{\mathfrak{pd}})$ are isomorphic, the category
$\mathbf{Sig}_{\mathfrak{pd}}$ has coproducts.
\end{proof}



After stating that a polyderivor, as was the case for a derivor, is
nothing more (nor less) than a morphism of a Kleisli category for
a convenient monad in $\mathbf{Sig}$, therefore confirming
category-theoretically its naturalness, our next goal is to lift
the contravariant functor $\Alg\colon \mathbf{Sig}\mor
\mathbf{Cat}$ up to a contravariant pseudo-functor
$\Alg_{\mathfrak{pd}}\colon \mathbf{Sig}_{\mathfrak{pd}}\mor
\mathbf{Cat}$, that will allow us, by applying, once more, the
construction of Ehresmann-Grothendieck, to get a new category of
algebras $\mathbf{Alg}_{\mathfrak{pd}}$ into which is embedded the
category $\mathbf{Alg}$.  But to achieve the just stated objective
we should define beforehand some auxiliary functors and natural
transformations.

\begin{proposition}\label{isonatcncat}
Let $S$ be a set of sorts. Then we have that
\begin{enumerate}
\item There exists an \emph{expansion} functor
      $(\farg)^{\nat_{S}}$ from $\mathbf{Set}^{S}$ to
      $\mathbf{Set}^{\fmon{S}}$ which sends an $S$-sorted set $A =
      (A_{s})_{s\in S}$ to the $\fmon{S}$-sorted set $A^{\nat_{S}} =
      (A_{w})_{w\in \fmon{S}}$, and an $S$-sorted mapping $f$ from $A$
      to $B$ to the $\fmon{S}$-sorted mapping $f^{\nat_{S}} =
      (f_{w})_{w\in \fmon{S}}$ from $(A_{w})_{w\in \fmon{S}}$ to
      $(B_{w})_{w\in \fmon{S}}$.  If $A$ is an $S$-sorted set and
      $f\colon A\mor B$ an $S$-sorted mapping, then we say that
      $A^{\nat_{S}}$ and $f^{\nat_{S}}$ are the \emph{expansions of}
      $A$ and $f$, respectively, \emph{to the words on} $S$ and, to
      simplify the notation, we will write $A^{\nat}$ and $f^{\nat}$
      instead of $A^{\nat_{S}}$ and $f^{\nat_{S}}$, respectively.

\item From the contravariant functor $\mathrm{MSet}$, from
      $\mathbf{Set}$ to $\mathbf{Cat}$, to the contravariant functor
      $\mathrm{MSet}\comp \mathrm{T}_{\star}^{\opp}$ between the same
      categories, where $\mathrm{T}_{\star}^{\opp}$ is the composite
      of $\mathbf{T}_{\star}^{\opp}$ (the dual of the free monoid
      functor $\mathbf{T}_{\star}$ from $\mathbf{Set}$ to $\mathbf{Mon}$, the
      category of monoids), and $G_{\mathbf{Mon}}$ (the forgetful functor
      from $\mathbf{Mon}$ to $\mathbf{Set}$), there exists a natural
      transformation $(\farg)^{\nat}$ which sends a set $S$ to the
      expansion functor $(\farg)^{\nat_{S}}$ from $\mathbf{Set}^{S}$ to
      $\mathbf{Set}^{\fmon{S}}$.

\item There exists a natural isomorphism $\iota_{S}$ from the
      functor $(\farg)^{\nat_{\fmon{S}}} \comp (\farg)^{\nat_{S}}$ to
      the functor $\Delta_{\cncat_{S}}\comp(\farg)^{\nat_{S}}$, both
      from the category $\mathbf{Set}^{S}$ to the category
      $\mathbf{Set}^{\ffmon{S}}$.
\end{enumerate}
\end{proposition}

\begin{proof}
We restrict ourselves to prove the second and third parts of the
proposition.

(2) $(\farg)^{\nat}$ is a natural transformation from $\mathrm{MSet}$
to $\mathrm{MSet}\comp \mathrm{T}_{\star}^{\opp}$ since, for a mapping
$\varphi\colon S\mor T$, the following diagram commutes
$$
\xymatrix@C=70pt@R=40pt{
\mathbf{Set}^{S}
\ar[r]^{(\farg)^{\nat_{S}}} &
\mathbf{Set}^{\fmon{S}}   \\
\mathbf{Set}^{T}
\ar[u]^{\Delta_{\varphi}}
\ar[r]_{(\farg)^{\nat_{T}}} &
\mathbf{Set}^{\fmon{T}}
\ar[u]_{\Delta_{\fmon{\varphi}}}
}
$$
Observe, in particular, that for a $T$-sorted set $B$, we have that
$(B_{\varphi})^{\nat_{S}} = (B^{\nat_{T}})_{\fmon{\varphi}}$.

(3) It is enough to define, for every $S$-sorted set $A$, the
component $(\iota_{S})_{A}$ of $\iota_{S}$ at $A$, as the
$\ffmon{S}$-isomorph\-ism $(\iota_{S})_{A}\colon A^{\nat\nat}\mor
(A^{\nat})_{\cncat}$ that has as $\ol{w}$-th coordinate, for  $\ol{w} =
(w_{\alpha})_{\alpha\in \bb{\ol{w}}}\in\ffmon{S}$, the canonical isomorphism%
$$
\xymatrix@C=22ex{%
A^{\nat\nat}_{\ol{w}} = \tprod_{\alpha\in\bb{\ol{w}}}
\tprod_{j\in\bb{w_{\alpha}}}A_{{w_{\alpha}}_{j}}
\ar[r]^{\tp{ \pr_{\alpha_{j}}
\comp \pr_{\alpha}}_{\alpha\in\bb{\ol{w}},j\in\bb{w_{\alpha}} }
         } &
\tprod_{\alpha\in\bb{\ol{w}},\,j\in\bb{w_{\alpha}}}A_{{w_{\alpha}}_{j}}
=
A^{\nat}_{\cncat \ol{w}},
}
$$
where $\pr_{\alpha}\colon A_{\ol{w}}\mor A_{w_{\alpha}}$ and
$\pr_{\alpha_{j}}\colon A_{w_{\alpha}} \mor
A_{{w_{\alpha}}_{j}}$ are the canonical projections. To simplify the
notation we will write $\iota^{A}$ instead of $(\iota_{S})_{A}$.
\end{proof}

\begin{corollary}
Let $\varphi\colon S\mor \fmon{T}$ and $\psi\colon T\mor \fmon{U}$ be
mappings.  Then, for every $T$-sorted set $B$ and $U$-sorted set $C$,
we have that
\begin{enumerate}
\item $((B^{\nat_{T}})^{\nat_{\fmon{T}}})_{\fmon{\varphi}}$, denoted
      by ${B_{\fmon{\varphi}}}$, and
      $(B^{\nat_{T}})_{\ext{\varphi}}$, denoted by $B_{\ext{\varphi}}$,
      are isomorphic $\fmon{S}$-sorted sets.

\item $(((C^{\nat_{U}})_{\psi})^{\nat_{T}})_{\varphi}$, denoted
      by ${C_{\psi}}_{\varphi}$, and
      $(C^{\nat_{U}})_{\ext{\psi}\comp\varphi}$, denoted by
      $C_{\ext{\psi}\comp\varphi}$, are isomorphic $S$-sorted sets.

\item There exists an isomorphism $\kappa_{\varphi}^{B}\colon
      \mathrm{BOp}_{T}(B)_{\ext{\varphi}\bprod\ext{\varphi}} \mor
      \mathrm{BOp}_{S}(B_{\varphi})$, where, to simplify, we have written
      $B_{\varphi}$ instead of $(B^{\nat_{T}})_{\varphi}$.
\end{enumerate}

\end{corollary}

\begin{proof}
(1) The isomorphism is $\iota_{\fmon{\varphi}}^{B}= (
\iota_{\fmon{\varphi}(w)}^{B}\colon B_{\fmon{\varphi}(w)} \mor
  B_{\ext{\varphi}(w)}
)_{w\in\fmon{S}}
$,
obtained from the natural isomorphism of the following diagram%
$$
\xymatrix{
\mathbf{Set}^{T}
  \ar@/^6ex/[rr]^{(\farg)^{\nat_{\fmon{T}}}\comp(\farg)^{\nat_{T}}}|{}="i0"
  \ar[r]_{(\farg)^{\nat_{T}}} &
\mathbf{Set}^{\fmon{T}}\xyn{i1}
  \ar[r]^{\Delta_{\cncat_{T}}}
  \ar@/_6ex/[rr]_{\Delta_{\ext{\varphi}}}|{}="y" &
\mathbf{Set}^{\ffmon{T}}\xyn{x}
  \ar[r]^{\Delta_{\fmon{\varphi}}} &
\mathbf{Set}^{\fmon{S}}
\ar @{} "i0";"i1"|{\dir{=>}}^{\,\iota}
\ar @{} "x";"y"|{\dir{ = }}
}
$$

(2) The isomorphism is $ \iota_{\fmon{\psi}\comp\varphi}^{C}= (
\iota_{\fmon{\psi}\comp\varphi(s)}^{C}\colon
C_{\fmon{\psi}\comp\varphi(s)} \mor
  C_{\ext{\psi}\comp\varphi(s)}
)_{s\in S}
$,
obtained from the natural isomorphism of the following diagram%
$$
\xymatrix@C=40pt{
\mathbf{Set}^{U}
  \ar@/^6ex/[rr]^{(\farg)^{\nat_{\fmon{U}}}\comp(\farg)^{\nat_{U}}}|{}="i0"
  \ar[r]_{(\farg)^{\nat_{U}}} &
\mathbf{Set}^{\fmon{U}}\xyn{i1}
  \ar[r]^{\Delta_{\cncat_{U}}}
  \ar@/_6ex/[rr]_{\Delta_{\ext{\psi}}}|{}="y" &
\mathbf{Set}^{{\ffmon{U}}}\xyn{x}
  \ar[r]^{\Delta_{\fmon{\psi}}} &
\mathbf{Set}^{\fmon{T}}
  \ar[r]^{\Delta_{\varphi}} &
\mathbf{Set}^{S}
\ar @{} "i0";"i1"|{\dir{=>}}^{\,\iota}
\ar @{} "x";"y"|{\dir{ = }}
}
$$

(3) It is enough to define, for two words $w,u\in \fmon{S}$,
$(\kappa_{\varphi}^{B})_{w,u}$, i.e., the component at $(w,u)$ of
$\kappa_{\varphi}^{B}$, as the isomorphism from
$\mathrm{Hom}(B_{\ext{\varphi}(w)},B_{\ext{\varphi}(u)})$ to
$\mathrm{Hom}(B_{\fmon{\varphi}(w)},B_{\fmon{\varphi}(u)})$ which sends
a mapping $h\colon B_{\ext{\varphi}(w)}\mor B_{\ext{\varphi}(u)}$ to
the composite mapping
$$
\xymatrix{
{B_{\fmon{\varphi}(w)}}
  \ar[r]^{ \iota_{\fmon{\varphi}(w)}^{B}} &
B_{\ext{\varphi}(w)}
  \ar[r]^{h} &
B_{\ext{\varphi}(u)}
  \ar[r]^{(\iota_{\fmon{\varphi}(u)}^{B})^{-1}} &
{B_{\fmon{\varphi}(u)}},
}
$$
where, we recall, $B_{\fmon{\varphi}(w)} = \textstyle\prod_{i\in
\bb{w}}B_{\varphi(w_{i})}$, and $B_{\fmon{\varphi}(u)} =
\textstyle\prod_{j\in \bb{u}}B_{\varphi(u_{j})}$.
\end{proof}

Once defined the above auxiliary functors and natural transformations
we prove in the following proposition that the polyderivors between
signatures determine functors, in the opposite direction, from the
category of algebras associated to the target signature to the
category of algebras associated to the source signature. These
functors will be the components of the morphism mapping of the
contravariant pseudo-functor $\Alg_{\mathfrak{pd}}$ from
$\mathbf{Sig}_{\mathfrak{pd}}$ to $\mathbf{Cat}$.

\begin{proposition}
Let $\mathbf{d}\colon \mathbf{\Sigma}\mor\mathbf{\Lambda}$ be a
morphism in $\mathbf{Sig}_{\mathfrak{pd}}$.  Then there exists a
functor $\Alg_{\mathfrak{pd}}(\mathbf{d}) =
\mathbf{d}^{\ast}_{\mathfrak{pd}}$ from
$\mathbf{Alg}(\mathbf{\Lambda})$ to $\mathbf{Alg}(\mathbf{\Sigma})$
defined as follows
\begin{enumerate}
\item $\mathbf{d}^{\ast}_{\mathfrak{pd}}$ assigns to a
      $\mathbf{\Lambda}$-algebra $\mathbf{B} = (B,G)$ the
      $\mathbf{\Sigma}$-algebra
      $\mathbf{d}^{\ast}_{\mathfrak{pd}}(\mathbf{B}) =
      (B_{\varphi},G^{\mathbf{d}})$, where $G^{\mathbf{d}}$ is $
      \kappa_{\varphi}^{B}\comp \ext{G}_{\ext{\varphi}\bprod
      \ext{\varphi}}\comp d $, obtained from
      $$
      \begin{aligned}
      \xymatrix@C=40pt@R=30pt{
      \coprod_{1\bprod\between_{T}}\Lambda
      \ar[r]^-{\eta_{\coprod_{1\bprod\between_{T}}\Lambda}^{\Ben_{T}}}
      \ar[rd]_-{G} &
      \mathrm{BTer}_{T}(\Lambda)
      \ar[d]^{\ext{G}}\\
      & \mathrm{BOp}_{T}(B)
      }
      \end{aligned}
      \quad \text{as} \quad \,\,
      \begin{aligned}
      \xymatrix@C=40pt@R=30pt{
      \mathrm{BTer}_{T}(\Lambda)_{\ext{\varphi}\bprod \ext{\varphi}}
      \ar[d]_{\ext{G}_{\ext{\varphi}\bprod \ext{\varphi}}} &
      \coprod_{1\bprod\between_{S}}\Sigma
      \ar[l]_-{d}\\
      \mathrm{BOp}_{T}(B)_{\ext{\varphi}\bprod \ext{\varphi}}
      \ar[r]_{\kappa_{\varphi}^{B}} &
      \mathrm{BOp}_{S}(B_{\varphi})
      }
      \end{aligned}
      $$

\item $\mathbf{d}^{\ast}_{\mathfrak{pd}}$ assigns to a
      $\mathbf{\Lambda}$-homomorphism $f$ from $\mathbf{B}$ to
      $\mathbf{B}'$ the $\mathbf{\Sigma}$-homo\-morphism
      $\mathbf{d}^{\ast}_{\mathfrak{pd}}(f) = f_{\varphi}$ from
      $\mathbf{d}^{\ast}_{\mathfrak{pd}}(\mathbf{B})$ to
      $\mathbf{d}^{\ast}_{\mathfrak{pd}}(\mathbf{B}')$.
\end{enumerate}
\end{proposition}

\begin{proof}
It is obvious that $G^{\mathbf{d}}$, as defined, is an algebraic
structure on $B_{\varphi}$.

Following this we prove that if $f$ is a
$\mathbf{\Lambda}$-homomorphism from $\mathbf{B}$ to $\mathbf{B}'$,
then $f_{\varphi}$ is a $\mathbf{\Sigma}$-homomorphism from
$\mathbf{d}^{\ast}_{\mathfrak{pd}}(\mathbf{B})$ to
$\mathbf{d}^{\ast}_{\mathfrak{pd}}(\mathbf{B}')$.

Let $\sigma\colon w\mor s$ be an  operation in $\Sigma$.  Then
in the following diagram
$$
 \xymatrix@C=40pt@R=20pt{
 &
{B_{\ext{\varphi}(w)}}
  \ar[rr]^{(\ext{G}_{\ext{\varphi}\bprod\ext{\varphi}}\comp d)_{w,s}(\sigma)}
  \ar[dd]_(.65){{f_{\ext{\varphi}(w)}}} & &
{B_{\ext{\varphi}(s)}}
  \ar[rd]^(.45){(\iota^{B}_{\fmon{\varphi}(s)})^{-1}}
  \ar[dd]^(.65){{f_{\ext{\varphi}(s)}}} \\
{B_{\varphi}}_{w}
  \ar[ru]^{\iota^{B}_{\fmon{\varphi}(w)}}
  \ar[rrrr]^{G^{\mathbf{d}}_{\sigma}}
  \ar[dd]_{{f_{\varphi}}_{w}} & & & &
{B_{\varphi}}_{s}
  \ar[dd]^{{f_{\varphi}}_{s}} \\
 &
{B'_{\ext{\varphi}(w)}}
  \ar[rr]_{(\ext{{G'}}_{\ext{\varphi}\bprod\ext{\varphi}}\comp d)_{w,s}(\sigma)} & &
{B'_{\ext{\varphi}(s)}}
  \ar[rd]^(.45){(\iota^{B'}_{\fmon{\varphi}(s)})^{-1}}  \\
{B'_{\varphi}}_{w}
  \ar[ru]^{\iota^{{B'}}_{\fmon{\varphi}(w)}}
  \ar[rrrr]_{{G'}^{\mathbf{d}}_{\sigma}} & & & &
{B'_{\varphi}}_{s} }
$$
the back face commutes because $f$ is a morphism, the top and bottom
faces commute by definition and the left and right faces commute
because $\iota$ is natural.  Hence, the front face commutes, but this
means that $f_{\varphi}$ is a $\mathbf{\Sigma}$-homo\-mor\-phism from
$\mathbf{d}^{\ast}_{\mathfrak{pd}}(\mathbf{B})$ to
$\mathbf{d}^{\ast}_{\mathfrak{pd}}(\mathbf{B}')$.

Finally we prove that $\mathbf{d}^{\ast}_{\mathfrak{pd}}$ is a
functor.  We restrict ourselves to verify that
$\mathbf{d}^{\ast}_{\mathfrak{pd}}$ preserves compositions.  Let
$f\colon \mathbf{B}\mor\mathbf{B}'$ and $g\colon
\mathbf{B}'\mor\mathbf{B}''$ be two
$\mathbf{\Lambda}$-homo\-mor\-phisms and $s\in S$, then we have that
\begin{align*}
    {(g\comp f)_{\varphi}}_{s} &=
    (g\comp f)_{\varphi(s)} \\
    &= (g\comp f)_{\varphi(s)_{0}}\bprod\cdots\bprod
         (g\comp f)_{\varphi(s)_{\bb{\varphi(s)}-1}} \\
    &= (g_{\varphi(s)_{0}}\comp f_{\varphi(s)_{0}}) \bprod\cdots\bprod
       (g_{\varphi(s)_{\bb{\varphi(s)}-1}} \comp f_{\varphi(s)_{\bb{\varphi(s)}-1}}) \\
    &= (g_{\varphi(s)_{0}}\bprod\cdots\bprod g_{\varphi(s)_{\bb{\varphi(s)}-1}})
       \comp
       (f_{\varphi(s)_{0}}\bprod\cdots\bprod f_{\varphi(s)_{\bb{\varphi(s)}-1}}) \\
    &= g_{\varphi(s)} \comp f_{\varphi(s)}\\
    &= {g_{\varphi}}_{s}\comp {f_{\varphi}}_{s}.
    \qedhere
\end{align*}
\end{proof}

Given a polyderivor $\mathbf{d}\colon
\mathbf{\Sigma}\mor\mathbf{\Lambda}$, a $\mathbf{\Lambda}$-algebra
$\mathbf{B}=(B,G)$ and an operation $\sigma\in\Sigma_{w,s}$, if we
agree that $w$ is the word $(s_{i})_{i\in m}$, that, for every $i\in
m$, $\varphi(s_{i})$ is the word $(t_{i,j})_{j\in n_{i}}$, and that
$\varphi(s)$ is the word $(t_{k})_{k\in p}$,
then we have that $\ext{\varphi}(w)$ is the word
$$
(t_{0,0},\ldots,t_{0,n_{0}-1},\ldots, t_{m-1,0},
\ldots, t_{m-1,n_{m-1}-1})
$$
and that $d(\sigma)\colon \ext{\varphi}(w)\mor\varphi(s)$ is a family
of terms $P = (P_{0},\ldots,P_{p-1})$ such that, for every $k\in p$,
$P_{k}\colon \ext{\varphi}(w)\mor t_{k}$.  Thus the realization of
$d(\sigma)$ in $\mathbf{B}$,
$\ext{G}_{\ext{\varphi}\bprod\ext{\varphi}}(P)$, is the term operation
$P^{\mathbf{B}}=\tp{P^{\mathbf{B}}_{0},\ldots,P^{\mathbf{B}}_{p-1}}$ of type %
$$
B_{t_{0,0}}\bprod\cdots\bprod B_{t_{0,n_{0}-1}}\bprod\cdots\bprod B_{t_{m-1,0}}
  \bprod\cdots\bprod B_{t_{m-1,n_{m-1}-1}}
\mor  B_{t_{0}}\bprod\cdots\bprod B_{t_{p-1}}
$$
that by composition with the isomorphism from ${B_{\varphi}}_{w}$ to
$B_{\ext{\varphi}(w)}$ provides the operation
$G^{\mathbf{d}}(\sigma)$
$$
\xymatrix@R=25pt{
{(B_{t_{0,0}}\bprod\cdots\bprod B_{t_{0,n_{0}-1}})
\bprod\cdots\bprod
(B_{t_{m-1,0}}\bprod\cdots\bprod B_{t_{m-1,n_{m-1}-1}})}
\ar[d]^{\iota^{B}_{(\varphi(s_{0}),\ldots,\varphi(s_{m-1}))}}
\\
{B_{t_{0,0}}\bprod\cdots\bprod B_{t_{0,n_{0}-1}}\bprod\cdots
\bprod B_{t_{m-1,0}}\bprod\cdots\bprod B_{t_{m-1,{n_{m-1}}-1}}}
\ar[d]^{P^{\mathbf{B}}} \\
B_{t_{0}}\bprod\cdots\bprod B_{t_{p-1}}
}
$$

It is now when we can prove that the contravariant functor $\Alg$ from
$\mathbf{Sig}$ to $\mathbf{Cat}$, defined in the second section, can
be lifted to a contravariant pseudo-functor $\Alg_{\mathfrak{pd}}$ from
$\mathbf{Sig}_{\mathfrak{pd}}$ to $\mathbf{Cat}$, and, as was the case
there, by applying to this contravariant pseudo-functor the
construction of Ehresmann-Grothendieck, we get the new category
$\mathbf{Alg}_{\mathfrak{pd}}$.

\begin{proposition}\label{defPseudofunctorAlgfuj}
There exists a contravariant pseudo-functor $\Alg_{\mathfrak{pd}}$ from
$\mathbf{Sig}_{\mathfrak{pd}}$ to the $2$-category $\mathbf{Cat}$ given by the
following data
\begin{enumerate}
\item The object mapping of $\Alg_{\mathfrak{pd}}$ is that which sends a
      signature $\mathbf{\Sigma}$ to
      $\Alg_{\mathfrak{pd}}(\mathbf{\Sigma}) = \mathbf{Alg}(\mathbf{\Sigma})$.

\item The morphism mapping of $\Alg_{\mathfrak{pd}}$ is that which sends a
      polyderivor $\mathbf{d}$ from $\mathbf{\Sigma}$ to $\mathbf{\Lambda}$ to
      $\mathbf{d}^{\ast}_{\mathfrak{pd}}\colon \mathbf{Alg}(\mathbf{\Lambda})\mor
      \mathbf{Alg}(\mathbf{\Sigma})$.

\item For every $\mathbf{d}\colon \mathbf{\Sigma}\mor \mathbf{\Lambda}$ and
      $\mathbf{e}\colon \mathbf{\Lambda}\mor \mathbf{\Omega}$, the natural
      isomorphism $\gamma^{\mathbf{d},\mathbf{e}}$ from
      $\mathbf{e}^{\ast}_{\mathfrak{pd}}\comp \mathbf{d}^{\ast}_{\mathfrak{pd}}$
      to $(\mathbf{e}\comp\mathbf{d})^{\ast}_{\mathfrak{pd}}$ is that which is
      defined, for every $\mathbf{\Omega}$-algebra $\mathbf{C}$, as the
      isomorphism $\iota_{\fmon{\psi}\comp\varphi}^{C}$.

\item For every $\mathbf{\Sigma}$, the natural isomorphism
      $\nu^{\mathbf{\Sigma}}$ from $\Id_{\mathbf{Alg}(\mathbf{\Sigma})}$ to
      $(\between_{S},\eta^{\Ben_{S}}_{\Sigma})^{\ast}_{\mathfrak{pd}}$
      is that which is defined, for every $\mathbf{\Sigma}$-algebra
      $\mathbf{A}$, as the canonical isomorphism $\delta_{S}^{A}\colon A\mor
      (A_{(s)})_{s\in S}$.
\end{enumerate}
\end{proposition}

\begin{proof}
Given $\mathbf{d}\colon \mathbf{\Sigma}\mor\mathbf{\Lambda}$ and
$\mathbf{e}\colon \mathbf{\Lambda}\mor \mathbf{\Omega}$,
we prove that%
$$
\iota_{\fmon{\psi}\comp\varphi}^{C}\colon
({C_{\psi}}_{\varphi}, {H^{\mathbf{e}}}^{{}^{\mathbf{d}}})\mor
(C_{\ext{\psi}\comp\varphi},H^{(\ext{\psi}\comp\varphi,
\ext{e}_{\ext{\psi}\bprod\ext{\psi}}\comp d)})
$$
is an isomorphism of $\mathbf{\Sigma}$-algebras, and for this it is enough
to prove that it is a morphism, because
$\iota_{\fmon{\psi}\comp\varphi}^{C}$ is a bijection.

But the following diagram commutes
$$
\xymatrix@C=35pt@R=25pt{
& \Sigma
  \ar[dl]_{{H^{\mathbf{e}}}^{{}^{\mathbf{d}}}}
  \ar[dr]^{\,\,H^{\mathbf{e}\comp\mathbf{d}}} \\
  \mathrm{BOp}_{\fmon{S}}({C_{\psi}}_{\varphi})
  \ar[rr]_{\Op_{\Ben}(\iota^{C}_{\fmon{\psi}\comp\varphi})} &&
  \mathrm{BOp}_{S}(C_{\ext{\psi}\comp\varphi})
}
$$
where $\Op_{\Ben}(\iota^{C}_{\fmon{\psi}\comp\varphi})$ is the
isomorphism from $\mathrm{BOp}_{\fmon{S}}({C_{\psi}}_{\varphi})$ to
$\mathrm{BOp}_{S}(C_{\ext{\psi}\comp\varphi})$ induced by the isomorphism
$\iota^{C}_{\fmon{\psi}\comp\varphi}$, because, on the one hand, we
have that
\begin{align*}
  {H^{\mathbf{e}}}^{{}^{\mathbf{d}}} &=
  \kappa_{\varphi}^{C_{\psi}}
    \comp \ext{(H^{\mathbf{e}})}_{\ext{\varphi}\bprod\ext{\varphi}}
    \comp d \\
  &=
  \kappa_{\varphi}^{C_{\psi}}
    \comp \ext{
      (\kappa_{\psi}^{C}\comp
       \ext{H}_{\ext{\psi}\bprod\ext{\psi}}\comp
       e)
    }_{\ext{\varphi}\bprod\ext{\varphi}}
    \comp d \\
  &=
  \kappa_{\varphi}^{C_{\psi}}
    \comp
      (\kappa_{\psi}^{C}\comp
       \ext{H}_{\ext{\psi}\bprod\ext{\psi}}\comp
       \ext{e})
    _{\ext{\varphi}\bprod\ext{\varphi}}
    \comp d \\
  &=
  \kappa_{\varphi}^{C_{\psi}}
    \comp
       (\kappa_{\psi}^{C})_{\ext{\varphi}\bprod\ext{\varphi}}\comp
       (\ext{H}_{\ext{\psi}\bprod\ext{\psi}})_{\ext{\varphi}\bprod\ext{\varphi}}\comp
       \ext{e}_{\ext{\varphi}\bprod\ext{\varphi}}
    \comp d \\
  &=
  \kappa_{\varphi}^{C_{\psi}}
    \comp
       (\kappa_{\psi}^{C})_{\ext{\varphi}\bprod\ext{\varphi}}\comp
       \ext{H}_{\ext{(\ext{\psi}\comp\varphi)}\bprod\ext{(\ext{\psi}\comp\varphi)}}
       \comp\ext{e}_{\ext{\varphi}\bprod\ext{\varphi}}
    \comp d,
\intertext{on the other hand, that}
  H^{\mathbf{e}\comp\mathbf{d}} &=
  H^{(\ext{\psi}\comp\varphi,\ext{e}_{\ext{\varphi}\bprod\ext{\varphi}}\comp
  d)} \\
  &=
  \kappa_{\ext{\psi}\comp\varphi}^{C}\comp
  \ext{H}_{\ext{(\ext{\psi}\comp\varphi)}\bprod\ext{(\ext{\psi}\comp\varphi)}}
  \comp\ext{e}_{\ext{\varphi}\bprod\ext{\varphi}} \comp d,
\end{align*}
and, lastly, that the following diagram commutes%
$$
\xymatrix@C=35pt@R=25pt{
(\mathrm{BOp}_{U}(C)_{\ext{\psi}\bprod\ext{\psi}})_{\ext{\varphi}\bprod\ext{\varphi}}
\ar[d]_{(\kappa^{C}_{\psi})_{\ext{\varphi}\bprod\ext{\varphi}}}
\ar[ddr]^{\kappa^{C}_{\ext{\psi}\comp\varphi}} \\
\mathrm{BOp}_{T}(C_{\psi})_{\ext{\varphi}\bprod\ext{\varphi}}
\ar[d]_{\kappa^{C_{\psi}}_{\varphi}} \\
\mathrm{BOp}_{\fmon{S}}({C_{\psi}}_{\varphi})
\ar[r]_{\Op_{\Ben}(\iota^{C}_{\fmon{\psi}\comp\varphi})} &
\mathrm{BOp}_{S}(C_{\ext{\psi}\comp\varphi})
}
$$
\end{proof}

\begin{definition}
The category $\mathbf{Alg}_{\mathfrak{pd}}$ of
algebras and morphisms, obtained by applying the
Ehresmann-Grothendieck construction to the contravariant
pseudo-functor $\Alg_{\mathfrak{pd}}$, is $\mathbf{Alg}_{\mathfrak{pd}} =
\int^{\mathbf{Sig}_{\mathfrak{pd}}}\Alg_{\mathfrak{pd}}$.
\end{definition}

Therefore the category $\mathbf{Alg}_{\mathfrak{pd}}$ has as objects
the pairs $(\mathbf{\Sigma},\mathbf{A})$, with $\mathbf{\Sigma}$ a
signature and $\mathbf{A}$ a $\mathbf{\Sigma}$-algebra, and as
morphisms from $(\mathbf{\Sigma},\mathbf{A})$ to
$(\mathbf{\Lambda},\mathbf{B})$, the pairs $(\mathbf{d},h)$, with
$\mathbf{d}$ a polyderivor from $\mathbf{\Sigma}$ to $\mathbf{\Lambda}$
and $h$ a $\mathbf{\Sigma}$-homomorphism from $\mathbf{A}$ to
$\mathbf{d}^{\ast}_{\mathfrak{pd}}(\mathbf{B})$.  Hence, for every
$(w,s)\in \fmon{S}\times S$ and $\sigma\in \Sigma_{w,s}$ the following
diagram commutes
$$
\xymatrix@C=80pt@R=40pt{
A_{w}
  \ar[r]^-{h_{w}}
  \ar[d]_{F_{\sigma}} &
(\prod_{j\in n_{0}}B_{t_{0},j})\times\cdots\times
(\prod_{j\in n_{m-1}}B_{t_{m-1},j})
  \ar[d]^{G^{\mathbf{d}}(\sigma)}
  \\
A_{s}
  \ar[r]_{h_{s}} &
\prod_{k\in p}B_{t_{k}}
}
$$
where we have agreed that $w$ is the word $(s_{i})_{i\in m}$, that,
for every $i\in m$, $\varphi(s_{i})$ is the word $(t_{i,j})_{j\in
n_{i}}$, and that $\varphi(s)$ is the word $(t_{k})_{k\in p}$.


\begin{example}
Let $\mathbf{\Sigma}$ be a signature, $p\in \mathbb{N}$, and
$\mathbf{d}= (\varphi,d)$ the polyderivor from $\mathbf{\Sigma}$ into
itself, where
\begin{enumerate}
\item $\varphi\colon S\mor \fmon{S}$ is the mapping which sends
      $s\in S$ to the word $\cncat_{\mu\in p}(s)$ and,

\item For $(w,s)\in \fmon{S}\times S$, $d_{w,s}$ is the mapping
      from $\Sigma_{w,s}$ to
      $\mathrm{T}_{\mathbf{\Sigma}}(\vs{\ext{\varphi}(w)})_{s}^{p}$
      which sends $\sigma\in\Sigma_{w,s}$ to
      $$
        (\sigma(v_{0}^{w_{0}}, v_{p}^{w_{1}},\ldots,
                v_{(\bb{w}-1)p}^{w_{\bb{w}-1}}),\ldots,
         \sigma(v_{p-1}^{w_{0}}, v_{2p-1}^{w_{1}},\ldots,
                v_{\bb{w}p-1}^{w_{\bb{w}-1}})),
      $$
      in $\mathrm{T}_{\mathbf{\Sigma}}(\vs{\ext{\varphi}(w)})_{s}^{p}$.
\end{enumerate}
Then, for the polyderivor $\mathbf{d}$ and two $\mathbf{\Sigma}$-algebras
$\mathbf{A}$ and $\mathbf{B}$, we have that
$(\mathbf{d},\left<h^{\mu}\right>_{\mu\in p})$, where, for every
$\mu\in p$, $h^{\mu} = (h^{\mu}_{s})_{s\in S}$ is a
$\mathbf{\Sigma}$-homomorphism from $\mathbf{A}$ to $\mathbf{B}$, is a
morphism from $(\mathbf{\Sigma},\mathbf{A})$ to
$(\mathbf{\Sigma},\mathbf{B})$, because
$\mathbf{d}^{\ast}_{\mathfrak{pd}}(\mathbf{B}) = \mathbf{B}^{p}$.

\end{example}

The following example of morphism between algebras (stated using the
terminology of Fujiwara) although redundant, since it is a particular
instance of the preceding one, is provided because, being refered to
single-sorted signatures and algebras, it is by far less troublesome
and more easily graspable.

\begin{example}
Let $\Sigma = (\Sigma_{n})_{n\in\mathbb{N}}$ be a single-sorted
signature, $\Phi = \{\,\varphi_{\mu}\mid \mu\in m\,\}$, and $P =
(P^{n})_{n\in \mathbb{N}}$ the family defined, for every natural
number $n\in \mathbb{N}$, as follows
$$
  P^{n}\nfunction
  {\Phi\times \Sigma_{n}}{\mathrm{T}_{\Sigma}(\Phi\times\vs{v_{n}})}
  {(\varphi_{\mu},\sigma)}
  {\sigma(\varphi_{\mu}(v_{0}),\ldots,\varphi_{\mu}(v_{n-1}))}
$$
Then, for the polyderivor $\mathbf{d}$ associated to $(\Phi,P)$ and two
$\Sigma$-algebras $\mathbf{A}$ and $\mathbf{B}$, we have
that $(\mathbf{d},\left<h_{\mu}\right>_{\mu\in p})$, where, for every
$\mu\in p$, $h_{\mu}$ is a $\Sigma$-homomorphism from
$\mathbf{A}$ to $\mathbf{B}$, is a morphism from
$(\Sigma,\mathbf{A})$ to $(\Sigma,\mathbf{B})$,
because $\mathbf{d}^{\ast}_{\mathfrak{pd}}(\mathbf{B}) =
\mathbf{B}^{m}$.
\end{example}

\begin{example}
Let $\Sigma = (\Sigma_{n})_{n\in\mathbb{N}}$ be a single-sorted
signature such that $\Sigma_{2} = \{\,+,-,\cdot\,\}$ and $\Sigma_{n} =
\vacio$, if $n\neq 2$, $\Phi = \{\,\varphi_{\mu,\nu}\mid (\mu,\nu)\in
2\times 2\,\}$, and $P = (P^{n})_{n\in \mathbb{N}}$ the family
defined, for  $n \neq 2$, as the unique mapping from $\vacio$ to
$\mathrm{T}_{\Sigma}(\Phi\times\vs{v_{n}})$, and, for $n = 2$, as follows
\begin{enumerate}
\item $P^{2}_{\varphi_{\mu,\nu},+} =
      \varphi_{\mu,\nu}(v_{0})+\varphi_{\mu,\nu}(v_{1})$.

\item $P^{2}_{\varphi_{\mu,\nu},-} =
      \varphi_{\mu,\nu}(v_{0})-\varphi_{\mu,\nu}(v_{1})$.

\item $P^{2}_{\varphi_{\mu,\nu},\cdot} =
      \sum_{\lambda=
      0}^{1}\varphi_{\mu,\lambda}(v_{0})\cdot\varphi_{\lambda,\nu}(v_{1})$.
\end{enumerate}
Then, for the polyderivor $\mathbf{d}$ associated to $(\Phi,P)$ and two
rings $\mathbf{A}$ and $\mathbf{B}$, we have
that $\left(\mathbf{d},\left(
\begin{smallmatrix}
f & \kappa_{0} \\
d & g
\end{smallmatrix}
\right) \right)$, where $f$ and $g$ are two ring homomorphisms from
$\mathbf{A}$ to $\mathbf{B}$, $\kappa_{0}$ is the mapping from $A$ to
$B$ that is constantly $0$, and $d$ an $(f,g)$-derivation from
$\mathbf{A}$ to $\mathbf{B}$, is a morphism from $(\Sigma,\mathbf{A})$
to $(\Sigma,\mathbf{B})$, because
$\mathbf{d}^{\ast}_{\mathfrak{pd}}(\mathbf{B}) = \mathbf{B}^{2\times
2}$, the matrix ring of degree $2$ over $\mathbf{B}$.
\end{example}

\begin{example}
Taking as set of sorts $S$ a set with two elements
$\{\,\mathsf{s},\mathsf{v}\,\}$, where the sort $\mathsf{s}$ is for
\lq\lq scalar\rq\rq and the sort $\mathsf{v}$ for \lq\lq vector\rq\rq,
and as $S$-sorted signature $\Sigma$ that one adequate for the concept
of vectorial space (with the scalar field variable), we have that for
a $\mathbf{K}$-vectorial space $\mathbf{V}$ and a
$\mathbf{K}'$-vectorial space $\mathbf{V}'$, if $f$ is a homomorphism
from the field $\mathbf{K}$ to the field $\mathbf{K}'$ and $g$ is an
$f$-semilineal morphism from the $\mathbf{K}$-vectorial space $\mathbf{V}$
to the $\mathbf{K}'$-vectorial space $\mathbf{V}'$, then
$(\mathbf{id}_{\mathbf{\Sigma}},(f,g))$ is a morphism from
$(\mathbf{\Sigma},(\mathbf{K},\mathbf{V}))$ to
$(\mathbf{\Sigma},(\mathbf{K}',\mathbf{V}'))$.

More examples like this one can be obtained, e.g., from the concept of
$\mathbf{R}$-module, for a ring $\mathbf{R}$, from that of
$\mathbf{G}$-set, for a group $\mathbf{G}$, from that of
$\mathbf{M}$-set, for a monoid $\mathbf{M}$, or from metric or
pseudo-metric spaces.
\end{example}

Additional examples related to computer sciences can be found
in~\cite{gtw76}.


The contravariant pseudo-functor $\Alg_{\mathfrak{pd}}$ is not only
useful to construct the category $\mathbf{Alg}_{\mathfrak{pd}}$.
Actually, as we prove in that which follows, it together with a
pseudo-functor $\mathrm{Ter}_{\mathfrak{pd}}$ from
$\mathbf{Sig}_{\mathfrak{pd}}$ to $\mathbf{Cat}$, and a
pseudo-extranatural transformation $(\mathrm{Tr},\theta)$ (from a
pseudo-functor on $\mathbf{Sig}^{\mathrm{op}}_{\mathfrak{pd}}\times
\mathbf{Sig}_{\mathfrak{pd}}$ to $\mathbf{Cat}$, induced by
$\Alg_{\mathfrak{pd}}$ and $\mathrm{Ter}_{\mathfrak{pd}}$, to the
functor, between the same categories, constantly $\mathbf{Set}$),
determine an institution $\mathfrak{Tm}_{\mathfrak{pd}} =
(\mathbf{Sig}_{\mathfrak{pd}},\mathrm{Alg}_{\mathfrak{pd}},
\mathrm{Ter}_{\mathfrak{pd}},(\mathrm{Tr},\theta))$ on $\mathbf{Set}$,
the so-called \emph{many-sorted term institution of Fujiwara}.

We define next some auxiliary functors and natural transformations
that we will use afterwards to prove, on the one hand, that there
exists a pseudo-functor $\mathrm{Ter}_{\mathfrak{pd}}$ from the
category $\mathbf{Sig}_{\mathfrak{pd}}$ to the $2$-category
$\mathbf{Cat}$, which generalizes the pseudo-functor $\mathrm{Ter}$
from the category $\mathbf{Sig}$ to the $2$-category $\mathbf{Cat}$,
and, on the other hand, that the category
$\mathbf{Alg}_{\mathfrak{pd}}$ has coproducts.

\begin{proposition}\label{adjuncionDaggerNatS}
Let $S$ be a set of sorts.  Then we have that
\begin{enumerate}
\item There exists a \emph{compression} functor $(\farg)^{\dagger_{S}}$
      from $\mathbf{Set}^{\fmon{S}}$ to $\mathbf{Set}^{S}$, left adjoint to the
      expansion functor $(\farg)^{\nat_{S}}$, defined, for every
      $\fmon{S}$-sorted set $C$ and $s\in S$, as follows
      $$
      C^{\dagger_{S}}_{s} =
      \bigcup\nolimits_{\substack{w\in \fmon{S}\And \\ w^{-1}[s]\neq
      \vacio}}(C_{w}\times \{w\}\times w^{-1}[s]),
      $$
      and, for every $\fmon{S}$-mapping $f\colon C\mor C'$, $s\in
      S$ and $(c,w,i)$ in $C^{\dagger_{S}}_{s}$, as follows
      $$
      f^{\dagger_{S}}_{s}(c,w,i) = (f_{w}(c),w,i).
      $$

\item From the contravariant functor
      $\mathrm{MSet}\comp \mathrm{T}_{\star}^{\opp}$, from
      $\mathbf{Set}$ to $\mathbf{Cat}$, to the contravariant functor
      $\mathrm{MSet}$ between the same
      categories, there exists a natural transformation
      $(\farg)^{\dagger}$ which sends a set $S$ to the
      compression functor $(\farg)^{\dagger_{S}}$ from
      $\mathbf{Set}^{\fmon{S}}$ to $\mathbf{Set}^{S}$.

\item There exists a natural isomorphism $\zeta_{S}$ from the
      functor $(\farg)^{\dagger_{S}} \comp
      (\farg)^{\dagger_{\fmon{S}}}$ to the functor
      $(\farg)^{\dagger_{S}}\comp\tcoprod_{\cncat_{S}}$.
\end{enumerate}
\end{proposition}

\begin{proof}
We restrict ourselves to prove the first and third part of the
proposition.

(1) For every $\fmon{S}$-sorted set $C$ and $S$-sorted set $A$, there
exists a natural isomorphism $\theta^{\dagger\nat}\colon
\mathrm{Hom}(C^{\dagger_{S}},A)\iso \mathrm{Hom}(C,A^{\nat_{S}})$
which assigns to an $S$-sorted mapping $f\colon C^{\dagger_{S}}\mor A$
the $\fmon{S}$- sorted mapping $\theta^{\dagger\nat}(f)$, defined, for
every $w\in \fmon{S}$ and $c\in C_{w}$, as
$\theta^{\dagger\nat}(f)_{w}(c) = (f_{w_{i}}(c,w,i))_{i\in\bb{w}}$.

Reciprocally, if $g\colon C\mor A^{\nat_{S}}$ is an $\fmon{S}$-sorted
mapping, then $(\theta^{\dagger\nat})^{-1}(g)$ is the $S$-sorted
mapping defined, for every $s\in S$ and $(c,w,i)\in
C^{\dagger_{S}}_{s}$, as $(\theta^{\dagger\nat})^{-1}(g)_{s}(c,w,i) =
g_{w}(c)_{i}$.

(3) By Proposition~\ref{isonatcncat}, the functors
$(\farg)^{\nat_{\fmon{S}}} \comp (\farg)^{\nat_{S}}$ and
$\Delta_{\cncat_{S}}\comp(\farg)^{\nat_{S}}$ are isomorphic.
Furthermore, $(\farg)^{\dagger_{S}} \comp (\farg)^{\dagger_{\fmon{S}}}$
is left adjoint to $(\farg)^{\nat_{\fmon{S}}} \comp
(\farg)^{\nat_{S}}$ and
$(\farg)^{\dagger_{S}}\comp\coprod_{\cncat_{S}}$ is left adjoint to
$\Delta_{\cncat_{S}}\comp(\farg)^{\nat_{S}}$, thus the functors
$(\farg)^{\dagger_{S}}\comp (\farg)^{\dagger_{\fmon{S}}}$ and
$(\farg)^{\dagger_{S}}\comp\tcoprod_{\cncat_{S}}$ are isomorphic. We
denote such a natural isomorphism by $\zeta_{S}$, and, to
simplify the notation, we will write $\zeta^{C}$ instead of
$(\zeta_{S})_{C}$ for the component of $\zeta_{S}$ at $C$.
\end{proof}

If $\varphi\colon S\mor\fmon{T}$ is a mapping, then from the
adjunctions $\coprod_{\varphi}\ladj\Delta_{\varphi}$ and
$(\farg)^{\dagger_{T}}\ladj (\farg)^{\nat_{T}}$, we get the
adjunction $\coprod_{\varphi}^{\dagger}\ladj
\Delta_{\varphi}^{\nat}$ as reflected in the following diagram
$$
\xymatrix@C=60pt@R=60pt{
{} & {\mathbf{Set}^{S}}
  \ar@/_1.5ex/[rd]_{\coprod_{\varphi}}
  \ar@{}[rd]|{\ladj}
  \ar@/_1.5ex/[rd];[]_{\Delta_{\varphi}}
  \ar@/_1.5ex/[ld]_{\coprod_{\varphi}^{\dagger}}
  \ar@{}[ld]|{\ladj}
  \ar@/_1.5ex/[ld];[]_{\Delta_{\varphi}^{\nat}}
  & {}\\
{\mathbf{Set}^{T}}
  \ar@/^1.5ex/[rr]^(.5){(\farg)^{\nat_{T}}}
  \ar@{}[rr]|(.5){\uadj}
  \ar@/^1.5ex/[rr];[]^(.5){(\farg)^{\dagger_{T}}}
  & {} &
{\mathbf{Set}^{\fmon{T}}}
}
$$
where, to simplify the notation, we have written
$\coprod_{\varphi}^{\dagger}$ instead of
$(\farg)^{\dagger_{T}}\comp\coprod_{\varphi}$ and
$\Delta_{\varphi}^{\nat}$ instead of $\Delta_{\varphi}\comp
(\farg)^{\nat_{T}}$.  Furthermore, we agree that
$\theta_{\varphi}^{\dagger\nat}$, $\eta_{\varphi}^{\dagger\nat}$, and
$\varepsilon_{\varphi}^{\dagger\nat}$ denote, respectively, the
natural isomorphism, the unit, and the counit of this composite
adjunction.

Since it could be of some help, next we recall the explicit
definitions of $\coprod_{\varphi}^{\dagger}$,
$\Delta_{\varphi}^{\nat}$, and $\theta_{\varphi}^{\dagger\nat}$.

\begin{enumerate}
\item The functor $\coprod_{\varphi}^{\dagger}$ assigns to an
      $S$-sorted set $A$ the $T$-sorted set
      $\coprod_{\varphi}^{\dagger}A$ whose $t$-th coordinate, for
      $t\in T$, is
      $$
      \textstyle (\coprod_{\varphi}^{\dagger}A)_{t} =
      \bigcup\nolimits_{\substack{x\in \fmon{T}\And \\ x^{-1}[t]\neq
      \vacio}}
      ((\bigcup_{s\in \varphi^{-1}[x]}(A_{s}\times \{s\}))\times
      \{x\}\times x^{-1}[t]),
      $$
      therefore $(\coprod_{\varphi}^{\dagger}A)_{t}$ has as members
      the ordered quadruples $(a,s,x,j)$ such that $a\in A_{s}$,
      $\varphi(s) = x$ and $\varphi(s)_{j} = t$.

\item The functor $\Delta_{\varphi}^{\nat}$ assigns to a
      $T$-sorted set $B$ the $S$-sorted set $\Delta_{\varphi}^{\nat}B
      = (B_{\varphi(s)})_{s\in S}$, where, we recall, $B_{\varphi(s)}
      = \prod_{i\in \bb{\varphi(s)}}B_{\varphi(s)_{i}}$.

\item For an $S$-sorted set $A$ and a $T$-sorted set $B$,
      the natural isomorphism $\theta_{\varphi}^{\dagger\nat}$
      of the adjunction
      $\coprod_{\varphi}^{\dagger}\ladj\Delta_{\varphi}^{\nat}$ sends
      a $T$-sorted mapping $f\colon \coprod_{\varphi}^{\dagger}A\mor
      B$ to the $S$-sorted mapping
      $\theta_{\varphi}^{\dagger\nat}(f)\colon A\mor
      \Delta_{\varphi}^{\nat}B$ that, for $s\in S$ and $a\in A_{s}$,
      is defined as follows
      $$
      \theta_{\varphi}^{\dagger\nat}(f)_{s}
      \nfunction
      {A_{s}}{B_{\varphi(s)}}
      {a}{(f_{\varphi(s)_{i}}(a,s,\varphi(s),i))_{i\in\bb{\varphi(s)}}}
      $$
      and $(\theta_{\varphi}^{\dagger\nat})^{-1}$ sends an $S$-sorted
      mapping $g\colon A\mor \Delta_{\varphi}^{\nat}B$ to the
      $T$-sorted mapping
      $(\theta_{\varphi}^{\dagger\nat})^{-1}(g)\colon
      \coprod_{\varphi}^{\dagger}A\mor B$ that, for $t\in T$, is
      defined as follows
      $$
      (\theta_{\varphi}^{\dagger\nat})^{-1}(g)_{t}
      \nfunction {(\coprod_{\varphi}^{\dagger}A)_{t}}{B_{t}}
      {(a,s,\varphi(s),i)}{g_{s}(a)_{i}}
      $$
\end{enumerate}

What we want to establish now is that the category
$\mathbf{Alg}_{\mathfrak{pd}}$ has coproducts and for this we begin by
proving that, for every polyderivor $\mathbf{d}\colon \mathbf{\Sigma}\mor
\mathbf{\Lambda}$, the functor $\mathbf{d}^{\ast}_{\mathfrak{pd}}$ from
$\mathbf{Alg}(\mathbf{\Lambda})$ to $\mathbf{Alg}(\mathbf{\Sigma})$
has a left adjoint $\mathbf{d}_{\ast}^{\mathfrak{pd}}$.

\begin{proposition}
Let $\mathbf{d}\colon \mathbf{\Sigma}\mor \mathbf{\Lambda}$ be a
polyderivor.  Then there exists a functor
$\mathbf{d}_{\ast}^{\mathfrak{pd}}$ from
$\mathbf{Alg}(\mathbf{\Sigma})$ to $\mathbf{Alg}(\mathbf{\Lambda})$
that is left adjoint to the functor $\mathbf{d}^{\ast}_{\mathfrak{pd}}$
from $\mathbf{Alg}(\mathbf{\Lambda})$ to
$\mathbf{Alg}(\mathbf{\Sigma})$.
\end{proposition}

\begin{proof}
We restrict ourselves to define the action of
$\mathbf{d}_{\ast}^{\mathfrak{pd}}$ on the objects (because being the
remaining details like those of Proposition~\ref{leftadjsigmorph},
although more cumbersome, they can be left to the reader).  Let
$\mathbf{A}$ be a $\mathbf{\Sigma}$-algebra.  Then
$\mathbf{d}_{\ast}^{\mathfrak{pd}}(\mathbf{A})$ is the
$\mathbf{\Lambda}$-algebra defined as %
$
\mathbf{T}_{\mathbf{\Lambda}}(\tcoprod_{\varphi}^{\dagger}A)/\cl{R}{}^{\mathbf{A}}
$, where $\cl{R}{}^{\mathbf{A}}$ is the congruence on
$\mathbf{T}_{\mathbf{\Lambda}}(\coprod_{\varphi}^{\dagger}A)$
generated by the $T$-sorted relation $R^{\mathbf{A}}$, defined, for
every $t\in T$, as
$$
R^{\mathbf{A}}_{t} =
  \biggl\{
    \left( (F_{\sigma}^{\mathbf{A}}(a_{i}\mid i\in\bb{w}),s,\varphi(s),j) ,
    d(\sigma)_{j}(\mathbf{a})
    \right)
    \biggm|
    \begin{gathered}
    j\in \varphi(s)^{-1}[t],\, w\in\fmon{S},\,  \\
    s\in S,\,\sigma\in\Sigma_{w,s},\, a\in A_{w}
    \end{gathered}
  \biggr\},
$$
being $\mathbf{a}$ the matrix
$$
\mathbf{a} =
    \left(
    \begin{smallmatrix}
    (a_{0},w_{0},\varphi(w_{0}),0) & \cdots &
    (a_{0},w_{0},\varphi(w_{0}),\bb{\varphi(w_{0})}-1) \\
    \vdots & \ddots & \vdots \\
    (a_{\bb{w}-1},w_{\bb{w}-1},\varphi(w_{\bb{w}-1}),0) & \cdots &
    (a_{\bb{w}-1},w_{\bb{w}-1},\varphi(w_{\bb{w}-1}),\bb{\varphi(w_{\bb{w}-1})}-1)
    \end{smallmatrix}
    \right),
$$
and $d(\sigma)_{j}(\mathbf{a})$ the result of replacing the variables
in the term $d(\sigma)_{j}$ with the entries in the matrix
$\mathbf{a}$ (recall that, for $\sigma\in\Sigma_{w,s}$, we have agreed
that $d(\sigma) = d_{w,s}(\sigma)$, where $d_{w,s}(\sigma)\in
\mathrm{T}_{\mathbf{\Lambda}}(\vs{\varphi^{\sharp}(w)})_{\varphi(s)}$,
hence, for every $j\in \bb{\varphi(s)}$, $d(\sigma)_{j}\in
\mathrm{T}_{\mathbf{\Lambda}}(\vs{\varphi^{\sharp}(w)})_{\varphi(s)_{j}}$).
\end{proof}

\begin{proposition}
The category  $\mathbf{Alg}_{\mathfrak{pd}}$ has coproducts.
\end{proposition}

\begin{proof}
The category $\mathbf{Sig}_{\mathfrak{pd}}$ has coproducts.  For every
signature $\mathbf{\Sigma}$, the category
$\mathbf{Alg}(\mathbf{\Sigma})$ has coproducts.  The functor
$\mathrm{Alg}_{\mathfrak{pd}}$ is locally reversible.  Therefore, by a
particular case of Proposition~\ref{CoLimitesGrt}, 
the category
$\mathbf{Sig}_{\mathfrak{pd}}$ has coproducts.
\end{proof}

Our next goal is to state that every polyderivor induces a functor between
the associated categories of terms as was the case for the signature
morphisms.

\begin{proposition}\label{polyderivor induces 1-cell}
Let $\mathbf{d}\colon \mathbf{\Sigma}\mor \mathbf{\Lambda}$ be a polyderivor.
Then there exists a functor $\mathbf{d}_{\diamond}^{\mathfrak{pd}}$ from
$\mathbf{Ter}(\mathbf{\Sigma})$ to $\mathbf{Ter}(\mathbf{\Lambda})$ defined
as follows
\begin{enumerate}
\item $\mathbf{d}_{\diamond}^{\mathfrak{pd}}$ sends an $S$-sorted set
      $X$ to the $T$-sorted set
      $\mathbf{d}_{\diamond}^{\mathfrak{pd}}(X) =
      \coprod_{\varphi}^{\dagger}X$.

\item $\mathbf{d}_{\diamond}^{\mathfrak{pd}}$ sends a morphism
      $P$ from $X$ to $Y$ in $\mathbf{Ter}(\mathbf{\Sigma})$ to the
      morphism $\mathbf{d}_{\diamond}^{\mathfrak{pd}}(P) =
      (\theta_{\varphi}^{\dagger\nat})^{-1}(\eta^{\mathbf{d}}_{X}\comp
      P)$ from $\coprod_{\varphi}^{\dagger}X$ to
      $\coprod_{\varphi}^{\dagger}Y$, where
      $\theta_{\varphi}^{\dagger\nat}$ is the natural isomorphism of
      the adjunction $\coprod_{\varphi}^{\dagger}\ladj
      \Delta_{\varphi}^{\nat}$, $\eta^{\mathbf{d}}_{X}$ the
      $\mathbf{\Sigma}$-homo\-morphism from
      $\mathbf{T}_{\mathbf{\Sigma}}(X)$ to $\Delta_{\varphi}^{\nat}
      (\mathrm{T}_{\mathbf{\Lambda}}(\coprod_{\varphi}^{\dagger}X))$
      that extends the $S$-sorted mapping
      $\Delta_{\varphi}^{\nat}(\eta_{\coprod_{\varphi}^{\dagger}X})
      \comp (\eta_{\varphi}^{\dagger\nat})_{X}$ from $X$ to
      $\Delta_{\varphi}^{\nat}
      (\mathrm{T}_{\mathbf{\Lambda}}(\coprod_{\varphi}^{\dagger}X))$,
      as in the following commutative diagram
      $$
      \xymatrix@C=13ex@R=10ex{
      X \ar[r]^-{\eta_{X}}
      \ar[d]_{(\eta_{\varphi}^{\dagger\nat})_{X}} &
      \mathrm{T}_{\mathbf{\Sigma}}(X)
      \ar[d]^{\eta^{\mathbf{d}}_{X} = \ext
      {(\Delta_{\varphi}^{\nat}(\eta_{\coprod_{\varphi}^{\dagger}X})
      \comp (\eta_{\varphi}^{\dagger\nat})_{X})}} \\
      \Delta_{\varphi}^{\nat}\coprod_{\varphi}^{\dagger}X
      \ar[r]_-{\Delta_{\varphi}^{\nat}(\eta_{\coprod_{\varphi}^{\dagger}X})} &
      \Delta_{\varphi}^{\nat}(\mathrm{T}
      _{\mathbf{\Lambda}}(\coprod_{\varphi}^{\dagger}X))
      }
      $$
      and $\eta_{\varphi}^{\dagger\nat}$ the unit of the adjunction
      $\coprod_{\varphi}^{\dagger}\ladj \Delta_{\varphi}^{\nat}$.
\end{enumerate}
\end{proposition}

\begin{proof}
The proof is structurally identical to that of
Proposition~\ref{functorMorfismoSignaturas}. 
However, we remark that
besides it, there is another alternative proof founded on the fact
that, for every term $P\colon X\mor Y$, the term
$\mathbf{d}_{\diamond}^{\mathfrak{pd}}(P)\colon
\coprod_{\varphi}^{\dagger}X \mor \coprod_{\varphi}^{\dagger}Y$ is the
composition of the morphisms in the following diagram
$$
\xymatrix@C=9.5 ex{
\coprod_{\varphi}^{\dagger}Y
\ar[r]^-{\coprod_{\varphi}^{\dagger}P} &
*+!<1.89ex,0ex>{\coprod_{\varphi}^{\dagger}\mathrm{T}_{\mathbf{\Sigma}}(X)}
\ar[r]^-{\coprod_{\varphi}^{\dagger}\eta^{\mathbf{d}}_{X}} &
\coprod_{\varphi}^{\dagger}\Delta_{\varphi}^{\nat}
(\mathrm{T}_{\mathbf{\Lambda}}(\coprod_{\varphi}^{\dagger}X))
\ar[d]^-{(\varepsilon_{\varphi}^{\dagger\nat})
_{\mathrm{T}_{\mathbf{\Lambda}}(\coprod_{\varphi}^{\dagger}X)}\,\,} \\
{} & {} &
\mathrm{T}_{\mathbf{\Lambda}}(\coprod_{\varphi}^{\dagger}X)  }
$$
where $(\varepsilon_{\varphi}^{\dagger\nat})
_{\mathrm{T}_{\mathbf{\Lambda}}(\coprod_{\varphi}^{\dagger}X)}$ is the
value at $\mathrm{T}_{\mathbf{\Lambda}}(\coprod_{\varphi}^{\dagger}X)$
of the counit of the adjunction
$\coprod_{\varphi}^{\dagger}\ladj\Delta_{\varphi}^{\nat}$.
\end{proof}

\begin{remark}
In the last proposition, $(\eta_{\varphi}^{\dagger\nat})_{X}$, the
value at $X$ of the unit of the adjunction
$\coprod_{\varphi}^{\dagger}\ladj\Delta_{\varphi}^{\nat}$, in its
$s$-th coordinate, assigns to $x\in X_{s}$, the family
$$
((x,s,\varphi(s),0),\ldots,(x,s,\varphi(s),\bb{\varphi(s)}-1)).
$$
We can say, informally, that to a variable $x\in X_{s}$ corresponds in
$\coprod_{\varphi}^{\dagger}X$ a family of variables of the form
$(x,s,\varphi(s),i)$, with sorts $\varphi(s)_{i}$, and that, for a
morphism $P\colon X\mor Y$ in $\mathbf{Ter}(\mathbf{\Sigma})$ and
$(y,s,\varphi(s),i)\in(\coprod_{\varphi}^{\dagger}Y)_{t}$,
$\mathbf{d}_{\diamond}^{\mathfrak{pd}}(P)_{\varphi(s)_{i}}(y,s,\varphi(s),i)$
is the term for $\mathbf{\Lambda}$ obtained by replacing, recursively,
in $P_{s}(y)$ the formal operations $\sigma\colon w\mor s$ by the
families of formal operations $d(\sigma)\colon
\ext{\varphi}(w)\mor\varphi(s)$ and the variables $x\in X_{s}$ by
families of variables $(x,s,\varphi(s),j)_{j\in\bb{\varphi(s)}}$.
\end{remark}

Before we state that the above construction can be lifted to a
pseudo-functor from the category $\mathbf{Sig}_{\mathfrak{pd}}$ to the
$2$-category $\mathbf{Cat}$, we point out that the relation of
satisfaction is also invariant under polyderivor change, i.e., that for
every polyderivor $\mathbf{d}\colon \mathbf{\Sigma}\mor \mathbf{\Lambda}$,
if $(P,Q)$ is a $\mathbf{\Sigma}$-equation of type $(X,Y)$ and
$\mathbf{A}$ a $\mathbf{\Lambda}$-algebra, then
$$
\mathbf{d}^{\ast}_{\mathfrak{pd}}(\mathbf{A})\models^{\mathbf{\Sigma}}_{X,Y}
(P,Q)\text{ iff } \mathbf{A}\models^{\mathbf{\Lambda}}_{\coprod_{\varphi}
^{\dagger}X,\coprod_{\varphi}^{\dagger}Y}
(\mathbf{d}_{\diamond}^{\mathfrak{pd}}(P),\mathbf{d}_{\diamond}^{\mathfrak{pd}}(Q)).
$$
This follows from the invariant character under signature change
through the polyderivors of the realization of terms as term operations in
arbitrary, but fixed, algebras, as stated in the following

\begin{proposition}
Let $\mathbf{d}\colon\mathbf{\Sigma}\mor\mathbf{\Lambda}$ be a
polyderivor.  Then, for every $\mathbf{\Lambda}$-algebra
$\mathbf{A}$ and term $P\colon X\mor Y$ for $\mathbf{\Sigma}$ of type
$(X,Y)$, the following diagram commutes
$$
\xymatrix@C=18ex@R=10ex{
(A_{\varphi})_{X}
\ar[r]^{P^{\mathbf{d}^{\ast}_{\mathfrak{pd}}(\mathbf{A})}} &
(A_{\varphi})_{Y} \\
A_{\coprod_{\varphi}^{\dagger}X}
\ar[r]_{\mathbf{d}_{\diamond}^{\mathfrak{pd}}(P)^{\mathbf{A}}}
\ar[u]^{(\theta_{\varphi}^{\dagger\nat})_{X,A}} &
A_{\coprod_{\varphi}^{\dagger}Y}
\ar[u]_{(\theta_{\varphi}^{\dagger\nat})_{Y,A}}
}
$$
where, to simplify the notation, we have written $A_{\varphi}$ instead
of the more accurate $\Delta^{\nat}_{\varphi}A$.
\end{proposition}

\begin{proof}
The proof is analogous to that of Proposition~$16$.
\end{proof}



It is now when we can state that the pseudo-functor $\mathrm{Ter}$
from $\mathbf{Sig}$ to the $2$\nobreakdash-category $\mathbf{Cat}$, defined in the
second section, can be lifted up to a pseudo-functor
$\mathrm{Ter}_{\mathfrak{pd}}$ from $\mathbf{Sig}_{\mathfrak{pd}}$ to
the $2$-category $\mathbf{Cat}$.

\begin{proposition}\label{defPseudofunctorTerfuj}
There exists a pseudo-functor $\mathrm{Ter}_{\mathfrak{pd}}$ from
$\mathbf{Sig}_{\mathfrak{pd}}$ to the $2$-category $\mathbf{Cat}$ given by the
following data
\begin{enumerate}
\item The object mapping of $\mathrm{Ter}_{\mathfrak{pd}}$ is that which
      sends a signature $\mathbf{\Sigma}$ to
      $\mathrm{Ter}_{\mathfrak{pd}}(\mathbf{\Sigma}) = \mathbf{Ter}(\mathbf{\Sigma})$.

\item The morphism mapping of $\mathrm{Ter}_{\mathfrak{pd}}$ is that
      which sends a signature morphism $\mathbf{d}\colon
      \mathbf{\Sigma}\mor \mathbf{\Lambda}$ to
      $\mathbf{d}_{\diamond}^{\mathfrak{pd}}\colon \mathbf{Ter}(\mathbf{\Sigma})\mor
      \mathbf{Ter}(\mathbf{\Lambda})$.

\item For $\mathbf{d}\colon \mathbf{\Sigma}\mor \mathbf{\Lambda}$ and
      $\mathbf{e}\colon \mathbf{\Lambda}\mor \mathbf{\Omega}$, the natural
      isomorphism $\gamma^{\mathbf{d},\mathbf{e}}$ from the composite
      $\mathbf{e}_{\diamond}^{\mathfrak{pd}} \comp
      \mathbf{d}_{\diamond}^{\mathfrak{pd}}$ to
      $(\mathbf{e}\comp\mathbf{d})_{\diamond}^{\mathfrak{pd}}$ is that which is
      defined, for every $S$-sorted set $X$, as the isomorphism
      $\gamma^{\mathbf{d},\mathbf{e}}_{X}\colon
      \coprod_{\psi}^{\dagger}\coprod_{\varphi}^{\dagger}X \mor
      \coprod_{\ext{\psi} \comp \varphi}^{\dagger} X$ in
      $\mathbf{Ter}(\mathbf{\Omega})$ that corresponds to the $U$-sorted
      mapping
      $$
      \xymatrix@C=60pt@R=40pt{ \coprod_{\ext{\psi} \comp
      \varphi}^{\dagger}X \ar[r]^{\rho_{X}} &
      \coprod_{\psi}^{\dagger}\coprod_{\varphi}^{\dagger}X
      \ar[r]^-{\eta_{\coprod_{\psi}^{\dagger}\coprod_{\varphi}^{\dagger}X}}
      &
      \mathrm{T}_{\mathbf{\Omega}}(\coprod_{\psi}^{\dagger}\coprod_{\varphi}^{\dagger}X),
      }
      $$
      where $\rho$ is the isomorphism obtained from the following diagram
$$
  \xymatrix@C=12ex@R=8ex{
  \mathbf{Set}^{S}
     \ar`r[rddd]+/r185pt/`[rddd]^{\coprod_{\ext{\psi}\comp \varphi}}="a" [rddd]
    \ar[rd]^{\coprod_{\varphi}}
    \ar[d]_{\coprod_{\varphi}^{\dagger}}
    \\
  \mathbf{Set}^{T}
    \ar[rd]^{\coprod_{\psi}}
    \ar[d]_{\coprod_{\psi}^{\dagger}}
    &
  \mathbf{Set}^{\fmon{T}}
  \ar`r[dd]+/r135pt/`[dd]^{\coprod_{\ext{\psi}}}="b" [dd]
    \ar@{}[d]|{=}
    \ar[rd]^{\coprod_{\fmon{\psi}}}
    \ar[l]_{(\farg)^{\dagger_{T}}}
      \\
  \mathbf{Set}^{U}
     &
  \mathbf{Set}^{\fmon{U}}
    \ar[l]_{(\farg)^{\dagger_{U}}}
     &
  \mathbf{Set}^{\ffmon{U}}
    \ar[l]_{(\farg)^{\dagger_{\fmon{U}}}}
    \ar[ld]^{\coprod_{\cncat_{U}}}
    \\
  &
  \mathbf{Set}^{\fmon{U}}
    \ar@{}[u]|{\dir{=>}}_{(\zeta^{\fmon{U}})^{-1}}
    \ar[lu]^{(\farg)^{\dagger_{U}}}
  \ar@{}"a";"b"|{\dir{=>}}_{(\gamma^{\varphi,\ext{\psi}})^{-1}}
  \ar@{}"3,3";"b"|{=}
    }
$$
  and $\gamma$ the isomorphism
  associated to the pseudo-functor $\mathrm{MSet}^{\ttcoprod}$.

\item For $\mathbf{\Sigma}$, the natural isomorphism
      $\nu^{\mathbf{\Sigma}}$ from $\Id_{\mathbf{Ter}(\mathbf{\Sigma})}$ to
      $(\between_{S},\eta^{\Ben_{S}}_{\Sigma})^{\mathfrak{pd}}_{\diamond}$
      is that which is defined, for an $S$-sorted set $X$, as the
      isomorphism $\nu^{\mathbf{\Sigma}}_{X}$ from  $X$ to
      $\coprod_{\between_{S}}^{\dagger}X $ that corresponds to the
      $S$-sorted mapping $\eta_{X}\comp \tau^{S}_{X}$ from
      $\coprod_{\between_{S}}^{\dagger}X$ to $\mathrm{T}_{\mathbf{\Omega}}(X)$,
      where $\tau^{S}$ is the natural isomorphism from
      $(\farg)^{\dagger_{S}}\comp \coprod_{\between_{S}}$ to
      $\Id_{\mathbf{Set}^{S}}$
      defined, for an $S$-sorted set $X$, as the $S$-sorted mapping
      whose $s$-th coordinate, for $s\in S$, sends an
      $((a,s),(s),0)\in (\coprod_{\between_{S}}^{\dagger}X)_{s}$ to
      $(\tau^{S}_{X})_{s}((a,s),(s),0) = a$.
\end{enumerate}
\end{proposition}
%


Following this we state a lemma from which we will get a
pseudo-extranatural transformation that formalizes the invariant
character of the realization of terms in algebras relative to the
polyderivors between signatures.

\begin{lemma}\label{lemmaPdExtranaturalFuj}
Let $\mathbf{\Sigma}$ be a signature and
$\mawt{\mathbf{id}}_{\mathfrak{pd}(\mathbf{\Sigma})}
=(\id_{\fmon{S}},\mawt{\id}_{\mathfrak{pd}(\mathbf{\Sigma})})$ the
polyderivor from $\mathfrak{pd}(\mathbf{\Sigma})$ to $\mathbf{\Sigma}$,
where $\id_{\fmon{S}}\colon\fmon{S}\mor \fmon{S}$ is the identity at
$\fmon{S}$ while $\mawt{\id}_{\mathfrak{pd}(\mathbf{\Sigma})}$ is the
canonical isomorphism from $\mathrm{T}_{\Ben_{S}}
(\tcoprod_{1\bprod\between_{S}}\Sigma)_{\cncat_{S}\bprod 1}$ to
$\mathrm{BTer}_{S}
(\tcoprod_{1\bprod\between_{S}}\Sigma)_{\cncat_{S}\bprod 1}$.  Then
the family $(\theta^{\dagger\nat}_{X,A})_{(\mathbf{A},X)\in
\mathbf{Alg}(\mathbf{\Sigma})\bprod
\mathbf{Ter}(\mathfrak{pd}(\mathbf{\Sigma}))}$ is a natural isomorphism
as shown in the following diagram
$$
\xymatrix@C=50pt@R=40pt{
{\mathbf{Alg}(\mathbf{\Sigma})\bprod
\mathbf{Ter}(\mathfrak{pd}(\mathbf{\Sigma}))}
  \ar[r]^-{\alpha_{\mathbf{\Sigma}}
           \bprod \mathrm{Id}}
  \ar[d]_{\mathrm{Id}\bprod
          \beta_{\mathbf{\Sigma}}} &
{\mathbf{Alg}(\mathfrak{pd}(\mathbf{\Sigma}))\bprod
\mathbf{Ter}(\mathfrak{pd}(\mathbf{\Sigma}))}\xyn{d}
  \ar[d]^{\mathrm{Tr}^{\mathfrak{pd}(\mathbf{\Sigma})}} \\
{\mathbf{Alg}(\mathbf{\Sigma})\bprod
\mathbf{Ter}(\mathbf{\Sigma})}\xyn{c}
  \ar[r]_-{\mathrm{Tr}^{\mathbf{\Sigma}}} &
\mathbf{Set}
\ar@{} "c";"d"|{\dir{=>}}^{\theta^{^{\dagger\nat}}}
%
}$$
where, to abbreviate, $\alpha_{\mathbf{\Sigma}} =
(\mawt{\mathbf{id}}_{\mathfrak{pd}(\mathbf{\Sigma})})^{\ast}_{\mathfrak{pd}}$
and $\beta_{\mathbf{\Sigma}} =
(\mawt{\mathbf{id}}_{\mathfrak{pd}(\mathbf{\Sigma})})^{\mathfrak{pd}}_{\diamond}$.
\end{lemma}

\begin{proof}
By Proposition~\ref{lemaPdextranatural}, 
we have that $\mathrm{Tr}$ is
a pseudo-extranatural transformation, hence, for the morphism
$\mawt{\mathbf{id}}_{\mathfrak{pd}(\mathbf{\Sigma})}\colon
\mathfrak{pd}(\mathbf{\Sigma})\mor\mathbf{\Sigma}$, the above diagram
iso-commutes. In particular, for a morphism $(f,P)\colon (\mathbf{A},X)\mor
(\mathbf{B},Y)$ in the product category $\mathbf{Alg}(\mathbf{\Sigma})\bprod
\mathbf{Ter}(\mathfrak{pd}(\mathbf{\Sigma}))$, we have the
configuration
$$
\xy
0;<1ex,0ex>:<0ex,1ex>::
%
\POS 0+(0,0)
\xymatrix@C=20ex@R=8ex@!0{
(\mathbf{A},X) \ar[r]^{(f,P)} & (\mathbf{B},Y)
\save"1,1"."1,2"!C*\frm{}="a"\restore 
}
\POS 0+(20,-10)
\xymatrix@C=21ex@R=8ex@!0{
(\alpha_{\mathbf{\Sigma}}(\mathbf{A}),X)
  \ar[r]^{(f^{\nat},P)} &
(\alpha_{\mathbf{\Sigma}}(\mathbf{B}),Y)
\save"1,1"."1,2"!C*+<5ex>\frm{}="b"\restore 
}
\POS 0+(-20,-10)
\xymatrix@C=21ex@R=8ex@!0{
(\mathbf{A},X^{\dagger})
  \ar[r]^{(f,\beta_{\mathbf{\Sigma}}(P))}&
(\mathbf{B},Y^{\dagger})
\save"1,1"."1,2"!C*+<5ex>\frm{}="c"\restore 
}
\POS 0+(14,-20)
\xymatrix"d"@C=11ex@R=11ex@!0{
&
A^{\nat}_{X}
  \ar[rd]|*+{P^{\alpha_{\mathbf{\Sigma}}(\mathbf{A})}}
  \ar[ld]|*+{f^{\nat}_{X}}
  \\
B^{\nat}_{X}
  \ar[rd]|*{P^{\alpha_{\mathbf{\Sigma}}(\mathbf{B})}}
  &&
A^{\nat}_{Y}
  \ar[ld]|*+{f^{\nat}_{Y}}
  \\
&B^{\nat}_{Y}
\save "d1,2"."d3,2"!C*+<1ex>\frm{}="d"\restore 
}
\POS (-14,-25)
\xymatrix"e"@C=12ex@R=11ex@!0{
&
A_{X^{\dagger}}
  \ar[rd]|*+{\beta_{\mathbf{\Sigma}}(P)^{\mathbf{A}}}
  \ar[ld]|*+{f_{X^{\dagger}}}
  \\
B_{X^{\dagger}}
  \ar[rd]|*+{\beta_{\mathbf{\Sigma}}(P)^{\mathbf{B}}}
  &&
A_{Y^{\dagger}}
  \ar[ld]|*+{f_{Y^{\dagger}}}
  \\
&
B_{Y^{\dagger}}
\save "e1,2"."e3,2"!C*+<5ex>\frm{}="e"\restore 
}
\ar "e1,2";"d1,2"|*+{\theta^{\dagger\nat}_{X,A}}
\ar "e2,1";"d2,1"|*+{\theta^{\dagger\nat}_{X,B}}
\ar "e2,3";"d2,3"|*+{\theta^{\dagger\nat}_{Y,A}}
\ar "e3,2";"d3,2"|*+{\theta^{\dagger\nat}_{Y,B}}
\ar @{|->}"a";"b"
\ar @{|->}"a";"c"
\ar @{|->}"b";"d"
\ar @{|->}"c";"e"
\endxy
$$
and $(\theta^{\dagger\nat}_{X,A})_{(\mathbf{A},X)\in
\mathbf{Alg}(\mathbf{\Sigma})\bprod
\mathbf{Ter}(\mathfrak{pd}(\mathbf{\Sigma}))}$ is a natural
isomorphism.
\end{proof}

\begin{remark}
The functors $\alpha_{\mathbf{\Sigma}} =
(\mawt{\mathbf{id}}_{\mathfrak{pd}(\mathbf{\Sigma})})^{\ast}_{\mathfrak{pd}}$
from $\mathbf{Alg}(\mathbf{\Sigma})$ to
$\mathbf{Alg}(\mathbf{\mathfrak{pd}(\Sigma}))$ are the components of a
natural transformation $\alpha$ from $\Alg$ into $\Alg\comp
\mathfrak{pd}^{\opp}$, both from $\ct{Sig}^{\opp}$ to $\ct{Cat}$,
as in the following diagram
$$
\xymatrix@C=40pt@R=8ex{
\ct{Sig}^{\opp}
\ar[rr]^{\mathfrak{pd}^{\opp}}
\ar[dr]_{\Alg}="x" &&
\ct{Sig}^{\opp}\xyn{y}
\ar[dl]^{\Alg} \\
& \ct{Cat}
\ar @{} "x"+<0pt,-7pt>;"y"+<0pt,-7pt>|{\dir{=>}}^*{\alpha}
}
$$
In its turn, the functors $\beta_{\mathbf{\Sigma}} =
(\mawt{\mathbf{id}}_{\mathfrak{pd}(\mathbf{\Sigma})})^{\mathfrak{pd}}_{\diamond}$
from $\mathbf{Ter}(\mathbf{\mathfrak{pd}(\Sigma}))$ to
$\mathbf{Ter}(\mathbf{\Sigma})$ are the components of a natural
transformation $\beta$ from $\mathrm{Ter}\comp \mathfrak{pd}$ into
$\mathrm{Ter}$, both from $\ct{Sig}$ to $\ct{Cat}$,
as in the following diagram
$$
\xymatrix@C=40pt@R=8ex{
\ct{Sig}
\ar[rr]^{\mathfrak{pd}}
\ar[dr]_{\mathrm{Ter}}="x" &&
\ct{Sig}\xyn{y}
\ar[dl]^{\mathrm{Ter}} \\
& \ct{Cat}
\ar @{} "y"+<0pt,-7pt>;"x"+<0pt,-7pt>|{\dir{=>}}_*{\beta}
}
$$
Besides, if for a polyderivor $\mathbf{d}\colon \mathbf{\Sigma}\mor
\mathbf{\Lambda}$ we denote by $\mawt{\mathbf{d}}\colon
\mathbf{\Sigma}\mor \mathfrak{pd}(\mathbf{\Lambda})$ the signatu\-re
morphism associated to $\mathbf{d}$, by the isomorphism between
$\mathbf{Sig}_{\mathfrak{pd}}$ and
$\mathbf{Kl}(\mathbb{T}_{\mathfrak{pd}})$ stated in
Proposition~\ref{isoSigfujKlfuj}, then we have that
\begin{enumerate}
\item $\mathbf{d}^{\ast}_{\mathfrak{pd}} =
      \mawt{\mathbf{d}}^{\ast}\comp \alpha_{\mathbf{\Lambda}}$, and

\item $\mathbf{d}_{\diamond}^{\mathfrak{pd}} =
      \beta_{\mathbf{\Lambda}}\comp \mawt{\mathbf{d}}_{\diamond}$.
\end{enumerate}
Therefore the morphism mappings of the pseudo-functors
$\mathrm{Alg}_{\mathfrak{pd}}$ and $\mathrm{Ter}_{\mathfrak{pd}}$ are
definable through the natural transformations $\alpha$ and $\beta$,
respectively.
\end{remark}

In the next proposition we construct a pseudo-functor
$\Alg_{\mathfrak{pd}}(\farg)\bprod \mathrm{Ter}_{\mathfrak{pd}}(\farg)$
from the product category $\mathbf{Sig}_{\mathfrak{pd}}^{\opp}\bprod
\mathbf{Sig}_{\mathfrak{pd}}$ to $\mathbf{Cat}$ (obtained from the
pseudo-functor $\Alg_{\mathfrak{pd}}$ and the pseudo-functor
$\mathrm{Ter}_{\mathfrak{pd}}$), and prove that the family
$\mathrm{Tr}=(\mathrm{Tr}^{\mathbf{\Sigma}})_{\mathbf{\Sigma}
\in\mathbf{Sig}_{\mathfrak{pd}}}$, together with the family
$\theta=(\theta^{\mathbf{d}})_{\mathbf{d}\in\Mor(\mathbf{Sig}_{\mathfrak{pd}})}$
is a pseudo-extranatural transformation from the pseudo-functor
$\Alg_{\mathfrak{pd}}(\farg)\bprod\mathrm{Ter}_{\mathfrak{pd}}(\farg)$
to the functor $\mathrm{K}_{\mathbf{Set}}$ (from
$\mathbf{Sig}_{\mathfrak{pd}}^{\opp}\bprod \mathbf{Sig}_{\mathfrak{pd}}$
to $\mathbf{Cat}$) that is constantly $\mathbf{Set}$.

\begin{proposition}\label{PdExtranaturalFuj}
There exists a pseudo-functor $\Alg_{\mathfrak{pd}}(\farg)\bprod
\mathrm{Ter}_{\mathfrak{pd}}(\farg)$ from the category
$\mathbf{Sig}^{\opp}_{\mathfrak{pd}}\bprod \mathbf{Sig}_{\mathfrak{pd}}$
to the $2$-category $\mathbf{Cat}$, obtained from the pseudo-functors
$\Alg_{\mathfrak{pd}}$ and $\mathrm{Ter}_{\mathfrak{pd}}$, which sends a
pair of signatures $(\mathbf{\Sigma},\mathbf{\Lambda})$ to the
category
$\mathbf{Alg}(\mathbf{\Sigma})\bprod\mathbf{Ter}(\mathbf{\Lambda})$,
and a pair of signature morphisms $(\mathbf{d},\mathbf{e})$ from
$(\mathbf{\Sigma},\mathbf{\Lambda})$ to
$(\mathbf{\Sigma}',\mathbf{\Lambda}')$ in
$\mathbf{Sig}_{\mathfrak{pd}}^{\opp}\bprod\mathbf{Sig}_{\mathfrak{pd}}$
to the functor $\mathbf{d}^{\ast}_{\mathfrak{pd}}\bprod
\mathbf{d}_{\diamond}^{\mathfrak{pd}}$ from
$\mathbf{Alg}(\mathbf{\Sigma})\bprod\mathbf{Ter}(\mathbf{\Lambda})$ to
$\mathbf{Alg}(\mathbf{\Sigma}')\bprod\mathbf{Ter}(\mathbf{\Lambda}')$.

Furthermore, the family of functors
$\mathrm{Tr}=(\mathrm{Tr}^{\mathbf{\Sigma}})
_{\mathbf{\Sigma}\in\mathbf{Sig}_{\mathfrak{pd}}}$,
together with the family
$\theta=(\theta^{\mathbf{d}})_{\mathbf{d}\in\Mor(\mathbf{Sig}_{\mathfrak{pd}})}$,
with $\theta^{\mathbf{d}}_{\mathbf{A},X}=\theta^{\dagger\nat}_{X,A}$, is a
pseudo-extranatural transformation from the pseudo-functor
$\Alg_{\mathfrak{pd}}(\farg)\bprod\mathrm{Ter}_{\mathfrak{pd}}(\farg)$
to the functor $\mathrm{K}_{\mathbf{Set}}$ from
$\mathbf{Sig}^{\opp}_{\mathfrak{pd}}\bprod \mathbf{Sig}_{\mathfrak{pd}}$ to
$\mathbf{Cat}$ that is constantly $\mathbf{Set}$.

\end{proposition}

\begin{proof}
We restrict ourselves to prove that, for every polyderivor $\mathbf{d}\colon
\mathbf{\Sigma}\mor \mathbf{\Lambda}$, the following diagram iso-commutes.
$$
\xymatrix@C=20ex@R=10ex{
{\mathbf{Alg}(\mathbf{\Lambda})\bprod\mathbf{Ter}(\mathbf{\Sigma})}
  \ar[r]^{\mathbf{d}^{\ast}_{\mathfrak{pd}}\bprod \mathrm{Id}}
  \ar[d]_{\mathrm{Id} \bprod \mathbf{d}_{\diamond}^{\mathfrak{pd}}} &
{\mathbf{Alg}(\mathbf{\Sigma})\bprod\mathbf{Ter}(\mathbf{\Sigma})}
  \ar[d]^{\mathrm{Tr}^{\mathbf{\Sigma}}} \\
{\mathbf{Alg}(\mathbf{\Lambda})\bprod\mathbf{Ter}(\mathbf{\Lambda})}
  \ar[r]_{\mathrm{Tr}^{\mathbf{\Lambda}}} &
\mathbf{Set}
}
$$
But in the following diagram, where, we recall, $\mawt{\mathbf{d}}\colon
\mathbf{\Sigma}\mor \mathfrak{pd}(\mathbf{\Lambda})$ is the signatu\-re
morphism associated to the polyderivor $\mathbf{d}\colon \mathbf{\Sigma}\mor
\mathbf{\Lambda}$, by Proposition~\ref{isoSigfujKlfuj},
$$
\xymatrix@C=11.5ex@R=12ex{
\xymn{\mathbf{Alg}(\mathbf{\Lambda})\bprod\mathbf{Ter}(\mathbf{\Sigma})} {1}
 &&&&
\xymn{\mathbf{Alg}(\mathbf{\Sigma})\bprod\mathbf{Ter}(\mathbf{\Sigma})}{2}
 \\
&&& \xymn[5ex,0ex]{\mathbf{Alg}(\mathfrak{pd}(\mathbf{\Lambda}))\bprod
\mathbf{Ter}(\mathbf{\Sigma})}{5} \\
&\xymn[-6ex,0ex]{\mathbf{Alg}(\mathbf{\Lambda})\bprod
\mathbf{Ter}(\mathfrak{pd}(\mathbf{\Lambda}))}{6}
&&\xymn[5ex,0ex]{\mathbf{Alg}(\mathfrak{pd}(\mathbf{\Lambda}))\bprod
\mathbf{Ter}(\mathfrak{pd}(\mathbf{\Lambda}))}{7}\\
\xymn{\mathbf{Alg}(\mathbf{\Lambda})\bprod\mathbf{Ter}(\mathbf{\Lambda})}{3}
 &&&&
\xymn[8ex,0ex]{\mathbf{Set}}{4}
\ar "1";"2"^{\mathbf{d}^{\ast}_{\mathfrak{pd}}\bprod \mathrm{Id}}
\ar "1";"5"|*+{\alpha_{\mathbf{\Lambda}}\bprod \mathrm{Id}}="a"
\ar "5";"2"|*+{\mawt{\mathbf{d}}^{\ast}\bprod \mathrm{Id}}
\ar "1"+<-7ex,-2ex>;"3"+<-7ex,+2ex>|*+{\mathrm{Id}\bprod
                                       \mathbf{d}_{\diamond}^{\mathfrak{pd}}}
\ar "1";"6"|*+{\mathrm{Id}\bprod \mawt{\mathbf{d}}_{\diamond}}
\ar "6";"3"^(.4){\mathrm{Id}\bprod \beta_{\mathbf{\Lambda}}}
\ar "5";"7"|*+{\mathrm{Id}\bprod \mawt{\mathbf{d}}_{\diamond}}="d"
\ar "6";"7"^-{\alpha_{\mathbf{\Lambda}}\bprod \mathrm{Id}}="b"
\ar "2"+<+8ex,-2ex>;"4"+<+0ex,+1.5ex>|*+{\mathrm{Tr}^{\mathbf{\Sigma}}}="e"
\ar "3";"4"_{\mathrm{Tr}^{\mathbf{\Lambda}}}="c"
\ar "7";"4"_(.4){\mathrm{Tr}^{\mathfrak{pd}(\mathbf{\Lambda})}}
\ar @{} "b";"c" |{(1)}
\ar @{} "d";"e" |{(2)}
}
$$
we have that the bottom trapezoid (1) iso-commutes by
Lemma~\ref{lemmaPdExtranaturalFuj}, the right-hand trapezoid (2)
iso-commutes because $\mathrm{Tr}$ is a pseudo-extranatural
transformation and the remaining subdiagrams commute by the
definitions of the involved entities.
\end{proof}

\begin{corollary}
The quadruple $\mathfrak{Tm}_{\mathfrak{pd}} =
(\mathbf{Sig}_{\mathfrak{pd}},\mathrm{Alg}_{\mathfrak{pd}},
\mathrm{Ter}_{\mathfrak{pd}},(\mathrm{Tr},\theta))$ is an institution
on $\mathbf{Set}$, the so-called \emph{many-sorted term institution of
Fujiwara}, or, to abbreviate, the \emph{term institution of Fujiwara}.
\end{corollary}

\begin{remark}
Since every standard signature morphism can be identified to a
polyderivor, the term institution is canonically embedded into the term
institution of Fujiwara.
\end{remark}

We close this section by constructing, for the category
$\mathbf{Sig}_{\mathfrak{pd}}$, the many-sorted equational institution
$\mathfrak{LEq}_{\mathfrak{pd}}$.  To do it we define a pseudo-functor
$\mathrm{LEq}_{\mathfrak{pd}}$ on $\mathbf{Sig}_{\mathfrak{pd}}$.

\begin{definition}
We denote by $\mathrm{LEq}_{\mathfrak{pd}}$ the pseudo-functor from
$\mathbf{Sig}_{\mathfrak{pd}}$ to
$\mathbf{Cat}_{\boldsymbol{\mathcal{V}}}$ given by the following data
\begin{enumerate}
\item The object mapping of $\mathrm{LEq}_{\mathfrak{pd}}$ is
      that which sends a signature $\mathbf{\Sigma}$ to the discrete
      category $\mathbf{LEq}(\mathbf{\Sigma})$ canonically associated
      to the set
      $$
      \textstyle \bigcup_{X,Y\in \boldsymbol{\mathcal{U}}}
      (\mathrm{Hom}(Y,\mathrm{T}_{\mathbf{\Sigma}}(X))^{2}\times \{(X,Y)\})
      $$
      of \emph{labelled} $\mathbf{\Sigma}$-\emph{equations}, i.e., the set of all
      pairs $((P,Q),(X,Y))$ with $(P,Q)$ a $\mathbf{\Sigma}$-equation
      of type $(X,Y)$, for some $X,Y\in \boldsymbol{\mathcal{U}}$.

\item The morphism mapping of $\mathrm{LEq}_{\mathfrak{pd}}$ is
      that which sends a polyderivor $\mathbf{d}$ from $\mathbf{\Sigma}$
      to $\mathbf{\Lambda}$ to the functor
      $\mathrm{LEq}_{\mathfrak{pd}}(\mathbf{d})$ from
      $\mathbf{LEq}(\mathbf{\Sigma})$ to
      $\mathbf{LEq}(\mathbf{\Lambda})$ which assigns to the labelled
      equation $((P,Q),(X,Y))$ in $\mathbf{LEq}(\mathbf{\Sigma})$ the
      labelled equation
      $$
      \textstyle\mathrm{LEq}_{\mathfrak{pd}}(\mathbf{d})((P,Q),(X,Y)) =
      ((\mathbf{d}_{\diamond}^{\mathfrak{pd}}(P),
      \mathbf{d}_{\diamond}^{\mathfrak{pd}}(Q)),
      (\coprod_{\varphi}^{\dagger}X,\coprod_{\varphi}^{\dagger}Y))
      $$
      in $\mathbf{LEq}(\mathbf{\Lambda})$.
\end{enumerate}
\end{definition}

\begin{corollary}
The quadruple $\mathfrak{LEq}_{\mathfrak{pd}} =
(\mathbf{Sig}_{\mathfrak{pd}},\mathrm{Alg}_{\mathfrak{pd}},
\mathrm{LEq}_{\mathfrak{pd}}, (\models,\theta))$ is an institution on
$\mathbf{2}$, the so-called \emph{many-sorted equational institution
of Fujiwara}, or, simply, the \emph{equational institution of
Fujiwara}.
\end{corollary}

\section{Transformations of Fujiwara.}

Continuing the work begun in~\cite{tF59}, Fujiwara defines
in~\cite{tF60} an equivalence relation, the \emph{conjugation}, on the
set of families of basic mapping-formulas from a given single-sorted
signature into a like one, relative to a set of equations for the
target signature.  But to properly define such an equivalence relation
it is necessary to define beforehand, for two families of basic
mapping-formulas $(\Phi,P)$, with $\Phi = \{\,\varphi_{\mu}\mid \mu\in
p\,\}$, and $(\Psi,Q)$, with $\Psi = \{\,\psi_{\nu}\mid \nu\in q\,\}$,
from a single-sorted signature $\Sigma$ into a like one $\Lambda$, the
concept of transformation from $(\Phi,P)$ to $(\Psi,Q)$, relative to a
set of equations for the target signature.  One such transformation
will be a mapping $L\colon \Psi\mor \mathrm{T}_{\Lambda}(\Phi)$
subject to satisfy a certain compatibility condition that involves
$P$, $Q$, and the given set of equations.

Next we proceed to give a short, but precise, description of the
transformations between families of basic mapping-formulas and of the
derived equivalence relation of conjugation as stated by Fujiwara
in~\cite{tF60}.  Let $L$ be a mapping from $\Psi$ to
$\mathrm{T}_{\Lambda}(\Phi)$, i.e., a family $(L_{\nu})_{\nu\in q}$ of
terms for $\Lambda$ with variables in $\Phi$.  Then, for every $\nu\in
q$, we get
$$
  L_{\nu}^{\mathbf{T}_{\Lambda}(\Phi\times \vs{v_{n}})}\colon
  \mathrm{T}_{\Lambda}(\Phi\times \vs{v_{n}})^{\Phi}\mor
  \mathrm{T}_{\Lambda}(\Phi\times \vs{v_{n}}),
$$
the term operation on $\mathbf{T}_{\Lambda}(\Phi\times \vs{v_{n}})$
determined by $L_{\nu}$.  This follows from the fact that, for every
$\nu\in q$, $L_{\nu}$ belongs to $\mathrm{T}_{\Lambda}(\Phi)$, and because,
by the universal property of the free $\Lambda$-algebra on $\Phi$, we
have the $\Lambda$-homomorphism $(\mathrm{pr}_{\mu})_{\mu\in
p}^{\sharp}$ from $\mathbf{T}_{\Lambda}(\Phi)$ to
$\mathbf{T}_{\Lambda}(\Phi\times \vs{v_{n}})^{\mathbf{T}_{\Lambda}(\Phi\times
\vs{v_{n}})^{\Phi}}$, as seen in the following diagram
$$
\xymatrix@C=80pt@R=40pt{
{} &
\Psi
  \ar[d]^-{L} \\
\Phi
\ar[r]^{\eta_{\Phi}}
\ar[rd]_{(\mathrm{pr}_{\mu})_{\mu\in p}} &
\mathrm{T}_{\Lambda}(\Phi)
\ar[d]^{(\mathrm{pr}_{\mu})_{\mu\in p}^{\sharp}
        (= \mathrm{Tr}^{\Phi,\mathbf{T}_{\Lambda}(\Phi\times \vs{v_{n}})})} \\
{} &
\mathrm{T}_{\Lambda}(\Phi\times \vs{v_{n}})^{\mathrm{T}_{\Lambda}(\Phi\times
\vs{v_{n}})^{\Phi}}
}
$$
where, for every $\mu\in p$, $\mathrm{pr}_{\mu}$ is the $\mu$-th
projection from $\mathbf{T}_{\Lambda}(\Phi\times
\vs{v_{n}})^{\Phi}$ to $\mathbf{T}_{\Lambda}(\Phi\times \vs{v_{n}})$.

But, for every $n\in \mathbb{N}$ and $\sigma\in \Sigma_{n}$, we have
that $(P^{n}_{\varphi_{\mu},\sigma})_{\mu\in p}\in
\mathrm{T}_{\Lambda}(\Phi\times \vs{v_{n}})^{\Phi}$, hence, for every
$\nu\in q$, the term operation
$L_{\nu}^{\mathbf{T}_{\Lambda}(\Phi\times \vs{v_{n}})}$ is such that
$$
  L_{\nu}^{\mathbf{T}_{\Lambda}(\Phi\times \vs{v_{n}})}\nfunction
  {\mathrm{T}_{\Lambda}(\Phi\times \vs{v_{n}})^{\Phi}}
  {\mathrm{T}_{\Lambda}(\Phi\times \vs{v_{n}})}
  {(P^{n}_{\varphi_{\mu},\sigma})_{\mu\in p}}
  {L_{\nu}^{\mathbf{T}_{\Lambda}(\Phi\times
  \vs{v_{n}})}((P^{n}_{\varphi_{\mu},\sigma})_{\mu\in p})}
$$
i.e., it sends, in particular, the given family of terms
$(P^{n}_{\varphi_{\mu},\sigma})_{\mu\in p}$ in
$\mathrm{T}_{\Lambda}(\Phi\times \vs{v_{n}})$ to the term
$L_{\nu}^{\mathbf{T}_{\Lambda}(\Phi\times
\vs{v_{n}})}((P^{n}_{\varphi_{\mu},\sigma})_{\mu\in p})$ in
$\mathrm{T}_{\Lambda}(\Phi\times \vs{v_{n}})$.

On the other hand, given $n\in \mathbb{N}$ and $i\in n$ we have the
following commutative diagram
$$
\xymatrix{
\Phi \ar[r]^{\eta_{\Phi}}
\ar[d]_{\tp{\mathrm{id}_{\Phi},\kappa_{v_{i}}}} &
\mathrm{T}_{\Lambda}(\Phi)
\ar[d]^{\tp{\mathrm{id}_{\Phi},\kappa_{v_{i}}}^{@}} \\
\Phi\times \vs{v_{n}}
  \ar[r]_{\eta_{\Phi\times \vs{v_{n}}}} &
  \mathrm{T}_{\Lambda}(\Phi\times \vs{v_{n}})
}
$$
where $\tp{\mathrm{id}_{\Phi},\kappa_{v_{i}}}$ is the mapping which
sends $\varphi_{\mu}$ in $\Phi$ to $(\varphi_{\mu},v_{i})$ in
$\Phi\times \vs{v_{n}}$, and
$\tp{\mathrm{id}_{\Phi},\kappa_{v_{i}}}^{@} = (\eta_{\Phi\times
\vs{v_{n}}}\comp \tp{\mathrm{id}_{\Phi},\kappa_{v_{i}}})^{\sharp}$ is
the underlying mapping of the value of the functor
$\mathbf{T}_{\Lambda}$ in $\tp{\mathrm{id}_{\Phi},\kappa_{v_{i}}}$.
Then, from the family of mappings
$(\tp{\mathrm{id}_{\Phi},\kappa_{v_{i}}}^{@})_{i\in n}$, and since we
dispose of the mapping $L\colon \Psi\mor
\mathrm{T}_{\Lambda}(\Phi)$, we get the family of terms
$$
((\tp{\mathrm{id}_{\Phi},\kappa_{v_{i}}}^{@}(L_{\nu}))_{\nu\in
q})_{i\in n}\in (\mathrm{T}_{\Lambda}(\Phi\times
\vs{v_{n}})^{\Psi})^{n},
$$
or, because of the isomorphism $(\mathbf{T}_{\Lambda}(\Phi\times
\vs{v_{n}})^{\Psi})^{n}\cong \mathbf{T}_{\Lambda}(\Phi\times
\vs{v_{n}})^{\Psi\times \vs{v_{n}}}$, the matrix
$$
\begin{pmatrix}
    \tp{\mathrm{id}_{\Phi},\kappa_{v_{0}}}^{@}(L_{0}) & \cdots &
    \tp{\mathrm{id}_{\Phi},\kappa_{v_{n-1}}}^{@}(L_{0}) \\
    \vdots & \ddots & \vdots \\
    \tp{\mathrm{id}_{\Phi},\kappa_{v_{0}}}^{@}(L_{q-1}) & \cdots &
    \tp{\mathrm{id}_{\Phi},\kappa_{v_{n-1}}}^{@}(L_{q-1})
\end{pmatrix}
$$
or, equivalently, by the exchange law (which, in this case, says that,
for $i\in n$, $\nu\in q$,
valuations $\eta_{\Phi\times \vs{v_{n}}}\comp
\tp{\mathrm{id}_{\Phi},\kappa_{v_{i}}}\colon
\Phi\mor \mathrm{T}_{\Lambda}(\Phi\times \vs{v_{n}})$, and terms
$L_{\nu}\in \mathrm{T}_{\Lambda}(\Phi)$, we have that
$\tp{\mathrm{id}_{\Phi},\kappa_{v_{i}}}^{@}(L_{\nu}) =
L_{\nu}^{\mathbf{T}_{\Lambda}(\Phi\times \vs{v_{n}})}
(\varphi_{\mu}(v_{i})\mid \mu\in p)$), the matrix
$$
\begin{pmatrix}
L_{0}^{\mathbf{T}_{\Lambda}(\Phi\times \vs{v_{n}})}
(\varphi_{\mu}(v_{0})\mid \mu\in p) & \cdots &
L_{0}^{\mathbf{T}_{\Lambda}(\Phi\times \vs{v_{n}})}
(\varphi_{\mu}(v_{n-1})\mid \mu\in p) \\
\vdots & \ddots & \vdots \\
L_{q-1}^{\mathbf{T}_{\Lambda}(\Phi\times \vs{v_{n}})}
(\varphi_{\mu}(v_{0})\mid \mu\in p) & \cdots &
L_{q-1}^{\mathbf{T}_{\Lambda}(\Phi\times \vs{v_{n}})}
(\varphi_{\mu}(v_{n-1})\mid \mu\in p)
\end{pmatrix}
$$
with entries in $\mathrm{T}_{\Lambda}(\Phi\times \vs{v_{n}})$.

Besides, by hypothesis, for every $n\in \mathbb{N}$ and $\sigma\in
\Sigma_{n}$, we have the family of terms
$$
(Q^{n}_{\psi_{\nu},\sigma})_{\nu\in q}\colon \Psi \mor
\mathrm{T}_{\Lambda}(\Psi\times \vs{v_{n}}).
$$
Therefore, for every $\nu\in q$, we get
$$
  Q^{n,\mathbf{T}_{\Lambda}(\Phi\times
  \vs{v_{n}})}_{\psi_{\nu},\sigma}\colon
  \mathrm{T}_{\Lambda}(\Phi\times \vs{v_{n}})^{\Psi\times \vs{v_{n}}} \mor
  \mathrm{T}_{\Lambda}(\Phi\times \vs{v_{n}}),
$$
the term operation on $\mathbf{T}_{\Lambda}(\Phi\times \vs{v_{n}})$
determined by $Q^{n}_{\psi_{\nu},\sigma}$.  This follows from the fact
that, for every $\nu\in q$, $Q^{n}_{\psi_{\nu},\sigma}$ is in
$\mathrm{T}_{\Lambda}(\Psi\times \vs{v_{n}})$, and because, by the
universal property of the free $\Lambda$-algebra on $\Psi\times
\vs{v_{n}}$, we have the $\Lambda$-homomorphism
$(\mathrm{pr}_{\nu,i})_{(\nu,i)\in q\times n}^{\sharp}$ from
$\mathbf{T}_{\Lambda}(\Psi\times \vs{v_{n}})$ to
$\mathbf{T}_{\Lambda}(\Phi\times
\vs{v_{n}})^{\mathbf{T}_{\Lambda}(\Phi\times \vs{v_{n}})^{\Psi\times
\vs{v_{n}}}}$, as seen in the following diagram
$$
\xymatrix@C=70pt@R=40pt{
{} &
\Psi
\ar[d]^-{(Q^{n}_{\psi_{\nu},\sigma})_{\nu\in q}} \\
\Psi\times \vs{v_{n}}
\ar[r]^{\eta_{\Psi\times \vs{v_{n}}}}
\ar[rd]_{(\mathrm{pr}_{\nu,i})_{(\nu,i)\in
q\times n}} &
\mathrm{T}_{\Lambda}(\Psi\times \vs{v_{n}})
\ar[d]^{(\mathrm{pr}_{\nu,i})_{(\nu,i)\in
q\times n}^{\sharp}
       (= \mathrm{Tr}^{\Psi\times \vs{v_{n}},\mathbf{T}_{\Lambda}
         (\Phi\times \vs{v_{n}})})} \\
{} &
\mathrm{T}_{\Lambda}(\Phi\times \vs{v_{n}})^{\mathrm{T}_{\Lambda}(\Phi\times
\vs{v_{n}})^{\Psi\times \vs{v_{n}}}}
}
$$
where, for every $(\nu,i)\in q\times n$, $\mathrm{pr}_{\nu,i}$ is
the $(\nu,i)$-th projection from
$\mathbf{T}_{\Lambda}(\Phi\times \vs{v_{n}})^{\Psi\times\vs{v_{n}}}$
to $\mathbf{T}_{\Lambda}(\Phi\times \vs{v_{n}})$.

Let $\tp{Q^{n,\mathbf{T}_{\Lambda}(\Phi\times
\vs{v_{n}})}_{\psi_{\nu},\sigma}}_{\nu\in q}$ be the unique
$\Lambda$-homomorphism from
$\mathbf{T}_{\Lambda}(\Phi\times \vs{v_{n}})^{\Psi\times \vs{v_{n}}}$ to
$\mathbf{T}_{\Lambda}(\Phi\times \vs{v_{n}})^{\Psi}$ such that, for every
$\nu\in q$, the following diagram commutes
$$
\xymatrix{
\mathbf{T}_{\Lambda}(\Phi\times \vs{v_{n}})^{\Psi\times \vs{v_{n}}}
\ar[rd]^*[l]{Q_{\psi_{\nu},\sigma}^{n,\mathbf{T}_{\Lambda}(\Phi\times \vs{v_{n}})}}
\ar[d]_{\tp{Q_{\psi_{\nu},\sigma}^{n,\mathbf{T}_{\Lambda}(\Phi\times
\vs{v_{n}})}}_{\nu\in q}} &
{} \\
\mathbf{T}_{\Lambda}(\Phi\times \vs{v_{n}})^{\Psi}
\ar[r]_{\mathrm{pr}_{\nu}} & \mathbf{T}_{\Lambda}(\Phi\times \vs{v_{n}})
}
$$
Then, since
$$
((\tp{\mathrm{id}_{\Phi},\kappa_{v_{i}}}^{@}(L_{\nu}))_{\nu\in
q})_{i\in n}\in (\mathrm{T}_{\Lambda}(\Phi\times
\vs{v_{n}})^{\Psi})^{n}\cong \mathrm{T}_{\Lambda}(\Phi\times
\vs{v_{n}})^{\Psi\times \vs{v_{n}}},
$$
the family
$
\left( Q^{n,\mathbf{T}_{\Lambda}(\Phi\times
\vs{v_{n}})}_{\psi_{\nu},\sigma}
\left(
\begin{smallmatrix}
    \tp{\mathrm{id}_{\Phi},\kappa_{v_{0}}}^{@}(L_{0}) & \cdots &
    \tp{\mathrm{id}_{\Phi},\kappa_{v_{n-1}}}^{@}(L_{0}) \\
    \vdots & \ddots & \vdots \\
    \tp{\mathrm{id}_{\Phi},\kappa_{v_{0}}}^{@}(L_{q-1}) & \cdots &
    \tp{\mathrm{id}_{\Phi},\kappa_{v_{n-1}}}^{@}(L_{q-1})
\end{smallmatrix}
\right)
\right)_{\nu\in q}
$
belongs to $\mathrm{T}_{\Lambda}(\Phi\times
\vs{v_{n}})^{\Psi}$.

After stating these technical preliminaries, and assuming as given a
set of equations $\mathcal{H}$ for $\Lambda$, finally we are ready to
provide the definition proposed by Fujiwara in~\cite{tF60} of the concept
of transformation from $(\Phi,P)$ to $(\Psi,Q)$.  It will be a mapping
$L\colon \Psi\mor \mathrm{T}_{\Lambda}(\Phi)$, or what is equivalent
an element of $\mathrm{T}_{\Lambda}(\Phi)^{q}$, such that, for every
$\nu\in q$, the equation
$$
Q^{n,\mathbf{T}_{\Lambda}(\Phi\times
\vs{v_{n}})}_{\psi_{\nu},\sigma}
\left(
\begin{smallmatrix}
    \tp{\mathrm{id}_{\Phi},\kappa_{v_{0}}}^{@}(L_{0}) & \cdots &
    \tp{\mathrm{id}_{\Phi},\kappa_{v_{n-1}}}^{@}(L_{0}) \\
    \vdots & \ddots & \vdots \\
    \tp{\mathrm{id}_{\Phi},\kappa_{v_{0}}}^{@}(L_{q-1}) & \cdots &
    \tp{\mathrm{id}_{\Phi},\kappa_{v_{n-1}}}^{@}(L_{q-1})
\end{smallmatrix}
\right) =
L_{\nu}^{\mathbf{T}_{\Lambda}(\Phi\times
\vs{v_{n}})}((P^{n}_{\varphi_{\mu},\sigma})_{\mu\in p})
$$
holds (not necessarily strictly but only) modulus the set of equations
$\mathcal{H}$.

From this, for two single-sorted signatures $\Sigma$, $\Lambda$, and a
set of equations $\mathcal{H}$ for $\Lambda$, Fujiwara defines
in~\cite{tF60} an equivalence relation on the set of families of basic
mapping-formulas from $\Sigma$ to $\Lambda$ (that he calls
$B_{W}$-\emph{conjugacy}, with $W$ identified to $\Lambda$ and $B_{W}$
to $\mathcal{H}$), by saying that two families of basic
mapping-formulas $(\Phi,P)$ and $(\Psi, Q)$ from $\Sigma$ to $\Lambda$
are equivalent, relative to the set of equations $\mathcal{H}$ for
$\Lambda$, if there exists a transformation $L\colon \Psi\mor
\mathrm{T}_{\Lambda}(\Phi)$ from $(\Phi,P)$ to $(\Psi,Q)$ such that,
for every $(\Lambda,\mathcal{H})$-algebra $\mathbf{B}$ the mapping
$\tp{L^{\mathbf{B}}_{\nu}}_{\nu\in q}$, from
$P(\mathbf{B})$ to $Q(\mathbf{B})$ is an isomorphism.  To this we add
that Fujiwara in~\cite{tF60} considers the vertical composition of
transformations between families of basic mapping-formulas, in order
to prove the transitivity, but he does not consider the horizontal
composition of transformations between families of basic
mapping-formulas between single-sorted signatures (no wonder, in his
time there was not any explicit notion of $2$-category).


\begin{example}
For the derivors $d\colon \Sigma\mor \Lambda$ and $e\colon \Lambda\mor
\Sigma$ of Higman and B.H. Neumann defined in the preceding section,
where
\begin{enumerate}
\item $\Sigma_{0} = \{\,1\,\}$, $\Sigma_{2} = \{\,/\,\}$ and
      $\Sigma_{n} = \vacio$, if $n\neq 0, 2$,

\item $\Lambda_{0} = \{\,1\,\}$, $\Lambda_{1} =
      \{\,{}^{-1}\,\}$, $\Lambda_{2} = \{\,\cdot\,\}$ and $\Lambda_{n}
      = \vacio$, if $n\neq 0,1,2$,

\item $d$ defines $/$ in terms of $\cdot$ and ${}^{-1}$, and

\item $e$ defines $\cdot$ and ${}^{-1}$ in terms of $/$,
\end{enumerate}
if we take as set $\mathcal{H}$ of defining axioms, relative to the
signature $\Sigma$, that system given by the following (unusual) group
axioms
\begin{enumerate}
\item $v_{0}/((((v_{0}/v_{0})/v_{1})/v_{2})/(((v_{0}/v_{0})/v_{0})/v_{2})) =
       v_{1}$, and

\item $v_{0}/1 = v_{0}$,
\end{enumerate}
then, taking as $L$ the term in $\mathrm{T}_{\Sigma}(1)$ whose term
realization on any given $\Sigma$-algebra is, essentially, the
corresponding identity mapping, we get a transformation from the
endoderivor $e\comp d$ at $\Sigma$ to the identity endoderivor at
$\Sigma$ because the equation
$$
v_{0}/(1/(1/v_{1})) = v_{0}/ v_{1}
$$
holds, modulus $\mathcal{H}$.

There is also a similar transformation from the endoderivor $d\comp e$ at
$\Lambda$ to the identity endoderivor at $\Lambda$.
\end{example}

\begin{example}
For the derivors $d\colon \Sigma\mor \Lambda$ and $e\colon \Lambda\mor
\Sigma$ of M. H. Stone defined in the preceding section where
\begin{enumerate}
\item $\Sigma_{0} = \{\,0,1\,\}$, $\Sigma_{1} = \{\,'\,\}$,
      $\Sigma_{2} = \{\,\wedge,\vee\,\}$ and $\Sigma_{n} = \vacio$, if
      $n\neq 0,1, 2$,

\item $\Lambda_{0} = \{\,0,1\,\}$, $\Lambda_{1} = \{\,-\,\}$,
      $\Lambda_{2} = \{\,\cdot,+\,\}$ and $\Lambda_{n} = \vacio$, if
      $n\neq 0,1,2$,

\item $d$ defines the \lq\lq Boolean algebra\rq\rq
      op\-era\-tions in terms of the \lq\lq Boolean ring\rq\rq
      op\-era\-tions, and

\item $e$ defines the \lq\lq Boolean ring\rq\rq op\-era\-tions
      in terms of the \lq\lq Boolean algebra\rq\rq op\-era\-tions,
\end{enumerate}
if we take as set $\mathcal{H}$ of defining axioms,
relative to the signature $\Lambda$, that system given by the usual
axioms for a Boolean ring, then, taking as $L$ the term in
$\mathrm{T}_{\Lambda}(1)$ whose term realization on any given
$\Lambda$-algebra is, essentially, the corresponding identity mapping,
we get a transformation from the endoderivor $d\comp e$ at
$\Lambda$ to the identity endoderivor at $\Lambda$ because the
equations
\begin{enumerate}
\item $(v_{0}\cdot(1+v_{1}))+((1+v_{0})\cdot v_{1})+
       ((v_{0}\cdot(1+v_{1}))\cdot ((1+v_{0})\cdot v_{1})) =
       v_{0}+v_{1}$,

\item $v_{0}\cdot v_{1} = v_{0}\cdot v_{1}$, and

\item $-v_{0} = 1+v_{0}$,
\end{enumerate}
hold, modulus $\mathcal{H}$.
\end{example}

\begin{example}
In~\cite{tF60}, pp.  260--268, Fujiwara provides examples of
transformations for families of basic mapping-formulas of
derivation-type when the operations and equations are those
corresponding to commutative linear algebras over a given field of
characteristic zero.
\end{example}


Having just sketched the theory of Fujiwara about the transformations
between families of basic mapping-formulas, what we ultimately try to
do is to define (once we have at our disposal a convenient notion of
morphism from a specification into a like one, but for polyderivors), for
two morphisms $\mathbf{d}$, $\mathbf{e}$ from a specification
$(\mathbf{\Sigma},\mathcal{E})$ into a like one
$(\mathbf{\Lambda},\mathcal{H})$ a concept of transformation from
$\mathbf{d}$ to $\mathbf{e}$, as well as a vertical and a horizontal
composition for these transformations, and all that in such a way that
specifications, morphisms and transformations constitute a
$2$-category.  But to succeed in doing it we should begin, as we do in
this section, by defining a structure of $2$-category on the category
$\mathbf{Sig}_{\mathfrak{pd}}$, of signatures and polyderivors, through the
concept of transformation between polyderivors.

The transformations between polyderivors that we define below are a
generalization (up to the many-sorted case) of the above concept of
transformation between families of basic mapping-formulas, due to
Fujiwara.  But, because the polyderivors are, simply, morphisms from a
signature into a like one, and not morphisms between specifications
(where there are involved equations), they will satisfy, for every
formal operation, a strict equation, instead of an equation modulus a
set of equations for the target signature.  However, after we define
in the last section the adequate morphisms between specifications
(through the polyderivors between the underlying signatures of the
specifications), we will get, also in that section, the full
generalization of the theory of Fujiwara in~\cite{tF60}, as announced
above.

In this section we also prove that the transformations between
polyderivors determine natural transformations between the functors
associated to the polyderivors, that allow us to lift the pseudo-functors
$\Alg_{\mathfrak{pd}}$ and $\mathrm{Ter}_{\mathfrak{pd}}$ up to
$2$-functors, and hence to get, by applying a construction of
Ehresmann-Grothendieck to $\Alg_{\mathfrak{pd}}$, a $2$-category
$\mathbf{Alg}_{\mathfrak{pd}}$.  Besides, we prove that the
transformations between polyderivors are also compatible with the
realization of the terms in the algebras and we characterize this
through the concept of pseudo-extranatural transformation between
pseudo-functors on $2$\nobreakdash-categories.  From this we get that the
relation between terms and algebras is an example of $2$-institution.

In order to define and investigate the transformations between
polyderivors it will be shown to be convenient to make use of some
derived operations in the Bénabou algebras of terms for the
different signatures, concretely of those in the following

\begin{definition}
Let $S$ be a set of sorts.
\begin{enumerate}
\item For every $\ol{w}\in\ffmon{S}$ and $\alpha\in\bb{\ol{w}}$, let
      $\pi^{\ol{w}}_{\alpha}$ be the derived operation of type $\lambda\mor
      (\cncat\ol{w},\ol{w}_{\alpha})$ defined as
      $$
      \tp{\pi^{\cncat \ol{w}}_{\sum_{\beta\in \alpha}p_{\beta}},
      \ldots, \pi^{\cncat \ol{w}}_{\sum_{\beta\in \alpha+1}p_{\beta}-1}
       }_{\cncat\ol{w},\ol{w}_{\alpha}},
      $$
      where $\ol{w}$ is of the form
      $$
      ((\cdot,\ldots,\cdot),\ldots,
      \overbrace{(\underset{\sum_{\beta\in \alpha}p_{\beta}}{\cdot},\ldots,
      \underset{\sum_{\beta\in \alpha+1}p_{\beta}-1}\cdot)}^{\ol{w}_{\alpha}},
      \ldots,(\cdot,\ldots,\cdot)),
      $$
      and, for every $\alpha\in \bb{\ol{w}}$, $p_{\alpha} =
      \bb{\ol{w}_{\alpha}}$.
\item For every $u\in\fmon{S}$ and
      $\ol{w}\in\ffmon{S}$, let $\tp{\,}_{u,\ol{w}}$ be
      the derived operation of type $((u,\ol{w}_{0}),
      \ldots, (u,\ol{w}_{\bb{\ol{w}}-1})) \mor (u,\cncat\ol{w})$
      defined as
      \begin{align*}
      \tp{P_{0},\ldots,P_{\bb{\ol{w}}-1}}_{u,\ol{w}}=
      \tp{
      &\pi^{\ol{w}_{0}}_{0} \comp
       P_{0},
      \ldots,
      \pi^{\ol{w}_{0}}_{\bb{\ol{w}_{0}}-1} \comp
       P_{0},
      \ldots,\\
      &\pi^{\ol{w}_{\bb{\ol{w}}-1}}_{0} \comp
      P_{\bb{\ol{w}}-1},
      \ldots,
      \pi^{\ol{w}_{\bb{\ol{w}}-1}}_{\bb{\ol{w}_{\bb{\ol{w}}-1}}-1} \comp
       P_{\bb{\ol{w}}-1}
      }_{u,\cncat\ol{w}}.
      \end{align*}

\item For every $n\in \mathbb{N}$, and
      $\ol{u},\,\ol{w}\in{\fmon{S}}^{n}$, let
      $\concat_{\ol{u},\ol{w}}$ be the derived operation of type
      $((\ol{u}_{0},\ol{w}_{0}), \ldots,
      (\ol{u}_{n-1},\ol{w}_{n-1}))\mor
      (\concat{\ol{u}},\concat{\ol{w}})$ defined as
      $$
      \concat_{\ol{u},\ol{w}}(P_{0},\ldots,P_{n-1})=
      \tp{P_{0}\comp \pi^{\ol{u}}_{0},\ldots,
      P_{n-1}\comp\pi^{\ol{u}}_{n-1}}_{\cncat\ol{u},\ol{w}}.
      $$
\end{enumerate}
\end{definition}

From now on, to simplify the notation, we will omit some subscripts in
the expressions.  Moreover, for the operations of the form
$\concat_{\ol{u},\ol{w}}$ we adopt the infix notation, and we will
write $P_{0}\cncat\cdots\cncat P_{n-1}$ instead of
$\concat_{\ol{u},\ol{w}}(P_{0},\ldots,P_{n-1})$, the type, in its
turn, will be $\ol{u}_{0}\cncat\cdots \cncat\ol{u}_{n-1} \mor
\ol{w}_{0}\cncat\cdots \cncat\ol{w}_{n-1}$.

For the algebras of terms $\mathbf{BTer}_{S}(\Sigma)$, the operations
$\concat_{\ol{u},\ol{w}}$ are, essentially, the result of gathering
into a family the corresponding terms, relabelling adequately the
variables.

Recalling that the Bénabou algebras are, up to isomorphism, the
finitary many-sorted algebraic theories of Bénabou (see
Proposition~\ref{isoBalgBth}), from now on, we will represent the
composition of terms diagram\-mat\-ica\-lly, and the equality of two
coterminal paths composed of terms by asserting the commutativity of
the appropriate diagram.

\begin{definition}
Let $\mathbf{d}$ and $\mathbf{e}$ be polyderivors from $\mathbf{\Sigma}$
to $\mathbf{\Lambda}$.  A \emph{transforma\-tion from} $\mathbf{d}$
\emph{to} $\mathbf{e}$ is a choice function $\xi$ for
$(\mathrm{BTer}_{T}(\Lambda)_{\varphi(s),\psi(s)})_{s\in S} =
(\mathrm{T}_{\mathbf{\Lambda}}(\vs{\varphi(s)})_{\psi(s)})_{s\in S}$,
i.e., an element of $\prod_{s\in
S}\mathrm{T}_{\mathbf{\Lambda}}(\vs{\varphi(s)})_{\psi(s)}$, such
that, for every operation $\sigma\colon w\mor s$, the following
diagram commutes
$$
\xymatrix@C=10pt{
1
\ar[r]^-{\left<\xi_{s},d(\sigma)\right>}
\ar[d]_{\left<e(\sigma),\xi_{w}\right>} &
\mathrm{T}_{\mathbf{\Lambda}}(\vs{\varphi(s)})_{\psi(s)}\times
\mathrm{T}_{\mathbf{\Lambda}}(\vs{\varphi^{\sharp}(w)})_{\varphi(s)}
\ar[d]^{\comp}    \\
\mathrm{T}_{\mathbf{\Lambda}}(\vs{\psi^{\sharp}(w)})_{\psi(s)}\times
\mathrm{T}_{\mathbf{\Lambda}}(\vs{\varphi^{\sharp}(w)})_{\psi^{\sharp}(w)}
\ar[r]_-{\comp} &
\mathrm{T}_{\mathbf{\Lambda}}(\vs{\varphi^{\sharp}(w)})_{\psi(s)}
}
$$
or more briefly, such that
$$
\xi_{s}\comp d(\sigma) = e(\sigma)\comp \xi_{w},
$$
where $\xi_{w}$ is $\xi_{w_{0}}\concat\cdots\concat
\xi_{w_{\bb{w}-1}}$.  We agree upon writing $\xi\colon
\mathbf{d}\mdf\mathbf{e}$ to denote the fact that $\xi$ is a
transforma\-tion from $\mathbf{d}$ to $\mathbf{e}$.
\end{definition}

Therefore for a transformation $\xi = (\xi_{s})_{s\in S}$ from
$\mathbf{d}$ to $\mathbf{e}$ we have, in particular, that, for every
$s\in S$, $\xi_{s}\in
\mathrm{T}_{\mathbf{\Lambda}}(\vs{\varphi(s)})_{\psi(s)}$, i.e., that
$$
\xi_{s} = ((\xi_{s})_{0},\ldots,(\xi_{s})_{\bb{\psi(s)}-1})
$$
is a tuple of length $\bb{\psi(s)}$ such that, for every
$i\in\bb{\psi(s)}$, $(\xi_{s})_{i}$ is a term for $\mathbf{\Lambda}$
of type $(\vs{\varphi(s)},\psi(s)_{i})$.

Henceforth, we agree to represent the commutativity condition for a
transformation $\xi\colon \mathbf{d}\mdf\mathbf{e}$ between polyderivors
also by the following diagram
$$
\xymatrix@C=70pt{
\ext{\varphi}(w) \ar[r]^{d(\sigma)}
              \ar[d]_{\xi_{w}} &
\varphi(s) \ar[d]^{\xi_{s}}    \\
\ext{\psi}(w) \ar[r]_{e(\sigma)} &
\psi(s)
}
$$
where, we recall, $\xi_{w}$ arises as
$$
\xymatrix@C=70pt{
\varphi(w_{0})\cncat\cdots\cncat\varphi(w_{\bb{w}-1})
\ar[r]^-{\pi^{\fmon{\varphi}(w)}_{i}}
\ar[d]_{\xi_{w}} &
\varphi(w_{i})
\ar[d]^{\xi_{w_{i}}}    \\
\psi(w_{0})\cncat\cdots\cncat\psi(w_{\bb{w}-1})
\ar[r]_-{\pi^{\fmon{\psi}(w)}_{i}} &
\psi(w_{i})
}
$$

\begin{example}
Let $\mathbf{\Sigma}$ be a signature, $p, q\in \mathbb{N}$, and
$\mathbf{d}= (\varphi,d)$, $\mathbf{e}= (\psi,e)$ two polyderivors from
$\mathbf{\Sigma}$ into itself, such that
\begin{enumerate}
\item $\varphi\colon S\mor \fmon{S}$ is the mapping which sends
      $s\in S$ to the word $\cncat_{\mu\in p}(s)$ and,

\item For $(w,s)\in \fmon{S}\times S$, $d_{w,s}$ is the mapping
      from $\Sigma_{w,s}$ to
      $\mathrm{T}_{\mathbf{\Sigma}}(\vs{\ext{\varphi}(w)})_{s}^{p}$
      which sends $\sigma\in\Sigma_{w,s}$ to
      $$
      (\sigma(v_{0}^{w_{0}}, v_{p}^{w_{1}},\ldots,
              v_{(\bb{w}-1)p}^{w_{\bb{w}-1}}),\ldots,
       \sigma(v_{p-1}^{w_{0}}, v_{2p-1}^{w_{1}},\ldots,
              v_{\bb{w}p-1}^{w_{\bb{w}-1}})),
      $$
\end{enumerate}
and
\begin{enumerate}
\item $\psi\colon S\mor \fmon{S}$ is the mapping which sends
      $s\in S$ to the word $\cncat_{\nu\in q}(s)$ and,

\item For $(w,s)\in \fmon{S}\times S$, $e_{w,s}$ is the mapping
      from $\Sigma_{w,s}$ to
      $\mathrm{T}_{\mathbf{\Sigma}}(\vs{\ext{\psi}(w)})_{s}^{q}$
      which sends $\sigma\in\Sigma_{w,s}$ to
      $$
      (\sigma(v_{0}^{w_{0}}, v_{q}^{w_{1}},\ldots,
              v_{(\bb{w}-1)q}^{w_{\bb{w}-1}}),\ldots,
       \sigma(v_{q-1}^{w_{0}}, v_{2q-1}^{w_{1}},\ldots,
              v_{\bb{w}q-1}^{w_{\bb{w}-1}})).
      $$
\end{enumerate}
Then, for an arbitrary, but fixed, mapping $f = (f(\nu))_{\nu\in q}$
from the natural number $q$ to the natural number $p$, taking as $\xi$
the element of $\prod_{s\in
S}\mathrm{T}_{\mathbf{\Lambda}}(\vs{\varphi(s)})_{s}^{q}$ defined, for
every $s\in S$, as
$$
  \xi_{s} = (v^{s}_{f(0)},\ldots,v^{s}_{f(q-1)}),
$$
where, to simplify the notation, we have identified the variables in
$\vs{\varphi(s)}$ with their images in
$\mathrm{T}_{\mathbf{\Sigma}}(\vs{\varphi(s)})$ under
$\eta_{\vs{\varphi(s)}}$, we have that $\xi$ is a transformation from
$\mathbf{d}$, $\mathbf{e}$.  We point out that the working out of all
the details of this example, even if a little troublesome, helps to
grasp the functioning of the polyderivors and the transformations between
them.

For more examples of transformations between polyderivors we refer to the
last section of this paper.
\end{example}

\begin{remark}
In the just stated example, due to the intended meaning of $\xi$ as a
$\mathrm{ms}$\nobreakdash-mapping from the direct $p$-power to the direct
$q$-power of some $\mathbf{\Sigma}$-algebra, the mappings of the type
$f\colon q\mor p$ act by selecting the coordinates for going from the
first direct power to the second direct power.
\end{remark}

\begin{example}
Let $(\Phi,P)$, with $\Phi = \{\,\varphi_{\mu}\mid \mu\in p\,\}$, and
$(\Psi,Q)$, with $\Psi = \{\,\psi_{\nu}\mid \nu\in q\,\}$, be two
families of basic mapping-formulas from the single-sorted signature
$\Sigma$ to the single-sorted signature $\Lambda$, and $L\in
\mathrm{T}_{\mathbf{\Lambda}}(\Phi)^{q}$.  Then $L$ is a
transformation from the polyderivor associated to $(\Phi,P)$ to the
polyderivor associated to $(\Psi,Q)$ iff, for every $n\in \mathbb{N}$ and
every $\sigma\in \Sigma_{n}$, the following diagram commutes
$$
\xymatrix@C=10pt{
1
\ar[r]^-{\left<L,d_{n,0}(\sigma)\right>}
\ar[d]_{\left<e_{n,0}(\sigma),L_{n}\right>} &
\mathrm{T}_{\mathbf{\Lambda}}(\Phi)^{q}\times
\mathrm{T}_{\mathbf{\Lambda}}(\Phi\times\vs{v_{n}})^{p}
\ar[d]^{\comp}    \\
\mathrm{T}_{\mathbf{\Lambda}}(\Psi\times\vs{v_{n}})^{q}\times
\mathrm{T}_{\mathbf{\Lambda}}(\Phi\times\vs{v_{n}})^{\Psi\times\vs{v_{n}}}
\ar[r]_-{\comp} &
\mathrm{T}_{\mathbf{\Lambda}}(\Phi\times \vs{v_{n}})^{q}
}
$$
where
\begin{enumerate}
\item $d_{n,0}(\sigma) = (P^{n}_{\varphi_{0},\sigma},\ldots,
                          P^{n}_{\varphi_{p-1},\sigma})$,

\item $e_{n,0}(\sigma) = (Q^{n}_{\psi_{0},\sigma},\ldots,
                          Q^{n}_{\psi_{q-1},\sigma})$, and

\item $L_{n} = \left(
               \begin{smallmatrix}
               \tp{\mathrm{id}_{\Phi},\kappa_{v_{0}}}^{@}(L_{0}) & \cdots &
               \tp{\mathrm{id}_{\Phi},\kappa_{v_{n-1}}}^{@}(L_{0}) \\
               \vdots & \ddots & \vdots \\
               \tp{\mathrm{id}_{\Phi},\kappa_{v_{0}}}^{@}(L_{q-1}) & \cdots &
               \tp{\mathrm{id}_{\Phi},\kappa_{v_{n-1}}}^{@}(L_{q-1})
               \end{smallmatrix}
               \right).
               $
\end{enumerate}
Observe that in this case the right-down path in the diagram is the
family
$$
(L_{\nu}^{\mathbf{T}_{\Lambda}(\Phi\times
\vs{v_{n}})}((P^{n}_{\varphi_{\mu},\sigma})_{\mu\in p}))_{\nu\in q},
$$
while the down-right path is the family
$$
\left(Q^{n,\mathbf{T}_{\Lambda}(\Phi\times
\vs{v_{n}})}_{\psi_{\nu},\sigma} \left(
\begin{smallmatrix}
    \tp{\mathrm{id}_{\Phi},\kappa_{v_{0}}}^{@}(L_{0}) & \cdots &
    \tp{\mathrm{id}_{\Phi},\kappa_{v_{n-1}}}^{@}(L_{0}) \\
    \vdots & \ddots & \vdots \\
    \tp{\mathrm{id}_{\Phi},\kappa_{v_{0}}}^{@}(L_{q-1}) & \cdots &
    \tp{\mathrm{id}_{\Phi},\kappa_{v_{n-1}}}^{@}(L_{q-1})
\end{smallmatrix}
\right)\right)_{\nu\in q}.
$$
Therefore, without having a set of equations $\mathcal{H}$ for the
target single-sorted signature $\Lambda$, any transformation of
Fujiwara from a family of basic mapping-formulas into a like one is an
example of transformation between the polyderivors associated to the
families of basic mapping-formulas.
\end{example}

The commutativity condition in the above definition of transformation
from a polyderivor into a like one can be extended up to the terms, as
proved in the following

\begin{proposition}\label{deformacionesSimDer}
Let $\mathbf{d}$ and $\mathbf{e}$ be polyderivors from $\mathbf{\Sigma}$ to
$\mathbf{\Lambda}$ and $\xi\colon\mathbf{d}\mdf\mathbf{e}$ a
transformation.  Then, for every  term
$P\colon u\mor w$ in $\mathrm{BTer}_{S}(\Sigma)$, $\xi_{w}\comp
\ext{d}(P) = \ext{e}(P)\comp \xi_{u}$, i.e.,
the following diagram commutes
$$
\xymatrix@C=70pt{
\ext{\varphi}(u) \ar[r]^{\ext{d}(P)}
              \ar[d]_{\xi_{u}} &
\ext{\varphi}(w) \ar[d]^{\xi_{w}}    \\
\ext{\psi}(u) \ar[r]_{\ext{e}(P)} &
\ext{\psi}(w)
}
$$
\end{proposition}

\begin{proof}
By algebraic induction in the Bénabou algebra
$\mathbf{\mathrm{BTer}}_{S}(\Sigma)$.  The basis of the induction
holds because it means that  $\xi$ is a transformation.

For the operations $\pi^{w}_{i}$, we have that $\ext{d}(\pi^{w}_{i}) =
\pi^{\fmon{\varphi}(w)}_{i}$, $\ext{e}(\pi^{w}_{i}) =
\pi^{\fmon{\psi}(w)}_{i}$, and
$$
\xi_{w_{i}}\comp \pi^{\fmon{\varphi}(w)}_{i}
= \pi^{\fmon{\psi}(w)}_{i}\comp \xi_{w},
$$
i.e., the following diagram commutes
$$
\xymatrix@C=70pt{
\ext{\varphi}(w) \ar[r]^{\pi^{\fmon{\varphi}(w)}_{i}}
              \ar[d]_{\xi_{w}} &
\varphi(w_{i}) \ar[d]^{\xi_{w_{i}}}    \\
\ext{\psi}(u) \ar[r]_{\pi^{\fmon{\psi}(w)}_{i}} &
\psi(w_{i})
}
$$

For the operations $\tp{\,}_{u,w}$, we have that
$$
\xi_{w}\comp
\ext{d}(\tp{P_{0},\ldots,P_{\bb{w}-1}}_{u,w}) =
\ext{e}(\tp{P_{0},\ldots,P_{\bb{w}-1}}_{u,w})\comp \xi_{u},
$$
i.e., that the following diagram commutes
$$
\xymatrixcolsep={23ex}
\xymatrixrowsep={8ex}
\xymatrix{
\ext{\varphi}(u) \ar[r]^{\ext{d}(\tp{P_{0},\ldots,P_{\bb{w}-1}}_{u,w})}
              \ar[d]_{\xi_{u}} &
\ext{\varphi}(w) \ar[d]^{\xi_{w}}    \\
\ext{\psi}(u) \ar[r]_{\ext{e}(\tp{P_{0},\ldots,P_{\bb{w}-1}}_{u,w})} &
\ext{\psi}(w)
}
$$
because
\begin{align*}
\xi_{w}\comp \ext{d}(\tp{P_{0},\ldots,P_{\bb{w}-1}}_{u,w})
&=
\tp{\ext{d}(P_{0}),\ldots,\ext{d}(P_{\bb{w}-1})}_{\ext{\varphi}(u),\fmon{\varphi}(w)}
\\
&=
\tp{(\xi_{w_{i}}\comp \ext{d}(P_{i})\mid i\in \bb{w})
   }_{\ext{\varphi}(u),\fmon{\varphi}(w)}
\\
&=
\tp{\ext{e}(P_{0})\comp \xi_{u}
    ,\ldots,
    \ext{e}(P_{\bb{w}-1})\comp \xi_{u}
   }_{\ext{\varphi}(u),\fmon{\varphi}(w)}
\\
&=
\tp{\ext{e}(P_{0})
    ,\ldots,
    \ext{e}(P_{\bb{w}-1})
   }_{\ext{\varphi}(u),\fmon{\varphi}(w)}
   \comp \xi_{u}
\\
&=
\ext{e}(\tp{P_{0},\ldots,P_{\bb{w}-1}}_{u,w}) \comp \xi_{u}.
\end{align*}

Finally, for the  operations $\comp_{u,x,w}$, it is obvious
that
$$
\xi_{x}\comp \ext{d}(Q)\comp \ext{d}(P) = \ext{e}(Q)\comp
\ext{e}(P)\comp \xi_{u},
$$
i.e., that the following diagram commutes
$$
\xymatrix@C=70pt{
\ext{\varphi}(u) \ar[r]^{\ext{d}(P)}
              \ar[d]_{\xi_{u}} &
\ext{\varphi}(w) \ar[r]^{\ext{d}(Q)}
              \ar[d]^{\xi_{w}} &
\ext{\varphi}(x) \ar[d]^{\xi_{x}}    \\
\ext{\psi}(u) \ar[r]_{\ext{e}(P)} &
\ext{\psi}(w) \ar[r]_{\ext{e}(Q)} &
\ext{\psi}(x)
}
$$
\end{proof}

What we want now is to endow the category
$\mathbf{Sig}_{\mathfrak{pd}}$ of signatures and polyderivors with a
structure of $2$-category.  For this we provide in the following
proposition the definitions of the horizontal and vertical composition
of the transformations between polyderivors, prove the law of Godement, and
define the identity transformations at the polyderivors.

\begin{proposition}
The signatures together with the polyderivors and the
trans\-for\-mations between the polyderivors have a structure of
$2$-category, denoted as $\mathbf{Sig}_{\mathfrak{pd}}$.
\end{proposition}

\begin{proof}
\emph{Definition of the vertical composition.}
Given the configuration
$$
\xymatrix@C=20ex{
\mathbf{\Sigma}
            \ar@/^6ex/[r]|*+{\mathbf{d}}="m1"
            \ar       [r]|*+{\mathbf{e}}="m2"
            \ar@/_6ex/[r]|*+{\mathbf{h}}="m3" &
\mathbf{\Lambda}
\ar @{} "m1";"m2"|{\dir{~>}}^{\;\xi}
\ar @{} "m2";"m3"|{\dir{~>}}^{\;\chi}
}
$$
the vertical composition of $\xi$ and $\chi$, denoted by
$\chi\comp\xi$ and defined as
$$
\chi\comp\xi = (\chi_{s}\comp\xi_{s})_{s\in S},
$$
is a transformation from $\mathbf{d}$ to $\mathbf{h}$, because, for
every $\sigma\colon w\mor s$, the following diagram commutes
$$
\xymatrixcolsep = {16ex}        %
\xymatrixrowsep = {6ex}
\xymatrix{
\ext{\varphi}(w)\ar[r]^{d(\sigma)}\ar[d]_{\xi_{w}} & \varphi(s)\ar[d]^{\xi_{s}} \\
\ext{\psi}(w) \ar[r]|{e(\sigma)}\ar[d]_{\chi_{w}} &\psi(s)\ar[d]^{\chi_{s}}\\
\ext{\gamma}(w)  \ar[r]_{h(\sigma)} &\gamma(s)
}
$$

\emph{Definition of the horizontal composition.}
Given the configuration
$$
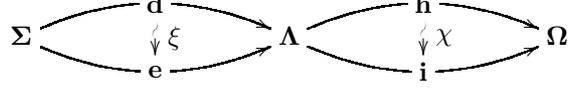
\xymatrix@C=20ex{
\mathbf{\Sigma}
  \ar@/^3ex/[r]|*+{\mathbf{d}}="f1"
  \ar@/_3ex/[r]|*+{\mathbf{e}}="f2" &
\mathbf{\Lambda}
  \ar@/^3ex/[r]|*+{\mathbf{h}}="f3"
  \ar@/_3ex/[r]|*+{\mathbf{i}}="f4" &
\mathbf{\Omega}
\ar @{} "f1";"f2"|{\dir{~>}}^{\,\xi}
\ar @{} "f3";"f4"|{\dir{~>}}^{\,\chi}
}
$$
the horizontal composition of $\xi$ and $\chi$, denoted by
$\chi\hcomp\xi$ and defined as
$$
\chi\hcomp\xi = (\chi_{\psi(s)}\comp \ext{h}(\xi_{s}))_{s\in S},
$$
or, equivalently, as $(\ext{i}(\xi_{s})\comp \chi_{\varphi(s)})_{s\in
S}$, is a transformation from $\mathbf{h}\comp\mathbf{d}$ to
$\mathbf{i}\comp\mathbf{e}$.  We have to prove that $\chi\hcomp\xi$ is
a transformation from $(\ext{\gamma}\comp\varphi,
\ext{h}_{\ext{\varphi}\bprod\ext{\varphi}}\comp d)$ to
$(\ext{\nu}\comp\psi, \ext{i}_{\ext{\psi}\bprod\ext{\psi}}\comp e)$,
i.e., that, for every $\sigma\colon w\mor s$, we have that
$$
(\chi\hcomp \xi)_{s}\comp \ext{h}(d(\sigma)) = \ext{i}(e(\sigma))
\comp (\chi\hcomp\xi)_{w}.
$$
But this happens since $\xi$, $\chi$ are transformations and
$\ext{h}$, $\ext{i}$ morphisms, i.e., because the following diagram
commutes $$\xymatrix@C=9ex@R=5ex{ & \ext{\gamma}(\ext{\varphi}(w))
\ar[rr]|*+{\ext{h}(d(\sigma))} \ar[dd]|(.3)*+{\chi_{\ext{\varphi}(w)}}
\ar[dl]|*+{\ext{h}(\xi_{w})} & & \ext{\gamma}(\varphi(s))
\ar[dd]|*+{\chi_{\varphi(s)}} \ar[dl]|*+{\ext{h}(\xi_{s})} \\
\ext{\gamma}(\ext{\psi}(w))
  \ar[rr]|(.65)*+{\ext{h}(e(\sigma))}
  \ar[dd]|*+{\chi_{\ext{\psi}(w)}} & &
\ext{\gamma}(\psi(s))
  \ar[dd]|(.3)*+{\chi_{\psi(s)}}  \\
&
\ext{\nu}(\ext{\varphi}(w))
  \ar[rr]|(.65)*+{\ext{i}(d(\sigma))}
  \ar[dl]|*+{\ext{i}(\xi_{w})} &
 &
\ext{\nu}(\varphi(s))
  \ar[dl]|*+{\ext{i}(\xi_{s})} \\
\ext{\nu}(\ext{\psi}(w))
  \ar[rr]|*+{\ext{i}(e(\sigma))} &
 &
\ext{\nu}(\psi(s))
}
$$

\emph{Law of Godement.}
Given the configuration
$$
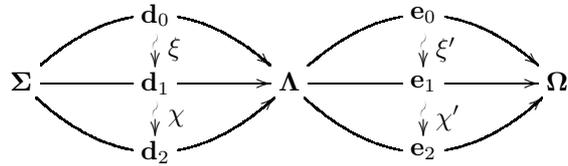
\xymatrix@C=20ex{
\mathbf{\Sigma}
  \ar@/^6ex/[r]|*+{\mathbf{d}_{0}}="f0"
  \ar       [r]|*+{\mathbf{d}_{1}}="f1"
  \ar@/_6ex/[r]|*+{\mathbf{d}_{2}}="f2" &
\mathbf{\Lambda}
  \ar@/^6ex/[r]|*+{\mathbf{e}_{0}}="g0"
  \ar       [r]|*+{\mathbf{e}_{1}}="g1"
  \ar@/_6ex/[r]|*+{\mathbf{e}_{2}}="g2" &
\mathbf{\Omega}
\ar @{} "f0";"f1"|{\dir{~>}}^{\,\xi}
\ar @{} "f1";"f2"|{\dir{~>}}^{\,\chi}
\ar @{} "g0";"g1"|{\dir{~>}}^{\,\xi'}
\ar @{} "g1";"g2"|{\dir{~>}}^{\,\chi'}
}
$$
we have that
$$
(\chi'\hcomp\chi)\comp(\xi'\hcomp\xi) = (\chi'\comp\xi')\hcomp (\chi\comp\xi).
$$
This is so since the following diagram commutes
\begin{narrow}{-3pt}{-3pt}
$$
\xymatrix@C=35pt@R=35pt{
\ext{\psi}_{0}(\varphi_{0}(s))
  \ar@/^25pt/[rr]^{\ext{e}_{0}(\chi_{s}\comp \xi_{s})}
  \ar[r]_{\ext{e}_{0}(\xi_{s})}
  \ar[d]_{\xi'_{\varphi_{0}(s)}}
  \ar@/_55pt/[dd]|{(\chi'\comp\xi')_{\varphi_{0}(s)}} &
\ext{\psi}_{0}(\varphi_{1}(s))
  \ar[r]_{\ext{e}_{0}(\chi_{s})}
  \ar[d]|{\xi'_{\varphi_{1}(s)}} &
\ext{\psi}_{0}(\varphi_{2}(s))
  \ar[d]^{\xi'_{\varphi_{2}(s)}}
  \ar@/^55pt/[dd]|{(\chi'\comp\xi')_{\varphi_{2}(s)}}\\
\ext{\psi}_{1}(\varphi_{0}(s))
  \ar[r]^{\ext{e}_{1}(\xi_{s})}
  \ar[d]_{\chi'_{\varphi_{0}(s)}} &
\ext{\psi}_{1}(\varphi_{1}(s))
  \ar[r]^{\ext{e}_{1}(\chi_{s})}
  \ar[d]|{\chi'_{\varphi_{1}(s)}} &
\ext{\psi}_{1}(\varphi_{2}(s))
  \ar[d]^{\chi'_{\varphi_{2}(s)}}\\
\ext{\psi}_{2}(\varphi_{0}(s))
  \ar[r]^{\ext{e}_{2}(\xi_{s})}
  \ar@/_25pt/[rr]_{\ext{e}_{2}(\chi_{s}\comp \xi_{s})} &
\ext{\psi}_{2}(\varphi_{1}(s))
  \ar[r]^{\ext{e}_{2}(\chi_{s})} &
\ext{\psi}_{2}(\varphi_{2}(s))
}
$$
\end{narrow}
\emph{Identities.} Finally, given polyderivor $\mathbf{d}\colon
\mathbf{\Sigma}\mor\mathbf{\Lambda}$ and $\mathbf{e}\colon
\mathbf{\Lambda}\mor\mathbf{\Omega}$ it is obvious that
\begin{enumerate}
\item The $S$-family
      $(\tp{\pi^{\varphi(s)}_{0},\ldots,\pi^{\varphi(s)}_{\bb{\varphi(s)}-1}}
      _{\varphi(s),\varphi(s)})_{s\in S}$, denoted by $\id_{\mathbf{d}}$,
      is the identity trans\-forma\-tion at $\mathbf{d}$, and that

\item $\id_{\mathbf{e}}\hcomp\id_{\mathbf{d}} = \id_{\mathbf{e}\comp\mathbf{d}}$.
\qedhere
\end{enumerate}
\end{proof}


Our next goal is to prove that the transformations between polyderivors
from a signature into a like one, determine natural transformations
between the functors between the categories of algebras associated to
the signatures.  To accomplish this we begin by proving that every
transformation $\xi$ from a polyderivor $\mathbf{d}$ to another one
$\mathbf{e}$, both from a signature $\mathbf{\Sigma}$ to a signature
$\mathbf{\Lambda}$, determines, for a given $\mathbf{\Lambda}$-algebra
$\mathbf{B}$, a $\mathbf{\Sigma}$-homomorphism $\xi^{\mathbf{B}}$ from
$\mathbf{d}^{\ast}_{\mathfrak{pd}}(\mathbf{B})$ to
$\mathbf{e}^{\ast}_{\mathfrak{pd}}(\mathbf{B})$.

\begin{proposition}\label{XialbHom}
Let $\mathbf{d}$ and $\mathbf{e}$ be polyderivors from $\mathbf{\Sigma}$
to $\mathbf{\Lambda}$, $\xi\colon \mathbf{d}\mdf\mathbf{e}$ a
transformation in $\mathbf{Sig}_{\mathfrak{pd}}$, and, for a
$\mathbf{\Lambda}$-algebra $\mathbf{B} = (B,G)$, let
$\xi^{\mathbf{B}}$ be the $S$-sorted mapping
$(\xi^{\mathbf{B}}_{s})_{s\in S}$ from $B_{\varphi}$ to $B_{\psi}$,
where, for every $s\in S$, we have that
$$
\xi_{s}^{\mathbf{B}} = \ext{G}_{\varphi(s),\psi(s)}(\xi_{s})\colon
B_{\varphi(s)}\mor B_{\psi(s)}.
$$
Then $\xi^{\mathbf{B}}$ is a $\mathbf{\Sigma}$-homomorphism from
$\mathbf{d}^{\ast}_{\mathfrak{pd}}(\mathbf{B})$ to
$\mathbf{e}^{\ast}_{\mathfrak{pd}}(\mathbf{B})$.
\end{proposition}

\begin{proof}
For every  operation $\sigma\colon w\mor s$, in $\Sigma$, we have
to prove that $G^{\mathbf{e}}_{\sigma}\comp \xi^{\mathbf{B}}_{w} =
\xi^{\mathbf{B}}_{s}\comp G^{\mathbf{d}}_{\sigma}$, and for this
it is enough to prove that every face, up to at most the frontal
one, in the following diagram commutes
$$
\xymatrix@C=40pt@R=20pt{
 &
{B_{\ext{\varphi}(w)}}
  \ar[rr]^{(\ext{G}_{\ext{\varphi}\bprod\ext{\varphi}}\comp d)_{w,s}(\sigma)}
  \ar[dd]_(.65){(\xi_{w})^{\mathbf{B}}} & &
{B_{\ext{\varphi}(s)}}
  \ar[rd]^(.45){(\iota^{B}_{\fmon{\varphi}(s)})^{-1}}
  \ar[dd]^(.65){(\xi_{s})^{\mathbf{B}}} \\
{B_{\varphi}}_{w}
  \ar[ru]^{\iota^{B}_{\fmon{\varphi}(w)}}
  \ar[rrrr]^{G^{\mathbf{d}}_{\sigma}}
  \ar[dd]|(.4){(\xi^{\mathbf{B}})_{w}} & & & &
{B_{\varphi}}_{s}
  \ar[dd]|(.4){(\xi^{\mathbf{B}})_{s}} \\
 &
{B_{\ext{\psi}(w)}}
  \ar[rr]_{(\ext{{G}}_{\ext{\psi}\bprod\ext{\psi}}\comp e)_{w,s}(\sigma)} & &
{B_{\ext{\psi}(s)}}
  \ar[rd]^(.45){(\iota^{B}_{\fmon{\psi}(s)})^{-1}}  \\
{B_{\psi}}_{w}
  \ar[ru]^{\iota^{{B}}_{\fmon{\psi}(w)}}
  \ar[rrrr]_{{G}^{\mathbf{e}}_{\sigma}} & & & &
{B_{\psi}}_{s} }
$$
from which it also follows, necessarily, that the frontal face
also commutes.

The top and bottom faces commute by definition.
The back face commutes because, being $\xi$ a transformation from
$\mathbf{d}$ to $\mathbf{e}$, from the commutativity of the
following diagram
$$
\xymatrix@C=70pt{ \ext{\varphi}(w) \ar[r]^{d(\sigma)}
\ar[d]_{\xi_{w}} & \varphi(s)
\ar[d]^{\xi_{s}} \\
\ext{\psi}(w) \ar[r]_{e(\sigma)} & \psi(s) }
$$
it follows that
\begin{align*}
(\xi_{s})^{\mathbf{B}}\comp
 (\ext{G}_{\ext{\varphi}\bprod\ext{\varphi}}\comp d)_{w,s}(\sigma)
&= \ext{G}_{\varphi(s),\psi(s)}(\xi_{s})\comp
  \ext{G}_{\ext{\varphi}(w),\varphi(s)}(d_{w,s}(\sigma)) \\
&=
\ext{G}_{\ext{\varphi}(w),\psi(s)}(\xi_{s}\comp d_{w,s}(\sigma)) \\
&=
\ext{G}_{\ext{\varphi}(w),\psi(s)}(e_{w,s}(\sigma)\comp \xi_{w}) \\
&= \ext{G}_{\ext{\psi}(w),\psi(s)}(e_{w,s}(\sigma))\comp
  \ext{G}_{\ext{\varphi}(w),\ext{\psi}(w)}(\xi_{w}) \\
&= (\ext{G}_{\ext{\psi}\bprod\ext{\psi}}\comp e)_{w,s}(\sigma)
\comp
  (\xi_{w})^{\mathbf{B}}.
\end{align*}
Relative to the lateral faces, let us verify, e.g., that the left
one commutes. For this it is enough to prove that
$$
(\xi_{w})^{\mathbf{B}}=\iota^{B}_{\fmon{\psi}(w)}\comp\xi_{w}^{\mathbf{B}}
\comp(\iota^{B}_{\fmon{\varphi}(w)})^{-1}.
$$
But we have that
\begin{align*}
    (\xi_{w})^{\mathbf{B}} &=
    \ext{G}_{\ext{\varphi}(w),\ext{\psi}(w)}(\xi_{w_{0}}\bconcat\cdots
    \bconcat \xi_{w_{\bb{w}-1}}) \\
    &=
    \ext{G}_{\varphi(w_{0}),\psi(w_{0})}(\xi_{w_{0}})\bconcat\cdots
    \bconcat\ext{G}_{\varphi(w_{\bb{w}-1},\psi(w_{\bb{w}-1})}
    (\xi_{w_{\bb{w}-1}})\\
    &=
    \xi_{w_{0}}^{\mathbf{B}}\bconcat\cdots\bconcat
    \xi_{w_{\bb{w}-1}}^{\mathbf{B}}\\
    &=
    \tp{
    \xi_{w_{0}}^{\mathbf{B}}\comp\pr^{B_{\ext{\varphi}(w)}}_{(0)},\ldots,
    \xi_{w_{\bb{w}-1}}^{\mathbf{B}}\comp\pr^{B_{\ext{\varphi}(w)}}_{(\bb{w}-1)},
    }
\end{align*}
hence it is enough to prove that, for every $i\in\bb{w}$, we have
that
$$
\pr^{B_{\ext{\psi}(w)}}_{(i)}\comp
\iota^{B}_{\fmon{\psi}(w)}\comp\xi_{w}^{\mathbf{B}}
\comp(\iota^{B}_{\fmon{\varphi}(w)})^{-1} =
\xi_{w_{i}}^{\mathbf{B}}\comp\pr^{B_{\ext{\varphi}(w)}}_{(i)},
$$
but this follows from the commutativity of the following diagram
$$
\xymatrix@C=50pt@R=28pt{ {} & B_{\ext{\varphi}(w)}
  \ar[rd]^{\pr^{B_{\ext{\varphi}(w)}}_{(i)}}
  \ar[ld]_{(\iota^{B}_{\fmon{\varphi}(w)})^{-1}} & {} \\
{B_{\varphi}}_{w}
  \ar[rr]^{\pr^{{B_{\varphi}}_{w}}_{i}}
  \ar[d]_{(\xi^{\mathbf{B}})_{w}} & {} &
B_{\varphi(w_{i})}
  \ar[d]^{\xi^{\mathbf{B}}_{w_{i}}} \\
{B_{\psi}}_{w}
  \ar[rr]_{\pr^{{B_{\psi}}_{w}}_{i}}
  \ar[dr]_{\iota^{B}_{\fmon{\psi}(w)}} & {} &
B_{\psi(w_{i})}
  \\
{} & B_{\ext{\psi}(w)}
  \ar[ru]_(.6){\pr^{B_{\ext{\psi}(w)}}_{(i)}} & {}
}
$$
\end{proof}

After having proved, for two polyderivors $\mathbf{d}$ and
$\mathbf{e}$ from $\mathbf{\Sigma}$ to $\mathbf{\Lambda}$, that
every transformation $\xi$ from $\mathbf{d}$ to $\mathbf{e}$,
induces, for every $\mathbf{\Lambda}$-algebra $\mathbf{B}$, a
$\mathbf{\Sigma}$-homomorphism $\xi^{\mathbf{B}}$ from
$\mathbf{d}^{\ast}_{\mathfrak{pd}}(\mathbf{B})$ to
$\mathbf{e}^{\ast}_{\mathfrak{pd}}(\mathbf{B})$, we prove in the
following proposition the naturalness of the involved process.

\begin{proposition}\label{AlgFujXiNat}
Let $\xi\colon \mathbf{d}\mdf\mathbf{e}$ be a transformation with
$\mathbf{d}$ and $\mathbf{e}$ polyderivors from $\mathbf{\Sigma}$ to
$\mathbf{\Lambda}$.  Then the family $(\xi^{\mathbf{B}})
_{\mathbf{B}\in\mathbf{Alg}(\mathbf{\Lambda})}$, denoted by
$\Alg_{\mathfrak{pd}}(\xi)$, is a natural transformation from the
functor $\mathbf{d}^{\ast}_{\mathfrak{pd}}$ to the functor
$\mathbf{e}^{\ast}_{\mathfrak{pd}}$, both from
$\mathbf{Alg}(\mathbf{\Lambda})$ to $\mathbf{Alg}(\mathbf{\Sigma})$.
\end{proposition}

\begin{proof}
We have to prove that, for every $\mathbf{\Lambda}$-algebras $\mathbf{B} =
(B,G)$, $\mathbf{C} = (C,H)$ and morphism $f\colon
\mathbf{B}\mor \mathbf{C}$ in $\mathbf{Alg}(\mathbf{\Lambda})$, the following
diagram commutes
$$
\xymatrix@C=70pt{
(B_{\varphi},G^{\mathbf{d}})
      \ar[r]^{\xi^{\mathbf{B}}}
      \ar[d]_{f_{\varphi}} &
(B_{\psi},G^{\mathbf{e}})
      \ar[d]^{f_{\psi}}  \\
(C_{\varphi},H^{\mathbf{d}})
      \ar[r]_{\xi^{\mathbf{C}}} &
(C_{\psi},H^{\mathbf{e}})
}
$$
But this is immediate because, for every $s\in S$, we have that
$\xi^{\mathbf{B}}_{s}$ and $\xi^{\mathbf{C}}_{s}$ are the realizations
of the term $\xi_{s}$ in the respective algebras, hence the following
diagram commutes
$$
\xymatrix@C=70pt{
B_{\varphi(s)} \ar[r]^{\xi^{\mathbf{B}}_{s}}
            \ar[d]_{f_{\varphi(s)}} &
B_{\psi(s)} \ar[d]^{f_{\psi(s)}} \\
C_{\varphi(s)} \ar[r]_{\xi^{\mathbf{C}}_{s}} &
C_{\psi(s)}
}
$$
\end{proof}

Once stated that the transformations between polyderivors from a
signature into a like one, induce natural transformations among the
functors between the categories of algebras associated to the
signatures, we can properly lift the pseudo-functor
$\Alg_{\mathfrak{pd}}\colon\mathbf{Sig}_{\mathfrak{pd}}\mor
\mathbf{Cat}$ up to the $2$-cells in the $2$-category
$\mathbf{Sig}_{\mathfrak{pd}}$.

\begin{proposition}
There exists a pseudo-functor $\Alg_{\mathfrak{pd}}$, contravariant in the
morphisms and covariant in the $2$-cells, from the $2$-category
$\mathbf{Sig}_{\mathfrak{pd}}$ to the $2$-category $\mathbf{Cat}$
given schematical\-ly by the following data
$$
\xymatrixcolsep={17ex}
\xymatrixrowsep={11ex}
\xymatrix{
\mathbf{\Sigma}
  \xymn[0pt,30pt]{\mathbf{Sig}_{\mathfrak{pd}}}{1}
  \ar @/_16pt/ [d]_*+{\mathbf{d}}="d"
  \ar @/^16pt/ [d]^*+{\mathbf{e}}="e"
  \ar @{} "d";"e"^*+/3pt/{\xi}|(.53){\dir{~>}} &&
\mathbf{Alg}(\mathbf{\Sigma})
  \xymn[0pt,30pt]{\mathbf{Cat}}{2}
  \ar "1";"2"^{\Alg_{\mathfrak{pd}}} \\
\mathbf{\Lambda} &&
\mathbf{Alg}(\mathbf{\Lambda})
  \ar @/^10pt/ []+<-4ex,2ex>;[u]+<-4ex,-2ex>
               ^*+{\mathbf{d}^{\ast}_{\mathfrak{pd}}}="dd"
  \ar @/_10pt/ []+<+4ex,2ex>;[u]+<+4ex,-2ex>
               _*+{\mathbf{e}^{\ast}_{\mathfrak{pd}}}="ee"
  \ar @{}"dd";"ee"^*+/3pt/{\Alg_{\mathfrak{pd}}(\xi)}|(.53){\dir{=>}}
\ar@{}"e";"dd" |{\lmapsto}
}
$$
together with the accompanying natural isomorphisms
$\gamma^{\mathbf{d},\mathbf{e}}$ and $\nu^{\mathbf{\Sigma}}$, as defined
in Proposition~\ref{defPseudofunctorAlgfuj}.
\end{proposition}

\begin{proof}
It follows from the fact that the natural isomorphisms of the
pseudo-functor are compatible with the structure of $2$-category of
$\mathbf{Sig}_{\mathfrak{pd}}$.
\end{proof}

On the basis of this last proposition we can lift the category
$\mathbf{Alg}_{\mathfrak{pd}}$ up to a 2-category as in the following

\begin{definition}
We denote by $\mathbf{Alg}_{\mathfrak{pd}} =
\iint^{\mathbf{Sig}_{\mathfrak{pd}}}\Alg_{\mathfrak{pd}}$ the
$2$-category which has
\begin{enumerate}
\item As objects ($0$-cells) the pairs
      $(\mathbf{\Sigma},\mathbf{A})$, where $\mathbf{\Sigma}$ is a
      signature and $\mathbf{A}$ a $\mathbf{\Sigma}$\nobreakdash-algebra,

\item As morphisms ($1$-cells) from
      $(\mathbf{\Sigma},\mathbf{A})$ to
      $(\mathbf{\Lambda},\mathbf{B})$ the pairs $(\mathbf{d},f)$,
      where $\mathbf{d}$ is a polyderivor from $\mathbf{\Sigma}$ to
      $\mathbf{\Lambda}$ and $f$ a $\mathbf{\Sigma}$-homomorphism from
      $\mathbf{A}$ to $\mathbf{d}^{\ast}_{\mathfrak{pd}}(\mathbf{B})$,
      and

\item As $2$-cells from $(\mathbf{d},f)$ to $(\mathbf{e},g)$,
      where $(\mathbf{d},f)$ and $(\mathbf{e},g)$ are morphisms from
      $(\mathbf{\Sigma},\mathbf{A})$ to
      $(\mathbf{\Lambda},\mathbf{B})$, the $2$-cells
      $\xi\colon\mathbf{\Sigma}\mdf\mathbf{\Lambda}$ in
      $\mathbf{Sig}_{\mathfrak{pd}}$ such that
      the following diagram commutes
      $$\xymatrix@R=10pt@C=60pt{
      {} &
      \mathbf{d}^{\ast}_{\mathfrak{pd}}(\mathbf{B})
      \ar[dd]^-{\xi^{\mathbf{B}}} \\
      \mathbf{A}
      \ar[ru]^-{f}
      \ar[rd]_-{g} \\
      {} & \mathbf{e}^{\ast}_{\mathfrak{pd}}(\mathbf{B})
      }
      $$
\end{enumerate}

\end{definition}

Relative to the above $2$-category we point out that for the following
configuration of $2$-cells
$$
\xymatrix@C=15ex@R=8ex{
(\mathbf{\Sigma},\mathbf{A})
  \ar@/^3ex/[r]|*+{(\mathbf{d},f)}="f1"
  \ar@/_3ex/[r]|*+{(\mathbf{e},g)}="f2" &
(\mathbf{\Lambda},\mathbf{B})
  \ar@/^3ex/[r]|*+{(\mathbf{h},p)}="f3"
  \ar@/_3ex/[r]|*+{(\mathbf{i},q)}="f4" &
(\mathbf{\Omega},\mathbf{C})
\ar @{} "f1";"f2"|{\dir{~>}}^{\xi}
\ar @{} "f3";"f4"|{\dir{~>}}^{\chi}
}
$$
we have the commutative diagram
$$
\xymatrix@C=25pt@R=2pt{
{} & {} & {} & {} & {} &
(\mathbf{h}\comp \mathbf{d})^{\ast}_{\mathfrak{pd}}(\mathbf{C})
\ar[dddd]|(.5){\xi^{\mathbf{h}^{\ast}_{\mathfrak{pd}}(\mathbf{C})}}
\ar[ddl]|(.5){\mathbf{d}^{\ast}_{\mathfrak{pd}}(\chi^{\mathbf{C}})} \\
{} & {} &
\mathbf{d}^{\ast}_{\mathfrak{pd}}(\mathbf{B})
\ar[rrru]|(.4){\mathbf{d}^{\ast}_{\mathfrak{pd}}(p)}
\ar[rrd]|(.5){\mathbf{d}^{\ast}_{\mathfrak{pd}}(q)}
\ar[dddd]|(.5){\xi^{\mathbf{B}}} & {} & {} & {} \\
{} & {} & {} & {} &
(\mathbf{i}\comp \mathbf{d})^{\ast}_{\mathfrak{pd}}(\mathbf{C})
\ar[dddd]|(.4){\xi^{\mathbf{i}^{\ast}_{\mathfrak{pd}}(\mathbf{C})}} & {} \\
\mathbf{A}\ar[rruu]^-{f}\ar[rrdd]_-{g} & {} & {} & {} & {} & {} \\
{} & {} & {} & {} & {} & (\mathbf{h}\comp
\mathbf{e})^{\ast}_{\mathfrak{pd}}(\mathbf{C})
\ar[ddl]|(.5){\mathbf{e}^{\ast}_{\mathfrak{pd}}(\chi^{\mathbf{C}})} \\
{} & {} &
\mathbf{e}^{\ast}_{\mathfrak{pd}}(\mathbf{B})
\ar[rrru]|(.4){\mathbf{e}^{\ast}_{\mathfrak{pd}}(p)}
\ar[rrd]|(.5){\mathbf{e}^{\ast}_{\mathfrak{pd}}(q)} & {} & {} & {} \\
{} & {} & {} & {} &
(\mathbf{i}\comp \mathbf{e})^{\ast}_{\mathfrak{pd}}(\mathbf{C}) & {}
}
$$


As was the case above for algebras and transformations, our goal now
is to prove that the transformations between polyderivors from a signature
into a like one, also determine natural transformations between the
functors between the categories of terms associated to the signatures.
To accomplish this we begin by proving that every transformation $\xi$
from a polyderivor $\mathbf{d}$ to another one $\mathbf{e}$, both from a
signature $\mathbf{\Sigma}$ to a signature $\mathbf{\Lambda}$,
determines, for a given $S$-sorted set $X$, a morphism $\xi_{X}$, in
the category $\mathbf{Ter}(\mathbf{\Lambda})$, from
$\coprod_{\varphi}^{\dagger}X$ to $\coprod_{\psi}^{\dagger}X$.

\begin{proposition}
Let $\mathbf{d}$ and $\mathbf{e}$ be polyderivors from $\mathbf{\Sigma}$
to $\mathbf{\Lambda}$, $\xi\colon \mathbf{d}\mdf\mathbf{e}$ a
transformation in $\mathbf{Sig}_{\mathfrak{pd}}$, and, for an
$S$-sorted set $X$, let $\xi_{X}\colon \coprod_{\psi}^{\dagger}X\mor
\mathrm{T}_{\mathbf{\Lambda}}(\coprod_{\varphi}^{\dagger}X)$ be the
$T$-sorted mapping defined, for $t\in T$ and $(x,s,\psi(s),i)\in
(\coprod_{\psi}^{\dagger}X)_{t}$, as follows
$$
(\xi_{X})_{t}(x,s,\psi(s),i) =
(\xi_{s})_{i}(v^{\varphi(s)_{j}}_{j}
             /(x,s,\varphi(s),j)\mid j\in \bb{\varphi(s)}).
$$
Then the mapping $\xi_{X}$ is a morphism, in the category
$\mathbf{Ter}(\mathbf{\Lambda})$, from $\coprod_{\varphi}^{\dagger}X$
to $\coprod_{\psi}^{\dagger}X$.
\end{proposition}

\begin{proof}
The definition of $\xi_{X}\colon \coprod_{\psi}^{\dagger}X\mor
\mathrm{T}_{\mathbf{\Lambda}}(\coprod_{\varphi}^{\dagger}X)$ is sound
since, for every $j\in\bb{\varphi(s)}$, we have that
$(x,s,\varphi(s),j)\in
(\coprod_{\varphi}^{\dagger}X)_{\varphi(s)_{j}}$ and $(\xi_{s})_{i}\in
\mathrm{T}_{\mathbf{\Lambda}}(\varphi(s))_{\psi(s)_{i}}$, hence
$(\xi_{X})_{t}(x,s,\varphi(s),i)$ is a term for $\mathbf{\Lambda}$ of
type $\psi(s)_{i}=t$.
\end{proof}

After having proved, for two polyderivors $\mathbf{d}$ and $\mathbf{e}$ from
$\mathbf{\Sigma}$ to $\mathbf{\Lambda}$, that every transformation
$\xi$ from $\mathbf{d}$ to $\mathbf{e}$, induces, for every $S$-sorted
set $X$, a morphism $\xi_{X}$ from $\coprod_{\varphi}^{\dagger}X$ to
$\coprod_{\psi}^{\dagger}X$, we prove in the following proposition
that they are the components of a natural transformation.

\begin{proposition}\label{TerFujXiNat}
Let $\xi\colon \mathbf{d}\mdf\mathbf{e}$ be a transformation in
$\mathbf{Sig}_{\mathfrak{pd}}$, with $\mathbf{d}$, $\mathbf{e}$ polyderivors from
$\mathbf{\Sigma}$ to $\mathbf{\Lambda}$.  Then
$\mathrm{Ter}_{\mathfrak{pd}}(\xi) =
(\xi_{X})_{X\in\mathbf{Ter}(\mathbf{\Sigma})}$ is a natural transformation
from $\mathbf{d}^{\mathfrak{pd}}_{\diamond}$ to
$\mathbf{e}^{\mathfrak{pd}}_{\diamond}$.
\end{proposition}

\begin{proof}
Because, for a morphism $P\colon X\mor Y$ in $\mathbf{Ter}(\mathbf{\Sigma})$,
the following diagram commutes
$$
\xymatrix@C=16ex@R=7ex{
\coprod_{\varphi}^{\dagger}X
  \ar[r]^{\xi_{X}}
  \ar[d]_{\mathbf{d}^{\mathfrak{pd}}_{\diamond}(P)} &
\coprod_{\psi}^{\dagger}X
  \ar[d]^{\mathbf{e}^{\mathfrak{pd}}_{\diamond}(P)}  \\
\coprod_{\varphi}^{\dagger}Y
  \ar[r]_{\xi_{Y}} &
\coprod_{\psi}^{\dagger}Y
}
$$
\end{proof}

Observe that this last proposition is analogous to
Proposition~\ref{deformacionesSimDer} but for derived operations with
variables in arbitrary many-sorted sets.

Once stated that the transformations between polyderivors from a
signature into a like one, induce natural transformations among the
functors between the categories of terms associated to the signatures,
we can properly lift the pseudo-functor
$\mathrm{Ter}_{\mathfrak{pd}}\colon\mathbf{Sig}_{\mathfrak{pd}}\mor
\mathbf{Cat}$ up to the $2$-cells of the $2$-category
$\mathbf{Sig}_{\mathfrak{pd}}$.

\begin{proposition}
There exists a pseudo-functor $\mathrm{Ter}_{\mathfrak{pd}}$ from the
$2$-category $\mathbf{Sig}_{\mathfrak{pd}}$ to $\mathbf{Cat}$,
covariant in the morphisms and the 2-cells, given schematical\-ly by the
following data
$$
\xymatrixcolsep={17ex}
\xymatrixrowsep={11ex}
\xymatrix{
\mathbf{\Sigma}
  \xymn[0pt,30pt]{\mathbf{Sig}_{\mathfrak{pd}}}{1}
  \ar @/_16pt/ [d]_*+{\mathbf{d}}="d"
  \ar @/^16pt/ [d]^*+{\mathbf{e}}="e"
  \ar @{} "d";"e"^*+/3pt/{\xi}|(.53){\dir{~>}} &&
\mathbf{Ter}(\mathbf{\Sigma})
  \xymn[0pt,30pt]{\mathbf{Cat}}{2}
  \ar "1";"2"^{\mathrm{Ter}_{\mathfrak{pd}}}
  \ar @/^10pt/ []+<4ex,-2ex>;[d]+<+4ex,+2ex>
               ^*+{\mathbf{e}^{\mathfrak{pd}}_{\diamond}}="ee"
  \ar @/_10pt/ []+<-4ex,-2ex>;[d]+<-4ex,+2ex>
               _*+{\mathbf{d}^{\mathfrak{pd}}_{\diamond}}="dd"
  \ar @{}"dd";"ee"^*+/3pt/{\mathrm{Ter}_{\mathfrak{pd}}(\xi)}|(.53){\dir{=>}}
  \\
\mathbf{\Lambda} &&
\mathbf{Ter}(\mathbf{\Lambda})
\ar@{}"e";"dd"|{\lmapsto}
}
$$
together with the accompanying natural isomorphisms
$\gamma^{\mathbf{d},\mathbf{e}}$ and $\nu^{\mathbf{\Sigma}}$, as
defined in Proposition~\ref{defPseudofunctorTerfuj}.
\end{proposition}

\begin{proof}
It follows from the fact that the natural isomorphisms of the
pseudo-functor are compatible with the structure of $2$-category of
$\mathbf{Sig}_{\mathfrak{pd}}$.
\end{proof}


Up to this point what we have at our disposal consists, essentially,
of the following data:
\begin{enumerate}
\item The pseudo-functor $\Alg_{\mathfrak{pd}}$, contravariant in
      the polyderivors and covariant in the transformations, from the
      $2$-category $\mathbf{Sig}_{\mathfrak{pd}}$ to the $2$-category
      $\mathbf{Cat}$,

\item The pseudo-functor $\mathrm{Ter}_{\mathfrak{pd}}$,
      covariant in the polyderivors and the transformations,
      from the $2$-category $\mathbf{Sig}_{\mathfrak{pd}}$ to
      the $2$-category $\mathbf{Cat}$,

\item The family of functors $\mathrm{Tr} =
      (\mathrm{Tr}^{\mathbf{\Sigma}})
      _{\mathbf{\Sigma}\in\mathbf{Sig}_{\mathfrak{pd}}}$, and

\item The family of natural isomorphisms
      $\theta =
      (\theta^{\mathbf{d}})
      _{\mathbf{d}\in\Mor(\mathbf{Sig}_{\mathfrak{pd}})}$, and

\item Taking $\mathbf{Sig}_{\mathfrak{pd}}$ as an ordinary
      category, the institution $\mathfrak{Tm}_{\mathfrak{pd}}$
      on $\mathbf{Set}$.
\end{enumerate}

But it happens that $\mathbf{Sig}_{\mathfrak{pd}}$ is a $2$-category,
hence our next goal will be to show that
$\mathfrak{Tm}_{\mathfrak{pd}}$ is not only an institution on
$\mathbf{Set}$ but actually a $2$-institution on $\mathbf{Set}$.  To
attain the just stated goal we should begin by proving that the
realization of the terms in the algebras is compatible with the
additional structure derived from the $2$-cells in
$\mathbf{Sig}_{\mathfrak{pd}}$, i.e., the transformations between
polyderivors.

\begin{lemma}
Let $\xi\colon \mathbf{d}\mdf\mathbf{e}$ be a transformation in
$\mathbf{Sig}_{\mathfrak{pd}}$ from the polyderivor $\mathbf{d}$ to the
polyderivor $\mathbf{e}$, both from $\mathbf{\Sigma}$ to
$\mathbf{\Lambda}$.  Then, for every $\mathbf{\Lambda}$-algebra
$\mathbf{B}$, and $S$-sorted set $X$, the mappings
$(\xi^{\mathbf{B}})_{X}\comp (\theta^{\dagger\nat}_{\varphi})_{X,B}$
and $(\theta^{\dagger\nat}_{\psi})_{X,B}\comp (\xi_{X})^{\mathbf{B}}$
from $B_{\coprod_{\varphi}^{\dagger}X}$ to
$(\Delta^{\nat}_{\psi}B)_{X}$ are identical, i.e., the following
diagram commutes
$$
\xymatrix@C=12ex@R=7ex{
B_{\coprod_{\varphi}^{\dagger}X}
  \ar[r]^-{(\xi_{X})^{\mathbf{B}}}
  \ar[d]_{(\theta^{\dagger\nat}_{\varphi})_{X,B}} &
B_{\coprod_{\psi}^{\dagger}X}
  \ar[d]^{(\theta^{\dagger\nat}_{\psi})_{X,B}}  \\
(\Delta^{\nat}_{\varphi}B)_{X}
  \ar[r]_-{(\xi^{\mathbf{B}})_{X}} &
(\Delta^{\nat}_{\psi}B)_{X}  }
$$
\end{lemma}

\begin{proof}
For every $f\in B_{\coprod_{\varphi}^{\dagger}X}$,
$(\xi_{X})^{\mathbf{B}}(f)\in B_{\coprod_{\psi}^{\dagger}X}$ is the
morphism $\ext{f}\comp \xi_{X}$, where $\ext{f}$ is the extension of
$f$ to $\mathbf{T}_{\mathbf{\Lambda}}(\coprod_{\varphi}^{\dagger}X)$,
obtained as shown in the following diagram
$$
\xymatrix@C=10ex@R=9ex{
\coprod_{\varphi}^{\dagger}X
  \ar[r]^-{\eta_{\coprod_{\varphi}^{\dagger}X}}
  \ar[rd]_-{f} &
\mathrm{T}_{\mathbf{\Lambda}}(\coprod_{\varphi}^{\dagger}X)
  \ar[d]^-{\ext{f}}  &
\coprod_{\psi}^{\dagger}X
  \ar[l]_-{\xi_{X}}
  \ar[ld]^-{(\xi_{X})^{\mathbf{B}}(f)} \\
&
B
}
$$
hence $(\theta^{\dagger\nat}_{\psi})_{X,B}((\xi_{X})^{\mathbf{B}}(f))$
is a morphism from $X$ to $\Delta^{\nat}_{\psi}B$.  Now, for $s\in S$
and $x\in X_{s}$, we have that
\begin{align*}
\bigl(
  (\theta^{\dagger\nat}_{\psi})_{X,B}
    \bigl((\xi_{X})^{\mathbf{B}}(f)\bigr)
\bigr)_{s}(x)
&=
\bigl(
  (\theta^{\dagger\nat}_{\psi})_{X,B}
    \bigl(\ext{f}\comp\xi_{X}\bigr)
\bigr)_{s}(x)
\\&=
\bigl(
  (\ext{f}\comp\xi_{X})_{\psi(s)_{i}}(x,s,\psi(s),i)\mid i\in\bb{\psi(s)}
\bigr)
\\&=
\bigl(
  \ext{f}_{\psi(s)_{i}}\bigl((\xi_{X})_{\psi(s)_{i}}(x,s,\psi(s),i)
\bigr)\mid i\in\bb{\psi(s)}
\bigr)
\\&=
\ext{f}_{\psi(s)}
  \bigl(
    (\xi_{X})_{\psi(s)_{i}}(x,s,\psi(s),i)\mid i\in\bb{\psi(s)}
  \bigr)
\\&=
\ext{f}_{\psi(s)}
  \Bigl(
    (\xi_{s})_{i}\bigl(v^{\varphi(s)_{j}}_{j}/(x,s,\varphi(s),j)
    \mid j\in\bb{\varphi(s)}\bigr)_{i\in\bb{\psi(s)}}
  \Bigr)
\\&=
\ext{f}_{\psi(s)}
 \bigl(
  \xi_{s}^{\mathbf{T}_{\mathbf{\Lambda}}(\coprod_{\varphi}^{\dagger}X)}
  \bigl((x,s,\varphi(s),j)\mid j\in\bb{\varphi(s)}\bigr)
 \bigr)
\\&=
\xi_{s}^{\mathbf{B}}
  \bigl(
    \ext{f}_{\varphi(s)} \bigl((x,s,\varphi(s),j)\mid j\in\bb{\varphi(s)}\bigr)
  \bigr)
\intertext{\hspace{10ex}because $\Alg_{\mathfrak{pd}}(\xi)$ is natural
           and $\ext{f}$ a morphism}
&=
\xi_{s}^{\mathbf{B}}
  \bigl(
     \ext{f}_{\varphi(s)_{j}}(x,s,\varphi(s),j)\mid j\in\bb{\varphi(s)}
  \bigr)
\\&=
\xi_{s}^{\mathbf{B}}
  \bigl(
       (\theta^{\dagger\nat}_{\varphi})_{X,B}(f)_{s} (x)
  \bigr)
\\&=
  \bigl(
     (\xi^{\mathbf{B}})_{X}
     \bigl(
       (\theta^{\dagger\nat}_{\varphi})_{X,B}(f)
     \bigr)
  \bigr)_{s}(x).
\end{align*}
Therefore $(\xi^{\mathbf{B}})_{X}\comp
(\theta^{\dagger\nat}_{\varphi})_{X,B} =
(\theta^{\dagger\nat}_{\psi})_{X,B}\comp (\xi_{X})^{\mathbf{B}}$, as
asserted.
\end{proof}

In the following proposition, which will be the basis to get the
many-sorted term $2$-institution of Fujiwara on $\mathbf{Set}$, we
construct a pseudo-functor $\Alg_{\mathfrak{pd}}(\farg)\bprod
\mathrm{Ter}_{\mathfrak{pd}}(\farg)$ from the $2$-category
$\mathbf{Sig}^{\opp}_{\mathfrak{pd}}\bprod \mathbf{Sig}_{\mathfrak{pd}}$
to the $2$-category $\mathbf{Cat}$ (obtained from the pseudo-functors
$\Alg$ and $\mathrm{Ter}$), and prove that the family
$\mathrm{Tr}=(\mathrm{Tr}^{\mathbf{\Sigma}})_{\mathbf{\Sigma}\in\mathbf{Sig}}$,
together with the family
$\theta=(\theta^{\mathbf{d}})_{\mathbf{d}\in\Mor(\mathbf{Sig})}$ is a
pseudo-extranatural transformation from the pseudo-functor
$\Alg_{\mathfrak{pd}}(\farg)\bprod\mathrm{Ter}_{\mathfrak{pd}}(\farg)$
to the functor $\mathrm{K}_{\mathbf{Set}}$ from
$\mathbf{Sig}^{\opp}_{\mathfrak{pd}}\bprod \mathbf{Sig}_{\mathfrak{pd}}$
to $\mathbf{Cat}$ that is constantly $\mathbf{Set}$.

\begin{proposition}\label{PdExtranaturalDef}
There exists a pseudo-functor $\Alg_{\mathfrak{pd}}(\farg)\bprod
\mathrm{Ter}_{\mathfrak{pd}}(\farg)$ from the $2$-category
$\mathbf{Sig}^{\opp}_{\mathfrak{pd}}\bprod \mathbf{Sig}_{\mathfrak{pd}}$
to the $2$-category $\mathbf{Cat}$, obtained from the pseudo-functors
$\Alg_{\mathfrak{pd}}$ and $\mathrm{Ter}_{\mathfrak{pd}}$, which sends a
pair of signatures $(\mathbf{\Sigma},\mathbf{\Lambda})$ to the
category
$\mathbf{Alg}(\mathbf{\Sigma})\bprod\mathbf{Ter}(\mathbf{\Lambda})$,
and a pair of signature morphisms $(\mathbf{d},\mathbf{e})$ from
$(\mathbf{\Sigma},\mathbf{\Lambda})$ to
$(\mathbf{\Sigma}',\mathbf{\Lambda}')$ in
$\mathbf{Sig}_{\mathfrak{pd}}^{\opp}\bprod\mathbf{Sig}_{\mathfrak{pd}}$
to the functor $\mathbf{d}^{\ast}_{\mathfrak{pd}}\bprod
\mathbf{d}^{\mathfrak{pd}}_{\diamond}$ from
$\mathbf{Alg}(\mathbf{\Sigma})\bprod\mathbf{Ter}(\mathbf{\Lambda})$ to
$\mathbf{Alg}(\mathbf{\Sigma}')\bprod\mathbf{Ter}(\mathbf{\Lambda}')$.

Furthermore, the family of functors
$\mathrm{Tr} =
(\mathrm{Tr}^{\mathbf{\Sigma}})_{\mathbf{\Sigma}\in\mathbf{Sig}_{\mathfrak{pd}}}$,
together with the family
$\theta=(\theta^{\mathbf{d}})_{\mathbf{d}\in\Mor(\mathbf{Sig}_{\mathfrak{pd}})}$,
with $\theta^{\mathbf{d}}_{\mathbf{A},X}=\theta^{\dagger\nat}_{X,A}$,
is a pseudo-extranatural transformation from the pseudo-functor
$\Alg_{\mathfrak{pd}}(\farg)\bprod\mathrm{Ter}_{\mathfrak{pd}}(\farg)$
to the functor $\mathrm{K}_{\mathbf{Set}}$ from
$\mathbf{Sig}^{\opp}_{\mathfrak{pd}}\bprod \mathbf{Sig}_{\mathfrak{pd}}$ to
$\mathbf{Cat}$ that is constantly $\mathbf{Set}$.
\end{proposition}

\begin{proof}
We restrict ourselves to prove that, for every transformation $\xi$
in $\mathbf{Sig}_{\mathfrak{pd}}$ from the polyderivor $\mathbf{d}$ to the
polyderivor $\mathbf{e}$, both from $\mathbf{\Sigma}$ to
$\mathbf{\Lambda}$, we have that the following equation holds
$$
\theta^{\mathbf{e}}\comp
(\mathrm{Tr}^{\mathbf{\Sigma}}\hcomp (\Alg_{\mathfrak{pd}}(\xi)\times 1)) =
(\mathrm{Tr}^{\mathbf{\Lambda}}\hcomp (1\bprod
\Ter_{\mathfrak{pd}}(\xi)))
\comp \theta^{\mathbf{d}}.
$$

Let $f\colon\mathbf{A}\mor \mathbf{B}$ be a morphism in
$\mathbf{Alg}(\mathbf{\Lambda})$ and $P\colon X\mor Y$ a morphism in
$\mathbf{Ter}(\mathbf{\Sigma})$.  Then we have the following
configuration
\begin{narrow}{-8pt}{-8pt}
$$
\xy
0;<1ex,0ex>:<0ex,1ex>::
%
\POS 0+(0,0)
\xymatrix@C=20ex@R=10ex@!0{
(\mathbf{A},X) \ar[r]^{(f,P)} & (\mathbf{B},Y)
\save"1,1"."1,2"!C*\frm{}="a"\restore 
}
\POS 0+(22,-13)
\xymatrix@C=23ex@R=12ex@!0{
(\mathbf{d}_{\mathfrak{pd}}^{\ast}(\mathbf{A}),X)
  \ar[r]^{(\xi ^{\mathbf{A}},X)}
  \ar[d]|{(f_{\varphi},P)} &
(\mathbf{e}_{\mathfrak{pd}}^{\ast}(\mathbf{A}),X)
  \ar[d]|{(f_{\psi},P)}  \\
(\mathbf{d}_{\mathfrak{pd}}^{\ast}(\mathbf{B}),Y)
  \ar[r]_{(\xi ^{\mathbf{B}},Y)} &
(\mathbf{e}_{\mathfrak{pd}}^{\ast}(\mathbf{B}),Y)
\save"1,1"."2,2"!C*+<5ex>\frm{}="b"\restore 
}
\POS 0+(-22,-13)
\xymatrix@C=23ex@R=12ex@!0{
(\mathbf{A},\coprod_{\varphi}^{\dagger}X)
  \ar[r]^{(\mathbf{A},\xi_{X})}
  \ar[d]|{(f,\mathbf{d}_{\diamond}^{\mathfrak{pd}}(P))}&
(\mathbf{A},\coprod_{\psi}^{\dagger}X)
  \ar[d]|{(f,\mathbf{e}_{\diamond}^{\mathfrak{pd}}(P))}  \\
(\mathbf{B},\coprod_{\varphi}^{\dagger}Y)
  \ar[r]_{(\mathbf{B},\xi_{Y})} &
(\mathbf{B},\coprod_{\psi}^{\dagger}Y)
\save"1,1"."2,2"!C*+<5ex>\frm{}="c"\restore 
}
\POS 0+(8,-38)
\xymatrix"d"@C=28ex@R=19ex@!0{
(A_{\varphi})_{X}
  \ar[r]|*+{(\xi^{\mathbf{A}})_{X}}
  \ar[d]|*+{P^{\mathbf{d}_{\mathfrak{pd}}^{\ast}(\mathbf{A})}}
  &
(A_{\psi})_{X}
  \ar[d]|*+{P^{\mathbf{e}_{\mathfrak{pd}}^{\ast}(\mathbf{A})}}
   \\
(A_{\varphi})_{Y}
  \ar[r]|*+{(\xi^{\mathbf{A}})_{Y}}
  \ar[d]|*+{(f_{\varphi})_{Y}}
  &
(A_{\psi})_{Y}
  \ar[d]|*+{(f_{\psi})_{Y}}
   \\
(B_{\varphi})_{Y}
  \ar[r]|*+{(\xi ^{\mathbf{B}})_{Y}}
  &
(B_{\psi})_{Y}
\save "d1,2"."d3,1"!C*+<5ex>\frm{}="d"\restore 
}
\POS (-12,-45)
\xymatrix"e"@C=28ex@R=19ex@!0{
A_{\coprod_{\varphi}^{\dagger}X}
  \ar[r]|*+{(\xi_{X})^{\mathbf{A}}}
  \ar[d]|*+{\mathbf{d}_{\diamond}^{\mathfrak{pd}}(P)^{\mathbf{A}}}&
A_{\coprod_{\psi}^{\dagger}X}
  \ar[d]|*+{\mathbf{e}_{\diamond}^{\mathfrak{pd}}(P)^{\mathbf{A}}}  \\
A_{\coprod_{\varphi}^{\dagger}Y}
  \ar[r]|*+{(\xi_{Y})^{\mathbf{A}}}
  \ar[d]|*+{f_{\coprod_{\varphi}^{\dagger}Y}}&
A_{\coprod_{\psi}^{\dagger}Y}
  \ar[d]|*+{f_{\coprod_{\psi}^{\dagger}Y}}  \\
B_{\coprod_{\varphi}^{\dagger}Y}
  \ar[r]|*+{(\xi_{Y})^{\mathbf{B}}} &
B_{\coprod_{\psi}^{\dagger}Y}
\save"e1,2"."e3,1"!C*+<10ex>\frm{}="e"\restore 
}
\ar "e1,1";"d1,1"|*{(\theta^{\dagger\nat}_{\varphi})_{X,A}}
\ar "e1,2";"d1,2"|*{(\theta^{\dagger\nat}_{\psi})_{X,A}}
\ar "e2,1";"d2,1"|*{(\theta^{\dagger\nat}_{\varphi})_{Y,A}}
\ar "e2,2";"d2,2"|*{(\theta^{\dagger\nat}_{\psi})_{Y,A}}
\ar "e3,1";"d3,1"|*{(\theta^{\dagger\nat}_{\varphi})_{Y,B}}
\ar "e3,2";"d3,2"|*{(\theta^{\dagger\nat}_{\psi})_{Y,B}}
\ar @{|->}"a";"b"
\ar @{|->}"a";"c"
\ar @{|->}"b";"d"
\ar @{|->}"c";"e"
\endxy
$$
\end{narrow}

In the cube, the top, middle and bottom faces commute by the preceding
lemma.  The lateral faces commute by
Lemma~\ref{lemmaPdExtranaturalFuj}.  The front face of the upper cube
commutes by Proposition~\ref{TerFujXiNat} and the front face of the
lower cube commutes because $f$ is a homomorphism.  The back face of
the top cube commutes because $\xi^{\mathbf{A}}$ is a homomorphism by
Proposition~\ref{XialbHom}, and the back face of the lower cube
commutes by Proposition~\ref{AlgFujXiNat}.
\end{proof}

From this proposition it follows immediately the following

\begin{corollary}
The quadruple $\mathfrak{Tm}_{\mathfrak{pd}} =
(\mathbf{Sig}_{\mathfrak{pd}},\Alg_{\mathfrak{pd}},
\mathrm{Ter}_{\mathfrak{pd}},(\mathrm{Tr},\theta))$ is a
$2$-institution on the category $\mathbf{Set}$, the \emph{many-sorted
term $2$-institution of Fujiwara}, or, simply, the \emph{term
$2$-institution of Fujiwara}.
\end{corollary}

To round off the work we have made up to this point we state in the
following corollary the existence of an embedding from the
$2$-category of signatures, polyderivors, and transformations between
polyderivors, into the $2$-category of $\boldsymbol{\mathcal{V}}$-monads,
alg-morphisms, and alg-transformations between alg-morphisms.

\begin{corollary}
There exists an embedding from the $2$-category
$\mathbf{Sig}_{\mathfrak{pd}}$ into the $2$-category
$\mathbf{Mnd}_{\boldsymbol{\mathcal{V}},\mathrm{alg}}$.
\end{corollary}

\begin{proof}
The embedding sends
\begin{enumerate}
\item A signature $\mathbf{\Sigma}$ to the monad
      $(\mathbf{Set}^{S},\mathbb{T}_{\mathbf{\Sigma}})$, where, we
      recall, $\mathbb{T}_{\mathbf{\Sigma}} =
      (\mathrm{T}_{\mathbf{\Sigma}},\eta,\mu)$ is the standard monad
      derived from the adjunction $\mathbf{T}_{\mathbf{\Sigma}}\dashv
      \mathrm{G}_{\mathbf{\Sigma}}$ between the category
      $\mathbf{Alg}(\mathbf{\Sigma})$ and the category
      $\mathbf{Set}^{S}$, with $\mathrm{T}_{\mathbf{\Sigma}} =
      \mathrm{G}_{\mathbf{\Sigma}}\comp \mathbf{T}_{\mathbf{\Sigma}}$,

\item A polyderivor $\mathbf{d}$ from $\mathbf{\Sigma}$ to
      $\mathbf{\Lambda}$ to the alg-morphism
      $$
      \xymatrix@C=60pt@R=40pt{
      *++{\mathbf{Set}^{S}}\xyn{1} &
      *++{\mathbf{Set}^{S}}\xyn{2} \\
      *++{\mathbf{Set}^{T}}\xyn{3} &
      *++{\mathbf{Set}^{T}}\xyn{4}
      \ar@{ ->}"1";"2"^{\mathrm{T}_{\mathbf{\Sigma}}}
      \ar@{ ->}"3";"4"_{\mathrm{T}_{\mathbf{\Lambda}}}
      \ar@<+1.5ex>@{<- }"1";"3"^{\Delta_{\varphi}^{\nat}}
      \ar@{}"1";"3"|{\ladj}
      \ar@<-1.5ex>@{ ->}"1";"3"_{\coprod_{\varphi}^{\dagger}}
      \ar@<+1.5ex>@{<- }"2";"4"^{\Delta_{\varphi}^{\nat}}
      \ar@{}"2";"4"|{\ladj}
      \ar@<-1.5ex>@{ ->}"2";"4"_{\coprod_{\varphi}^{\dagger}}
      \ar @{} "1";"4"|{\lambda}
      }
      $$
      also denoted by $\mathbb{T}_{\mathbf{d}} =
      (\coprod_{\varphi}^{\dagger}\ladj\Delta_{\varphi}^{\nat},\lambda)$,
      from $(\mathbf{Set}^{S},\mathbb{T}_{\mathbf{\Sigma}})$ to
      $(\mathbf{Set}^{T},\mathbb{T}_{\mathbf{\Lambda}})$, where the
      component $\lambda_{1}$ of the matrix $\lambda$ at $X$ is the
      underlying $\mathrm{ms}$-mapping of $\ext{\bigl(
      \theta_{\varphi}^{\dagger\nat}
      (\eta_{\coprod_{\varphi}^{\dagger}X}) \bigl)}$, the canonical
      extension to $\mathbf{T}_{\mathbf{\Sigma}}(X)$ of the
      $\mathrm{ms}$-mapping
      $\Delta_{\varphi}^{\nat}(\eta_{\coprod_{\varphi}^{\dagger}X})\comp
      (\eta_{\varphi}^{\dagger\nat})_{X}$ from $X$ to
      $\Delta_{\varphi}^{\nat}(\mathrm{T}_{\mathbf{\Lambda}}
      (\coprod_{\varphi}^{\dagger}X))$, as stated in
      Proposition~\ref{polyderivor induces 1-cell}, and

\item A transformation $\xi$ from $\mathbf{d}$ to
      $\mathbf{d}'$ to the alg-transformation
      $$
      \xymatrix@C=60pt@R=40pt{
      *++{\mathbf{Set}^{S}}\xyn{1} &
      *++{\mathbf{Set}^{S}}\xyn{2} \\
      *++{\mathbf{Set}^{T}}\xyn{3} &
      *++{\mathbf{Set}^{T}}\xyn{4}
      \ar@{ ->}"1";"2"^{1}
      \ar@{ ->}"3";"4"_{\mathrm{T}_{\mathbf{\Lambda}}}
      \ar@<+1.5ex>@{<- }"1";"3"^{\Delta_{\varphi}^{\nat}}
      \ar@{}"1";"3"|{\ladj}
      \ar@<-1.5ex>@{ ->}"1";"3"_{\coprod_{\varphi}^{\dagger}}
      \ar@<+1.5ex>@{<- }"2";"4"^{\Delta_{\varphi'}^{\nat}}
      \ar@{}"2";"4"|{\ladj}
      \ar@<-1.5ex>@{ ->}"2";"4"_{\coprod_{\varphi'}^{\dagger}}
      \ar @{} "1";"4"|{\xi}
      }
      $$
      also denoted by $\mathbb{T}_{\xi}$, from
      $\mathbb{T}_{\mathbf{d}}$ to $\mathbb{T}_{\mathbf{d}'}$, where
      the component $\xi_{0}$ of the matrix $\xi$ at $X$ is the
      $\mathrm{ms}$-mapping $\xi_{X}$, as stated in
      Proposition~\ref{TerFujXiNat}.
      \qedhere
\end{enumerate}
\end{proof}

From this embedding and taking into account the work by Street
in~\cite{st72}, it follows that the polyderivors and transformations
between polyderivors are a concrete foundation for a bidimensional
many-sorted general algebra.


\section{Equivalence of the specifications of Hall and Bénabou.}

In this section we define a $2$-category of specifications,
$\mathbf{Spf}_{\mathfrak{pd}}$, with objects the specifications,
morphisms from a specification into a like one the polyderivors between the
underlying signatures of the specifications that are compatible with
the equations, and $2$-cells from a morphism into a like one a
convenient class of transforma\-tions between the polyderivors.  In such a
2-category we prove, for every set of sorts $S$, the equivalence of
the specifications of Hall and Bénabou for $S$, from which, through
the pseudo-functor $\Alg_{\mathfrak{pd}}$, the equivalence between the
corresponding categories of algebras, $\mathbf{Alg}(\mathrm{H}_{S})$
and $\mathbf{Alg}(\mathrm{B}_{S})$, is obtained as an easy corollary.

For a polyderivor $\mathbf{d}\colon\Sigma\mor\Lambda$, the functor
$\mathbf{d}^{\mathfrak{pd}}_{\diamond}$ of translation from
$\mathbf{Ter}(\mathbf{\Sigma})$ to $\mathbf{Ter}(\mathbf{\Lambda})$
enables us to define the concept of $\mathfrak{pd}$-specification
morphism from a specification into a like one.

\begin{definition}
Let $(\mathbf{\Sigma},\ec{E})$ and $(\mathbf{\Lambda},\ec{H})$ be
specifications.  An $\mathfrak{pd}$-\emph{specification morphism from}
$(\mathbf{\Sigma},\ec{E})$ \emph{to} $(\mathbf{\Lambda},\ec{H})$ is a
polyderivor $\mathbf{d}\colon\Sigma\mor\Lambda$ such that
$(\mathbf{d}^{\mathfrak{pd}}_{\diamond})^{2}[\ec{E}]\incl
\Cn_{\mathbf{\Lambda}}(\ec{H})$.  We denote by
$\mathbf{Spf}_{\mathfrak{pd}}$ the corresponding category.
\end{definition}

Given two $\mathfrak{pd}$-specification morphisms $\mathbf{d}$ and
$\mathbf{e}$ from $(\mathbf{\Sigma},\ec{E})$ to
$(\mathbf{\Lambda},\ec{H})$, since $\mathbf{d}$ and $\mathbf{e}$ are,
in particular, polyderivors from $\mathbf{\Sigma}$ to $\mathbf{\Lambda}$,
we have, in principle, at our disposal all the transformations
$\xi\colon \mathbf{d}\mdf\mathbf{e}$ from $\mathbf{d}$ to $\mathbf{e}$
as potential candidates for a concept of transformation between these
$\mathfrak{pd}$-specification morphisms.

However, the condition of commutativity for the transformations
between polyderivors is too much strict, because it requires, for every
formal operation $\sigma\colon w\mor s$ in $\Sigma_{w,s}$, the strict
equality
$$
\xi_{s}\comp d(\sigma) = e(\sigma)\comp \xi_{w},
$$
and, actually, what could happen (and probably the most one reasonably
can hope for), as was pointed out by Fujiwara in~\cite{tF60}, is that,
under the presence of equations, such a type of equation holds only
modulus the congruence generated by the equations in the target
specification.  Therefore, for the $\mathfrak{pd}$-specification
morphisms, the notion of transformation that we adopt, following the
example of Fujiwara in~\cite{tF60}, is that one where the strict
equality between terms is replaced by the equality between them but
relative to the congruence generated by the equations in the target
specification.  These transformations, in its turn, allow us to
endow the category $\mathbf{Spf}_{\mathfrak{pd}}$ of a structure of
2-category.

\begin{definition}
Let $\mathbf{d}$ and
$\mathbf{e}\colon(\mathbf{\Sigma},\ec{E})\mor(\mathbf{\Lambda},\ec{H})$
be $\mathfrak{pd}$-speci\-fication morphisms.  A \emph{transformation from}
$\mathbf{d}$ \emph{to} $\mathbf{e}$ is a choice function $\xi$ for
$(\mathrm{BTer}_{T}(\Lambda)_{\varphi(s),\psi(s)})_{s\in S}$, such
that, for every formal operation $\sigma\colon w\mor s$, we have that
$$
\xi_{s}\comp d(\sigma)\equiv_{\cl{\ec{H}}}e(\sigma)\comp \xi_{w}.
$$
\end{definition}

\begin{proposition}
The specifications, the $\mathfrak{pd}$-specification morphisms, and the
transformations between $\mathfrak{pd}$-specification morphisms determine a
$2$-category $\mathbf{Spf}_{\mathfrak{pd}}$.
\end{proposition}


The pseudo-functor $\Alg_{\mathfrak{pd}}$ from
$\mathbf{Sig}_{\mathfrak{pd}}$ to $\mathbf{Cat}$ can be lifted up
to the $2$-category $\mathbf{Spf}_{\mathfrak{pd}}$, by assigning,
in particular, to a specification $(\mathbf{\Sigma},\ec{E})$ the
category $\mathbf{Alg}(\mathbf{\Sigma},\ec{E})$ of its models.

\begin{proposition}
There exists a pseudo-functor $\Alg^{\mathrm{sp}}_{\mathfrak{pd}}$ from
$\mathbf{Spf}_{\mathfrak{pd}}$ to $\mathbf{Cat}$ defined as follows
\begin{enumerate}
\item $\Alg^{\mathrm{sp}}_{\mathfrak{pd}}$ sends a specification
      $(\mathbf{\Sigma},\ec{E})$ to the category
      $\mathrm{Alg}^{\mathrm{sp}}_{\mathfrak{pd}}(\mathbf{\Sigma},\ec{E})
      = \mathbf{Alg}(\mathbf{\Sigma},\ec{E})$ of its models, i.e., the
      full subcategory of $\mathbf{Alg}(\mathbf{\Sigma})$ determined
      by those $\mathbf{\Sigma}$\nobreakdash-algebras that satisfy all the
      equations in $\ec{E}$.

\item $\Alg^{\mathrm{sp}}_{\mathfrak{pd}}$ sends a
      $\mathfrak{pd}$-specification morphism $\mathbf{d}$ from
      $(\mathbf{\Sigma},\ec{E})$ to $(\mathbf{\Lambda},\ec{H})$ to the
      functor $\Alg^{\mathrm{sp}}_{\mathfrak{pd}}(\mathbf{d}) =
      \mathbf{d}^{\ast}_{\mathfrak{pd}}$ from
      $\mathbf{Alg}(\mathbf{\Lambda},\ec{H})$ to
      $\mathbf{Alg}(\mathbf{\Sigma},\ec{E})$, obtained from the
      functor $\mathbf{d}^{\ast}_{\mathfrak{pd}}$ from
      $\mathbf{Alg}(\mathbf{\Lambda})$ to
      $\mathbf{Alg}(\mathbf{\Sigma})$ by bi-restriction.

\item $\Alg^{\mathrm{sp}}_{\mathfrak{pd}}$ sends a transformation
      $\xi\colon \mathbf{d}\mdf\mathbf{e}$ from $\mathbf{d}$ to
      $\mathbf{e}$ to the natural transformation
      $\mathrm{Alg}_{\mathfrak{pd}}(\xi)$ from
      $\mathbf{d}^{\ast}_{\mathfrak{pd}}$ to
      $\mathbf{e}^{\ast}_{\mathfrak{pd}}$.
\end{enumerate}

\end{proposition}

The pseudo-functor $\mathrm{Ter}_{\mathfrak{pd}}$ from
$\mathbf{Sig}_{\mathfrak{pd}}$ to $\mathbf{Cat}$ can also be lifted up to the
$2$-category $\mathbf{Spf}_{\mathfrak{pd}}$.

\begin{proposition}
There exists a pseudo-functor $\mathrm{Ter}^{\mathrm{sp}}_{\mathfrak{pd}}$ from
$\mathbf{Spf}_{\mathfrak{pd}}$ to $\mathbf{Cat}$ defined as follows
\begin{enumerate}
\item $\mathrm{Ter}^{\mathrm{sp}}_{\mathfrak{pd}}$ sends a
      specification $(\mathbf{\Sigma},\ec{E})$ to the
      category $\Ter^{\mathrm{sp}}_{\mathfrak{pd}}(\mathbf{\Sigma},\ec{E}) =
      \mathbf{Ter}(\mathbf{\Sigma},\ec{E})$,
      where $\mathbf{Ter}(\mathbf{\Sigma},\ec{E})$
      is the quotient category
      $\mathbf{Ter}(\mathbf{\Sigma})/\cl{\ec{E}}$.

\item $\mathrm{Ter}^{\mathrm{sp}}_{\mathfrak{pd}}$ sends a
      $\mathfrak{pd}$-specification morphism
      $\mathbf{d}$ from $(\mathbf{\Sigma},\ec{E})$ to
      $(\mathbf{\Lambda},\ec{H})$ to the functor
      $\Ter^{\mathrm{sp}}_{\mathfrak{pd}}(\mathbf{d})$ from the quotient category
      $\mathbf{Ter}(\mathbf{\Sigma},\ec{E}) =
      \mathbf{Ter}(\mathbf{\Sigma})/\cl{\ec{E}}$ to the quotient category
      $\mathbf{Ter}(\mathbf{\Lambda},\ec{H}) =
      \mathbf{Ter}(\mathbf{\Lambda})/\cl{\ec{H}}$, which assigns to
      a morphism
      $[P]_{\cl{\ec{E}}}$ from  $X$ to $Y$ in
      $\mathbf{Ter}(\mathbf{\Sigma},\ec{E})$ the morphism
      $$
      \mathrm{Ter}^{\mathrm{sp}}_{\mathfrak{pd}}(\mathbf{d})([P]_{\cl{\ec{E}}}) =
      [\mathbf{d}_{\diamond}^{\mathfrak{pd}}(P)]_{\cl{\ec{H}}}\colon
      \tcoprod_{\varphi}^{\dagger}X\mor
      \tcoprod_{\varphi}^{\dagger}Y
      $$
      in $\mathbf{Ter}(\mathbf{\Lambda},\ec{H})$.

\item $\mathrm{Ter}^{\mathrm{sp}}_{\mathfrak{pd}}$ sends a transformation
      $\xi\colon \mathbf{d}\mdf\mathbf{e}$ from $\mathbf{d}$ to
      $\mathbf{e}$ to the natural transformation
      $\mathrm{Ter}_{\mathfrak{pd}}(\xi)$ from
      $\Ter^{\mathrm{sp}}_{\mathfrak{pd}}(\mathbf{d})$ to
      $\Ter^{\mathrm{sp}}_{\mathfrak{pd}}(\mathbf{e})$.
\end{enumerate}

\end{proposition}

Furthermore, from the $2$-category $\mathbf{Spf}^{\opp}_{\mathfrak{pd}}\bprod
\mathbf{Spf}_{\mathfrak{pd}}$ to the $2$-category $\mathbf{Cat}$ there exists a
pseudo-functor $\Alg^{\mathrm{sp}}_{\mathfrak{pd}}(\farg)\times
\mathrm{Ter}^{\mathrm{sp}}_{\mathfrak{pd}}(\farg)$ and a
pseudo-extranatural transformation
$(\mathrm{Tr}^{\mathrm{sp}},\theta^{\mathrm{sp}})$ from
$\Alg^{\mathrm{sp}}_{\mathfrak{pd}}(\farg)\bprod
\mathrm{Ter}^{\mathrm{sp}}_{\mathfrak{pd}}(\farg)$ to the functor
constantly $\mathbf{Set}$, and from this we get the following

\begin{corollary}
The quadruple $\mathfrak{Spf}_{\mathfrak{pd}} =
(\mathbf{Spf}_{\mathfrak{pd}},\Alg^{\mathrm{sp}}_{\mathfrak{pd}},
\mathrm{Ter}^{\mathrm{sp}}_{\mathfrak{pd}},
(\mathrm{Tr}^{\mathrm{sp}},\theta^{\mathrm{sp}}))$ is a
$2$-institu\-tion on the category $\mathbf{Set}$, the so-called
\emph{many-sorted specification $2$-institution of Fujiwara}, or,
simply, the \emph{specification $2$-institution of Fujiwara}.
\end{corollary}

From the pseudo-functor functor
$\Alg^{\mathrm{sp}}_{\mathfrak{pd}}$, from
$\mathbf{Spf}_{\mathfrak{pd}}$ to $\mathbf{Cat}$, to the
pseudo-functor $\Alg_{\mathfrak{pd}}\comp
\mathrm{sig}^{\mathrm{op}}$, between the same $2$-categories,
there exists a pseudo-natural transformation,
$\mathrm{In}_{\mathfrak{pd}}$, which sends a specification
$(\mathbf{\Sigma},\ec{E})$ to the full embedding
$\mathrm{In}_{(\mathbf{\Sigma},\ec{E})}$ of
$\mathbf{Alg}(\mathbf{\Sigma},\ec{E})$ into
$\mathbf{Alg}(\mathbf{\Sigma})$.  Besides, from the pseudo-functor
$\mathrm{Ter}_{\mathfrak{pd}}\comp \mathrm{sig}$, from
$\mathbf{Spf}_{\mathfrak{pd}}$ to $\mathbf{Cat}$, to the
pseudo-functor $\mathrm{Ter}^{\mathrm{sp}}_{\mathfrak{pd}}$,
between the same $2$-categories, there exists a (strict)
pseudo-natural transformation, $\mathrm{Pr}_{\mathfrak{pd}}$, given
by the following data
\begin{enumerate}
\item For each specification
      $(\mathbf{\Sigma},\ec{E})$, the projection functor
      $\mathrm{Pr}_{\cl{\ec{E}}}$ from
      $\mathbf{Ter}(\mathbf{\Sigma})$ to the quotient category
      $\mathbf{Ter}(\mathbf{\Sigma})/\cl{\ec{E}}$.

\item For each specification morphism $\mathbf{d}$
      from $(\mathbf{\Sigma},\ec{E})$ to $(\mathbf{\Lambda},\ec{H})$, the
      isomorphic identity natural transformation, denoted in this case by
      $\mathrm{Pr}_{\mathbf{d}}$, from the functor
      $\mathrm{Pr}_{\cl{\ec{H}}}\comp (\mathrm{Ter}\comp
      \mathrm{sig})(\mathbf{d})$, from the category
      $\mathbf{Ter}(\mathbf{\Sigma})$ to the category
      $\mathbf{Ter}(\mathbf{\Lambda})/\cl{\ec{H}}$, to the functor
      $\mathrm{Ter}^{\mathrm{sp}}(\mathbf{d})\comp
      \mathrm{Pr}_{\cl{\ec{E}}}$, between the same categories.
\end{enumerate}

Therefore we have obtained the following

\begin{corollary}
The pair
$(\mathrm{sig},(\mathrm{In}_{\mathfrak{pd}},\mathrm{Pr}_{\mathfrak{pd}}))$
is a morphism of $2$-institutions from the
$2$-institution $\mathfrak{Spf}_{\mathfrak{pd}}$ to the $2$-institution
$\mathfrak{Tm}_{\mathfrak{pd}}$.
\end{corollary}


Our next goal is to prove that the specifications of Bénabou and Hall
are equivalent in the $2$-category $\mathbf{Spf}_{\mathfrak{pd}}$.

\begin{proposition}
The specifications $\Ben_{S} =
(\mathbf{\Sigma}^{\Ben_{S}},\ec{E}^{\Ben_{S}})$, of Bénabou for $S$,
and $\Hall_{S} = (\mathbf{\Sigma}^{\Hall_{S}},\ec{E}^{\Hall_{S}})$, of
Hall for $S$, are equivalent in the 2-category
$\mathbf{Spf}_{\mathfrak{pd}}$.
\end{proposition}

\begin{proof}
Let $S$ be a set of sorts.  From the signature
$\mathbf{\Sigma}^{\Ben_{S}}$ to the signature
$\mathbf{\Sigma}^{\Hall_{S}}$, we have the polyderivor $\mathbf{d} = (\varphi,d)$,
where $\varphi$ is the mapping
$$
\function
{\fmon{S}\bprod\fmon{S}}
{\fmon{(\fmon{S}\bprod S)}}
{(u,v)}
{((u,v_{0}),\ldots,(u,v_{\bb{v}-1}))}
$$
while $d\colon \Sigma^{\Ben_{S}}\mor \mathrm{BTer}_{\fmon{S}\bprod
S}(\Sigma^{\Hall_{S}})_{\ext{\varphi}\bprod\varphi}$ is defined as
\begin{enumerate}
\item  For every $w\in \fmon{S}$, and $i\in\bb{w}$,
       $$
       d(\pi^{w}_{i})=(\pi^{w}_{i}).
       $$

\item  For every $u$, $w\in\fmon{S}$,
       $$
       d(\tp{\,}_{u,w})=(v^{u,w_{0}}_{0},\ldots,v^{u,w_{\bb{w}-1}}_{\bb{w}-1}).
       $$

\item  For every $u$, $v$, $w \in \fmon{S}$,
\begin{multline*}
    d(\comp_{u,v,w})=(\xi_{u,v,w_{0}}(v^{u,w_{0}}_{\bb{v}},v^{u,v_{0}}_{0},
    \ldots,v^{u,v_{\bb{v}-1}}_{\bb{v}-1}),\ldots, \\
    \xi_{u,v,w_{\bb{w}-1}}(v^{u,w_{\bb{w}-1}}_{\bb{v}+\bb{w}-1},
    v^{u,v_{0}}_{0},\ldots,v^{u,v_{\bb{v}-1}}_{\bb{v}-1})).
\end{multline*}
\end{enumerate}
Now we prove that the definition of $\mathbf{d}$ is sound. For the  operations
$\pi^{w}_{i}\in\Sigma^{\Ben_{S}}_{\lambda,(w,(w_{i}))}$, we have that
\begin{align*}
        d(\pi^{w}_{i}) &\in
    \mathrm{BTer}_{\fmon{S}\bprod S}
    (\Sigma^{\Hall_{S}})_{\ext{\varphi}(\lambda),\varphi(w,(w_{i}))} \\
    &=
    \mathrm{BTer}_{\fmon{S}\bprod S}
    (\Sigma^{\Hall_{S}})_{\lambda,((w,(w_{i})))} \\
    &=
    \mathrm{T}_{\Sigma^{\Hall_{S}}}(\vs{\lambda})_{((w,(w_{i})))},
\end{align*}
because $d(\pi^{w}_{i})$ is a word of length 1 that has as its unique
component an operation of coarity $(w,(w_{i}))$.

For the  operations $\tp{\,}_{u,w}\in\Sigma^{\Ben_{S}}
_{((u,(w_{0})),\ldots,(u,(w_{\bb{w}-1}))),(u,w)}$, we have that
\begin{align*}
  d(\tp{\,}_{u,w}) &\in
  \mathrm{BTer}_{\fmon{S}\bprod S}(\Sigma^{\Hall_{S}})
  _{\ext{\varphi}((u,w_{0}),\ldots,(u,(w_{\bb{w}-1}))),
  \varphi(u,w)} \\
&=
  \mathrm{BTer}_{\fmon{S}\bprod S}(\Sigma^{\Hall_{S}})
  _{((u,w_{0}),\ldots,(u,w_{\bb{w}-1})),
  ((u,w_{0}),\ldots,(u,w_{\bb{w}-1}))} \\
&=
  \mathrm{T}_{\Sigma^{\Hall_{S}}}(\vs{((u,w_{0}),\ldots,(u,w_{\bb{w}-1}))})
  _{((u,w_{0}),\ldots,(u,w_{\bb{w}-1}))},
\end{align*}
because $d(\tp{\,}_{u,w})$ is a word of length $\bb{w}$ that, for
every $i\in\bb{w}$, has as $i$-th component a term of coarity
$(u,(w_{i}))$.  For the operations
$\comp_{u,v,w}\in\Sigma^{\Ben_{S}}_{((u,v),(v,w)),(u,w)}$, we have
that
\begin{align*}
  d(\comp_{u,v,w}) &\in
  \mathrm{BTer}_{\fmon{S}\bprod S}(\Sigma^{\Hall_{S}})
  _{\ext{\varphi}((u,v),(v,w)),
  \varphi(u,w)} \\
&=
  \mathrm{BTer}_{\fmon{S}\bprod S}(\Sigma^{\Hall_{S}})
  _{((u,v_{0}),\ldots,(u,v_{\bb{v}-1}),
  (v,w_{0}),\ldots,(v,w_{\bb{w}-1})),
  ((u,w_{i})\mid i\in\bb{w})} \\
&=
  \mathrm{T}_{\Sigma^{\Hall_{S}}}(\vs{(((u,v_{j})\mid j\in \bb{v}),
  ((v,w_{i})\mid i\in\bb{w}))}
  _{((u,w_{i})\mid i\in\bb{w})},
\end{align*}
because $d(\comp_{u,v,w})$ is a word of length $\bb{w}$ that, for
every $i\in\bb{w}$, has as $i$-th component a term of coarity
$(u,w_{i})$.

From the signature $\mathbf{\Sigma}^{\Hall_{S}}$ to the
signature $\mathbf{\Sigma}^{\Ben_{S}}$ we have the
polyderivor $\mathbf{e} = (\psi,e)$, where $\psi$ is the mapping
$$
\function
{\fmon{S}\bprod S}
{\fmon{(\fmon{S}\bprod \fmon{S})}}
{(w,s)}
{((w,(s)))}
$$
while $e\colon\Sigma^{\Hall_{S}}\mor
\mathrm{BTer}_{\fmon{S}\bprod S}(\Sigma^{\Ben_{S}})_{\ext{\psi}\bprod\psi}$
is defined as
\begin{enumerate}
\item  For every $w\in \fmon{S}$ and $i\in\bb{w}$,
       $$e(\pi^{w}_{i})=(\pi^{w}_{i}).$$

\item  For every $u$, $w\in\fmon{S}$ and $s\in S$,
       $$e(\xi_{u,w,s})=(v^{w,s}_{0}\comp
       \langle v^{u,w_{0}}_{1},\ldots,v^{u,w_{\bb{w}-1}}_{\bb{w}}\rangle).$$
\end{enumerate}
The polyderivors $\mathbf{d}$ and $\mathbf{e}$ are, obviously, compatible with the
respective equations, hence are $\mathfrak{pd}$-specification morphisms.

Finally we should prove that there are invertible transformations
between the identity at $\mathbf{\Sigma}^{\Ben_{S}}$ and the polyderivor
$\mathbf{e}\comp \mathbf{d}$, as well as between the identity at
$\mathbf{\Sigma}^{\Hall_{S}}$ and the polyderivor $\mathbf{d}\comp
\mathbf{e}$.  Since both proofs are analogous, we restrict ourselves to
present only the first one.

From the identity at $\mathbf{\Sigma}^{\Ben_{S}}$ into $\mathbf{e}\comp
\mathbf{d}$ we have the transformation $\chi$, defined, for every
$(u,w)\in\fmon{S}\bprod\fmon{S}$, as the term
$$
\chi_{(u,w)}=
(\pi^{w}_{0}\comp v_{0},\ldots,\pi^{w}_{\bb{w}-1}\comp v_{0})
\in \mathrm{T}_{\mathbf{\Sigma}^{\Ben_{S}}}
(\vs{((u,w))})_{((u,(w_{0})),\ldots,(u,(w_{\bb{w}-1})))},%
$$
and from $\mathbf{e}\comp \mathbf{d}$ into the identity at
$\mathbf{\Sigma}^{\Ben_{S}}$ we have the transformation $\rho$, defined,
for every $(u,w)\in\fmon{S}\bprod\fmon{S}$, as the term
$$
\rho_{(u,w)}=
\tp{v_{0},\ldots,v_{\bb{w}-1}}_{u,w}
\in \mathrm{T}_{\mathbf{\Sigma}^{\Ben_{S}}}
(\vs{((u,(w_{0})),\ldots,(u,(w_{\bb{w}-1})))})_{((u,w))}.%
$$
Then  $\rho_{(u,w)}\comp \chi_{(u,w)}$ is the  term
$$
\tp{\pi^{w}_{0}\comp v_{0},\ldots,\pi^{w}_{\bb{w}-1}\comp
v_{0}}_{u,w} = v_{0},
$$
and $\chi_{(u,w)}\comp \rho_{(u,w)}$ is the  term
$$
(\pi^{w}_{0}\comp \tp{v_{0},\ldots,v_{\bb{w}-1}}_{u,w},
\ldots,
\pi^{w}_{\bb{w}-1}\comp \tp{v_{0},\ldots,v_{\bb{w}-1}}_{u,w})
= (v_{0},\ldots,v_{\bb{w}-1}),
$$
hence $\xi$ and $\rho$ are inverses.
\end{proof}

\begin{corollary}
For every set of sorts $S$, the category
$\mathbf{Alg}(\mathrm{H}_{S})$, of Hall algebras for $S$, is equivalent
to the category $\mathbf{Alg}(\mathrm{B}_{S})$, of Bénabou algebras
for $S$.
\end{corollary}

\begin{proof}
It follows from the existence of the pseudo-functor
$\Alg_{\mathfrak{pd}}^{\mathrm{sp}}$ from the $2$-category
$\mathbf{Spf}_{\mathfrak{pd}}$ to the $2$-category $\mathbf{Cat}$ and
from the fact that the specifications
$(\mathbf{\Sigma}^{\Ben_{S}},\ec{E}^{\Ben_{S}})$ and
$(\mathbf{\Sigma}^{\Hall_{S}},\ec{E}^{\Hall_{S}})$ are equivalent in
the 2-category $\mathbf{Spf}_{\mathfrak{pd}}$. We summarize these facts
by means of the following picture:
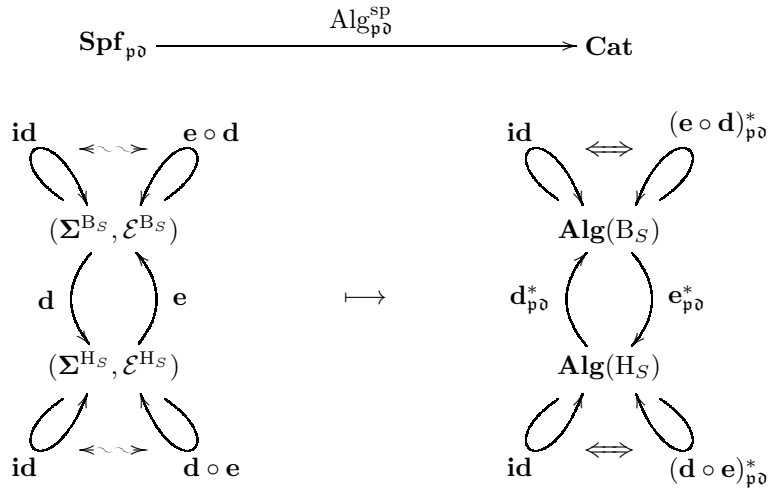
\newdir{<~~>}{!/21pt/\dir{<}*!/21pt/\dir{-}!/12pt/\dir{~}!/5pt/\dir{~}*!/1pt/\dir{-}*!/-3pt/\dir{>}}
\newdir{<==>}{!\dir2{<}!/-2pt/\dir{=}!/-6pt/\dir{=}%
  !/-10pt/\dir{=}*!/-17pt/\dir2{>}}
$$
\xymatrixcolsep={15ex} \xymatrixrowsep={8ex} \xymatrix{
(\mathbf{\Sigma}^{\Ben_{S}},\ec{E}^{\Ben_{S}})
  \xymn[0pt,70pt]{\mathbf{Spf}_{\mathfrak{pd}}}{1}
  \xymn[-16pt,+10pt]{}{o1}
  \xymn[-09pt,+7pt]{}{d1}
  \xymn[+16pt,+10pt]{}{o2}
  \xymn[+09pt,+7pt]{}{d2}
  \ar @`{{**{} ?+<-35pt,20pt>,?+<-10pt,40pt>}}
     "o1";"d1"^*{\mathbf{id}}|*{}="od1"
  \ar @`{{**{} ?+<+35pt,20pt>,?+<+10pt,40pt>}}
     "o2";"d2"_*{\mathbf{e}\comp\mathbf{d}}|*{}="od2"
  \ar @{} "od1";"od2"^*+/-3pt/{}|(.50){\dir{<~~>}}
  \ar @/_16pt/ [d]_*+{\mathbf{d}}="d"
  &&
\mathbf{Alg}(\mathrm{B}_{S})
  \xymn[0pt,70pt]{\mathbf{Cat}}{2}
  \ar "1";"2"^{\Alg_{\mathfrak{pd}}^{\mathrm{sp}}}
  \xymn[-16pt,+10pt]{}{ro1}
  \xymn[-09pt,+7pt]{}{rd1}
  \xymn[+16pt,+10pt]{}{ro2}
  \xymn[+09pt,+7pt]{}{rd2}
  \ar @`{{**{} ?+<-35pt,20pt>,?+<-10pt,40pt>}}
     "ro1";"rd1"^*{\mathbf{id}}|*{}="rod1"
  \ar @`{{**{} ?+<+35pt,20pt>,?+<+10pt,40pt>}}
     "ro2";"rd2"_*{(\mathbf{e}\comp\mathbf{d})^{\ast}_{\mathfrak{pd}}}|*{}="rod2"
  \ar @{} "rod1";"rod2"^*+/-3pt/{}|(.50){\dir{<==>}}
  \ar @/^16pt/ [d]^*+{\mathbf{e}^{\ast}_{\mathfrak{pd}}}
  \\
(\mathbf{\Sigma}^{\Hall_{S}},\ec{E}^{\Hall_{S}})
  \xymn[-16pt,-10pt]{}{o3}
  \xymn[-09pt,-7pt]{}{d3}
  \xymn[+16pt,-10pt]{}{o4}
  \xymn[+09pt,-7pt]{}{d4}
  \ar @`{{**{} ?+<-35pt,-20pt>,?+<-10pt,-40pt>}}
     "o3";"d3"_*{\mathbf{id}}|*{}="od3"
  \ar @`{{**{} ?+<+35pt,-20pt>,?+<+10pt,-40pt>}}
     "o4";"d4"^*{\mathbf{d}\comp\mathbf{e}}|*{}="od4"
  \ar @{} "od3";"od4"_*+/+3pt/{}|(.50){\dir{<~~>}}
  \ar @/_16pt/ [u]_*+{\mathbf{e}}="e"
  &&
\mathbf{Alg}(\mathrm{H}_{S})
  \xymn[-16pt,-10pt]{}{ro3}
  \xymn[-09pt,-7pt]{}{rd3}
  \xymn[+16pt,-10pt]{}{ro4}
  \xymn[+09pt,-7pt]{}{rd4}
  \ar @`{{**{} ?+<-35pt,-20pt>,?+<-10pt,-40pt>}}
     "ro3";"rd3"_*{\mathbf{id}}|*{}="rod3"
  \ar @`{{**{} ?+<+35pt,-20pt>,?+<+10pt,-40pt>}}
     "ro4";"rd4"^*{(\mathbf{d}\comp\mathbf{e})^{\ast}_{\mathfrak{pd}}}|*{}="rod4"
  \ar @{} "rod3";"rod4"_*+/+3pt/{}|(.50){\dir{<==>}}
  \ar @/^16pt/ [u]^*+{\mathbf{d}^{\ast}_{\mathfrak{pd}}}="df"
\ar@{}"e";"df"|(.53){\lmapsto} }
$$
\end{proof}



\begin{thebibliography}{99}
\bibitem
{b74}
J. Barwise,
\emph{Axioms for abstract model theory},
Annals of Mathematical Logic,
\textbf{7} (1974),
pp. 221--265.


\bibitem
{jB68}
J. B{\'{e}}nabou,
\emph{Structures algebriques dans les categories},
Cahiers de topologie et g\'{e}ometrie diff\'{e}rentielle,
\textbf{10} (1968),
pp. 1--126.


\bibitem
{bf48}
G. Birkhoff and O. Frink,
\emph{Representation of lattices by sets},
Trans. {A}mer. {M}ath. {S}oc.,
\textbf{64} (1948),
pp. 299--316.

\bibitem
{bl70}
G. Birkhoff and J. D. Lipson,
\emph{Heterogeneous algebras},
J. {C}ombinatorial {T}heory,
\textbf{8} (1970),
pp. 115--133.

\bibitem
{fB94}
F. Borceux,
\emph{Handbook of categorical algebra 1. Basic category
                theory},
Cambridge University Press,
1994.

\bibitem
{fB94a}
F. Borceux,
\emph{Handbook of categorical algebra 2. Categories and
                structures},
Cambridge University Press,
1994.

\bibitem
{fB94b}
F. Borceux,
\emph{Handbook of categorical algebra 3. Categories of
                sheaves},
Cambridge University Press,
1994.

\bibitem
{nB70}
N. Bourbaki,
\emph{Th{\'{e}}orie des ensembles},
Hermann,
1970.

\bibitem
{djb69}
D. J. Brown,
\emph{Abstract logics},
Ph. D. Thesis. Stevens Institute of Technology,
New Jersey,
1969.

\bibitem
{bs73}
D. J. Brown and R. Suszko,
\emph{Abstract logics},
Dissertationes Math. (Rozprawy Matematyczne),
\textbf{102} (1973),
pp. 9--42.

\bibitem
{bs81}
S. Burris and H. P. Sankappanavar,
\emph{A course in universal algebra},
Springer-Verlag,
1981.

\bibitem
{cs04}
J. Climent and J. Soliveres,
\emph{On many-sorted algebraic closure operators},
Mathematische Nachrichten,
\textbf{266} (2004),
pp. 81--84.

\bibitem
{cs05} J. Climent and J. Soliveres, \emph{The completeness theorem
for monads in categories of sorted sets}, Houston Journal of
Mathematics, \textbf{31} (2005), pp. 103--129.


\bibitem
{pC81}
P. Cohn,
\emph{Universal algebra},
D. Reidel Publishing Company,
1981.

\bibitem
{dw00}
K. Denecke and S. Wismath,
\emph{Hyperidentities and clones},
Gordon and Breach Science Publishers,
2000.

\bibitem
{Ehr65}
Ch. Ehresmann.
\emph{Cat\'{e}gories et structures},
Dunod,
1965.

\bibitem
{es52}
S. Eilenberg and N. Steenrod,
\emph{Foundations of algebraic topology},
Princeton, New Jersey, Princeton University Press,
1952.

\bibitem
{hE77}
H. Enderton,
\emph{Elements of set theory},
Academic Press,
1977.

\bibitem
{fd01}
H. A. Feitosa and I. M. L. D'Ottaviano,
\emph{Conservative translations},
Annals of Pure and Applied Logic,
\textbf{108} (2001),
pp. 205--227.

\bibitem
{f74}
S. Feferman,
\emph{Two notes on abstract model theory I. Properties invariant on
the range of definable relations between structures},
Fundamenta Mathematicae,
\textbf{82} (1974),
pp. 153--164.

\bibitem
{tF59}
T. Fujiwara,
\emph{On mappings between algebraic systems},
Osaka Math. J.,
\textbf{11} (1959),
pp. 153--172.

\bibitem
{tF60}
T. Fujiwara,
\emph{On mappings between algebraic systems, {II}},
Osaka Math. J.,
\textbf{12} (1960),
pp. 253--268.


\bibitem
{kG33}
K. G\"{o}del, \emph{Eine interpretation des intuitionistischen
{A}ussagenkalk\"{u}ls},
Ergebnisse eines mathematischen {K}olloquiums,
\textbf{4} (1933),
pp.  39--40.

\bibitem
{gb84}
J. Goguen and R. Burstall,
\emph{Introducing institutions}.
In E. Clarke, ed. \emph{Proc. Logics of Programming Workshop},
Springer-Verlag,
1984,
pp. 221--256.

\bibitem
{gb86}
J. Goguen and R. Burstall,
\emph{A study in the foundations of programming methodology:
                Specifications, institutions, charters and parchments}.
In D. Pitt et alii. eds. \emph{Proc. Summer Workshop on Category Theory
          and Computer Programming},
Springer-Verlag,
1986,
pp. 313--333.


\bibitem
{gm85}
J. Goguen and J. Meseguer,
\emph{Completeness of many-sorted equational logic},
Houston Journal of Mathematics,
\textbf{11}(1985),
pp. 307--334.

\bibitem
{gtw76}
J. Goguen, J. Thatcher and E. Wagner,
\emph{An initial algebra approach to the specification, correctness,
and implementation of abstract data types},
IBM Thomas J. Watson Research Center,
Tecnical Report RC 6487,
October 1976.


\bibitem
{gG79}
G. Gr\"{a}tzer,
\emph{Universal algebra, 2nd ed.},
Springer-Verlag,
1979.

\bibitem
{jwG66}
J. W. Gray,
\emph{Fibred and cofibred categories.}
In S. Eilenberg et alii. eds. \emph{Proc. Conf. Categorical Algebra},
Springer-Verlag,
1966,
pp. 21--83.

\bibitem
{jwG74}
J. W. Gray.
\emph{Formal category theory: adjointness for 2-categories},
Springer-Verlag,
1974.

\bibitem
{Gro71}
A. Grothendieck.
\emph{Cat\'{e}gories fibr\'{e}es et descente
(Expos\'{e} VI)}.  In A. Grothendieck, ed., Rev\^{e}tements \'{e}tales
et groupe fondamental (SGA 1),
Springer-Verlag,
1971,
pp.  145--194.

\bibitem
{jh30}
J. Herbrand,
\emph{Recherches sur la th\'{e}orie de la d\'{e}monstration},
Thesis at the University of Paris,
1930.

\bibitem
{aH30}
A. Heyting,
\emph{Die formalen Regeln der intuitionistischen Logik},
Sitzungsberichte der Preussischen Akademie der
Wissenschaften. Physikalisch-Mathematische Klasse,
(1930),
pp. 42--56.



\bibitem
{h63}
P.J. Higgins,
\emph{Algebras with a scheme of operators},
Mathematische Nachrichten,
\textbf{27} (1963),
pp. 115--132.

\bibitem
{hn52}
G. Higman and B.H. Neumann,
\emph{Groups as groupoids with one law},
Pub. Math. Debrecen,
\textbf{2} (1952),
pp. 215--221.

\bibitem
{nJ68} N. Jacobson,
\emph{Structure of rings},
Amer.  Math. Soc. Colloq. Pub., vol 37, Providence,
1968.

\bibitem
{bJ72}
B. J\'{o}nsson,
\emph{Topics in universal algebra},
Springer-Verlag,
1972.

\bibitem
{kl65}
H. Kleisli,
\emph{Every standard construction is induced by a pair of adjoint
functors},
Proc. {A}mer. {M}ath. {S}oc.,
\textbf{16}(1965),
pp. 544--546.

\bibitem
{agK67}
A. G. Kurosh,
\emph{Alg\`{e}bre g\'{e}n\'{e}rale},
Dunod,
1967.

\bibitem
{fL63}
F. W. Lawvere,
\emph{Functorial semantics of algebraic theories},
Dissertation. Columbia University,
1963.

\bibitem
{mL55}
M. Lazard,
\emph{Lois de groupes et analyseurs},
Ann. Sci. Ecole Norm. Sup.,
\textbf{72}(1955),
pp. 299--400.

\bibitem
{sM98}
S. Mac Lane,
\emph{Categories for the working mathematician. 2nd ed.},
Springer-Verlag,
1998.

\bibitem
{aim59}
A. I. Mal'cev,
\emph{Model correspondences}.
In \emph{A. I. Mal'cev. The metamathematics of algebraic
systems. Collected Papers: 1936--1967},
trans. and ed. by  B. Franklin Wells III,
North-Holland,
1971,
pp. 66--94.

\bibitem
{jmm65}
J.M. Maranda,
\emph{Formal categories},
Can. J. of Math.,
\textbf{17}(1965),
pp. 758--801.

\bibitem
{m76}
G. Matthiessen,
\emph{Theorie der {H}eterogenen {A}lgebren},
Mathematik-{A}rbeits-{P}apiere,
Nr. 3,
Universit\"{a}t, Bremen, 1976.

\bibitem
{mmt87}
R. McKenzie, G. McNulty and W. Taylor,
\emph{Algebras, lattices, varieties. Vol. I},
Wadsworth $\And$ Brooks/Cole Mathematics Series,
1987.

\bibitem
{m89}
J. Meseguer,
\emph{General logics}.
In H.-D. Ebbinghaus et alii. eds. \emph{Logic Colloquium'87},
North-Holland,
1989,
pp. 275--329.


\bibitem
{mg85}
J. Meseguer and J. Goguen,
\emph{Initiality, induction and computability.}
In M. Nivat and J. Reynolds, eds.,
\emph{Algebraic Methods in Semantics},
Cambridge University Press,
1985,
pp. 459--541.

\bibitem
{dM69}
D. Monk,
\emph{Introduction to set theory},
McGraw-Hill, Inc.,
1969.

\bibitem
{pP71}
P.~H. Palmquist.
\emph{The double category of adjoint squares}.
In J.W. Gray and S. Mac Lane, eds.,
\emph{Reports of the Midwest Category Seminar V},
Springer-Verlag,
1971,
pp. 123--153.

\bibitem
{jP65}
J. Porte,
\emph{Recherches sur la th{\'{e}}orie g{\'{e}}n{\'{e}}rale des
syst{\`{e}}mes formels et sur les syst{\`{e}}mes connectifs},
Gauthier-Villars,
1965.

\bibitem
{eP21}
E. Post.
\emph{Introduction to a general theory of elementary propositions},
Amer. J. Math.,
\textbf{43} (1921)
pp. 163--185.

\bibitem
{eP41}
E. Post,
\emph{The two-valued iterative systems of mathematical logic},
Princeton Univ. Press,
1941.

\bibitem
{jSch61}
J. Schmidt,
\emph{Algebraic operations and algebraic independence in algebras
      with infinitary operations},
Math. Japon.,
\textbf{6} (1961/62)
pp. 77--112.

\bibitem
{jS62}
J. Sonner,
\emph{On the formal definition of categories},
Math. Zeitschr.,
\textbf{80} (1962)
pp. 163--176.

\bibitem
{mst36}
M. H. Stone,
\emph{The theory of representations for Boolean algebras},
Trans. Amer. Math. Soc.,
\textbf{40} (1936)
pp. 37--111.

\bibitem
{st72}
R. Street,
\emph{The formal theory of monads},
Journal of Pure and Applied Algebra,
\textbf{2} (1972)
pp. 149--168.

\bibitem
{tbg91}
A. Tarlecki, R. Burstall and J. Goguen,
\emph{Some fundamental algebraic tools for the semantics
              of computation: Part 3. Indexed categories},
Theoretical Computer Science,
\textbf{91} (1991)
pp. 239--264.

\bibitem
{tv57}
A. Tarski and R.L. Vaught,
\emph{Arithmetical extensions of relational systems},
Compositio Math.,
\textbf{13} (1957)
pp. 81--102.

\bibitem
{aT38}
A. Tarski,
\emph{\"{U}ber unerreichbare Kardinalzhalen},
Fundamenta Mathematicae,
\textbf{30} (1938)
pp. 68--89.



\end{thebibliography}
\end{document}